\documentclass[12pt]{article}
\RequirePackage[OT1]{fontenc}
\RequirePackage{amsthm,amsmath,amssymb,subcaption,graphicx,epstopdf,enumitem,lmodern}
\RequirePackage[round,colon,authoryear]{natbib}
\RequirePackage[colorlinks,citecolor=blue,urlcolor=blue]{hyperref}
\usepackage{bm,color,fancyvrb,xcolor,mathtools}
\usepackage{amscd}
\usepackage{mathrsfs}
\usepackage{bbm}
\usepackage{accents}
\def\utilde{\undertilde}

\def\eps{\varepsilon}
\def\epsi{\epsilon}
\font\tencmmib=cmmib10 \skewchar\tencmmib '60
\newfam\cmmibfam
\textfont\cmmibfam=\tencmmib

\def\lessim{\ \lower4pt\hbox{$
		\buildrel{\displaystyle <}\over\sim$}\ }
\def\gessim{\ \lower4pt\hbox{$\buildrel{\displaystyle >}
		\over\sim$}\ }

\usepackage{algorithm,algorithmicx}
\usepackage[noend]{algpseudocode}
\usepackage{tensor}

\theoremstyle{plain}

\newtheorem{theorem}{Theorem}[section]

\newtheorem{lemma}{Lemma}

\newtheorem{corollary}[theorem]{Corollary}

\DeclarePairedDelimiter\floor{\lfloor}{\rfloor}
\DeclarePairedDelimiter{\norm}{\lVert}{\rVert}
\DeclarePairedDelimiter{\abs}{\lvert}{\rvert}

\providecommand{\abs}[1]{\left\lvert#1\right\rvert}
\providecommand{\norm}[1]{\left\lVert#1\right\rVert}
\renewcommand{\hat}{\widehat}
\renewcommand{\tilde}{\widetilde}

\renewcommand{\hat}{\widehat}

\newcommand{\bfm}[1]{\ensuremath{\mathbf{#1}}}

\newcommand\numberthis{\addtocounter{equation}{1}\tag{\theequation}}

\def\ba{\bfm a}   \def\bA{\bfm A}  
\def\bb{\bfm b}   \def\bB{\bfm B}

   \def\bD{\bfm D}  
\def\be{\bfm e}   \def\bE{\bfm E}  \def\EE{\mathbb{E}}
    
   \def\bG{\bfm G}  
     
     \def\II{\mathbb{I}}

   \def\bM{\bfm M}  
   \def\bN{\bfm N}  \def\NN{\mathbb{N}}
\def\bo{\bfm o}   \def\bO{\bfm O}  
\def\bp{\bfm p}   \def\bP{\bfm P}  \def\PP{\mathbb{P}}
\def\bq{\bfm q}   \def\bQ{\bfm Q}  
     \def\RR{\mathbb{R}}
   \def\bS{\bfm S}  \def\SS{\mathbb{S}}
     
\def\bu{\bfm u}   \def\bU{\bfm U}  
\def\bv{\bfm v}   \def\bV{\bfm V}  
\def\bw{\bfm w}     
\def\bx{\bfm x}   \def\bX{\bfm X}  
   \def\bY{\bfm Y}  
\def\bz{\bfm z}   \def\bZ{\bfm Z}  

\def\calA{{\cal  A}}

\def\calL{{\cal  L}} 
\def\calM{{\cal  M}} 
\def\calN{{\cal  N}} 
\def\calO{{\cal  O}} 
\def\calP{{\cal  P}} 
\def\calQ{{\cal  Q}} 
 
\def\calS{{\cal  S}}

\newcommand{\bfsym}[1]{\ensuremath{\boldsymbol{#1}}}

\def\balpha{\bfsym \alpha}

             \def\bGamma{\bfsym \Gamma}

\def\btheta{\bfsym {\theta}}           
          
             \def\bSigma{\bfsym \Sigma}

\DeclareMathOperator{\argmin}{argmin}

\DeclareMathOperator{\Cov}{Cov}

\DeclareMathOperator{\E}{E}

\DeclareMathOperator{\Var}{Var}

\DeclareMathOperator{\argmax}{argmax}

\def\eps{\varepsilon}

\newcommand{\vertiii}[1]{{\left\vert\kern-0.25ex\left\vert\kern-0.25ex\left\vert #1 
		\right\vert\kern-0.25ex\right\vert\kern-0.25ex\right\vert}}

\usepackage{tikz}
\usetikzlibrary{decorations.pathreplacing,calc}

\def\scrT{\mathscr{T}} 
\def\scrX{\mathscr{X}}

\def\scrE{\mathscr{E}}
\def\scrA{\mathscr{A}}
\def\scrM{\mathscr{M}}

\newcommand{\COMMENT}[2][.5\linewidth]{%
	\leavevmode\hfill\makebox[#1][l]{//~#2}}

\begin{document}

	\title{Large Dimensional Independent Component Analysis:\\ Statistical Optimality and Computational Tractability$^\ast$}
	\author{Arnab Auddy and Ming Yuan\\
		Department of Statistics\\
		Columbia University}
	\date{(\today)}
	
	\maketitle
	
\begin{abstract}
In this paper, we investigate the optimal statistical performance and the impact of computational constraints for independent component analysis (ICA). Our goal is twofold. On the one hand, we characterize the precise role of dimensionality on sample complexity and statistical accuracy, and how computational consideration may affect them. In particular, we show that the optimal sample complexity is linear in dimensionality, and interestingly, the commonly used sample kurtosis-based approaches are necessarily suboptimal. However, the optimal sample complexity becomes quadratic, up to a logarithmic factor, in the dimension if we restrict ourselves to estimates that can be computed with low-degree polynomial algorithms. On the other hand, we develop computationally tractable estimates that attain both the optimal sample complexity and minimax optimal rates of convergence. We study the asymptotic properties of the proposed estimates and establish their asymptotic normality that can be readily used for statistical inferences. Our method is fairly easy to implement and numerical experiments are presented to further demonstrate its practical merits.
\end{abstract}
	\footnotetext[1]{
	This research was supported by NSF Grants DMS-2015285 and DMS-2052955.}
	
	\noindent
	
\section{Introduction}	
Independent component analysis (ICA) is a powerful and general data analysis tool. It was initially introduced as a blind source separation technique \citep{jutten1991blind,comon1994independent} and has since found tremendous success in applications from numerous scientific and engineering fields such as neuroimaging, cognitive science, signal processing, and machine learning, just to name a few. The basic premise of ICA is that the coordinates of a random vector $\bX\in \RR^d$ can be expressed as linear combinations of $d$ independent latent variables $S_1,\ldots, S_d$ such that $\bX=\bA\bS$ where $\bS=(S_1,\ldots, S_d)^\top$ and $\bA\in \RR^{d\times d}$ is a deterministic but unknown matrix, referred to as the mixing matrix. ICA can be viewed as a refinement of the principal component analysis (PCA) which allows us to write each coordinate of $\bX$ as linear combinations of $d$ uncorrelated random variables. While there are infinitely many ways to express $\bX$ as linear transformations of uncorrelated variables, remarkably, there is essentially only one way to write it as linear transformations of independent random variables if there is at most one Gaussian source among $S_1,\ldots, S_d$, i.e., the mixing matrix $\bA$ is identifiable up to rescaling and signed permutation of its columns. See, e.g. \cite{comon1994independent,eriksson2004identifiability}.
	
The goal of ICA is to recover from a sample, consisting of $n$ independent copies $\bX_1,\ldots, \bX_n$ of $\bX$, the independent components $\bS_1=\bA^{-1}\bX_1, \ldots, \bS_n=\bA^{-1}\bX_n$, or equivalently to estimate the mixing matrix $\bA$. Over the years, numerous methods and algorithms have been developed for this purpose. Notable examples include \cite{delfosse1995adaptive, pham1996blind, pham1997blind, amari1997blind, cardoso1999high, hyvarinen1999fast, lee1999independent, hastie2002independent, eriksson2003characteristic},  \cite{samarov2004nonparametric, chen2005consistent, chen2006efficient, ilmonen2011semiparametrically, samworth2012independent, belkin2018eigenvectors} among numerous others. See, e.g., \cite{cardoso1996independent, hyvarinen2000independent, roberts2001independent, stone2004independent,comon2010handbook, nordhausen2018independent} for overviews and surveys of a wide-ranging list of algorithms and applications of ICA. Most of these earlier developments for ICA have focused on the classical fixed dimension paradigm when the number of sources is fixed and the sample size diverges. Both the statistical and algorithmic aspects of ICA in such a setting are by now well-understood.

There are, however, an increasing amount of empirical evidence that these theories cannot fully capture the complexity, both statistically and computationally, of the problem or explain the lackluster performance of these methods often observed in modern applications. This predicament can be largely attributed to the effect of dimension $d$ which is largely unaccounted for and has inspired a flurry of recent research to better understand its role in ICA. See, e.g., \cite{belkin2013blind, anand2014tensor, anand2014sample, bhaskara2014uniqueness, goyal2014fourier, vempala2015max, voss2015pseudo, belkin2018eigenvectors} among others. These works typically take an algorithmic viewpoint and focus on developing estimation procedures with both sample and computational complexities that are polynomials of $d$. Yet it remains unclear what the best possible sample complexity is, and how considerations of computational tractability may affect it. Even less is known about the statistical properties of ICA when $d$ is not assumed fixed. The goal of the current work is to specifically address these fundamental questions, and develop estimating procedures for ICA that are both statistically efficient and computationally tractable.

In particular, we derive information theoretical limits for ICA by establishing the minimax optimal rates for estimating the mixing matrix $\bA$. The minimax optimal rates of convergence can be equivalently expressed in terms of the optimal sample complexity: a sample size $n\gtrsim d/\varepsilon^2$ is necessary and sufficient to ensure the existence of estimates of the unmixing directions with an error up to $\varepsilon$. To prove these bounds, we show that no estimates can converge at faster rates and also introduce an estimate that can attain these rates. The minimax optimal estimation technique is based on two crucial observations. The first is that, if $\bX$ is centered and pre-whitened so that its covariance matrix is the identity matrix, then the column vectors of $\bA$ can be identified with linear transforms of $\bX$ that maximize the kurtosis. Interestingly, we show that while maximizing sample kurtosis leads to consistent estimates of column vectors of the mixing matrix when $n\gg d^2$, it is necessarily inconsistent when $n\ll d^2$ and thus cannot be minimax optimal at least in the regime $d\ll n\ll d^2$. To overcome this difficulty, we appeal to another key insight that, for any given $\bu\in \SS^{d-1}$, we can derive estimates of the kurtosis of $\bu^\top\bX$ with improved tail behavior than the sample kurtosis. This builds upon the sub-gaussian mean estimation framework that has recently received significant attention; see, e.g., \cite{catoni2012challenging, minsker2015geometric,hsu2016loss, lugosi2016risk, minsker2018sub, ke2019user}. These estimates of kurtosis can be ensembled to allow for better control of the estimation error of kurtosis uniformly over all linear transformations, and subsequently, lead to minimax optimal estimate of the mixing matrix.
	
The minimax optimal estimates we introduced however incur a prohibitive computational cost and are not feasible for problems of large or even moderate dimensions. A natural question is whether or not a computationally tractable estimating procedure can be minimax optimal. More specifically, we focus our attention on a large class of computationally tractable algorithms called low-degree polynomial algorithms. See, e.g.,  \cite{hopkins2018statistical, kunisky2022notes, schramm2022computational}. We show that no consistent estimate of the mixing matrix can be computed using such algorithms if $n\lesssim d^2$, up to a logarithmic factor, which leads to a lower bound of $n\gtrsim \max\{d^2, d/\varepsilon^2\}$ (up to a logarithmic factor) for optimal sample size over this class of algorithms. This sample size requirement is indeed (nearly) optimal as we also develop a computationally tractable estimating procedure that attains this sample complexity.

Our estimate is motivated by a close inspection of the popular FastICA algorithm and its pitfalls when $d$ is not small or fixed. FastICA can be viewed as a fixed-point algorithm for maximizing sample kurtosis. Because of the inherent nonconvexity of the underlying optimization problem, it is well known that the performance of FastICA is sensitive to its initialization. See, e.g., \cite{zarzoso2006fast}. And this is especially problematic in moderate or large dimensions when random initialization can only yield a nontrivial estimate of the mixing direction with exponentially (in $d$) many tries. To address this challenge, we propose a new initialization strategy that is guaranteed to produce a nontrivial initial estimate whenever $n\gtrsim d^2$. The proposed initialization procedure is prompted by a careful inspection of how random slicing behaves and is based on improved moment estimators. With such an initialization, we show that the fixed-point iteration and deflation steps of FastICA are in fact effective in high dimensions, yielding an estimate that is both computationally tractable and statistically optimal. In addition to the convergence rates, we also establish the distributional properties of our estimate. Specifically, we provide normal approximations to estimates of linear and bilinear forms of the mixing matrix. These results can be readily used to construct confidence intervals or conduct hypothesis testing. The close connection with the FastICA also means that our estimating procedure can be easily implemented and the practical merits of the proposed method are further demonstrated by numerical, both simulated and real data, examples.

Our work is closely related to a fast-growing literature on using tensor methods in statistics and machine learning. See, e.g., \cite{kolda2009tensor,anandkumar2014tensor,cichocki2015tensor,sidiropoulos2017tensor}. In particular, we leverage the fact that unmixing directions can be identified with the eigenvectors of a cumulant tensor and the estimation error of the mixing matrix can therefore be viewed as the perturbation effect of the sample cumulant tensor. There is a fruitful line of research in developing algorithm-dependent bounds for general perturbed low-rank tensors. See, e.g., \cite{anand2014tensor, mu2015successive,mu2017greedy, belkin2018eigenvectors, zhang2018tensor, auddy2021estimating,auddy2020perturbation}. But an observation key to our development is that the estimation error of the sample cumulant tensor has unique structures and appropriately leveraging them can lead to more efficient algorithms and sharper statistical performance. In particular, our results suggest that the (un)mixing directions can be estimated at a much faster rate (i.e., $\sqrt{d/n}$) than the sample kurtosis tensor (i.e., $\sqrt{d^2/n}$). Also as a byproduct of our analysis, we show that if the signals have nonvanishing excess kurtosis, then there is virtually no benefit in considering higher-order moments in ICA.

The rest of this paper is organized as follows. In the next section, we study the information theoretical limits of ICA as well as the limits for estimates that can be computed using low-degree polynomial algorithms. Section~\ref{sec:tract} introduces our proposed ICA method and proves that it is both polynomial-time computable and minimax optimal when $n\gtrsim d^2$. Section~\ref{sec:asy-dist} presents the asymptotic properties of the proposed estimate. To complement the theoretical and methodological developments, numerical experiments are given in Section~\ref{sec:simul} to demonstrate the practical merits of our approach. We conclude with a brief summary in Section~\ref{sec:conclude}. All proofs are relegated to the Appendix for space considerations.

\section{Information Theoretical and Computational Limits}\label{sec:inf-th}

In this section, we establish the information theoretical limits for ICA and develop an approach that achieves them. This is followed up by a study of the impact of computational constraints on the optimal sample complexity.

\subsection{Setup and Minimax Lower Bounds}

Without loss of generality, we shall assume in the rest of the paper that each component of source $\bS$ is centered and has a unit variance. In addition, we shall assume that each signal $S_j$ has finite $(8+\epsilon)$th moment for some $\epsilon>0$, and their (excess) kurtosis $\kappa_4(S_j):=\EE(S_j^4)-3$ is bounded away from zero and infinity. Note that all our results can be straightforwardly generalized to the case when one source has zero excess kurtosis, e.g., when it is Gaussian. This is because its corresponding (un)mixing direction can be determined by virtue of all the other directions. For brevity, we shall not delve into such a special case in what follows. Denote by $\calP_{\rm ICA}(\bA; \epsilon, M_1, M_2)$ the collection of all probability measures of $\bX\in {\mathbb R}^d$ that follows an ICA model with mixing matrix $\bA$ and latent signal $\bS$ obeying these assumptions, that is,
\begin{eqnarray*}
\calP_{\rm ICA}(\bA; \epsilon, M_1, M_2)=\{\calL(\bX): \bX=\bA\bS,\, M_1^{-1}\le |\kappa_4(S_j)|\le M_1,\, \EE(|S_j|^{8+\epsilon})\le M_2,\\
 \EE S_j=0,\, \EE S_j^2=1,\, S_j{\rm s\ are \ independent}\},
\end{eqnarray*}
where $\calL(\bX)$ is the law of the random variable $\bX$.

It is also customary in the ICA literature to pre-whiten the data so that one can assume that $\texttt{cov}(\bX)=\II_d$. To see the rationale behind pre-whitening and the simplification it brings about, consider the case when $\texttt{cov}(\bX)=\Sigma$ is known. Let $\tilde{\bX}=\Sigma^{-1/2}\bX$. It is clear that $\tilde{\bX}$ has an identity covariance matrix and also follows an ICA model with mixing matrix $\bB:=\Sigma^{-1/2}\bA$. In other words, there is no loss of generality to assume that $\texttt{cov}(\bX)=\II_d$ if the covariance matrix of $\bX$ is known in advance. Note that both the new mixing matrix $\bB$ and its corresponding unmixing matrix ($\bB^{-1}=\bB^\top$) belong to the set $\calO(d)$ of orthonormal $d \times d$ matrices. Note that there is no loss of generality to adopt this simplifying assumption when discussing information theoretical or computational lower bounds since a lower bound for the special case (assuming $\bA\in \calO(d)$) is necessarily a lower bound for the more general case. How to appropriately pre-whiten, however, is an important practical aspect for ICA. To fix ideas, we shall nonetheless follow the convention and assume that $\bA\in \calO(d)$ with the exception of Section \ref{sec:prewhiten} where we  specifically address the issue of prewhitening and its implications.

Note that any signed permutation of the columns of a mixing matrix leads to essentially the same ICA representation, and one cannot differentiate between an orthonormal matrix $\bA$ and its signed permutation based on $\bX$. To explicitly account for such ambiguity in measuring the quality of an estimate of $\bA$, the following loss function is commonly used:
$$
\ell_M(\bA,\bB)=\min_{\pi:[d]\to [d]}\max_{1\le k\le d} \sin\angle(\ba_k, \bb_{\pi(k)})
$$
where the minimum is taken over all permutations, $\angle (\bu,\bv)$ is the angle between two vectors $\bu$ and $\bv$ taking value in $[0,\pi/2]$, and $\ba_k$ and $\bb_k$ are the $k$th column vectors of $\bA$ and $\bB$ respectively. In the context of ICA, $\ell_M$ measures how poorly each ``unmixing directions'' $\ba_j$ is estimated since $S_j=\ba_j^\top \bX$, $j=1,2,\ldots,d$. Alternatively, we can also look at the quality of an estimate in an averaged sense:
$$
\ell_A(\bA,\bB)=\min_{\pi}\left({1\over d}\sum_{1\le k\le d} \sin^2\angle(\ba_k, \bb_{\pi(k)})\right)^{1/2}.
$$
Both loss functions quantify the overall accuracy of the estimated ``unmixing directions''. Our first result gives a minimax lower bound for estimating $\bA$:

\begin{theorem}
\label{th:ICAlower}
For any $\epsilon, M_1, M_2>0$, there exists a constant $c>0$ such that 
$$
\inf_{\tilde{\bA}}\sup_{\substack{\calL(\bX)\in \calP_{\rm ICA}(\bA; \epsilon, M_1, M_2)\\ \bA\in \calO(d)}} 
\EE[\ell(\tilde{\bA}, \bA)]\ge c\cdot \min\left\{\sqrt{d\over n},1\right\},
$$
where the loss function $\ell$ is either $\ell_M$ or $\ell_A$ defined above and the infimum is taken over all estimators of $\bA$ based on observing $n$ independent copies $\bX_1,\ldots,\bX_n$ of $\bX$.
\end{theorem}

It is worth noting that the lower bound given in Theorem \ref{th:ICAlower} holds for any $\epsilon>0$. In other words, one cannot hope to improve it with a higher moment (larger $\epsilon$) requirement for the sources. In fact, what we proved is slightly stronger than stated here: the lower bound holds even if all source variables are subgaussian.

The above minimax lower bound can be equivalently expressed in terms of the sample complexity commonly used in the literature. Specifically, to ensure that there exists an estimate $\tilde{\bA}$ satisfying $\ell(\tilde{\bA},\bA)\le_p \varepsilon$, we need a sample size $n\gtrsim d/\varepsilon^2$. This sample complexity is much smaller than those established for the state-of-the-art ICA techniques. See, e.g., \cite{anand2014sample, belkin2018eigenvectors}. Naturally, one asks whether this sample size requirement is indeed attainable. The answer is affirmative, if doubtful at the first glance, and we shall introduce an estimate that attains the optimal sample complexity later in the section. Our approach is motivated by a close inspection of the connection between ICA and moment estimation underlying many existing ICA techniques. Our investigation starts by revealing a fundamental yet subtle limitation of how this connection is commonly exploited, and hence why the earlier approaches are necessarily suboptimal.

\subsection{ICA and Moment Estimation}

Many popular ICA methods are based upon the crucial observation that, if $\bX$ is centered and pre-whitened so that its covariance matrix is the identity matrix, then the column vectors, $\ba_1,\ldots, \ba_d$, of $\bA$ can be identified with linear transforms of $\bX$ that optimize the kurtosis. A useful way to understand this property is through the decomposition of the fourth-order cumulant tenor of $\bX$.

Recall that $\bX=\bA\bS$ with $\bA\in \calO(d)$. Since $\{S_j:1\le j\le d\}$ are independent random variables, the characteristic function of $\bX\in \RR^d$ can be written as
$$
\phi(\bv)=\EE\exp(i\bv^{\top}\bX)=\prod_{j=1}^d\EE\exp(i(\bv^{\top}\ba_j)S_j)
=\prod_{j=1}^d\phi_j(\bv^{\top}\ba_j)
$$
for any $\bv\in \RR^d$. Here $\ba_j$ is the $j$th column of $\bA$ and $\phi_j(\cdot)$ is the characteristic function of $S_j$. This implies that the cumulant generating function of $\bX$ can be expressed as
\begin{equation}\label{eq:log-cf}
	\psi(\bv):=\log \phi(\bv)=\sum_{j=1}^d\log(\phi_j(\bv^{\top}\ba_j))
	=\sum_{j=1}^d\psi_j(\bv^{\top}\ba_j),
\end{equation}
where $\psi_j(\cdot)=\log(\phi_j(\cdot))$ is the cumulant generating function of $S_j$. Differentiating both sides of the equation \eqref{eq:log-cf}, four times with respect to $\bv$, one has
$$
\kappa_4(\bv^\top\bX)
=\sum_{j=1}^d\kappa_4(S_j)(\langle\ba_j,\bv\rangle)^4=(\scrM_4(\bX)-\scrM_0)\times_1\bv\times_2\bv\times_3\bv\times_4\bv,
$$
where
$$\scrM_4(\bX)=\EE (\bX\circ\bX\circ\bX\circ \bX)$$
and 
\begin{equation}\label{eq:defM0}
	\scrM_0=\sum_{\{i,j,k,l\}=\{i_1,i_2\}}\sum_{i_1,i_2=1}^d \be_i\circ\be_j\circ\be_k\circ\be_l.
\end{equation}
Hereafter $\circ$ represents the outer product, i.e., the $(i,j,k,l)$ entry of $\scrM_4(\bX)$ is $[\scrM_4(\bX)]_{ijkl}=\EE(X_iX_jX_kX_l)$, $\be_i$ is the $i$th canonical basis of $\RR^d$, and $\times_j$ means multiplication between a tensor along the $j$th mode and a vector of conformable dimension.

The expression above holds for all $\bv\in \RR^d$, and hence, $\scrM_4(\bX)-\scrM_0$, the fourth cumulant tensor of $\bX$, is orthogonally decomposable (ODECO):
\begin{equation}\label{eq:odec}
	\scrM_4(\bX)-\scrM_0=\sum_{j=1}^d\kappa_4(S_j)\ba_j\circ\ba_j\circ\ba_j\circ\ba_j.
\end{equation}
We shall refer to $\ba_j$s and $\kappa_4(S_j)$s as the eigenvectors and eigenvalues of $\scrM_4(\bX)-\scrM_0$. As such, the $\ba_j$s are the complete enumeration of the local optima of $f(\bu):=\kappa_4(\bu^\top \bX)$ with respect to the domain ${\mathbb S}^{d-1}$. This, in particular, implies that random initialization followed by gradient iterations can recover $\ba_j$s with probability one. See, e.g., \cite{belkin2018eigenvectors} for detailed discussions. Inspired by this observation, many approaches to ICA aim at finding the local optima of the sample kurtosis:
$$
\hat{f}(\bu):=\hat{\kappa}^{\rm sample}_4(\bu^\top \bX)=(\hat{\scrM}^{\rm sample}_4(\bX)-\scrM_0)\times_1\bu\times_2\bu\times_3\bu\times_4\bu
$$
where
$$\hat{\scrM}^{\rm sample}_4(\bX)={1\over n}\sum_{i=1}^n (\bX_i\circ\bX_i\circ\bX_i\circ \bX_i),$$
is the fourth-order sample moment tensor. For example, the popular FastICA can be viewed as a fixed-point algorithm to optimize $\hat{f}$ over ${\mathbb S}^{d-1}$.

It is important to note that, even though $f$ only has $d$ local optima, $\hat{f}$ may have many more. The success of any algorithm hinges upon its ability to find the local optima of $\hat{f}$ that are close to those of $f$. Clearly, regardless of which algorithm to use, how well this strategy works depends on how well the sample kurtosis approximates its population counterpart. Perhaps, somewhat surprisingly, the following result indicates that, no matter which algorithm is used to optimize $\hat{f}$, the resulting estimating procedure cannot achieve the optimal sample complexity given by Theorem \ref{th:ICAlower}.

\begin{theorem}\label{th:samp-lbd}
Suppose that $\bX_1,\ldots, \bX_n$ are $n$ independent copies of a random vector $\bX$ such that $\calL(\bX)\in \calP_{\rm ICA}(\bA;\epsilon, M_1, M_2)$ for some $\bA\in \calO(d)$ and constants $\epsilon, M_1,M_2>0$. There exist constants $C_1,C_2>0$ such that if $n\le C_1d^2$ then,
$$
\sup_{\bu\in \SS^{d-1}}|\hat{\kappa}_4^{\rm sample}(\bu^\top \bX)-\kappa_4(\bu^\top \bX)|\ge 4
$$
with probability at least $1-d^{-4}$. Moreover,
$$
\max_{1\le j\le d}|\langle\ba_j,\hat{\bu}\rangle|\le C_2\sqrt{d^{-1/4}\log d},
$$
with probability at least $1-n^{-\epsilon/8}$, where
$$
\hat{\bu}\in\argmax_{\bu:\|\bu\|=1}\hat{f}(\bu).
$$
\end{theorem}

The first statement of Theorem \ref{th:samp-lbd} suggests that the sample kurtosis does not approximate the true kurtosis well for some $\bu\in \SS^{d-1}$ and hence maximizing the sample kurtosis may not lead to consistent estimates of the maximizer of the population kurtosis. Indeed, the second statement states more precisely how maximizing the sample kurtosis fails to recover any of the unmixing directions when $n\ll d^2$: even if we can overcome the nonconvexity of $\hat{f}$ and optimize it, the optimizer is not a good estimate. More specifically, when $d$ diverges,
$$
\max_{1\le j\le d}|\langle\ba_j,\hat{\bu}\rangle|\to_p 0.
$$
In other words, $\hat{\bu}$ behaves like the worst possible estimate of $\ba_j$s as it is asymptotically orthogonal to all of them!	It is worth noting that this peculiar phenomenon is unique in a high dimensional situation -- although $\{\ba_1, \ldots, \ba_d\}$ is a complete basis for $\RR^d$, the vector $\hat{\bu}$ can still be nearly perpendicular to all of them.

Theorem \ref{th:samp-lbd} indicates that any ICA algorithm aiming to optimize the sample kurtosis, if performing as desired, cannot be consistent when $n\ll d^2$, and thus is not minimax optimal at least in the regime $d\ll n\ll d^2$. This limitation is due to the inefficiencies of sample kurtosis and thankfully can be overcome with an improved moment estimate. As a result, we shall argue, the lower bound in Theorem \ref{th:ICAlower} is indeed attainable.

\subsection{Minimax Optimal ICA}
The main idea is to first construct a more reliable estimate of $\theta_\bu:=\EE(\bu^\top \bX)^4$ separately for a carefully chosen set of $\bu$s, and then use them as anchors to derive an improved estimate of $\theta_\bu$ for all $\bu\in\SS^{d-1}$. We now describe the procedure in detail.

\paragraph{Estimating $\theta_{\bu}$ for a fixed $\bu$.} We first discuss how to estimate $\theta_{\bu}$ for a fixed $\bu\in\SS^{d-1}$. Note that this amounts to estimating the mean of $(\bu^\top\bX)^4$ and we shall appeal to a general strategy developed by \cite{catoni2012challenging}. More specifically, with a slight abuse of notation, let $\psi: \RR\mapsto\RR$ be a non-decreasing influence function such that
$$
-\log(1-x+x^2/2)\le \psi(x)\le\log(1+x+x^2/2).
$$
Denote by $\psi_\alpha(x)=\psi(\alpha x)$. See \cite{catoni2012challenging} for more concrete examples of the influence function $\psi$. For any $\bu\in \SS^{d-1}$, denote by $\hat{\theta}_{\bu}$ the solution to
\begin{equation}\label{eq:def-hat-theta-u}
	\sum_{i=1}^n \psi_\alpha \left\{(\bu^\top \bX_i)^4-{\theta}_{\bu}\right\}=0.
\end{equation}
Under the moment assumptions for the sources, the results from \cite{catoni2012challenging} imply that, for any $t>\exp(-n/2)$,
\begin{equation}\label{eq:catoni}
	\left|\hat{\theta}_{\bu}-\EE(\bu^\top \bX_i)^4\right|\le \sqrt{2M_2^{8/(8+\epsilon)}\log(t^{-1})\over n\left(1-2\log(t^{-1})/n\right)},
\end{equation}
with probability at least $1-2t$. In other words, for any $\bu\in \SS^{d-1}$, $\hat{\theta}_{\bu}$ has subgaussian tails.

\paragraph{Estimating $\theta_{\bu}$ for all $\bu$s.} 

Now let $\calN$ be a $1/4$ covering set of $\SS^{d-1}$ with $|\calN|\le 9^d$. We first estimate $\theta_\bu$ for every $\bu\in \calN$ as above and then use them as anchors to construct estimates for $\theta_{\bu}$ for any $\bu\in \SS^{d-1}$. Recall that $\theta_{\bu}=\langle\scrM_4(\bX),\bu\circ\cdots\circ\bu\rangle$. Estimating $\theta_{\bu}$, therefore, amounts to estimating the moment tensor $\scrM_4(\bX)$. To this end, we shall take advantage of the key property that we noticed in the previous subsection: $\scrM_4(\bX)-\scrM_0$ is ODECO. Specifically, we shall estimate $\scrM_4(\bX)$ by
\begin{equation}
	\label{eq:defAhatica1}
	\hat{\scrM}_4(\bX):=\argmin_{\scrA: \scrA-\scrM_0\text{ is ODECO}} \max_{\bu\in \calN}\left|\hat{\theta}_{\bu}-\langle \scrA,\bu\circ\bu\circ\bu\circ\bu\rangle\right|.
\end{equation}
%
%

\paragraph{Estimating the mixing matrix.} By definition, $\hat{\scrM}_4(\bX)-\scrM_0$ is an ODECO tensor and we can then estimate $\ba_j$s by its singular vectors. Denote by
\begin{equation}\label{eq:defuhat}
	\hat{\scrM}_4(\bX)-\scrM_0=\sum_{j=1}^d \hat{\lambda}_j \hat{\bu}_j\circ\hat{\bu}_j\circ\hat{\bu}_j\circ\hat{\bu}_j.
\end{equation}
where the singular vectors $\hat{\bu}_j$s are also the stationary points of
\begin{equation}\label{eq:khat}
\hat{\kappa}_4(\bu^\top \bX)=\langle \hat{\scrM}_4(\bX)-\scrM_0,\bu\circ\bu\circ\bu\circ\bu\rangle.
\end{equation}

\medskip
The following theorem shows that $\hat{\bu}_j$s are indeed minimax optimal estimates of the column vectors of $\bA$.

\begin{theorem}
\label{th:ICA-infth}
Suppose that $\bX_1,\ldots, \bX_n$ are $n$ independent copies of a random vector $\bX$ such that $\calL(\bX)\in\calP_{\rm ICA}(\bA;\epsilon, M_1, M_2)$ for some $\bA\in \calO(d)$ and $\epsilon, M_1,M_2>0$. There exist constants $C_1,C_2>0$ such that if $n>C_1d$, then with probability at least $1-\exp(-d)$,
$$
\sup_{\bu\in \SS^{d-1}}|\hat{\kappa}_4(\bu^\top \bX)-\kappa_4(\bu^\top \bX)|\le C_2\sqrt{d\over n},
$$
where $\hat{\kappa}_4$ is defined by \eqref{eq:khat}. And consequently, with probability at least $1-\exp(-d)$,
$$
\ell(\hat{\bU}, \bA)\le C_2\sqrt{d\over n}.
$$
where the loss function $\ell$ is either $\ell_M$ or $\ell_A$, $\hat{\bU}=[\hat{\bu}_1,\ldots, \hat{\bu}_d]$, and $\hat{\bu}_j$s are defined by \eqref{eq:defuhat}.
\end{theorem}

Although the estimate we introduced above is minimax optimal, it also incurs prohibitive computational costs when $d$ is large. It not only requires estimating the kurtosis for an exponential (in $d$) number of $\bu$s, but also involves a highly nonconvex optimization problem of \eqref{eq:defuhat}. This naturally brings about the question: are there alternative approaches that are both minimax optimal and computationally tractable? The answer turns out to be mixed: we can but only with a more stringent sample size requirement. Specifically, we shall first argue that if we restrict our attention to a large class of computationally efficient algorithms, then the minimum sample size required for consistent estimation is at least of the order $d^2$, up to a logarithmic factor, rather than linear in $d$ as Theorem \ref{th:ICAlower} suggests. 

\subsection{Computational Limits for ICA}\label{sec:comp-low}

We now turn our attention to estimates that can be computed using a general class of polynomial-time (in $d$) algorithms. In particular, we make use of the recently developed low-degree polynomial bounds. The low-degree polynomial method originated from the sum-of-squares literature in theoretical computer science and has since been developed into a general framework to study the average-case computational complexity. See, e.g., \cite{hopkins2018statistical} or \cite{kunisky2022notes} for further details. Many popular algorithms such as power iteration or approximate message passing can be cast within this framework. It has been shown to provide a unified explanation of information-computation gaps in a number of important statistical problems including sparse PCA, tensor PCA, and planted submatrix among many others in the sense that the best-known polynomial-time computable algorithms are low-degree polynomials and conversely low-degree polynomial algorithms fail in the ``hard'' regime.

In the setting of ICA, a low-degree polynomial algorithm takes as input $\bX_1,\ldots, \bX_n$, and produces an output, an estimate of the mixing matrix, that can be expressed as a polynomial of the input of degree $O(\log d)$. Note that the dimension of a space of polynomials up to a certain degree grows exponentially with the degree so that the space of polynomials of degree $O(\log d)$ has dimension $d^{O(1)}$.
Our next result shows that if we restrict our attention to this class of algorithms, sample size $n\gg d^2/{\rm polylog}(d)$ is required to ensure consistency of the estimated mixing matrix.

\begin{theorem}\label{th:comp-lower}
Let $\RR_{\le D}[\bX_1,\ldots, \bX_n]$ be the set of all real polynomials of degree up to $D$ over independent copies $\bX_1,\ldots, \bX_n$ of $\bX$. Suppose that $D\le c_0\log d$ for some constant $c_0>0$. Then, for any $\epsilon, M_1, M_2>0$, there exist constants $c_1,c_2$ such that
$$
\inf_{\tilde{\bA}\in \RR_{\le D}[\bX_1,\ldots, \bX_n]}\sup_{\substack{\calL(\bX)\in \calP_{\rm ICA}(\bA; \epsilon, M_1, M_2)\\ \bA\in \calO(d)}} 
\EE[\ell_M(\tilde{\bA}, \bA)]\ge c_1,
$$
whenever $n\le c_2d^2(\log d)^{-10}$.
\end{theorem}

We outline here the main idea behind the proof of Theorem \ref{th:comp-lower}. As often done in establishing minimax lower bounds, the proof of Theorem \ref{th:comp-lower} proceeds by first reducing the problem of estimating $\bA$ to a related hypothesis testing problem about $\bA$. Specifically, we consider testing
$$
H_0:\bA=\bP\quad
\text{vs.}\quad
H_1:\bA=\bQ
$$
based on observing $n$ independent observations from the ICA model $\bX=\bA\bS$ where $\bP\neq\bQ\in \calO(d)$ are two fixed orthonormal matrices. The idea is that if we can not differentiate between $\bP$ and $\bQ$, then we cannot estimate $\bA$ up to an error of $\ell(\bP,\bQ)/2$.

Without the computational constraint, we can consistently differentiate between $\bP$ and $\bQ$ if the likelihood ratio diverges. Intuitively, if only low-degree polynomial algorithms are allowed, then the low-degree polynomial projection of the likelihood ratio must diverge. The projection of the likelihood ratio onto the space of low-degree polynomials is given by
\begin{equation}\label{eq:lrcs-low-deg}
	\Lambda_{\le D}:=\max_{f\in \RR_{\le D}[\bX_1,\ldots, \bX_n]}\dfrac{\EE_{\bA=\bQ} f(\bX_1,\ldots, \bX_n)}{\sqrt{\EE_{\bA=\bP} f^2(\bX_1,\ldots, \bX_n)}}.
\end{equation}
See, e.g., \cite{hopkins2018statistical}. The above optimization over all possible polynomials has a solution in terms of the ratio of the likelihood functions of $\bX$ under $H_0$ and $H_1$. Let $\mu$ and $\nu$ be the laws of $\bX=\bP\bS$ and $\bX=\bQ\bS$ respectively. We assume that both $\mu$ and $\nu$ are absolutely continuous with respect to the Lebesgue measure, and write the likelihood ratio under these two measures as
$$
L(\bX_1,\dots,\bX_n)=
\dfrac{L_{\nu}(\bX_1,\dots,\bX_n)}{L_{\mu}(\bX_1,\dots,\bX_n)}
=\prod_{i=1}^n
\prod_{k=1}^d
\dfrac{f_k(\bq_k^{\top}\bX_i)}{f_k(\bp_k^{\top}\bX_i)}
$$
where $f_k$ is the density of $S_k$. With $\underline{\bX}=(\bX_1,\dots,\bX_n)$, it can be shown that
\begin{equation}\label{eq:def-low-Lamb-0}
	\Lambda^2_{\le D}= \sum_{t=1}^N\langle L,b_t\rangle^2
	=\sum_{t=1}^N
	\left(\EE_{\mu^{\otimes n}}\left(L(\underline{\bX})b_t(\underline{\bX})\right)
	\right)^2
	=\sum_{t=1}^N
	\left(\EE_{\nu^{\otimes n}}\left(b_t(\underline{\bX})\right)
	\right)^2,
\end{equation}
where $\{b_t\}_{t=1}^N$ is a basis for polynomials of degree at most $D$, and the last equality follows by the definition of the likelihood ratio. By choosing an appropriate basis $\{b_t\}_{t=1}^N$, we can show that
\begin{lemma}\label{lem:comp-low} 
Let $\bX_1,\ldots, \bX_n$ be independent copies of $\bX=\bA\bS$. Suppose $D\le c_0\log d$ for some constant $c_0>0$. Then there is a constant $C>0$ such that for any $\bP\in\calO(d)$, there exists a $\bQ\in\calO(d)$ such that
\begin{equation}\label{eq:comp-low-2}
		\ell_M(\bP,\bQ)=\dfrac{1}{\sqrt{2}}
		\quad
		{\rm and}
		\quad 
		\EE\Lambda_{\le D} \le 2,
	\end{equation}
	whenever $n\le d^2/(CD^2(\log d)^{10})$. 
\end{lemma}

We can then leverage Lemma \ref{lem:comp-low} to prove Theorem \ref{th:comp-lower}. To this end, suppose we have an estimator $\hat{\ba}_1$ of the first column of $\bA$ that can be expressed as a polynomial of $\underline{\bX}$ with order up to $D$. For $\bP\in\calO(d)$, consider the low-degree polynomial
$$
g(\underline{\bX})=(\hat{\ba}_1^{\top}\bX_n)^4-(\bp_1^{\top}\bX_n)^4.
$$
Intuitively, if $\hat{\ba}_1$ is a good estimate, say, $\PP_{\bA}\left(|\langle\hat{\ba},\ba_1\rangle|>0.9\right)\ge 0.9$, then we shall be able to tell $\bP$ and $\bQ$ apart based on the value of $g(\underline{\bX})$, and consequently, ensure that the ratio ${\EE_{\bA=\bQ}g(\underline{\bX})}/{\sqrt{\EE_{\bA=\bP}g^2(\underline{\bX})}}$ is sufficiently large. This, however, contradicts Lemma \ref{lem:comp-low} which states that it cannot be larger than 2.

\section{Computationally Tractable ICA}\label{sec:tract}

In light of the discussion from the previous section, we shall focus on the case when $n\gtrsim d^2$ and develop a computationally tractable ICA approach that is minimax optimal in this regime. The main challenge for ICA when $d$ is large is the fact that $\hat{f}$ is highly nonconvex and can have many more than $d$ local optima most of which are not close to any of the $\ba_j$s. There are two general strategies to overcome this challenge. One is to incorporate occasional jumps to the iterations to make it possible to escape a bad local optimum. See, e.g., \cite{belkin2018eigenvectors}. The other is to ensure the initialization is sufficiently close to a good local optimum. In particular, our approach is closely related to and inspired by a careful inspection of the popular FastICA algorithm. To this end, we shall first briefly review the FastICA procedure and highlight its perils and challenges when $d$ is large.

\subsection{FastICA and Its Pitfalls when $d$ is Large}

As noted, FastICA is a fixed-point algorithm to optimize the sample kurtosis:
$$
\hat{f}(\bu)=\hat{\kappa}_4^{\rm sample}(\bu^\top \bX)={1\over n}\sum_{i=1}^n (\bu^\top\bX_i)^4-3,
$$
over $\bu\in {\mathbb S}^{d-1}$. The algorithm consists of three parts: initialization which is typically sampled uniformly from the sphere ${\mathbb S}^{d-1}$, fixed-point iterations either with a prespecified number of iterations or until a certain convergence criterion is met, and deflation where an estimated independent component is removed by projecting the observations onto its orthogonal complement. These steps are repeated until all independent components are recovered. See \cite{hyvarinen2000independent} for details.
	
	\begin{algorithm}[htbp]
		\caption{FastICA}\label{alg:alg-ica}
		\hspace*{\algorithmicindent} \textbf{Input:} $\bX_1,\ldots, \bX_n$
		\begin{algorithmic}[1]
			\For {$j=1$ to $d$}
			\State $\hat{\ba}_j^{[0]}\leftarrow$ 
			initialize 
			\COMMENT{Initialization}
			\For{$t=1$ to $T$} \COMMENT{Fixed-point iteration}
			\State $\hat{\ba}_j\leftarrow \hat{\ba}_j^{[t-1]}-{1\over 3n}
			\displaystyle
			\sum_{i=1}^n[\bX_i((\hat{\ba}_j^{[t-1]})^\top\bX_i)^3]$
			\State $\hat{\ba}_j^{[t]}\leftarrow {\hat{\ba}_j\over \|\hat{\ba}_j\|}$
			\EndFor
			\State $\bX_i\leftarrow \bX_i-((\hat{\ba}_j^{[t]})^\top\bX_i)\hat{\ba}_j^{[t]}, \quad i=1,\ldots, n$ \COMMENT{Deflation}
			\EndFor\\
			\Return $\{\hat{\ba}_j^{[T]}:1\le j\le d\}$.
		\end{algorithmic}
	\end{algorithm}

A good initialization is arguably the most critical to the success of the algorithm because $\hat{f}(\cdot)$ may have exponentially many local optima most of which are not necessarily close to any of the $\ba_j$s. Indeed, the fact that initialization plays an important role in FastICA is widely known in practice. When $d$ is small, this issue can be resolved by running the FastICA algorithm with multiple random initializations to ensure that we start from somewhere close to a mixing direction at least for one try. However, the difficulty is exacerbated with a larger $d$, and this simple strategy may no longer be practical. To see this, note that a uniformly sampled direction $\bu\in \SS^{d-1}$ satisfies
$$
\max_{1\le k\le d}|\langle \bu,\ba_k\rangle|=O_p(\sqrt{d^{-1}\log d}).
$$
For small $d$, this means that there exists a direction $\ba_k$ such that  $|\langle \bu,\ba_k\rangle|$ is bounded away from 0. But, when $d$ is large, this is no longer the case and $\bu$ will be nearly orthogonal to all $\ba_k$s! 

This problem, unfortunately, cannot be easily resolved with multiple rounds of random initialization. Consider repeatedly sampling $L$ directions $\bu_1,\ldots, \bu_L$. An application of the union bound yields
$$
\max_{1\le l\le L}\max_{1\le k\le d}|\langle \bu_l,\ba_k\rangle|=O_p(\sqrt{d^{-1}(\log d+\log L)}).
$$
Thus, it is going to take exponentially many, i.e., $\exp(\Omega(d))$, random initializations to ensure that at least one $\bu_l$ has a nonvanishing inner product with one of the mixing directions, which makes it computationally infeasible. 

Another common strategy of initialization for ICA is via suitable matricization of the cumulant tensor. See, e.g., \cite{hyvarinen2000independent}. Specifically, one collapses the first two modes and the last two modes of $\hat{\scrM}^{\rm sample}_4(\bX)-\scrM_0$ to form a $d^2\times d^2$ matrix  $\calM_{(1,2)(3,4)}(\hat{\scrM}^{\rm sample}_4(\bX)-\scrM_0)$. Let $\calM_{(1,2)(3,4)}(\hat{\scrM}^{\rm sample}_4(\bX)-\scrM_0)=UDV^\top$ be its singular value decomposition and denote by $\bu_1$ the first column vector of $U$. Next $\bu_1$ is reshaped into a $d\times d$ matrix, denoted by $U_1$. We set the initial value $\ba_1^{[0]}$ to be the leading left singular vector of $U_1$.  The rationale behind this approach is that $\scrM^{\rm sample}_4(\bX)-\scrM_0$ estimates
$$
\scrM_4(\bX)-\scrM_0=\sum_{k=1}^d \kappa_4(S_k) \ba_k\circ \ba_k\circ \ba_k\circ \ba_k.
$$
Note that $\ba_k\otimes\ba_k$s are the eigenvectors of $\calM_{(1,2)(3,4)}(\scrM_4(\bX)-\scrM_0)$:
\begin{equation}\label{eq:mat-cuml4}
	\calM_{(1,2)(3,4)}(\scrM_4(\bX)-\scrM_0)=\sum_{k=1}^d \kappa_4(S_k)(\ba_k\otimes \ba_k)\circ (\ba_k\otimes \ba_k),
\end{equation}
where $\otimes$ represents the Kronecker product. However, the eigenvectors of $\calM_{(1,2)(3,4)}(\scrM_4(\bX)-\scrM_0)$ are uniquely defined only if $\kappa_4(S_k)$s are distinct. Even if this is true, the validity of this approach will depend on the differences among $\kappa_4(S_k)$s relative to the estimation error $\hat{\scrM}^{\rm sample}_4(\bX)-\scrM_4(\bX)$. All of these impose unnecessary and sometimes unrealistic assumptions on the sources.

In addition to the initialization, when $d$ is more than a handful, one also needs to be concerned with the computational complexity of any ICA algorithm: how does it scale with an increasing dimensionality and how many iterations are required to ensure a certain precision even if we start with a ``good'' initial value? Of course, the situation gets even more involved as we sequentially recover the independent components. As each unmixing direction is estimated and removed through deflation, how their estimation error accumulates and affects the estimation of remaining directions is essential to understand if the algorithm works for moderate or large $d$. Also pertinent is the statistical performance of the final estimates: perhaps the most important distinction from earlier analysis of this type of iterative algorithm is that we prove that the unmixing directions can be estimated at the minimax optimal rate of convergence that is far superior to the estimated kurtosis.

We shall discuss how all these challenges can be addressed through careful analyses and propose an improved estimating procedure that is both computationally tractable and statistically efficient.

\subsection{Initialization}

The goal of initialization is to start with a nontrivial estimate $\hat{\ba}_k^{[0]}$. More specifically, we want to make sure that $\sin\angle (\hat{\ba}_k^{[0]},\ba_k)$ is bounded away from $1$. Two main ideas behind our proposal is random slicing and improved moment estimation which we shall now describe in detail. 

\subsubsection{Random Slicing}
We begin with random slicing. The idea of random slicing is commonly used for tensor decomposition. See, e.g., \cite{anand2014sample}. We shall take a closer look at its operating characteristics which, interestingly, suggest existing implementations of the random slicing are suboptimal, before discussing ways to improve them.

Let $\tilde{\scrM}_4(\bX)$ be an estimate of $\scrM_4(\bX)$ and write $\tilde{\scrM}_4(\bX)=\scrM_4(\bX)+\scrE$ where $\scrE$ is its estimation error. Suppose that $G$ is a $d\times d$ random matrix whose entries are independently drawn from the standard normal distribution. Then $(\tilde{\scrM}_4(\bX)-\scrM_0)\times_{3,4}G$ is a random combination of the $(1,2)$ slices of the estimated cumulant tenor $\tilde{\scrM}_4(\bX)-\scrM_0$:
\begin{equation}\label{eq:slice}
	(\tilde{\scrM}_4(\bX)-\scrM_0)\times_{3,4}G=\sum_{k=1}^d \kappa_4(S_k)(\ba_k^\top G\ba_k)\ba_k\circ\ba_k+\scrE\times_{3,4}G.
\end{equation}

\paragraph{Eigengap in the signal.} The first part on the right-hand side of \eqref{eq:slice} represents the signal that we wish to recover. It has eigenvalues $\{\kappa_4(S_k)(\ba_k^\top G\ba_k): 1\le k\le d\}$. Because of the randomness in $G$, even if all $\kappa_4(S_k)$s are equal, these eigenvalues are necessarily different. Moreover, we can create a large eigengap between the top two singular values of the signal through multiple rounds of random slicing. More specifically, suppose that we repeat the random slicing $L$ times with independently generated Gaussian random matrix $G_1,\ldots, G_L$. Assume, without loss of generality, that $|\kappa_4(S_1)|\ge |\kappa_4(S_k)|$ for all $k>1$. The ``eigenvalues'' corresponding to $\ba_1$ are $\{\kappa_4(S_1)(\ba_1^\top G_l\ba_1): 1\le l\le L\}$. These are independent normal random variables so that
\begin{equation}\label{eq:defLstar}
	\max_{1\le l\le L}|\kappa_4(S_1)(\ba_1^\top G_l\ba_1)|=|\kappa_4(S_1)|\sqrt{2\log L}(1+o_p(1)).
\end{equation}
Denote by $L^\ast$ be the maximizing index. Note that for any $k>1$, $\ba_k^\top G_{L^\ast}\ba_k$ is independent of $\ba_1^\top G_{L^\ast}\ba_1$. Therefore
$$
\max_{k>1}|\kappa_4(S_k)(\ba_k^\top G_{L^\ast}\ba_k)|=\max_{k>1}|\kappa_4(S_k)|\sqrt{2\log d}(1+o_p(1)).
$$
If we take $L\ge d^2$, then for the $L^\ast$ round of random slicing, the leading singular value of the first term on the rightmost-hand side of \eqref{eq:slice} is $2|\kappa_4(S_1)|\sqrt{\log d}(1+o_p(1))$ whereas its second largest singular value is no bigger than $|\kappa_4(S_1)|\sqrt{2\log d}(1+o_p(1))$, and therefore creating an eigengap $(2-\sqrt{2})|\kappa_4(S_1)|(\log d)^{1/2}(1+o_p(1))$.

\paragraph{Error bounds for the noise.} The second term on the right-hand side of \eqref{eq:slice} is the ``noise'' that we wish to be able to control. Write $\scrE=\scrE_1+\scrE_2$ where
$$
\scrE_1=(\scrE\times_3 \ba_1\times_4\ba_1)\circ \ba_1\circ \ba_1.
$$
It is easy to see that
$$
\|\scrE_1\times_{3,4}G_{L^\ast}\|=|\ba_1^\top G_{L^\ast}\ba_1|\|\scrE\times_3 \ba_1\times_4\ba_1\|\le |\ba_1^\top G_{L^\ast}\ba_1|\|\scrE\|,
$$
where, as before, $L^\ast$ is the slicing that maximizes the left-hand side of \eqref{eq:defLstar}. Therefore we can make sure it is dominated by $|\ba_1^\top G_{L^\ast}\ba_1|$ if $\|\scrE\|$ is made sufficiently small. 

On the other hand, conditional on the observations $\bX_1,\ldots, \bX_n$,
$$
\scrE_2\times_{3,4}G_{L^\ast}=\scrE_2\times_{3,4} (\II_d-\ba_1\circ \ba_1)G_{L^\ast}
$$
is again independent of $\ba_1^\top G_{L^\ast}\ba_1$. An application of concentration inequality for matrix Gaussian sequence from \cite{tropp2015introduction} yields, with probability at least $1-e^{-t^2/2}$,
$$
\|\scrE_2\times_{3,4}G_{L^\ast}\|\le \|\calM_{(1,3,4)(2)}(\scrE_2)\|(\sqrt{2\log d}+t)\le \|\calM_{(1,3,4)(2)}(\scrE)\|(\sqrt{2\log d}+t),
$$
assuming that $\tilde{\scrM}_4$ is symmetric. Here, similar to before, $\calM_{(1,3,4)(2)}(\cdot)$ denotes the reshaping of fourth-order tensor into a matrix by collapsing its first, third, and fourth indices.

Together, we have
\begin{eqnarray*}
\|\scrE\times_{3,4}G_{L^\ast}\|&\le& \|\scrE_1\times_{3,4}G_{L^\ast}\|+\|\scrE_2\times_{3,4}G_{L^\ast}\|\\
&\le& \sqrt{2\log L}\|\scrE\|(1+o_p(1))+\|\calM_{(1,3,4)(2)}(\scrE)\|(\sqrt{2\log d}+O_p(1)).
\end{eqnarray*}

\paragraph{Conditions for moment estimation.} If we can make sure that
\begin{equation}\label{eq:estcond}
\|\tilde{\scrM}_4-\scrM_4\|\le c|\kappa_4(S_1)|, \qquad{\rm and}\qquad \|\calM_{(1,3,4)(2)}(\tilde{\scrM}_4-\scrM_4)\|\le c|\kappa_4(S_1)|
\end{equation}
for a sufficiently small constant $c>0$, then an application of the Davis-Kahn-Wedin Theorem immediately yields that $\bu_{L^\ast}$ is a nontrivial estimate of $\ba_1$ in that $\sin\angle(\bu_{L^\ast}, \ba_1)$ can be bounded away from 1, where $\bu_{L^\ast}$ is the leading singular vector of $(\tilde{\scrM}_4(\bX)-\scrM_0)\times_{3,4}G_{L^\ast}$. This brings our attention back to moment estimation.

\subsubsection{Moment Estimation Revisited}

An obvious choice for $\tilde{\scrM}_4$ in our previous discussion is the sample moment tensor $\hat{\scrM}_4^{\rm sample}(\bX)$. Indeed it can be shown to obey the first inequality in \eqref{eq:estcond} when $n\gtrsim d^2$. More specifically,

\begin{lemma}\label{lem:tens-conc}
Let $\bX_1,\ldots,\bX_n$ be $n$ independent copies of $\bX$ such that $\calL(\bX)\in\calP_{\rm ICA}(\bA;\epsilon, M_1, M_2)$ for some $\bA\in \calO(d)$ and constants $\epsilon, M_1,M_2>0$.Then there exist constants $C_1,C_2>0$ such that, for any $t_1,t_2>0$,
$$
\norm*{\hat{\scrM}_4^{\rm sample}(\bX)-\scrM_4(\bX)}
\ge C_1(t_1+t_2)\sqrt{\dfrac{d^2}{n}}
$$
with probability at most $\exp(-C_2n^{\epsi/8})+\exp(-C_2t_1)+
(C_2n^{\epsi/8}t_2^{2+\epsi/4})^{-1}$.
\end{lemma}

However, a more careful inspection reveals that the second condition of \eqref{eq:estcond} does not hold for the sample moment tensor unless $n\gtrsim d^3$. This is because direct matricization of the sample moment tensor can induce significant bias. 

To illustrate this point, let us consider estimating $\calM_{(12)(3,4)}(\scrM_4(\bX))$. Note that it is also the second moment matrix of $\bY:=\bX\otimes \bX$, i.e.,
$$
\calM_{(12)(3,4)}(\scrM_4(\bX))=\texttt{cov}(\bY)+\EE(\bY)\EE(\bY)^\top=\texttt{cov}(\bY)+\texttt{vec}(\II_d)\texttt{vec}(\II_d)^\top.
$$
Likewise, $\calM_{(12)(3,4)}(\hat{\scrM}^{\rm sample}_4(\bX))$ is also the sample moment of $\bY$ and amounts to estimating both the covariance and expectation of $\bY$ by their sample counterpart. In particular, estimating $\texttt{vec}(\II_d)\texttt{vec}(\II_d)^\top$ by $\bar{\bY}\bar{\bY}^\top$ incurs a dominating and unnecessary estimation error of the order $\sqrt{d^3/n}$. Instead, a better estimate of $\calM_{(12)(3,4)}(\scrM_4(\bX))$ is
\begin{equation}\label{eq:defhatM}
\hat{\bM}:={1\over n}\sum_{i=1}^n (\bY_i-\bar{\bY})(\bY_i-\bar{\bY})^\top + \texttt{vec}(\II_d)\texttt{vec}(\II_d)^\top.
\end{equation}

\begin{lemma}\label{lem:mat-init-conc}
Let $\bX_1,\dots,\bX_n$ be $n$ independent copies of a random vector $\bX$ such that $\calL(\bX)\in\calP_{\rm ICA}(\bA;\epsi,M_1,M_2)$ for some $\bA\in \calO(d)$ and constants $\epsi,M_1,M_2>0$.Then there exist constants $C_1,C_2>0$ such that, for any $t_1,t_2>0$,
$$
\norm*{
		\hat{\bM}-\calM_{(12)(3,4)}(\scrM_4(\bX))
	}
	\ge C_1(t_1+t_2)\sqrt{\dfrac{d^2}{n}}
$$
with probability at most $\exp(-C_2n^{\epsi/8})+\exp(-C_2t_1)+(C_2n^{\epsi/8}t_2^{2+\epsi/4})^{-1}$.
\end{lemma}

In principle, we can view $\calM_{(12)(3,4)}^{-1}(\hat{\bM})$ as an estimate of the moment tensor $\scrM_4(\bX)$. However, it is designed for estimating $\calM_{(12)(3,4)}(\scrM_4(\bX))$ and can be further improved for estimating other matricizations of $\scrM_4(\bX)$. The main idea is the observation that $\calM_{(12)(3,4)}(\scrM_4(\bX)-\scrM_0)$ in \eqref{eq:mat-cuml4} is a $d^2\times d^2$ matrix but has rank $d$. In the following, we use a sample splitting idea to utilize this additional information.%

To this end, we divide the samples into $\calS_1$ and $\calS_2$ of size $n/2$. We then compute $\hat{\bM}_1$ and $\hat{\bM}_2$ as above using the samples in $\calS_1$ and $\calS_2$ respectively. Let $\hat{\bP}$ be the projection matrix onto the top-$d$ singular space of $\hat{\bM}_2-\calM_{(12)(3,4)}(\scrM_0)$. We shall proceed to estimate $\scrM_4(\bX)$ by
\begin{equation}\label{eq:def-scrM-init}
	\hat{\scrM}=\calM_{(12)(3,4)}^{-1}(\hat{\bP}(\hat{\bM}_1-\calM_{(12)(3,4)}(\scrM_0))\hat{\bP})+\scrM_0.
\end{equation}

The following lemma shows that the moment estimate $\hat{\scrM}$ satisfies the condition \eqref{eq:estcond}.
\begin{lemma}\label{lem:mat-init-conc2}
Let $\bX_1,\dots,\bX_n$ be $n$ independent copies of a random vector $\bX$ such that $\calL(\bX)\in\calP_{\rm ICA}(\bA;\epsi,M_1,M_2)$ for some $\bA\in \calO(d)$ and constants $\epsi,M_1,M_2>0$. Then there exist constants $C_1,C_2>0$ such that, for any $t_1,t_2>0$,
$$
\max\left\{\left\|\hat{\scrM}-\scrM_4(\bX)\right\|, \left\|\calM_{(123)(4)}(\hat{\scrM}-\scrM_4(\bX))\right\|\right\}\ge C_1(t_1+t_2)\sqrt{\dfrac{d^2}{n}}
$$
with probability at most $\exp(-C_2n^{\epsi/8})+\exp(-C_2t_1)+(C_2n^{\epsi/8}t_2^{2+\epsi/4})^{-1}$.
\end{lemma}

\subsubsection{Initialization}

The following algorithm summarizes our proposed initialization scheme.

\begin{algorithm}[htbp]
	\caption{Intialization}\label{alg:init}
	\hspace*{\algorithmicindent} \textbf{Input:} $\bX_1,\ldots, \bX_n$
	\begin{algorithmic}[1]
		\State Divide the samples into disjoint sets $\calS_1$ and $\calS_2$ of size $n/2$.
		\State Compute $\hat{\bM}_1$ and $\hat{\bM}_2$ as defined by \eqref{eq:defhatM} on $\calS_1$ and $\calS_2$ respectively. 
		\State Compute the top rank-$d$ projection $\hat{\bP}$ of $\hat{\bM}_2-\mathcal{M}_{(12)(3,4)}(\scrM_0)$.
		\State Compute the cumulant tensor estimate $\hat{\scrM}$ as defined by \eqref{eq:def-scrM-init} 
		\For {$l=1$ to $L$}
		\State Generate a $d\times d$ Gaussian random matrix $G_l$
		\State Compute the leading singular value and left singular vector of $(\hat{\scrM}-\mathcal{M}_{(12)(3,4)}(\scrM_0))\times_{3,4} G_l$, denoted by $\sigma_l$ and $\bu_l$
		\EndFor
		\State Compute $L^\ast=\argmax_{1\le l\le L}\sigma_l$.\\
		\Return $\bu_{L^\ast}$.
	\end{algorithmic}
\end{algorithm}

The following theorem shows that Algorithm \ref{alg:init} indeed yields a nontrivial estimate of one of the column vectors of $\bA$.

\begin{theorem}\label{th:init}
Let $\bX_1,\dots,\bX_n$ be $n$ independent copies of a random vector $\bX$ such that $\calL(\bX)\in\calP_{\rm ICA}(\bA;\epsi,M_1,M_2)$ for some $\bA\in \calO(d)$ and constants $\epsi,M_1,M_2>0$, and $\bu$ the output from Algorithm~\ref{alg:init} with $\bX_1,\dots,\bX_n$ as input. Then there exist constants $C_1,C_2>0$ such that 
$$
\min_{1\le k\le d}
\sin\angle
\left(
\bu,\ba_k
\right) \le \dfrac{1}{4}
$$
with probability at least $1-C_1d^{-3}$, provided that $n>C_2d^2$ and $L>C_2d^2$.
\end{theorem}

It is worth noting that thus far our discussion has centered around initializing the first direction. For the remaining directions, we need to apply the algorithm to deflated observations. Because of the error associated with the already estimated directions, the input to Algorithm \ref{alg:init} does not necessarily follow an ICA model. Nonetheless, as we shall see later that with a sufficient number of fixed-point iterations and appropriate deflation, Algorithm \ref{alg:init} can be used successively to yield good initialization for all columns vectors of $\bA$.

\subsection{Convergence Analysis}

Now that we have a nontrivial initialization, we turn our attention to the complexity of fixed-point iteration and the effect of deflation. A careful inspection of the dynamics of these steps reveal that we can estimate $\ba_k$s at the minimax optimal rate of $\sqrt{d/n}$ which is much faster than rate $\sqrt{d^2/n}$ for estimating the moment tensor under the operator norm.

\subsubsection{Fixed-point iteration}

We begin with the fixed-point iterations. We shall focus on recovering the first mixing direction to illustrate the main ideas. Following the discussion in the previous subsection, we shall from now on assume that we have an initial value $\hat{\ba}^{[0]}$ obeying
\begin{equation}\label{eq:ica-init}
\min_{1\le k\le d}\sin\angle\left(\hat{\ba}^{[0]},\ba_k\right)\le 1-\eta,
\end{equation}
for some $\eta>0$. Assume, without loss of generality, that the minimizer of left-hand side is $k=1$ for brevity. We now investigate how the fixed-point iterations could lead to a rate optimal estimate of $\ba_1$. Particular attention needs to be paid to how different errors of the sample moment tensor may affect the dynamics of the iteration.

Recall that
\begin{eqnarray*}
\hat{\ba}^{[t]}&\propto& (\hat{\scrM}^{\rm sample}_4(\bX)-\scrM_0)\times_2\hat{\ba}^{[t-1]}\times_3\hat{\ba}^{[t-1]}\times_4\hat{\ba}^{[t-1]}\\
&=&\sum_{j=1}^d \kappa_4(S_j)\langle \ba_j,\hat{\ba}^{[t-1]}\rangle^3\ba_j+(\hat{\scrM}^{\rm sample}_4(\bX)-\scrM_4(\bX))\times_2\hat{\ba}^{[t-1]}\times_3\hat{\ba}^{[t-1]}\times_4\hat{\ba}^{[t-1]}.
\end{eqnarray*}
Denote by $\eta_t=\sin\angle(\hat{\ba}^{[t]},\ba_j)$. Then
$$
\left|\kappa_4(S_j)\langle \ba_1,\hat{\ba}^{[t-1]}\rangle^3\right|=|\kappa_4(S_j)|(1-\eta_{t-1}^2)^{3/2}
$$
and
$$
\left\|\sum_{l\neq j} \kappa_4(S_l)\langle \ba_l,\hat{\ba}^{[t-1]}\rangle^3\ba_l\right\|=\left(\sum_{l\neq j} \kappa_4^2(S_l)\langle \ba_l,\hat{\ba}^{[t-1]}\rangle^6\right)^{1/2}\le M_1\eta_{t-1}^3.
$$
On the other hand, by triangular inequality,
\begin{eqnarray*}
&&\|(\hat{\scrM}^{\rm sample}_4(\bX)-\scrM_4(\bX))\times_2\hat{\ba}^{[t-1]}\times_3\hat{\ba}^{[t-1]}\times_4\hat{\ba}^{[t-1]}\|\\
&\le& \|(\hat{\scrM}^{\rm sample}_4(\bX)-\scrM_4(\bX))\times_2\ba_j\times_3\ba_j\times_4\ba_j\|+3\eta_{t-1}\|(\hat{\scrM}^{\rm sample}_4(\bX)-\scrM_4(\bX))\times_3\ba_j\times_4\ba_j\|\\
&&+3\eta_{t-1}^2\|(\hat{\scrM}^{\rm sample}_4(\bX)-\scrM_4(\bX))\times_4\ba_j\|+\eta_{t-1}^3\|(\hat{\scrM}^{\rm sample}_4(\bX)-\scrM_4(\bX))\|\\
&=&\Delta_3+3\eta_{t-1}\Delta_2+3\eta_{t-1}^2\Delta_1+\eta_{t-1}^3\Delta,
\end{eqnarray*}
where
$$
\Delta=\|\hat{\scrM}^{\rm sample}_4(\bX)-\scrM_4(\bX)\|, \qquad \Delta_1=\max_{1\le j\le d}\|(\hat{\scrM}^{\rm sample}_4(\bX)-\scrM_4(\bX))\times_4\ba_j\|,
$$
$$
\Delta_2=\max_{1\le j\le d}\|(\hat{\scrM}^{\rm sample}_4(\bX)-\scrM_4(\bX))\times_3\ba_j\times_4\ba_j\|,
$$
and
$$
\Delta_3=\max_{1\le j\le d}\|(\hat{\scrM}^{\rm sample}_4(\bX)-\scrM_4(\bX))\times_2\ba_j\times_3\ba_j\times_4\ba_j\|.
$$

The key observation here is that the final bound depends on the estimation error of the sample moment tensors measured in various matricized norms. We shall leverage the fact that these errors are of different magnitudes to better characterize the dynamics of the fix-point iteration.

\begin{lemma}\label{lem:sample-conc} 
Let $\bX_1,\dots,\bX_n$ be $n$ independent copies of a random vector $\bX$ such that $\calL(\bX)\in\calP_{\rm ICA}(\bA;\epsi,M_1,M_2)$ for some $\bA\in \calO(d)$ and constants $\epsi,M_1,M_2>0$. There exist constants $C_1,C_2>0$ such that
$$
\Delta_1\le \Delta\le C_1(t_1+t_2)\sqrt{\dfrac{d^2}{n}},$$
and
$$
\Delta_3\le\Delta_2\le C_1(t_1+t_2)\sqrt{\dfrac{d}{n}}, 
$$
with probability at least $1-\exp(-C_2t_1)-
(C_2n^{\epsi/8}t_2^{2+\epsi/4})^{-1}$.
\end{lemma}

Note that
\begin{eqnarray*}
\eta_t&\le& \left|\kappa_4(S_j)\langle \ba_j,\hat{\ba}^{[t-1]}\rangle^3\right|^{-1}\left\|\sum_{l\neq j} \kappa_4(S_l)\langle \ba_l,\hat{\ba}^{[t-1]}\rangle^3\ba_l\right\|\\
&&+\left|\kappa_4(S_j)\langle \ba_j,\hat{\ba}^{[t-1]}\rangle^3\right|^{-1}\left\|(\hat{\scrM}^{\rm sample}_4(\bX)-\scrM_4(\bX))\times_2\hat{\ba}^{[t-1]}\times_3\hat{\ba}^{[t-1]}\times_4\hat{\ba}^{[t-1]}\right\|\\
&\le& 2|\kappa_4(S_j)|^{-1}\Delta_3+6|\kappa_4(S_j)|^{-1}\Delta_2\eta_{t-1}
+6|\kappa_4(S_j)|^{-1}\Delta_1\eta_{t-1}^2+2|\kappa_4(S_j)|^{-1}\Delta\eta_{t-1}^3,
\end{eqnarray*}
provided that $\eta_{t-1}<1/2$, say. Now Lemma \ref{lem:sample-conc} and Theorem~\ref{th:init} yield

\begin{theorem}\label{th:ica-piter}
Suppose that $\hat{\ba}^{[0]}\in \SS^{d-1}$ satisfies \eqref{eq:ica-init} with a sufficiently large $0<\eta<1$. Then if $T\ge C\log d$, for any fixed  $\delta\in (0,1)$, the estimator obtained after $T$ fixed-point iterations satisfies
$$
\PP
\left\{\min_{1\le j\le d}\sin\angle\left(\hat{\ba}^{[T]},\ba_j\right)\le C\sqrt{\dfrac{d(\log (1/\delta))}{n}}\right\}\ge 1-\delta-d^{-3}-n^{-\epsi/8}
$$ 
for a constant $C>0$, provided $n\ge Cd^2(\log \delta)^2$.
\end{theorem}

In other words, the number of fixed-point iterations required scales with the dimension $d$ in a logarithmic fashion. This suggests that, as long as a good initialization is provided, the computational complexity can be well managed. Of course, as before, we are only concerned with recovering one unmixing direction at this point. The situation beyond the first direction is more involved as estimation error of previously recovered directions may accumulate. To address this issue, we now turn to the deflation step.

\subsubsection{Deflation}
The deflation step is essential to ensure that all mixing directions can be identified sequentially. In particular, suppose that we already have estimates, $\{\hat{\ba}_k:1\le k< j\}$, for the $j-1$ unmixing directions by running Algorithm \ref{alg:alg-ica} with an initialization using Algorithm \ref{alg:init}. Without loss of generality, assume that for $\hat{\ba}_k$ is a $\sqrt{d/n}$-consistent estimate of $\ba_k$ for $k=1,\ldots, j-1$. We now consider the effect of deflation and how well we can estimate the next unmixing direction.

It is not hard to see that the deflation step of the FastICA amounts to the subtraction of the estimated components from $\hat{\scrM}^{\rm sample}_4(\bX)$:
\begin{equation}
\hat{\scrM}_4^{(j),{\rm sample}}(\bX):=\hat{\scrM}^{\rm sample}_4(\bX) -\sum_{k=1}^{j-1}\hat{f}(\hat{\ba}_k)
	\hat{\ba}_k\circ \hat{\ba}_k\circ\hat{\ba}_k\circ\hat{\ba}_k
\end{equation}
Note that $\hat{\scrM}_4^{(j),{\rm sample}}(\bX)$ can be viewed as an estimate of
$$
\scrM_4^{(j)}(\bX)=\scrM_4(\bX)-\sum_{k=1}^{j-1}\kappa_4(S_k)
\ba_k\circ \ba_k\circ\ba_k\circ\ba_k=\sum_{k=j}^{d}\kappa_4(S_k)
\ba_k\circ \ba_k\circ\ba_k\circ\ba_k
$$
The same argument for initialization and fixed-point iterations remains valid if we can bound the estimation error of $\hat{\scrM}_4^{(j),{\rm sample}}(\bX)$ in the same way as the original sample moment $\hat{\scrM}_4^{{\rm sample}}(\bX)$, which is indeed the case.

For example, in light of the discussion from previous subsections, we can show by induction that
$$
\left|\hat{f}(\hat{\ba}_k)
\hat{\ba}_k\circ \hat{\ba}_k\circ\hat{\ba}_k\circ\hat{\ba}_k-\kappa_4(S_k)
\ba_k\circ \ba_k\circ\ba_k\circ\ba_k\right|=O_p\left(\sqrt{d\over n}\right).
$$
By triangular inequality,
\begin{eqnarray*}
&&\|\hat{\scrM}_4^{(j),{\rm sample}}(\bX)-\scrM_4^{(j)}(\bX)\|\\
&\le& \|\hat{\scrM}^{\rm sample}_4(\bX)-\scrM_4(\bX)\|
+\left\|\sum_{k=1}^{j-1}\left[\hat{f}(\hat{\ba}_k)
\hat{\ba}_k\circ \hat{\ba}_k\circ\hat{\ba}_k\circ\hat{\ba}_k-\kappa_4(S_k)
\ba_k\circ \ba_k\circ\ba_k\circ\ba_k\right]\right\|\\
&=&O_p\left(\sqrt{d^2\over n}\right).
\end{eqnarray*}
Other norms of the error can also be bounded similarly, thus leading to the following conclusion.
\begin{theorem}\label{th:comp-rate}
Let $\bX_1,\ldots,\bX_n$ be $n$ independent copies of $\bX$ s.t.  $\calL(\bX)\in\calP_{\rm ICA}(\bA;\epsilon, M_1, M_2)$ for some $\bA\in \calO(d)$ and $\epsilon, M_1,M_2>0$. Suppose we run Algorithm \ref{alg:alg-ica} with the initialization step via Algorithm \ref{alg:init}, and with parameters $L\ge C_1d^2$ for initialization and $T\ge C_2\log d$ for fixed-point iteration for some constants $C_1,C_2>0$. Denote by $\hat{\bA}=[\hat{\ba}_1,\ldots,\hat{\ba}_d]$ the output. There exist constants $C_3,C_4>0$ such that if $n\ge C_3d^2(\log \delta)^2$, then
$$
\ell(\hat{\bA},\bA)\le C_4\sqrt{{d\log(1/\delta)}\over n},
$$
with probability at least $1-\delta-d^{-3}-n^{-\epsi/8}$ with the loss function $\ell$ being either $\ell_M$ or $\ell_A$.
\end{theorem}

The rate of convergence given by Theorem \ref{th:comp-rate} matches the information-theoretical lower bound of Theorem \ref{th:ICAlower}, indicating that our estimate is minimax optimal when the sample size $n\gtrsim d^2$. In addition, in light of Theorem \ref{th:comp-lower}, the sample complexity required by our estimate is also optimal, up to a logarithmic factor, among all estimates that can be computed using low-order polynomial algorithms.

\subsection{Pre-whitening and Unknown Covariance Matrix}
\label{sec:prewhiten}

Throughout our discussion so far, we have made the simplifying assumption that the data are pre-whitened so that the mixing matrix $\bA$ is orthonormal. As noted before, this is always possible if the covariance matrix of $\bX$ is known apriori. In practice, however, $\texttt{cov}(\bX)$ needs to be estimated. Prewhitening with an estimated covariance matrix could adversely affect the estimate for ICA. We shall now describe a sample splitting scheme to overcome this challenge so that all the properties we presented when $\bA\in \calO(d)$ remain valid even if $\texttt{cov}(\bX)$ is not known in advance.

\begin{algorithm}[h]
	\caption{Prewhitening}\label{alg:alg-ica-2}
	\hspace*{\algorithmicindent} \textbf{Input:} $\bX_1,\ldots, \bX_n$
	\begin{algorithmic}[1]
		\State Partition the samples into two halves with index sets $S_1$ and $S_2$.
		\State Compute the sample covariance matrix $\hat{\bSigma}$ for $\{\bX_i: i\in S_1\}$.
		\State Prewhiten data from $S_2$: $\tilde{\bX}_i\leftarrow \hat{\bSigma}^{-1/2}\bX_i$ for $i\in S_2$.
		\State Run Algorithm \ref{alg:alg-ica} with input $\{\tilde{\bX}_i: i\in S_2\}$ and initialized with Algorithm \ref{alg:init}. Denote by $\{\hat{\ba}_j:1\le j\le d\}$ the output.
		\State Update $\hat{\ba}_j\leftarrow \hat{\Sigma}^{1/2}\hat{\ba}_j$\\
		\Return $\{\hat{\ba}_j:1\le j\le d\}$.
	\end{algorithmic}
\end{algorithm}

The validity of this algorithm is justified by the following result.

\begin{theorem}\label{th:ica-gen-rate}
Let $\bX_1,\ldots,\bX_n$ be $n$ independent copies of $\bX$ s.t.  $\calL(\bX)\in\calP_{\rm ICA}(\bA;\epsilon, M_1, M_2)$ for some$\epsilon, M_1,M_2>0$ and $\bA\in \RR^{d\times d}$ with bounded condition number. Suppose we run Algorithm \ref{alg:alg-ica-2} that calls Algorithm \ref{alg:alg-ica} and Algorithm \ref{alg:init} with parameters $L\ge C_1d^2$ for initialization and $T\ge C_2\log d$ for fixed-point iteration for some constants $C_1,C_2>0$. Denote by $\hat{\bA}=[\hat{\ba}_1,\ldots,\hat{\ba}_d]$ the output. There exist constants $C_3,C_4>0$ such that if $n\ge C_3d^2(\log \delta)^2$, then
$$
\ell(\hat{\bA},\bA)\le C_4\sqrt{{d\log(1/\delta)}\over n},
$$
with probability at least $1-\delta-d^{-3}-n^{-\epsi/8}$ with the loss function $\ell$ being either $\ell_M$ or $\ell_A$.
\end{theorem}

Theorem \ref{th:ica-gen-rate} shows that with the prewhitening via sample splitting as described by Algorithm \ref{alg:alg-ica-2}, we can ensure that the estimated mixing matrix enjoys the same asymptotic properties as if the true mixing matrix is orthonormal. In fact, in addition to the convergence rates, the asymptotic distributional properties we present in the next section continue to hold beyond orthonormal mixing matrices as well.

\section{Asymptotic Normality}
\label{sec:asy-dist}
The computationally tractable estimators from the previous section are minimax rate optimal when $n\gtrsim d^2$. We now show that it is in fact possible to derive a finer bound, by an explicit characterization of the leading error term, enabling us to derive the asymptotic distribution of the estimators $\hat{\ba}_j$. For brevity, we shall assume, as before, that the mixing matrix $\bA\in \calO(d)$, and the estimated mixing directions are appropriately permuted and signed so that $\pi(j)=j$ and $\langle\hat{\ba}_j,\ba_j\rangle>0$. Specifically, this can be done by defining
\begin{equation}\label{eq:def-A-hat}
	\hat{\bA}
	:=
	\left[
	{\rm sign}(\langle\hat{\ba}_{\pi(1)},\ba_1\rangle)\hat{\ba}_{\pi(1)}\,\,
	\dots\,\,
	{\rm sign}(\langle\hat{\ba}_{\pi(d)},\ba_d\rangle)
	\hat{\ba}_{\pi(d)}
	\right]
\end{equation}
where $\hat{\ba}_j$s are the output from Algorithm \ref{alg:alg-ica} with initialization step given by Algorithm \ref{alg:init}. Specifically, we shall show that
$$
{\rm sign}
\left(
\langle\hat{\ba}_j,\ba_j\rangle
\right)
\hat{\ba}_{\pi(j)}-\ba_j 
=\dfrac{1}{n}\sum_{k=1}^n(S_{kj})^3\bX_k+O_p\left(\dfrac{d(\log d)}{n}\right).
$$
This allows us to derive asymptotic distributions for both linear and bilinear forms of $\bA$.

\paragraph{Marginal Distribution of Linear Forms:}

We begin with linear forms of unmixing directions. Consider, for example, estimating $\bu^\top \ba_j$ for some fixed vector $\bu$. A natural and consistent estimate is $\bu^\top\hat{\ba}_j$. The following theorem shows that it is also an asymptotically normal estimate.

\begin{theorem}\label{th:ica-perp}
Let $\bX_1,\ldots,\bX_n$ be $n$ independent copies of $\bX$ s.t.  $\calL(\bX)\in\calP_{\rm ICA}(\bA;\epsilon, M_1, M_2)$ for some $\bA\in \calO(d)$, $M_0,M_2>0$, and $\epsilon\ge 4$. Then for any $j\in [d]$, and any $\bu\in \RR^d$ such that $\liminf_{d\to\infty}\|\bP_{\ba_j,\perp}\bu\|>0$ we have
for sufficiently large $d$ that
	$$ 
	\sup_{x\in \RR}
	\abs*{
		\PP\left\{\dfrac{\sqrt{n}}{\sigma_{\bu}}
		\bu^{\top}(\hat{\ba}_j-\ba_j)
		\le x\right\} 
		-\Phi(x)
	}
	\le~  \dfrac{Cd(\log d)}{\sqrt{n}}
	+\dfrac{C(\log d)^{3/2}}{\sqrt{d}},
	$$
	where $\sigma_{\bu}^2=\bu^{\top}(\II_d-\ba_j\ba_j^{\top})\bu\cdot \Var(S_{1j}^3)/\kappa_4(S_j)^2$, provided $n\ge Cd^2(\log d)^2$.
\end{theorem}

It is clear from Theorem \ref{th:ica-perp} that
$$\sqrt{n}\bu^{\top}(\hat{\ba}_j-\ba_j)\to_d N(0,\sigma_{\bu}^2),$$
when $d\to\infty$ and $n\gg d^2(\log d)^2$. Note that,  in Theorem \ref{th:ica-perp}, we restrict ourselves to vector $\bu$ such that $\|\bP_{\ba_j,\perp}\bu\|>0$ for large $d$. This amounts to assuming that $\sin\angle\left(\bu,\ba_j\right)>0$. On the other hand when $\sin\angle\left(\bu,\ba_j\right)=0$, the estimated linear form converges at a faster rate and its asymptotic behavior can be more precisely characterized as follows:
\begin{corollary}\label{cor:ica-par}
	Under the assumptions of Theorem \ref{th:ica-perp}, we have for any $j\in \{1,\dots,d\}$ that
	\begin{itemize}
		\item[i)] for any $0<\delta<1$,
		$$\PP\left(\left(1-\langle\hat{\ba}_j,\ba_j\rangle^2
		\right)\ge \dfrac{Cd\log (1/\delta)}{n}
		+\dfrac{C(d\log d)^{3/2}}{n^{3/2}}
		\right)
		\le \delta+d^{-3}+n^{-\epsi/8}.
		$$
		\item[ii)] Moreover, if $n\ge Cd^3(\log d)^2$, then
	\end{itemize}
	$$
	\sup_{x>0}\,
	\abs*{
		\PP\left(
		\dfrac{n\kappa_4(S_j)^2}{\EE(S_{1j}^6)^2}
		\left(1-\langle\hat{\ba}_j,\ba_j\rangle^2
		\right)
		\le x
		\right)
		-\PP(\chi^2_d\le x)
	}
	\le 		
	\dfrac{C(d\log d)^{3/2}}{\sqrt{n}}.
	$$
\end{corollary}

Note in particular that when $\|\bP_{\ba_j,\perp}\bu\|\to 0$, then 
$$
\sqrt{n}\bu^{\top}(\hat{\ba}_j-\ba_j)\to_p 0.
$$
as $d,n\to\infty$ while $n\gg d^2(\log d)^2$.

It is helpful to draw a comparison with the usual PCA, where an additional debiasing step may be required to obtain asymptotic normality for linear forms of the principal components. See, e.g., \cite{koltchinskii2016asymptotics,koltchinskii2017concentration,koltchinskii2020efficient,xia2021normal}. In the case of ICA, the sample size $n\gg d^2$ required by computational consideration is sufficient to ensure that bias becomes negligible. A similar phenomenon was observed by \cite{xia2022inference} in the case of tensor denoising.

It is oftentimes of interest to derive the asymptotic distribution for individual entries of $\bA$. This follows immediately from Theorem \ref{th:ica-perp} by taking $\bu$ as the canonical basis.

\begin{corollary}\label{cor:aij-dist}
Let $\bX_1,\ldots,\bX_n$ be $n$ independent copies of $\bX$ s.t.  $\calL(\bX)\in\calP_{\rm ICA}(\bA;\epsilon, M_1, M_2)$ for some $\bA\in \calO(d)$, $M_0,M_2>0$, and $\epsilon\ge 4$. If $|a_{ij}|=|\be_i^{\top}\ba_j|< 1-c$ for some constant $c>0$, then under the assumptions of Theorem~\ref{th:ica-perp} there exists a constant $C>0$ such that
	$$
	\sup_{x\in \RR}
	~
	\abs*{
		\PP\left(
		\dfrac{\sqrt{n}}{\sigma_{ij}}
		(\hat{a}_{ij}-a_{ij})
		\le x\right) 
		-\Phi(x)
	}
	\le~  \dfrac{Cd(\log d)}{\sqrt{n}}
	+\dfrac{C(\log d)^{3/2}}{\sqrt{d}}
	$$
	where $\hat{a}_{ij}$ and $a_{ij}$ are the $(i,j)$ entry of $\hat{\bA}$ and $\bA$ respectively, and
	$\sigma_{ij}^2=(1-a_{ij}^2)\Var(S_j^3)/\kappa_4(S_j)^2$.
\end{corollary}


\paragraph{Joint Distribution of Linear Forms:}
We can also derive the joint distribution of estimates of linear forms of multiple columns of $\bA$. To this end, for $\bu_1,\bu_2,\dots,\bu_d\in \RR^{d}$, denote by $\bSigma_{\bu_1,\ldots,\bu_d}$ a $d\times d$ matrix whose $(i,j)$ entry is given by
$$
\left(\bSigma_{\bu_1,\ldots,\bu_d}\right)_{ij}
=\begin{cases}
	\bu^{\top}_j(\II_d-\ba_j\ba_j^{\top})\bu_j\cdot\ \dfrac{\Var(S_j^3)}{\kappa_4(S_j)^2}
	\quad &\text{if }i=j\\
	\bu_i^{\top}\ba_j\bu_j^{\top}\ba_i\cdot\
	\dfrac{\EE(S_i^4)\EE(S_j^4)}{\kappa_4(S_i)\kappa_4(S_j)}
	\quad &\text{if }i\neq j.
\end{cases}
$$
Let us define the vectors $\bv,\hat{\bv}\in \RR^d$ with elements
$
v_j:=\bu_j^{\top}\ba_j$
 and 
$\hat{v}_j=\bu_j^{\top}\hat{\ba}_j$. We then have
\begin{theorem}\label{th:joint}
Let $\bX_1,\ldots,\bX_n$ be $n$ independent copies of $\bX$ s.t.  $\calL(\bX)\in\calP_{\rm ICA}(\bA;\epsilon, M_1, M_2)$ for some $\bA\in \calO(d)$, $M_1,M_2>0$, and $\epsilon\ge 4$. Then for any $\bu_1,\dots,\bu_d\in \RR^d$ such that $\liminf_{d\to \infty}\min_{1\le j\le d}\|\bP_{\ba_j,\perp}\bu_j\|>0$  there exists a constant $C>0$ such that
	$$
	\sup_{\bx \in \RR^d}
	\abs*{
		\PP\left(
		\sqrt{n}
		\bSigma_{\bu_1,\ldots,\bu_d}^{-1/2}(\hat{\bv}-\bv)
		\in \prod_{j=1}^d(-\infty, x_j]
		\right)
		-\prod_{j=1}^d\Phi(x_j)
	}
	\le \dfrac{Cd(\log d)}{\sqrt{n}}+\dfrac{C(\log d)^{3/2}}{\sqrt{d}}.
	$$
for sufficiently large $d$, provided $n\ge Cd^2(\log d)^2$.	
\end{theorem}

\paragraph{Bilinear Forms:}

In addition to the linear forms described above, one may also be interested in bilinear forms $\bu^{\top}\bA\bv$ for two vectors $\bu$ and $\bv$. The asymptotic distribution of its estimate can also be derived. 

\begin{theorem}\label{th:ica-clt}
Let $\bX_1,\ldots,\bX_n$ be $n$ independent copies of $\bX$ s.t.  $\calL(\bX)\in\calP_{\rm ICA}(\bA;\epsilon, M_1, M_2)$ for some $\bA\in \calO(d)$, $M_0,M_2>0$, and $\epsilon\ge 4$. Then for any $\bu,\bv\in \RR^d$ such that $\liminf_{d\to \infty}\sigma_{\bu,\bv}>0$, there exists a constant $C>0$ such that
	$$
	\sup_{x\in \RR}
	~\abs*{
		\PP\left(\dfrac{\sqrt{n}}{\sigma_{\bu,\bv}}
		\bu^{\top}(\hat{\bA}-\bA)\bv
		\le x\right) 
		-\Phi(x)
	}
	\le~  \dfrac{Cd^{3/2}(\log d)}{\sqrt{n}}
	+\dfrac{C(\log d)^{3/2}}{\sqrt{d}},
	$$
	for sufficiently large $d$, provided $n\ge Cd^3(\log d)^2$. Here $
	\sigma^2_{\bu,\bv}:=\bu^{\top}\bA\bD_{\bv}\bA^{\top}\bu $
	and 
	$$\bD_{\bv}:={\rm diag}
	\left(
	\displaystyle\sum_{j\neq 1}\dfrac{v_j^2{\rm Var}(S_{1j}^3)}{\kappa_4(S_j)^2}\,\,
	\sum_{j\neq 2}\dfrac{v_j^2{\rm Var}(S_{1j}^3)}{\kappa_4(S_j)^2}\,\,
	\dots\,\,
	\sum_{j\neq d}\dfrac{v_j^2{\rm Var}(S_{1j}^3)}{\kappa_4(S_j)^2}
	\right).
	$$

\end{theorem}

Theorem \ref{th:ica-clt} states the asymptotic normality of all bilinear forms $\bu^{\top}\bA\bv$ for which the variance $\sigma_{\bu,\bv}$ is positive. Notice that we require a larger sample size in Theorem \ref{th:ica-clt} than for the linear forms. We suspect that the weaker bound in the previous theorem is an artifact of our proof and not a fundamental barrier.

\section{Numerical Experiments}\label{sec:simul}

To complement our theoretical analysis and demonstrate the practical merits of our approach,
we carried out several sets of numerical experiments.

\subsection{Finite Sample Properties}
Our first set of numerical experiments is a number of simulation studies to explore the finite sample properties of the proposed method and its comparison with several popular ICA techniques. 

\paragraph{Comparison with Other Popular Techniques.} In these experiments, we fixed the dimension $d=25$, generated the mixing matrix $\bA$ from the uniform distribution over $\calO(d)$, and $S_k$s from the ${\rm Laplace}$ distribution with mean zero and variance one. We varied the sample size $n=500$, $1500$, and $2000$. For each simulated data, we computed the mixing matrix using six different methods:  FastICA using neg-entropy as the contrast function \citep{hyvarinen2000independent}, JADE \citep{cardoso1993blind, cardoso1999high}, FastICA initialized by naive matricization, FastICA with random initialization, FastICA with random slicing (\cite{anand2014tensor}), and the proposed method. In each run, we measure the performance of both estimates using $\ell_M$ and $\ell_A$. The two error measures yield qualitatively very similar results and we only report the result with $\ell_A$, for brevity. The results, summarized from 200 runs for each sample size, are presented in Figure \ref{fig:lo-dim}.

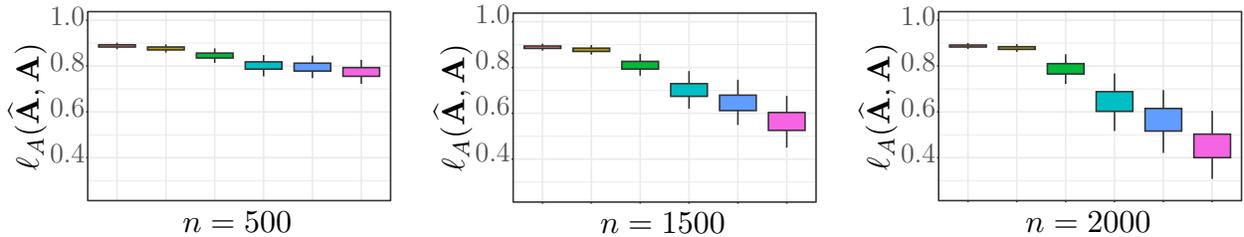
\begin{figure}[htb]
	\centering
	\begin{subfigure}{0.25\textwidth}
		\begin{tikzpicture}[x=.22pt,y=.22pt]
	\hspace{-1cm}
\definecolor{fillColor}{RGB}{255,255,255}
\begin{scope}
\definecolor{drawColor}{RGB}{255,255,255}
\definecolor{fillColor}{RGB}{255,255,255}

\end{scope}
\begin{scope}
\definecolor{fillColor}{RGB}{255,255,255}

\path[fill=fillColor] ( 34.16, 30.69) rectangle (554.90,355.85);
\definecolor{drawColor}{gray}{0.92}

\path[draw=drawColor,line width= 0.3pt,line join=round] ( 34.16, 65.17) --
	(554.90, 65.17);

\path[draw=drawColor,line width= 0.3pt,line join=round] ( 34.16,144.00) --
	(554.90,144.00);

\path[draw=drawColor,line width= 0.3pt,line join=round] ( 34.16,222.83) --
	(554.90,222.83);

\path[draw=drawColor,line width= 0.3pt,line join=round] ( 34.16,301.66) --
	(554.90,301.66);

\path[draw=drawColor,line width= 0.6pt,line join=round] ( 34.16,104.59) --
	(554.90,104.59);

\path[draw=drawColor,line width= 0.6pt,line join=round] ( 34.16,183.41) --
	(554.90,183.41);

\path[draw=drawColor,line width= 0.6pt,line join=round] ( 34.16,262.24) --
	(554.90,262.24);

\path[draw=drawColor,line width= 0.6pt,line join=round] ( 34.16,341.07) --
	(554.90,341.07);

\path[draw=drawColor,line width= 0.6pt,line join=round] ( 84.55, 30.69) --
	( 84.55,355.85);

\path[draw=drawColor,line width= 0.6pt,line join=round] (168.54, 30.69) --
	(168.54,355.85);

\path[draw=drawColor,line width= 0.6pt,line join=round] (252.53, 30.69) --
	(252.53,355.85);

\path[draw=drawColor,line width= 0.6pt,line join=round] (336.52, 30.69) --
	(336.52,355.85);

\path[draw=drawColor,line width= 0.6pt,line join=round] (420.52, 30.69) --
	(420.52,355.85);

\path[draw=drawColor,line width= 0.6pt,line join=round] (504.51, 30.69) --
	(504.51,355.85);
\definecolor{drawColor}{gray}{0.20}

\path[draw=drawColor,line width= 0.6pt,line join=round] ( 84.55,298.80) -- ( 84.55,302.44);

\path[draw=drawColor,line width= 0.6pt,line join=round] ( 84.55,294.95) -- ( 84.55,291.13);
\definecolor{fillColor}{RGB}{248,118,109}

\path[draw=drawColor,line width= 0.6pt,fill=fillColor] ( 53.05,298.80) --
	( 53.05,294.95) --
	(116.05,294.95) --
	(116.05,298.80) --
	( 53.05,298.80) --
	cycle;

\path[draw=drawColor,line width= 0.6pt,line join=round] (168.54,295.03) -- (168.54,299.63);

\path[draw=drawColor,line width= 0.6pt,line join=round] (168.54,289.95) -- (168.54,285.36);
\definecolor{fillColor}{RGB}{183,159,0}

\path[draw=drawColor,line width= 0.6pt,fill=fillColor] (137.05,295.03) --
	(137.05,289.95) --
	(200.04,289.95) --
	(200.04,295.03) --
	(137.05,295.03) --
	cycle;

\path[draw=drawColor,line width= 0.6pt,line join=round] (252.53,284.55) -- (252.53,292.72);

\path[draw=drawColor,line width= 0.6pt,line join=round] (252.53,276.08) -- (252.53,267.62);
\definecolor{fillColor}{RGB}{0,186,56}

\path[draw=drawColor,line width= 0.6pt,fill=fillColor] (221.04,284.55) --
	(221.04,276.08) --
	(284.03,276.08) --
	(284.03,284.55) --
	(221.04,284.55) --
	cycle;

\path[draw=drawColor,line width= 0.6pt,line join=round] (336.52,269.49) -- (336.52,281.07);

\path[draw=drawColor,line width= 0.6pt,line join=round] (336.52,256.94) -- (336.52,244.45);
\definecolor{fillColor}{RGB}{0,191,196}

\path[draw=drawColor,line width= 0.6pt,fill=fillColor] (305.03,269.49) --
	(305.03,256.94) --
	(368.02,256.94) --
	(368.02,269.49) --
	(305.03,269.49) --
	cycle;

\path[draw=drawColor,line width= 0.6pt,line join=round] (420.52,267.20) -- (420.52,280.25);

\path[draw=drawColor,line width= 0.6pt,line join=round] (420.52,253.63) -- (420.52,241.48);
\definecolor{fillColor}{RGB}{97,156,255}

\path[draw=drawColor,line width= 0.6pt,fill=fillColor] (389.02,267.20) --
	(389.02,253.63) --
	(452.01,253.63) --
	(452.01,267.20) --
	(389.02,267.20) --
	cycle;

\path[draw=drawColor,line width= 0.6pt,line join=round] (504.51,259.54) -- (504.51,272.93);

\path[draw=drawColor,line width= 0.6pt,line join=round] (504.51,244.60) -- (504.51,231.73);
\definecolor{fillColor}{RGB}{245,100,227}

\path[draw=drawColor,line width= 0.6pt,fill=fillColor] (473.01,259.54) --
	(473.01,244.60) --
	(536.00,244.60) --
	(536.00,259.54) --
	(473.01,259.54) --
	cycle;

\path[draw=drawColor,line width= 0.6pt,line join=round,line cap=round] ( 34.16, 30.69) rectangle (554.90,355.85);
\end{scope}
\begin{scope}
\definecolor{drawColor}{gray}{0.30}

\node[text=drawColor,anchor=base east,inner sep=0pt, outer sep=0pt, scale=  0.88] at ( 29.21,96.56) {0.4};

\node[text=drawColor,anchor=base east,inner sep=0pt, outer sep=0pt, scale=  0.88] at ( 29.21,175.38) {0.6};

\node[text=drawColor,anchor=base east,inner sep=0pt, outer sep=0pt, scale=  0.88] at ( 29.21,254.21) {0.8};

\node[text=drawColor,anchor=base east,inner sep=0pt, outer sep=0pt, scale=  0.88] at ( 29.21,333.04) {1.0};
\end{scope}
\begin{scope}
\definecolor{drawColor}{gray}{0.20}

\path[draw=drawColor,line width= 0.6pt,line join=round] ( 31.41,104.59) --
	( 34.16,104.59);

\path[draw=drawColor,line width= 0.6pt,line join=round] ( 31.41,183.41) --
	( 34.16,183.41);

\path[draw=drawColor,line width= 0.6pt,line join=round] ( 31.41,262.24) --
	( 34.16,262.24);

\path[draw=drawColor,line width= 0.6pt,line join=round] ( 31.41,341.07) --
	( 34.16,341.07);
\end{scope}
%
%
%
%
%
%
%
%
%
%
%
%
\begin{scope}
\definecolor{drawColor}{RGB}{0,0,0}

\node[text=drawColor,rotate= 90.00,anchor=base,inner sep=0pt, outer sep=0pt, scale=  1.10] at ( -50.08,193.27) {\small $\ell_A(\hat{\bA},\bA)$};
\end{scope}
\begin{scope}
\definecolor{fillColor}{RGB}{255,255,255}

\path[fill=fillColor] (565.90,136.80) rectangle (623.25,249.74);
\end{scope}
\begin{scope}
	\definecolor{drawColor}{RGB}{0,0,0}
	
	\node[text=drawColor,anchor=base,inner sep=0pt, outer sep=0pt, scale=  1.10] at (291.36,  -22.64) {\small $n=500$};
\end{scope}
\end{tikzpicture}
	\end{subfigure}\hfil
	\begin{subfigure}{0.25\textwidth}
		\begin{tikzpicture}[x=.22pt,y=.22pt]
		\hspace{-.5cm}
\definecolor{fillColor}{RGB}{255,255,255}
\begin{scope}
\definecolor{drawColor}{RGB}{255,255,255}
\definecolor{fillColor}{RGB}{255,255,255}

\end{scope}
\begin{scope}
\definecolor{fillColor}{RGB}{255,255,255}

\path[fill=fillColor] ( 34.16, 30.69) rectangle (554.90,355.85);
\definecolor{drawColor}{gray}{0.92}

\path[draw=drawColor,line width= 0.3pt,line join=round] ( 34.16, 65.17) --
	(554.90, 65.17);

\path[draw=drawColor,line width= 0.3pt,line join=round] ( 34.16,144.00) --
	(554.90,144.00);

\path[draw=drawColor,line width= 0.3pt,line join=round] ( 34.16,222.83) --
	(554.90,222.83);

\path[draw=drawColor,line width= 0.3pt,line join=round] ( 34.16,301.66) --
	(554.90,301.66);

\path[draw=drawColor,line width= 0.6pt,line join=round] ( 34.16,104.59) --
	(554.90,104.59);

\path[draw=drawColor,line width= 0.6pt,line join=round] ( 34.16,183.41) --
	(554.90,183.41);

\path[draw=drawColor,line width= 0.6pt,line join=round] ( 34.16,262.24) --
	(554.90,262.24);

\path[draw=drawColor,line width= 0.6pt,line join=round] ( 34.16,341.07) --
	(554.90,341.07);

\path[draw=drawColor,line width= 0.6pt,line join=round] ( 84.55, 30.69) --
	( 84.55,355.85);

\path[draw=drawColor,line width= 0.6pt,line join=round] (168.54, 30.69) --
	(168.54,355.85);

\path[draw=drawColor,line width= 0.6pt,line join=round] (252.53, 30.69) --
	(252.53,355.85);

\path[draw=drawColor,line width= 0.6pt,line join=round] (336.52, 30.69) --
	(336.52,355.85);

\path[draw=drawColor,line width= 0.6pt,line join=round] (420.52, 30.69) --
	(420.52,355.85);

\path[draw=drawColor,line width= 0.6pt,line join=round] (504.51, 30.69) --
	(504.51,355.85);
\definecolor{drawColor}{gray}{0.20}

\path[draw=drawColor,line width= 0.6pt,line join=round] ( 84.55,299.19) -- ( 84.55,303.15);

\path[draw=drawColor,line width= 0.6pt,line join=round] ( 84.55,295.08) -- ( 84.55,291.32);
\definecolor{fillColor}{RGB}{248,118,109}

\path[draw=drawColor,line width= 0.6pt,fill=fillColor] ( 53.05,299.19) --
	( 53.05,295.08) --
	(116.05,295.08) --
	(116.05,299.19) --
	( 53.05,299.19) --
	cycle;

\path[draw=drawColor,line width= 0.6pt,line join=round] (168.54,295.51) -- (168.54,300.81);

\path[draw=drawColor,line width= 0.6pt,line join=round] (168.54,289.91) -- (168.54,284.49);
\definecolor{fillColor}{RGB}{183,159,0}

\path[draw=drawColor,line width= 0.6pt,fill=fillColor] (137.05,295.51) --
	(137.05,289.91) --
	(200.04,289.91) --
	(200.04,295.51) --
	(137.05,295.51) --
	cycle;

\path[draw=drawColor,line width= 0.6pt,line join=round] (252.53,272.76) -- (252.53,285.46);

\path[draw=drawColor,line width= 0.6pt,line join=round] (252.53,259.55) -- (252.53,248.17);
\definecolor{fillColor}{RGB}{0,186,56}

\path[draw=drawColor,line width= 0.6pt,fill=fillColor] (221.04,272.76) --
	(221.04,259.55) --
	(284.03,259.55) --
	(284.03,272.76) --
	(221.04,272.76) --
	cycle;

\path[draw=drawColor,line width= 0.6pt,line join=round] (336.52,234.45) -- (336.52,256.07);

\path[draw=drawColor,line width= 0.6pt,line join=round] (336.52,212.62) -- (336.52,191.35);
\definecolor{fillColor}{RGB}{0,191,196}

\path[draw=drawColor,line width= 0.6pt,fill=fillColor] (305.03,234.45) --
	(305.03,212.62) --
	(368.02,212.62) --
	(368.02,234.45) --
	(305.03,234.45) --
	cycle;

\path[draw=drawColor,line width= 0.6pt,line join=round] (420.52,214.68) -- (420.52,241.22);

\path[draw=drawColor,line width= 0.6pt,line join=round] (420.52,188.13) -- (420.52,163.42);
\definecolor{fillColor}{RGB}{97,156,255}

\path[draw=drawColor,line width= 0.6pt,fill=fillColor] (389.02,214.68) --
	(389.02,188.13) --
	(452.01,188.13) --
	(452.01,214.68) --
	(389.02,214.68) --
	cycle;

\path[draw=drawColor,line width= 0.6pt,line join=round] (504.51,184.83) -- (504.51,213.48);

\path[draw=drawColor,line width= 0.6pt,line join=round] (504.51,154.06) -- (504.51,124.46);
\definecolor{fillColor}{RGB}{245,100,227}

\path[draw=drawColor,line width= 0.6pt,fill=fillColor] (473.01,184.83) --
	(473.01,154.06) --
	(536.00,154.06) --
	(536.00,184.83) --
	(473.01,184.83) --
	cycle;

\path[draw=drawColor,line width= 0.6pt,line join=round,line cap=round] ( 34.16, 30.69) rectangle (554.90,355.85);
\end{scope}
\begin{scope}
\definecolor{drawColor}{gray}{0.30}

\node[text=drawColor,anchor=base east,inner sep=0pt, outer sep=0pt, scale=  0.88] at ( 29.21,101.56) {0.4};

\node[text=drawColor,anchor=base east,inner sep=0pt, outer sep=0pt, scale=  0.88] at ( 29.21,180.38) {0.6};

\node[text=drawColor,anchor=base east,inner sep=0pt, outer sep=0pt, scale=  0.88] at ( 29.21,259.21) {0.8};

\node[text=drawColor,anchor=base east,inner sep=0pt, outer sep=0pt, scale=  0.88] at ( 29.21,338.04) {1.0};
\end{scope}
\begin{scope}
\definecolor{drawColor}{gray}{0.20}

\path[draw=drawColor,line width= 0.6pt,line join=round] ( 31.41,104.59) --
	( 34.16,104.59);

\path[draw=drawColor,line width= 0.6pt,line join=round] ( 31.41,183.41) --
	( 34.16,183.41);

\path[draw=drawColor,line width= 0.6pt,line join=round] ( 31.41,262.24) --
	( 34.16,262.24);

\path[draw=drawColor,line width= 0.6pt,line join=round] ( 31.41,341.07) --
	( 34.16,341.07);
\end{scope}
\begin{scope}
\definecolor{drawColor}{gray}{0.20}

\path[draw=drawColor,line width= 0.6pt,line join=round] ( 84.55, 27.94) --
	( 84.55, 30.69);

\path[draw=drawColor,line width= 0.6pt,line join=round] (168.54, 27.94) --
	(168.54, 30.69);

\path[draw=drawColor,line width= 0.6pt,line join=round] (252.53, 27.94) --
	(252.53, 30.69);

\path[draw=drawColor,line width= 0.6pt,line join=round] (336.52, 27.94) --
	(336.52, 30.69);

\path[draw=drawColor,line width= 0.6pt,line join=round] (420.52, 27.94) --
	(420.52, 30.69);

\path[draw=drawColor,line width= 0.6pt,line join=round] (504.51, 27.94) --
	(504.51, 30.69);
\end{scope}
\begin{scope}
	\definecolor{drawColor}{RGB}{0,0,0}
	
	\node[text=drawColor,rotate= 90.00,anchor=base,inner sep=0pt, outer sep=0pt, scale=  1.10] at ( -50.08,193.27) {\small $\ell_A(\hat{\bA},\bA)$};
\end{scope}
\begin{scope}
	\definecolor{fillColor}{RGB}{255,255,255}
	
	\path[fill=fillColor] (565.90,136.80) rectangle (623.25,249.74);
\end{scope}
\begin{scope}
	\definecolor{drawColor}{RGB}{0,0,0}
	
	\node[text=drawColor,anchor=base,inner sep=0pt, outer sep=0pt, scale=  1.10] at (291.36,  -22.64) {\small $n=1500$};
\end{scope}
\end{tikzpicture}
	\end{subfigure}\hfil
	\begin{subfigure}{0.25\textwidth}
		\begin{tikzpicture}[x=.22pt,y=.22pt]
\definecolor{fillColor}{RGB}{255,255,255}
\begin{scope}
\definecolor{drawColor}{RGB}{255,255,255}
\definecolor{fillColor}{RGB}{255,255,255}

\end{scope}
\begin{scope}
\definecolor{fillColor}{RGB}{255,255,255}

\path[fill=fillColor] ( 34.16, 30.69) rectangle (554.90,355.85);
\definecolor{drawColor}{gray}{0.92}

\path[draw=drawColor,line width= 0.3pt,line join=round] ( 34.16, 65.17) --
	(554.90, 65.17);

\path[draw=drawColor,line width= 0.3pt,line join=round] ( 34.16,144.00) --
	(554.90,144.00);

\path[draw=drawColor,line width= 0.3pt,line join=round] ( 34.16,222.83) --
	(554.90,222.83);

\path[draw=drawColor,line width= 0.3pt,line join=round] ( 34.16,301.66) --
	(554.90,301.66);

\path[draw=drawColor,line width= 0.6pt,line join=round] ( 34.16,104.59) --
	(554.90,104.59);

\path[draw=drawColor,line width= 0.6pt,line join=round] ( 34.16,183.41) --
	(554.90,183.41);

\path[draw=drawColor,line width= 0.6pt,line join=round] ( 34.16,262.24) --
	(554.90,262.24);

\path[draw=drawColor,line width= 0.6pt,line join=round] ( 34.16,341.07) --
	(554.90,341.07);

\path[draw=drawColor,line width= 0.6pt,line join=round] ( 84.55, 30.69) --
	( 84.55,355.85);

\path[draw=drawColor,line width= 0.6pt,line join=round] (168.54, 30.69) --
	(168.54,355.85);

\path[draw=drawColor,line width= 0.6pt,line join=round] (252.53, 30.69) --
	(252.53,355.85);

\path[draw=drawColor,line width= 0.6pt,line join=round] (336.52, 30.69) --
	(336.52,355.85);

\path[draw=drawColor,line width= 0.6pt,line join=round] (420.52, 30.69) --
	(420.52,355.85);

\path[draw=drawColor,line width= 0.6pt,line join=round] (504.51, 30.69) --
	(504.51,355.85);
\definecolor{drawColor}{gray}{0.20}

\path[draw=drawColor,line width= 0.6pt,line join=round] ( 84.55,298.38) -- ( 84.55,301.78);

\path[draw=drawColor,line width= 0.6pt,line join=round] ( 84.55,294.97) -- ( 84.55,291.58);
\definecolor{fillColor}{RGB}{248,118,109}

\path[draw=drawColor,line width= 0.6pt,fill=fillColor] ( 53.05,298.38) --
	( 53.05,294.97) --
	(116.05,294.97) --
	(116.05,298.38) --
	( 53.05,298.38) --
	cycle;

\path[draw=drawColor,line width= 0.6pt,line join=round] (168.54,295.55) -- (168.54,299.94);

\path[draw=drawColor,line width= 0.6pt,line join=round] (168.54,290.93) -- (168.54,286.57);
\definecolor{fillColor}{RGB}{183,159,0}

\path[draw=drawColor,line width= 0.6pt,fill=fillColor] (137.05,295.55) --
	(137.05,290.93) --
	(200.04,290.93) --
	(200.04,295.55) --
	(137.05,295.55) --
	cycle;

\path[draw=drawColor,line width= 0.6pt,line join=round] (252.53,266.41) -- (252.53,282.56);

\path[draw=drawColor,line width= 0.6pt,line join=round] (252.53,248.57) -- (252.53,231.58);
\definecolor{fillColor}{RGB}{0,186,56}

\path[draw=drawColor,line width= 0.6pt,fill=fillColor] (221.04,266.41) --
	(221.04,248.57) --
	(284.03,248.57) --
	(284.03,266.41) --
	(221.04,266.41) --
	cycle;

\path[draw=drawColor,line width= 0.6pt,line join=round] (336.52,218.19) -- (336.52,249.55);

\path[draw=drawColor,line width= 0.6pt,line join=round] (336.52,184.31) -- (336.52,150.58);
\definecolor{fillColor}{RGB}{0,191,196}

\path[draw=drawColor,line width= 0.6pt,fill=fillColor] (305.03,218.19) --
	(305.03,184.31) --
	(368.02,184.31) --
	(368.02,218.19) --
	(305.03,218.19) --
	cycle;

\path[draw=drawColor,line width= 0.6pt,line join=round] (420.52,189.28) -- (420.52,220.91);

\path[draw=drawColor,line width= 0.6pt,line join=round] (420.52,150.35) -- (420.52,112.85);
\definecolor{fillColor}{RGB}{97,156,255}

\path[draw=drawColor,line width= 0.6pt,fill=fillColor] (389.02,189.28) --
	(389.02,150.35) --
	(452.01,150.35) --
	(452.01,189.28) --
	(389.02,189.28) --
	cycle;

\path[draw=drawColor,line width= 0.6pt,line join=round] (504.51,145.09) -- (504.51,185.27);

\path[draw=drawColor,line width= 0.6pt,line join=round] (504.51,104.66) -- (504.51, 68.10);
\definecolor{fillColor}{RGB}{245,100,227}

\path[draw=drawColor,line width= 0.6pt,fill=fillColor] (473.01,145.09) --
	(473.01,104.66) --
	(536.00,104.66) --
	(536.00,145.09) --
	(473.01,145.09) --
	cycle;

\path[draw=drawColor,line width= 0.6pt,line join=round,line cap=round] ( 34.16, 30.69) rectangle (554.90,355.85);
\end{scope}
\begin{scope}
\definecolor{drawColor}{gray}{0.30}

\node[text=drawColor,anchor=base east,inner sep=0pt, outer sep=0pt, scale=  0.88] at ( 29.21,96.56) {0.4};

\node[text=drawColor,anchor=base east,inner sep=0pt, outer sep=0pt, scale=  0.88] at ( 29.21,175.38) {0.6};

\node[text=drawColor,anchor=base east,inner sep=0pt, outer sep=0pt, scale=  0.88] at ( 29.21,254.21) {0.8};

\node[text=drawColor,anchor=base east,inner sep=0pt, outer sep=0pt, scale=  0.88] at ( 29.21,333.04) {1.0};
\end{scope}
\begin{scope}
\definecolor{drawColor}{gray}{0.20}

\path[draw=drawColor,line width= 0.6pt,line join=round] ( 31.41,104.59) --
	( 34.16,104.59);

\path[draw=drawColor,line width= 0.6pt,line join=round] ( 31.41,183.41) --
	( 34.16,183.41);

\path[draw=drawColor,line width= 0.6pt,line join=round] ( 31.41,262.24) --
	( 34.16,262.24);

\path[draw=drawColor,line width= 0.6pt,line join=round] ( 31.41,341.07) --
	( 34.16,341.07);
\end{scope}
\begin{scope}
\definecolor{drawColor}{gray}{0.20}

\path[draw=drawColor,line width= 0.6pt,line join=round] ( 84.55, 27.94) --
	( 84.55, 30.69);

\path[draw=drawColor,line width= 0.6pt,line join=round] (168.54, 27.94) --
	(168.54, 30.69);

\path[draw=drawColor,line width= 0.6pt,line join=round] (252.53, 27.94) --
	(252.53, 30.69);

\path[draw=drawColor,line width= 0.6pt,line join=round] (336.52, 27.94) --
	(336.52, 30.69);

\path[draw=drawColor,line width= 0.6pt,line join=round] (420.52, 27.94) --
	(420.52, 30.69);

\path[draw=drawColor,line width= 0.6pt,line join=round] (504.51, 27.94) --
	(504.51, 30.69);
\end{scope}
\begin{scope}
	\definecolor{drawColor}{RGB}{0,0,0}
	
	\node[text=drawColor,rotate= 90.00,anchor=base,inner sep=0pt, outer sep=0pt, scale=  1.10] at ( -50.08,193.27) {\small $\ell_A(\hat{\bA},\bA)$};
\end{scope}
\begin{scope}
	\definecolor{fillColor}{RGB}{255,255,255}
	
	\path[fill=fillColor] (565.90,136.80) rectangle (623.25,249.74);
\end{scope}
\begin{scope}
	\definecolor{drawColor}{RGB}{0,0,0}
	
	\node[text=drawColor,anchor=base,inner sep=0pt, outer sep=0pt, scale=  1.10] at (291.36,  -22.64) {\small $n=2000$};
\end{scope}
\end{tikzpicture}
	\end{subfigure}
	\caption{Comparison of ICA methods for $d=25$. In each panel, from left to right, we respectively plot the results using FastICA-neg-entropy and  JADE, followed by FastICA initialized respectively by naive matricization, random unit vectors, randomly sliced matricization, and Algorithm~\ref{alg:init}.}
	\label{fig:lo-dim}
\end{figure}

Figure~\ref{fig:lo-dim} shows the benefit of using randomization to initialize iteration algorithms for ICA. The standard algorithms of FastICA with negative entropy and JADE, when used with the deflation variant (as provided in \texttt{R} packages {\sf FastICA} and {\sf JADE}) do not show desirable performance even when the sample size is high. Similarly, the naive matricization method, which applies SVD to the $(1,2)$-matricization of the fourth-order cumulant tensor, does not perform well. As we point out in Section~\ref{sec:inf-th}, this is because there is no eigengap in this matrix, and consequently SVD is not unique. Algorithmically, the method ran into convergence issues for the same reason.

In contrast, the randomization-based methods perform increasingly well for moderately large sample sizes. We compared three such initialization schemes: a random unit vector, SVD of a random linear combination of the $(1,2)$ slices of the fourth-order moment tensor (as done by \cite{anand2014tensor}), as well as our method based on the improved fourth-order moment estimate (Algorithm \ref{alg:init}). Note that the improved moment estimates perform better than the other two initialization algorithms. This trend is more pronounced in the high dimensional regime, given in Figures~\ref{fig:ica-0}, \ref{fig:ica-1}, and \ref{fig:trunc}.


\paragraph{Distributional Properties.} Next, we assess the accuracy of the normal approximation to the distribution of the estimated mixing matrix. To fix ideas, we focus on the distribution of $\langle \hat{\ba}_1,\ba_2\rangle$. Figure~\ref{fig:lo-dim-clt} shows the histogram constructed from 200 simulation runs each for sample size $n=400$, $800$, and $1200$. Since the proposed initialized method (Algorithm~\ref{alg:init}) is one of the better-performing methods in Figure~\ref{fig:lo-dim}, we plot the histograms of only that estimator.  We overlay the histograms with the normal distribution centered with the sampled mean and variance given by Theorem \ref{th:ica-clt}.

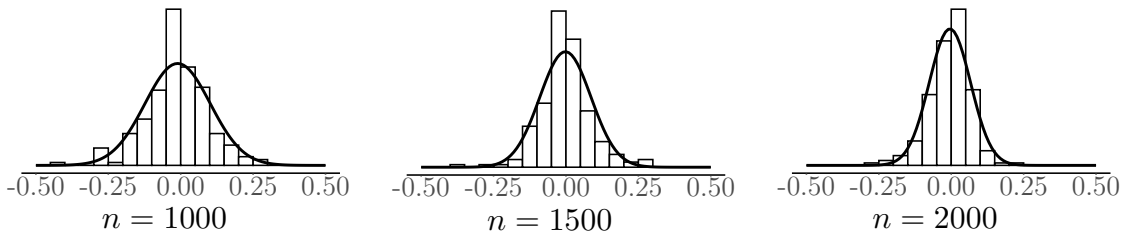
\begin{figure}[htb]
	\centering
	\begin{subfigure}{0.25\textwidth}
		\begin{tikzpicture}[x=.2pt,y=.2pt]
\definecolor{fillColor}{RGB}{255,255,255}
\begin{scope}
\definecolor{drawColor}{RGB}{255,255,255}
\definecolor{fillColor}{RGB}{255,255,255}

\end{scope}
\begin{scope}
\definecolor{fillColor}{RGB}{255,255,255}

\path[fill=fillColor] ( 20.71, 30.69) rectangle (623.25,355.85);
\definecolor{drawColor}{RGB}{0,0,0}

\path[draw=drawColor,line width= 0.6pt,fill=fillColor] ( 48.10, 45.47) rectangle ( 75.49, 45.47);

\path[draw=drawColor,line width= 0.6pt,fill=fillColor] ( 75.49, 45.47) rectangle (102.88, 50.94);

\path[draw=drawColor,line width= 0.6pt,fill=fillColor] (102.88, 45.47) rectangle (130.27, 45.47);

\path[draw=drawColor,line width= 0.6pt,fill=fillColor] (130.27, 45.47) rectangle (157.65, 45.47);

\path[draw=drawColor,line width= 0.6pt,fill=fillColor] (157.65, 45.47) rectangle (185.04, 78.31);

\path[draw=drawColor,line width= 0.6pt,fill=fillColor] (185.04, 45.47) rectangle (212.43, 50.94);

\path[draw=drawColor,line width= 0.6pt,fill=fillColor] (212.43, 45.47) rectangle (239.82,105.68);

\path[draw=drawColor,line width= 0.6pt,fill=fillColor] (239.82, 45.47) rectangle (267.21,133.05);

\path[draw=drawColor,line width= 0.6pt,fill=fillColor] (267.21, 45.47) rectangle (294.59,187.79);

\path[draw=drawColor,line width= 0.6pt,fill=fillColor] (294.59, 45.47) rectangle (321.98,341.07);

\path[draw=drawColor,line width= 0.6pt,fill=fillColor] (321.98, 45.47) rectangle (349.37,231.59);

\path[draw=drawColor,line width= 0.6pt,fill=fillColor] (349.37, 45.47) rectangle (376.76,193.27);

\path[draw=drawColor,line width= 0.6pt,fill=fillColor] (376.76, 45.47) rectangle (404.15,105.68);

\path[draw=drawColor,line width= 0.6pt,fill=fillColor] (404.15, 45.47) rectangle (431.53, 83.79);

\path[draw=drawColor,line width= 0.6pt,fill=fillColor] (431.53, 45.47) rectangle (458.92, 61.89);

\path[draw=drawColor,line width= 0.6pt,fill=fillColor] (458.92, 45.47) rectangle (486.31, 56.41);

\path[draw=drawColor,line width= 0.6pt,fill=fillColor] (486.31, 45.47) rectangle (513.70, 45.47);

\path[draw=drawColor,line width= 0.6pt,fill=fillColor] (513.70, 45.47) rectangle (541.09, 45.47);

\path[draw=drawColor,line width= 0.6pt,fill=fillColor] (541.09, 45.47) rectangle (568.47, 45.47);

\path[draw=drawColor,line width= 0.6pt,fill=fillColor] (568.47, 45.47) rectangle (595.86, 45.47);

\path[draw=drawColor,line width= 1.1pt,line join=round] ( 48.10, 45.48) --
	( 53.58, 45.49) --
	( 59.06, 45.50) --
	( 64.54, 45.51) --
	( 70.01, 45.53) --
	( 75.49, 45.56) --
	( 80.97, 45.60) --
	( 86.45, 45.66) --
	( 91.92, 45.73) --
	( 97.40, 45.83) --
	(102.88, 45.96) --
	(108.36, 46.13) --
	(113.83, 46.36) --
	(119.31, 46.66) --
	(124.79, 47.04) --
	(130.27, 47.53) --
	(135.74, 48.15) --
	(141.22, 48.94) --
	(146.70, 49.91) --
	(152.18, 51.12) --
	(157.65, 52.59) --
	(163.13, 54.38) --
	(168.61, 56.53) --
	(174.09, 59.09) --
	(179.56, 62.11) --
	(185.04, 65.65) --
	(190.52, 69.74) --
	(196.00, 74.43) --
	(201.47, 79.76) --
	(206.95, 85.75) --
	(212.43, 92.42) --
	(217.91, 99.76) --
	(223.39,107.75) --
	(228.86,116.37) --
	(234.34,125.53) --
	(239.82,135.18) --
	(245.30,145.20) --
	(250.77,155.47) --
	(256.25,165.84) --
	(261.73,176.16) --
	(267.21,186.24) --
	(272.68,195.92) --
	(278.16,205.00) --
	(283.64,213.30) --
	(289.12,220.65) --
	(294.59,226.89) --
	(300.07,231.87) --
	(305.55,235.48) --
	(311.03,237.65) --
	(316.50,238.32) --
	(321.98,237.47) --
	(327.46,235.12) --
	(332.94,231.34) --
	(338.41,226.20) --
	(343.89,219.82) --
	(349.37,212.35) --
	(354.85,203.94) --
	(360.32,194.78) --
	(365.80,185.04) --
	(371.28,174.92) --
	(376.76,164.59) --
	(382.24,154.22) --
	(387.71,143.97) --
	(393.19,133.99) --
	(398.67,124.40) --
	(404.15,115.30) --
	(409.62,106.76) --
	(415.10, 98.84) --
	(420.58, 91.58) --
	(426.06, 84.99) --
	(431.53, 79.08) --
	(437.01, 73.83) --
	(442.49, 69.21) --
	(447.97, 65.19) --
	(453.44, 61.72) --
	(458.92, 58.76) --
	(464.40, 56.25) --
	(469.88, 54.14) --
	(475.35, 52.40) --
	(480.83, 50.96) --
	(486.31, 49.78) --
	(491.79, 48.83) --
	(497.26, 48.07) --
	(502.74, 47.47) --
	(508.22, 46.99) --
	(513.70, 46.62) --
	(519.17, 46.33) --
	(524.65, 46.11) --
	(530.13, 45.94) --
	(535.61, 45.81) --
	(541.09, 45.72) --
	(546.56, 45.65) --
	(552.04, 45.60) --
	(557.52, 45.56) --
	(563.00, 45.53) --
	(568.47, 45.51) --
	(573.95, 45.50) --
	(579.43, 45.49) --
	(584.91, 45.48) --
	(590.38, 45.48) --
	(595.86, 45.47);
\end{scope}
\begin{scope}
\definecolor{drawColor}{RGB}{0,0,0}

\end{scope}
\begin{scope}
\definecolor{drawColor}{RGB}{0,0,0}

\path[draw=drawColor,line width= 0.6pt,line join=round] ( 20.71, 30.69) --
	(623.25, 30.69);
\end{scope}
\begin{scope}
\definecolor{drawColor}{gray}{0.20}

\path[draw=drawColor,line width= 0.6pt,line join=round] ( 48.10, 27.94) --
	( 48.10, 30.69);

\path[draw=drawColor,line width= 0.6pt,line join=round] (185.04, 27.94) --
	(185.04, 30.69);

\path[draw=drawColor,line width= 0.6pt,line join=round] (321.98, 27.94) --
	(321.98, 30.69);

\path[draw=drawColor,line width= 0.6pt,line join=round] (458.92, 27.94) --
	(458.92, 30.69);

\path[draw=drawColor,line width= 0.6pt,line join=round] (595.86, 27.94) --
	(595.86, 30.69);
\end{scope}
\begin{scope}
\definecolor{drawColor}{gray}{0.30}

\node[text=drawColor,anchor=base,inner sep=0pt, outer sep=0pt, scale=  0.88] at ( 48.10, -9.68) {-0.50};

\node[text=drawColor,anchor=base,inner sep=0pt, outer sep=0pt, scale=  0.88] at (185.04, -9.68) {-0.25};

\node[text=drawColor,anchor=base,inner sep=0pt, outer sep=0pt, scale=  0.88] at (321.98, -9.68) {0.00};

\node[text=drawColor,anchor=base,inner sep=0pt, outer sep=0pt, scale=  0.88] at (458.92, -9.68) {0.25};

\node[text=drawColor,anchor=base,inner sep=0pt, outer sep=0pt, scale=  0.88] at (595.86, -9.68) {0.50};
\end{scope}
\begin{scope}
	\definecolor{drawColor}{RGB}{0,0,0}
	
	\node[text=drawColor,anchor=base,inner sep=0pt, outer sep=0pt, scale=  1.10] at (291.36,  -72.64) {\small $n=1000$};
\end{scope}
\end{tikzpicture}
	\end{subfigure}\hfil
	\begin{subfigure}{0.25\textwidth}
		\begin{tikzpicture}[x=.2pt,y=.2pt]
\definecolor{fillColor}{RGB}{255,255,255}
\begin{scope}
\definecolor{drawColor}{RGB}{255,255,255}
\definecolor{fillColor}{RGB}{255,255,255}

\path[draw=drawColor,line width= 0.6pt,line join=round,line cap=round,fill=fillColor] (  0.00,  0.00) rectangle (628.75,361.35);
\end{scope}
\begin{scope}
\definecolor{fillColor}{RGB}{255,255,255}

\path[fill=fillColor] ( 20.71, 30.69) rectangle (623.25,355.85);
\definecolor{drawColor}{RGB}{0,0,0}

\path[draw=drawColor,line width= 0.6pt,fill=fillColor] ( 48.10, 45.47) rectangle ( 75.49, 45.47);

\path[draw=drawColor,line width= 0.6pt,fill=fillColor] ( 75.49, 45.47) rectangle (102.88, 45.47);

\path[draw=drawColor,line width= 0.6pt,fill=fillColor] (102.88, 45.47) rectangle (130.27, 50.31);

\path[draw=drawColor,line width= 0.6pt,fill=fillColor] (130.27, 45.47) rectangle (157.65, 45.47);

\path[draw=drawColor,line width= 0.6pt,fill=fillColor] (157.65, 45.47) rectangle (185.04, 50.31);

\path[draw=drawColor,line width= 0.6pt,fill=fillColor] (185.04, 45.47) rectangle (212.43, 50.31);

\path[draw=drawColor,line width= 0.6pt,fill=fillColor] (212.43, 45.47) rectangle (239.82, 60.00);

\path[draw=drawColor,line width= 0.6pt,fill=fillColor] (239.82, 45.47) rectangle (267.21,123.00);

\path[draw=drawColor,line width= 0.6pt,fill=fillColor] (267.21, 45.47) rectangle (294.59,166.62);

\path[draw=drawColor,line width= 0.6pt,fill=fillColor] (294.59, 45.47) rectangle (321.98,341.07);

\path[draw=drawColor,line width= 0.6pt,fill=fillColor] (321.98, 45.47) rectangle (349.37,287.76);

\path[draw=drawColor,line width= 0.6pt,fill=fillColor] (349.37, 45.47) rectangle (376.76,152.08);

\path[draw=drawColor,line width= 0.6pt,fill=fillColor] (376.76, 45.47) rectangle (404.15, 93.93);

\path[draw=drawColor,line width= 0.6pt,fill=fillColor] (404.15, 45.47) rectangle (431.53, 69.70);

\path[draw=drawColor,line width= 0.6pt,fill=fillColor] (431.53, 45.47) rectangle (458.92, 55.16);

\path[draw=drawColor,line width= 0.6pt,fill=fillColor] (458.92, 45.47) rectangle (486.31, 60.00);

\path[draw=drawColor,line width= 0.6pt,fill=fillColor] (486.31, 45.47) rectangle (513.70, 45.47);

\path[draw=drawColor,line width= 0.6pt,fill=fillColor] (513.70, 45.47) rectangle (541.09, 45.47);

\path[draw=drawColor,line width= 0.6pt,fill=fillColor] (541.09, 45.47) rectangle (568.47, 45.47);

\path[draw=drawColor,line width= 0.6pt,fill=fillColor] (568.47, 45.47) rectangle (595.86, 45.47);

\path[draw=drawColor,line width= 1.1pt,line join=round] ( 48.10, 45.47) --
	( 53.58, 45.47) --
	( 59.06, 45.47) --
	( 64.54, 45.47) --
	( 70.01, 45.47) --
	( 75.49, 45.47) --
	( 80.97, 45.47) --
	( 86.45, 45.47) --
	( 91.92, 45.47) --
	( 97.40, 45.47) --
	(102.88, 45.47) --
	(108.36, 45.48) --
	(113.83, 45.49) --
	(119.31, 45.50) --
	(124.79, 45.53) --
	(130.27, 45.56) --
	(135.74, 45.61) --
	(141.22, 45.69) --
	(146.70, 45.80) --
	(152.18, 45.97) --
	(157.65, 46.21) --
	(163.13, 46.55) --
	(168.61, 47.02) --
	(174.09, 47.67) --
	(179.56, 48.55) --
	(185.04, 49.73) --
	(190.52, 51.29) --
	(196.00, 53.31) --
	(201.47, 55.91) --
	(206.95, 59.18) --
	(212.43, 63.25) --
	(217.91, 68.23) --
	(223.39, 74.24) --
	(228.86, 81.37) --
	(234.34, 89.69) --
	(239.82, 99.26) --
	(245.30,110.07) --
	(250.77,122.06) --
	(256.25,135.12) --
	(261.73,149.07) --
	(267.21,163.67) --
	(272.68,178.62) --
	(278.16,193.55) --
	(283.64,208.06) --
	(289.12,221.72) --
	(294.59,234.10) --
	(300.07,244.79) --
	(305.55,253.40) --
	(311.03,259.62) --
	(316.50,263.23) --
	(321.98,264.09) --
	(327.46,262.16) --
	(332.94,257.51) --
	(338.41,250.33) --
	(343.89,240.87) --
	(349.37,229.48) --
	(354.85,216.55) --
	(360.32,202.51) --
	(365.80,187.79) --
	(371.28,172.81) --
	(376.76,157.95) --
	(382.24,143.57) --
	(387.71,129.94) --
	(393.19,117.27) --
	(398.67,105.73) --
	(404.15, 95.40) --
	(409.62, 86.32) --
	(415.10, 78.46) --
	(420.58, 71.78) --
	(426.06, 66.18) --
	(431.53, 61.57) --
	(437.01, 57.82) --
	(442.49, 54.83) --
	(447.97, 52.47) --
	(453.44, 50.64) --
	(458.92, 49.24) --
	(464.40, 48.18) --
	(469.88, 47.39) --
	(475.35, 46.82) --
	(480.83, 46.40) --
	(486.31, 46.11) --
	(491.79, 45.90) --
	(497.26, 45.75) --
	(502.74, 45.66) --
	(508.22, 45.59) --
	(513.70, 45.55) --
	(519.17, 45.52) --
	(524.65, 45.50) --
	(530.13, 45.49) --
	(535.61, 45.48) --
	(541.09, 45.47) --
	(546.56, 45.47) --
	(552.04, 45.47) --
	(557.52, 45.47) --
	(563.00, 45.47) --
	(568.47, 45.47) --
	(573.95, 45.47) --
	(579.43, 45.47) --
	(584.91, 45.47) --
	(590.38, 45.47) --
	(595.86, 45.47);
\end{scope}
\begin{scope}
\definecolor{drawColor}{RGB}{0,0,0}

\end{scope}
\begin{scope}
\definecolor{drawColor}{RGB}{0,0,0}

\path[draw=drawColor,line width= 0.6pt,line join=round] ( 20.71, 30.69) --
	(623.25, 30.69);
\end{scope}
\begin{scope}
\definecolor{drawColor}{gray}{0.20}

\path[draw=drawColor,line width= 0.6pt,line join=round] ( 48.10, 27.94) --
	( 48.10, 30.69);

\path[draw=drawColor,line width= 0.6pt,line join=round] (185.04, 27.94) --
	(185.04, 30.69);

\path[draw=drawColor,line width= 0.6pt,line join=round] (321.98, 27.94) --
	(321.98, 30.69);

\path[draw=drawColor,line width= 0.6pt,line join=round] (458.92, 27.94) --
	(458.92, 30.69);

\path[draw=drawColor,line width= 0.6pt,line join=round] (595.86, 27.94) --
	(595.86, 30.69);
\end{scope}
\begin{scope}
\definecolor{drawColor}{gray}{0.30}

\node[text=drawColor,anchor=base,inner sep=0pt, outer sep=0pt, scale=  0.88] at ( 48.10, -9.68) {-0.50};

\node[text=drawColor,anchor=base,inner sep=0pt, outer sep=0pt, scale=  0.88] at (185.04, -9.68) {-0.25};

\node[text=drawColor,anchor=base,inner sep=0pt, outer sep=0pt, scale=  0.88] at (321.98, -9.68) {0.00};

\node[text=drawColor,anchor=base,inner sep=0pt, outer sep=0pt, scale=  0.88] at (458.92, -9.68) {0.25};

\node[text=drawColor,anchor=base,inner sep=0pt, outer sep=0pt, scale=  0.88] at (595.86, -9.68) {0.50};
\end{scope}
\begin{scope}
	\definecolor{drawColor}{RGB}{0,0,0}
	
	\node[text=drawColor,anchor=base,inner sep=0pt, outer sep=0pt, scale=  1.10] at (291.36,  -72.64) {\small $n=1500$};
\end{scope}
\end{tikzpicture}
	\end{subfigure}\hfil 
	\begin{subfigure}{0.25\textwidth}
		\begin{tikzpicture}[x=.2pt,y=.2pt]
\definecolor{fillColor}{RGB}{255,255,255}
\begin{scope}
\definecolor{drawColor}{RGB}{255,255,255}
\definecolor{fillColor}{RGB}{255,255,255}

\end{scope}
\begin{scope}
\definecolor{fillColor}{RGB}{255,255,255}

\path[fill=fillColor] ( 20.71, 30.69) rectangle (623.25,355.85);
\definecolor{drawColor}{RGB}{0,0,0}

\path[draw=drawColor,line width= 0.6pt,fill=fillColor] ( 48.10, 45.47) rectangle ( 75.49, 45.47);

\path[draw=drawColor,line width= 0.6pt,fill=fillColor] ( 75.49, 45.47) rectangle (102.88, 45.47);

\path[draw=drawColor,line width= 0.6pt,fill=fillColor] (102.88, 45.47) rectangle (130.27, 45.47);

\path[draw=drawColor,line width= 0.6pt,fill=fillColor] (130.27, 45.47) rectangle (157.65, 45.47);

\path[draw=drawColor,line width= 0.6pt,fill=fillColor] (157.65, 45.47) rectangle (185.04, 50.08);

\path[draw=drawColor,line width= 0.6pt,fill=fillColor] (185.04, 45.47) rectangle (212.43, 54.70);

\path[draw=drawColor,line width= 0.6pt,fill=fillColor] (212.43, 45.47) rectangle (239.82, 63.94);

\path[draw=drawColor,line width= 0.6pt,fill=fillColor] (239.82, 45.47) rectangle (267.21, 91.65);

\path[draw=drawColor,line width= 0.6pt,fill=fillColor] (267.21, 45.47) rectangle (294.59,179.41);

\path[draw=drawColor,line width= 0.6pt,fill=fillColor] (294.59, 45.47) rectangle (321.98,281.03);

\path[draw=drawColor,line width= 0.6pt,fill=fillColor] (321.98, 45.47) rectangle (349.37,341.07);

\path[draw=drawColor,line width= 0.6pt,fill=fillColor] (349.37, 45.47) rectangle (376.76,188.65);

\path[draw=drawColor,line width= 0.6pt,fill=fillColor] (376.76, 45.47) rectangle (404.15, 73.18);

\path[draw=drawColor,line width= 0.6pt,fill=fillColor] (404.15, 45.47) rectangle (431.53, 50.08);

\path[draw=drawColor,line width= 0.6pt,fill=fillColor] (431.53, 45.47) rectangle (458.92, 50.08);

\path[draw=drawColor,line width= 0.6pt,fill=fillColor] (458.92, 45.47) rectangle (486.31, 45.47);

\path[draw=drawColor,line width= 0.6pt,fill=fillColor] (486.31, 45.47) rectangle (513.70, 45.47);

\path[draw=drawColor,line width= 0.6pt,fill=fillColor] (513.70, 45.47) rectangle (541.09, 45.47);

\path[draw=drawColor,line width= 0.6pt,fill=fillColor] (541.09, 45.47) rectangle (568.47, 45.47);

\path[draw=drawColor,line width= 0.6pt,fill=fillColor] (568.47, 45.47) rectangle (595.86, 45.47);

\path[draw=drawColor,line width= 1.1pt,line join=round] ( 48.10, 45.47) --
	( 53.58, 45.47) --
	( 59.06, 45.47) --
	( 64.54, 45.47) --
	( 70.01, 45.47) --
	( 75.49, 45.47) --
	( 80.97, 45.47) --
	( 86.45, 45.47) --
	( 91.92, 45.47) --
	( 97.40, 45.47) --
	(102.88, 45.47) --
	(108.36, 45.47) --
	(113.83, 45.47) --
	(119.31, 45.47) --
	(124.79, 45.47) --
	(130.27, 45.47) --
	(135.74, 45.47) --
	(141.22, 45.47) --
	(146.70, 45.48) --
	(152.18, 45.49) --
	(157.65, 45.51) --
	(163.13, 45.55) --
	(168.61, 45.61) --
	(174.09, 45.72) --
	(179.56, 45.89) --
	(185.04, 46.16) --
	(190.52, 46.58) --
	(196.00, 47.21) --
	(201.47, 48.16) --
	(206.95, 49.54) --
	(212.43, 51.51) --
	(217.91, 54.26) --
	(223.39, 58.00) --
	(228.86, 62.99) --
	(234.34, 69.49) --
	(239.82, 77.76) --
	(245.30, 88.03) --
	(250.77,100.48) --
	(256.25,115.17) --
	(261.73,132.08) --
	(267.21,150.99) --
	(272.68,171.53) --
	(278.16,193.13) --
	(283.64,215.06) --
	(289.12,236.46) --
	(294.59,256.37) --
	(300.07,273.83) --
	(305.55,287.91) --
	(311.03,297.85) --
	(316.50,303.09) --
	(321.98,303.31) --
	(327.46,298.51) --
	(332.94,288.96) --
	(338.41,275.22) --
	(343.89,258.02) --
	(349.37,238.29) --
	(354.85,216.98) --
	(360.32,195.06) --
	(365.80,173.40) --
	(371.28,152.74) --
	(376.76,133.67) --
	(382.24,116.58) --
	(387.71,101.68) --
	(393.19, 89.04) --
	(398.67, 78.58) --
	(404.15, 70.15) --
	(409.62, 63.50) --
	(415.10, 58.39) --
	(420.58, 54.54) --
	(426.06, 51.72) --
	(431.53, 49.69) --
	(437.01, 48.26) --
	(442.49, 47.28) --
	(447.97, 46.62) --
	(453.44, 46.19) --
	(458.92, 45.91) --
	(464.40, 45.73) --
	(469.88, 45.62) --
	(475.35, 45.56) --
	(480.83, 45.52) --
	(486.31, 45.49) --
	(491.79, 45.48) --
	(497.26, 45.47) --
	(502.74, 45.47) --
	(508.22, 45.47) --
	(513.70, 45.47) --
	(519.17, 45.47) --
	(524.65, 45.47) --
	(530.13, 45.47) --
	(535.61, 45.47) --
	(541.09, 45.47) --
	(546.56, 45.47) --
	(552.04, 45.47) --
	(557.52, 45.47) --
	(563.00, 45.47) --
	(568.47, 45.47) --
	(573.95, 45.47) --
	(579.43, 45.47) --
	(584.91, 45.47) --
	(590.38, 45.47) --
	(595.86, 45.47);
\end{scope}
\begin{scope}
\definecolor{drawColor}{RGB}{0,0,0}

\end{scope}
\begin{scope}
\definecolor{drawColor}{RGB}{0,0,0}

\path[draw=drawColor,line width= 0.6pt,line join=round] ( 20.71, 30.69) --
	(623.25, 30.69);
\end{scope}
\begin{scope}
\definecolor{drawColor}{gray}{0.20}

\path[draw=drawColor,line width= 0.6pt,line join=round] ( 48.10, 27.94) --
	( 48.10, 30.69);

\path[draw=drawColor,line width= 0.6pt,line join=round] (185.04, 27.94) --
	(185.04, 30.69);

\path[draw=drawColor,line width= 0.6pt,line join=round] (321.98, 27.94) --
	(321.98, 30.69);

\path[draw=drawColor,line width= 0.6pt,line join=round] (458.92, 27.94) --
	(458.92, 30.69);

\path[draw=drawColor,line width= 0.6pt,line join=round] (595.86, 27.94) --
	(595.86, 30.69);
\end{scope}
\begin{scope}
\definecolor{drawColor}{gray}{0.30}

\node[text=drawColor,anchor=base,inner sep=0pt, outer sep=0pt, scale=  0.88] at ( 48.10, -9.68) {-0.50};

\node[text=drawColor,anchor=base,inner sep=0pt, outer sep=0pt, scale=  0.88] at (185.04, -9.68) {-0.25};

\node[text=drawColor,anchor=base,inner sep=0pt, outer sep=0pt, scale=  0.88] at (321.98, -9.68) {0.00};

\node[text=drawColor,anchor=base,inner sep=0pt, outer sep=0pt, scale=  0.88] at (458.92, -9.68) {0.25};

\node[text=drawColor,anchor=base,inner sep=0pt, outer sep=0pt, scale=  0.88] at (595.86, -9.68) {0.50};
\end{scope}
\begin{scope}
	\definecolor{drawColor}{RGB}{0,0,0}
	
	\node[text=drawColor,anchor=base,inner sep=0pt, outer sep=0pt, scale=  1.10] at (291.36,  -72.64) {\small $n=2000$};
\end{scope}
\end{tikzpicture}
	\end{subfigure}
	\vspace{-5mm}
	\caption{Sampling distributions for $\langle\hat{\ba}_1,\ba_2\rangle$.
	}
	\label{fig:lo-dim-clt}
\end{figure}
The results match well with our theoretical development and as expected, the agreement between the histograms and the normal distribution improves with the increasing sample size.


\subsection{Asymptotic Properties}

In the next set of numerical experiments, we increase the sample size and dimensionality and provide further numerical evidence corroborating our theoretical developments.

\paragraph{Sample Complexity.} To fix ideas, we focus on the effect of initialization and compare our estimator with FastICA with random slicing as suggested by \cite{anand2014sample}. The only difference between the two approaches is in the initialization step and our method uses an improved moment estimate. To evaluate the effect of this on the quality of estimated unmixing directions, we vary $d$ from $d=90$ to $d=150$ at an interval of 10 and the sample size from 10000 to 24000 at intervals of 2000. The results, averaged over 200 runs, are given in Figure~\ref{fig:ica-0} below.

\begin{figure}[htbp]
	\centering
	\begin{subfigure}{.5\textwidth}
		\centering
		\input{fig-ica-mean.tex}
	\end{subfigure}%
	\begin{subfigure}{.5\textwidth}
		\centering
		\begin{tikzpicture}[x=.4pt,y=.4pt]
			\definecolor{fillColor}{RGB}{255,255,255}
			\begin{scope}
				\definecolor{drawColor}{RGB}{255,255,255}
				\definecolor{fillColor}{RGB}{255,255,255}
				
			\end{scope}
			\begin{scope}
				\definecolor{fillColor}{gray}{0.92}
				
				\definecolor{drawColor}{RGB}{255,255,255}
				
				\path[draw=drawColor,line width= 0.6pt,line join=round] (139.77, 54.48) --
				(464.93, 54.48);
				
				\path[draw=drawColor,line width= 0.6pt,line join=round] (139.77, 94.13) --
				(464.93, 94.13);
				
				\path[draw=drawColor,line width= 0.6pt,line join=round] (139.77,133.79) --
				(464.93,133.79);
				
				\path[draw=drawColor,line width= 0.6pt,line join=round] (139.77,173.44) --
				(464.93,173.44);
				
				\path[draw=drawColor,line width= 0.6pt,line join=round] (139.77,213.09) --
				(464.93,213.09);
				
				\path[draw=drawColor,line width= 0.6pt,line join=round] (139.77,252.75) --
				(464.93,252.75);
				
				\path[draw=drawColor,line width= 0.6pt,line join=round] (139.77,292.40) --
				(464.93,292.40);
				
				\path[draw=drawColor,line width= 0.6pt,line join=round] (139.77,332.06) --
				(464.93,332.06);
				
				\path[draw=drawColor,line width= 0.6pt,line join=round] (163.56, 30.69) --
				(163.56,355.85);
				
				\path[draw=drawColor,line width= 0.6pt,line join=round] (203.22, 30.69) --
				(203.22,355.85);
				
				\path[draw=drawColor,line width= 0.6pt,line join=round] (242.87, 30.69) --
				(242.87,355.85);
				
				\path[draw=drawColor,line width= 0.6pt,line join=round] (282.52, 30.69) --
				(282.52,355.85);
				
				\path[draw=drawColor,line width= 0.6pt,line join=round] (322.18, 30.69) --
				(322.18,355.85);
				
				\path[draw=drawColor,line width= 0.6pt,line join=round] (361.83, 30.69) --
				(361.83,355.85);
				
				\path[draw=drawColor,line width= 0.6pt,line join=round] (401.49, 30.69) --
				(401.49,355.85);
				
				\path[draw=drawColor,line width= 0.6pt,line join=round] (441.14, 30.69) --
				(441.14,355.85);
				\definecolor{fillColor}{RGB}{163,219,202}
				
				\path[fill=fillColor] (143.74, 34.65) rectangle (183.39, 74.31);
				\definecolor{fillColor}{RGB}{183,206,183}
				
				\path[fill=fillColor] (183.39, 34.65) rectangle (223.04, 74.31);
				\definecolor{fillColor}{RGB}{205,186,156}
				
				\path[fill=fillColor] (223.04, 34.65) rectangle (262.70, 74.31);
				\definecolor{fillColor}{RGB}{231,146,107}
				
				\path[fill=fillColor] (262.70, 34.65) rectangle (302.35, 74.31);
				\definecolor{fillColor}{RGB}{245,107,66}
				
				\path[fill=fillColor] (302.35, 34.65) rectangle (342.01, 74.31);
				\definecolor{fillColor}{RGB}{251,76,39}
				
				\path[fill=fillColor] (342.01, 34.65) rectangle (381.66, 74.31);
				\definecolor{fillColor}{RGB}{253,48,20}
				
				\path[fill=fillColor] (381.66, 34.65) rectangle (421.31, 74.31);
				\definecolor{fillColor}{RGB}{255,0,0}
				
				\path[fill=fillColor] (421.31, 34.65) rectangle (460.97, 74.31);
				\definecolor{fillColor}{RGB}{139,231,219}
				
				\path[fill=fillColor] (143.74, 74.31) rectangle (183.39,113.96);
				\definecolor{fillColor}{RGB}{155,224,208}
				
				\path[fill=fillColor] (183.39, 74.31) rectangle (223.04,113.96);
				\definecolor{fillColor}{RGB}{172,213,193}
				
				\path[fill=fillColor] (223.04, 74.31) rectangle (262.70,113.96);
				\definecolor{fillColor}{RGB}{194,197,170}
				
				\path[fill=fillColor] (262.70, 74.31) rectangle (302.35,113.96);
				\definecolor{fillColor}{RGB}{226,157,119}
				
				\path[fill=fillColor] (302.35, 74.31) rectangle (342.01,113.96);
				\definecolor{fillColor}{RGB}{242,119,77}
				
				\path[fill=fillColor] (342.01, 74.31) rectangle (381.66,113.96);
				\definecolor{fillColor}{RGB}{249,90,50}
				
				\path[fill=fillColor] (381.66, 74.31) rectangle (421.31,113.96);
				\definecolor{fillColor}{RGB}{252,65,31}
				
				\path[fill=fillColor] (421.31, 74.31) rectangle (460.97,113.96);
				\definecolor{fillColor}{RGB}{119,238,230}
				
				\path[fill=fillColor] (143.74,113.96) rectangle (183.39,153.61);
				\definecolor{fillColor}{RGB}{134,233,222}
				
				\path[fill=fillColor] (183.39,113.96) rectangle (223.04,153.61);
				\definecolor{fillColor}{RGB}{149,226,212}
				
				\path[fill=fillColor] (223.04,113.96) rectangle (262.70,153.61);
				\definecolor{fillColor}{RGB}{165,218,200}
				
				\path[fill=fillColor] (262.70,113.96) rectangle (302.35,153.61);
				\definecolor{fillColor}{RGB}{186,204,180}
				
				\path[fill=fillColor] (302.35,113.96) rectangle (342.01,153.61);
				\definecolor{fillColor}{RGB}{222,164,128}
				
				\path[fill=fillColor] (342.01,113.96) rectangle (381.66,153.61);
				\definecolor{fillColor}{RGB}{240,126,84}
				
				\path[fill=fillColor] (381.66,113.96) rectangle (421.31,153.61);
				\definecolor{fillColor}{RGB}{247,96,56}
				
				\path[fill=fillColor] (421.31,113.96) rectangle (460.97,153.61);
				\definecolor{fillColor}{RGB}{101,243,237}
				
				\path[fill=fillColor] (143.74,153.61) rectangle (183.39,193.27);
				\definecolor{fillColor}{RGB}{116,239,231}
				
				\path[fill=fillColor] (183.39,153.61) rectangle (223.04,193.27);
				\definecolor{fillColor}{RGB}{132,234,223}
				
				\path[fill=fillColor] (223.04,153.61) rectangle (262.70,193.27);
				\definecolor{fillColor}{RGB}{144,228,215}
				
				\path[fill=fillColor] (262.70,153.61) rectangle (302.35,193.27);
				\definecolor{fillColor}{RGB}{158,222,205}
				
				\path[fill=fillColor] (302.35,153.61) rectangle (342.01,193.27);
				\definecolor{fillColor}{RGB}{179,209,187}
				
				\path[fill=fillColor] (342.01,153.61) rectangle (381.66,193.27);
				\definecolor{fillColor}{RGB}{221,165,130}
				
				\path[fill=fillColor] (381.66,153.61) rectangle (421.31,193.27);
				\definecolor{fillColor}{RGB}{240,125,83}
				
				\path[fill=fillColor] (421.31,153.61) rectangle (460.97,193.27);
				\definecolor{fillColor}{RGB}{84,247,243}
				
				\path[fill=fillColor] (143.74,193.27) rectangle (183.39,232.92);
				\definecolor{fillColor}{RGB}{100,244,238}
				
				\path[fill=fillColor] (183.39,193.27) rectangle (223.04,232.92);
				\definecolor{fillColor}{RGB}{116,239,231}
				
				\path[fill=fillColor] (223.04,193.27) rectangle (262.70,232.92);
				\definecolor{fillColor}{RGB}{128,235,225}
				
				\path[fill=fillColor] (262.70,193.27) rectangle (302.35,232.92);
				\definecolor{fillColor}{RGB}{142,230,217}
				
				\path[fill=fillColor] (302.35,193.27) rectangle (342.01,232.92);
				\definecolor{fillColor}{RGB}{154,224,209}
				
				\path[fill=fillColor] (342.01,193.27) rectangle (381.66,232.92);
				\definecolor{fillColor}{RGB}{182,207,184}
				
				\path[fill=fillColor] (381.66,193.27) rectangle (421.31,232.92);
				\definecolor{fillColor}{RGB}{223,162,125}
				
				\path[fill=fillColor] (421.31,193.27) rectangle (460.97,232.92);
				\definecolor{fillColor}{RGB}{65,250,248}
				
				\path[fill=fillColor] (143.74,232.92) rectangle (183.39,272.58);
				\definecolor{fillColor}{RGB}{87,247,242}
				
				\path[fill=fillColor] (183.39,232.92) rectangle (223.04,272.58);
				\definecolor{fillColor}{RGB}{102,243,237}
				
				\path[fill=fillColor] (223.04,232.92) rectangle (262.70,272.58);
				\definecolor{fillColor}{RGB}{114,240,232}
				
				\path[fill=fillColor] (262.70,232.92) rectangle (302.35,272.58);
				\definecolor{fillColor}{RGB}{126,236,226}
				
				\path[fill=fillColor] (302.35,232.92) rectangle (342.01,272.58);
				\definecolor{fillColor}{RGB}{138,231,219}
				
				\path[fill=fillColor] (342.01,232.92) rectangle (381.66,272.58);
				\definecolor{fillColor}{RGB}{152,225,210}
				
				\path[fill=fillColor] (381.66,232.92) rectangle (421.31,272.58);
				\definecolor{fillColor}{RGB}{190,200,176}
				
				\path[fill=fillColor] (421.31,232.92) rectangle (460.97,272.58);
				\definecolor{fillColor}{RGB}{43,253,252}
				
				\path[fill=fillColor] (143.74,272.58) rectangle (183.39,312.23);
				\definecolor{fillColor}{RGB}{71,250,247}
				
				\path[fill=fillColor] (183.39,272.58) rectangle (223.04,312.23);
				\definecolor{fillColor}{RGB}{88,246,242}
				
				\path[fill=fillColor] (223.04,272.58) rectangle (262.70,312.23);
				\definecolor{fillColor}{RGB}{102,243,237}
				
				\path[fill=fillColor] (262.70,272.58) rectangle (302.35,312.23);
				\definecolor{fillColor}{RGB}{115,240,232}
				
				\path[fill=fillColor] (302.35,272.58) rectangle (342.01,312.23);
				\definecolor{fillColor}{RGB}{125,236,227}
				
				\path[fill=fillColor] (342.01,272.58) rectangle (381.66,312.23);
				\definecolor{fillColor}{RGB}{136,232,221}
				
				\path[fill=fillColor] (381.66,272.58) rectangle (421.31,312.23);
				\definecolor{fillColor}{RGB}{151,225,211}
				
				\path[fill=fillColor] (421.31,272.58) rectangle (460.97,312.23);
				\definecolor{fillColor}{RGB}{0,255,255}
				
				\path[fill=fillColor] (143.74,312.23) rectangle (183.39,351.88);
				\definecolor{fillColor}{RGB}{53,252,250}
				
				\path[fill=fillColor] (183.39,312.23) rectangle (223.04,351.88);
				\definecolor{fillColor}{RGB}{74,249,246}
				
				\path[fill=fillColor] (223.04,312.23) rectangle (262.70,351.88);
				\definecolor{fillColor}{RGB}{89,246,242}
				
				\path[fill=fillColor] (262.70,312.23) rectangle (302.35,351.88);
				\definecolor{fillColor}{RGB}{102,243,237}
				
				\path[fill=fillColor] (302.35,312.23) rectangle (342.01,351.88);
				\definecolor{fillColor}{RGB}{114,240,232}
				
				\path[fill=fillColor] (342.01,312.23) rectangle (381.66,351.88);
				\definecolor{fillColor}{RGB}{108,241,235}
				
				\path[fill=fillColor] (381.66,312.23) rectangle (421.31,351.88);
				\definecolor{fillColor}{RGB}{136,232,221}
				
				\path[fill=fillColor] (421.31,312.23) rectangle (460.97,351.88);
			\end{scope}
			\begin{scope}
				\definecolor{drawColor}{gray}{0.30}
				
				\node[text=drawColor,anchor=base east,inner sep=0pt, outer sep=0pt, scale=  0.88] at (134.82, 51.45) {10000};
				
				\node[text=drawColor,anchor=base east,inner sep=0pt, outer sep=0pt, scale=  0.88] at (134.82, 91.10) {12000};
				
				\node[text=drawColor,anchor=base east,inner sep=0pt, outer sep=0pt, scale=  0.88] at (134.82,130.76) {14000};
				
				\node[text=drawColor,anchor=base east,inner sep=0pt, outer sep=0pt, scale=  0.88] at (134.82,170.41) {16000};
				
				\node[text=drawColor,anchor=base east,inner sep=0pt, outer sep=0pt, scale=  0.88] at (134.82,210.06) {18000};
				
				\node[text=drawColor,anchor=base east,inner sep=0pt, outer sep=0pt, scale=  0.88] at (134.82,249.72) {20000};
				
				\node[text=drawColor,anchor=base east,inner sep=0pt, outer sep=0pt, scale=  0.88] at (134.82,289.37) {22000};
				
				\node[text=drawColor,anchor=base east,inner sep=0pt, outer sep=0pt, scale=  0.88] at (134.82,329.03) {24000};
			\end{scope}
			\begin{scope}
				\definecolor{drawColor}{gray}{0.20}
				
				\path[draw=drawColor,line width= 0.6pt,line join=round] (137.02, 54.48) --
				(139.77, 54.48);
				
				\path[draw=drawColor,line width= 0.6pt,line join=round] (137.02, 94.13) --
				(139.77, 94.13);
				
				\path[draw=drawColor,line width= 0.6pt,line join=round] (137.02,133.79) --
				(139.77,133.79);
				
				\path[draw=drawColor,line width= 0.6pt,line join=round] (137.02,173.44) --
				(139.77,173.44);
				
				\path[draw=drawColor,line width= 0.6pt,line join=round] (137.02,213.09) --
				(139.77,213.09);
				
				\path[draw=drawColor,line width= 0.6pt,line join=round] (137.02,252.75) --
				(139.77,252.75);
				
				\path[draw=drawColor,line width= 0.6pt,line join=round] (137.02,292.40) --
				(139.77,292.40);
				
				\path[draw=drawColor,line width= 0.6pt,line join=round] (137.02,332.06) --
				(139.77,332.06);
			\end{scope}
			\begin{scope}
				\definecolor{drawColor}{gray}{0.20}
				
				\path[draw=drawColor,line width= 0.6pt,line join=round] (163.56, 27.94) --
				(163.56, 30.69);
				
				\path[draw=drawColor,line width= 0.6pt,line join=round] (203.22, 27.94) --
				(203.22, 30.69);
				
				\path[draw=drawColor,line width= 0.6pt,line join=round] (242.87, 27.94) --
				(242.87, 30.69);
				
				\path[draw=drawColor,line width= 0.6pt,line join=round] (282.52, 27.94) --
				(282.52, 30.69);
				
				\path[draw=drawColor,line width= 0.6pt,line join=round] (322.18, 27.94) --
				(322.18, 30.69);
				
				\path[draw=drawColor,line width= 0.6pt,line join=round] (361.83, 27.94) --
				(361.83, 30.69);
				
				\path[draw=drawColor,line width= 0.6pt,line join=round] (401.49, 27.94) --
				(401.49, 30.69);
				
				\path[draw=drawColor,line width= 0.6pt,line join=round] (441.14, 27.94) --
				(441.14, 30.69);
			\end{scope}
			\begin{scope}
				\definecolor{drawColor}{gray}{0.30}
				
				\node[text=drawColor,anchor=base,inner sep=0pt, outer sep=0pt, scale=  0.88] at (163.56, 7.68) {80};
				
				\node[text=drawColor,anchor=base,inner sep=0pt, outer sep=0pt, scale=  0.88] at (203.22, 7.68) {90};
				
				\node[text=drawColor,anchor=base,inner sep=0pt, outer sep=0pt, scale=  0.88] at (242.87, 7.68) {100};
				
				\node[text=drawColor,anchor=base,inner sep=0pt, outer sep=0pt, scale=  0.88] at (282.52, 7.68) {110};
				
				\node[text=drawColor,anchor=base,inner sep=0pt, outer sep=0pt, scale=  0.88] at (322.18, 7.68) {120};
				
				\node[text=drawColor,anchor=base,inner sep=0pt, outer sep=0pt, scale=  0.88] at (361.83, 7.68) {130};
				
				\node[text=drawColor,anchor=base,inner sep=0pt, outer sep=0pt, scale=  0.88] at (401.49, 7.68) {140};
				
				\node[text=drawColor,anchor=base,inner sep=0pt, outer sep=0pt, scale=  0.88] at (441.14, 7.68) {150};
			\end{scope}
			\begin{scope}
				\definecolor{drawColor}{RGB}{0,0,0}
				
				\node[text=drawColor,anchor=base,inner sep=0pt, outer sep=0pt, scale=  1.10] at (302.35,  -20.64) {Dimension};
			\end{scope}
			\begin{scope}
				\definecolor{drawColor}{RGB}{0,0,0}
				
			\end{scope}
			\begin{scope}
				\definecolor{fillColor}{RGB}{255,255,255}
				
				\path[fill=fillColor] (475.93,144.03) rectangle (528.39,242.51);
			\end{scope}
			\begin{scope}
				\node[inner sep=0pt,outer sep=0pt,anchor=south west,rotate=  0.00] at (481.43, 149.53) {
					\includegraphics[width= 14.45pt,height= 28.91pt,interpolate=true]{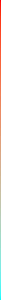}};
			\end{scope}
			\begin{scope}
				\definecolor{drawColor}{RGB}{0,0,0}
				
				\node[text=drawColor,anchor=base west,inner sep=0pt, outer sep=0pt, scale=  0.88] at (531.39,147.75) {0.2};
				
				\node[text=drawColor,anchor=base west,inner sep=0pt, outer sep=0pt, scale=  0.88] at (531.39,167.58) {0.4};
				
				\node[text=drawColor,anchor=base west,inner sep=0pt, outer sep=0pt, scale=  0.88] at (531.39,187.40) {0.6};
				
				\node[text=drawColor,anchor=base west,inner sep=0pt, outer sep=0pt, scale=  0.88] at (531.39,207.23) {0.8};
			\end{scope}
			\begin{scope}
				\definecolor{drawColor}{RGB}{0,0,0}
				
				\node[text=drawColor,anchor=base west,inner sep=0pt, outer sep=0pt, scale=  1.10] at (481.43,228.36) {\small Distance};
			\end{scope}
			\begin{scope}
				\definecolor{drawColor}{RGB}{255,255,255}
				
				\path[draw=drawColor,line width= 0.2pt,line join=round] (481.43,150.78) -- (484.32,150.78);
				
				\path[draw=drawColor,line width= 0.2pt,line join=round] (481.43,170.61) -- (484.32,170.61);
				
				\path[draw=drawColor,line width= 0.2pt,line join=round] (481.43,190.43) -- (484.32,190.43);
				
				\path[draw=drawColor,line width= 0.2pt,line join=round] (481.43,210.26) -- (484.32,210.26);
				
				\path[draw=drawColor,line width= 0.2pt,line join=round] (493.00,150.78) -- (495.89,150.78);
				
				\path[draw=drawColor,line width= 0.2pt,line join=round] (493.00,170.61) -- (495.89,170.61);
				
				\path[draw=drawColor,line width= 0.2pt,line join=round] (493.00,190.43) -- (495.89,190.43);
				
				\path[draw=drawColor,line width= 0.2pt,line join=round] (493.00,210.26) -- (495.89,210.26);
			\end{scope}
		\end{tikzpicture}
	\end{subfigure}
	\vspace{-5mm}
	\caption{
		Comparison of random slicing (left) and projection-based  (right) based initializations.}
	\label{fig:ica-0}
\end{figure}

These results show different sample complexities between the two methods with the proposed method achieving consistency with much smaller sample sizes. This confirms our theoretical findings that existing approaches, because of the way they use sample moments, have a sample complexity of $n\asymp d^3$ whereas the proposed method has a sample complexity of $n\asymp d^2$.

To gain further insights into the effect of dimensionality, in the second set of simulations, we fixed the sample size at $n=24000$ and vary the dimension from $d=90$ to $d=150$ at an interval of 10. For each value of $d$, we report the average estimation error, measured by both $\ell_M$ and $\ell_A$, over 200 replications of the experiment. The results are in Figure \ref{fig:ica-1} below.
\begin{figure}[htbp]
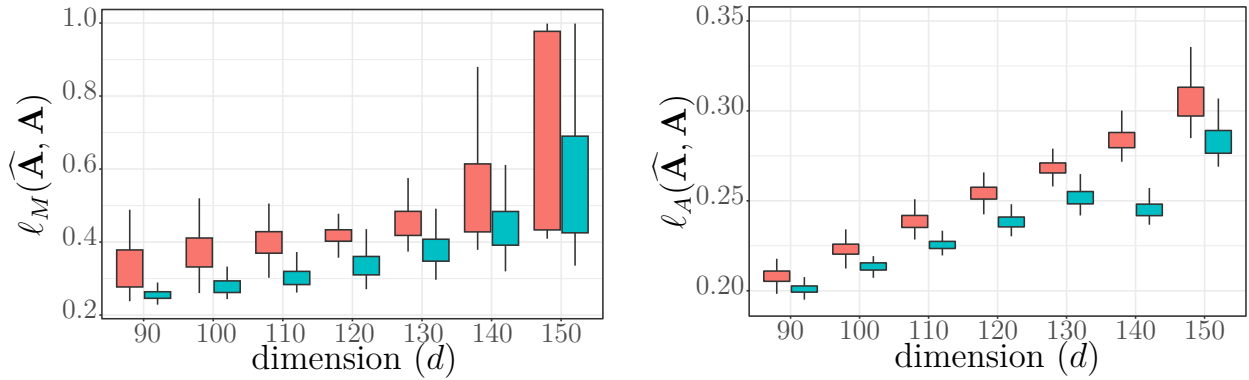

	\centering
	\begin{subfigure}{.49\textwidth}
		\input{ica-max-dims.tex}
	\end{subfigure}\hfil
	\begin{subfigure}{.49\textwidth}
		\input{ica-mean-dims.tex}
	\end{subfigure}
	\vspace{-5mm}
	\caption{
		Comparison of random slicing (red) and projection-based   (green) initializations at $n=24000$.}
	\label{fig:ica-1}
\end{figure}

Again the comparison between the two methods is rather similar when using either error measure. As expected from our theoretical results, the performance of both methods worsens when the dimension grows while the sample size is fixed. However, the green boxplots, corresponding to Algorithm~\ref{alg:alg-ica} show strictly smaller errors with respect to either metric, when compared with the prevalent random slicing initialization used in FastICA.

The third set of numerical experiments focuses on the effect of the sample size. As earlier, we consider $S_k$ to be i.i.d. Laplace random variables. We fix $d=150$ and vary the sample size from 24000 to 30000 at an interval of 1000 each. The results are summarized in Figure \ref{fig:trunc}. As before, each boxplot is constructed with results from 200 replications. Again they clearly demonstrate the superiority of the proposed approach.

\begin{figure}[htbp]
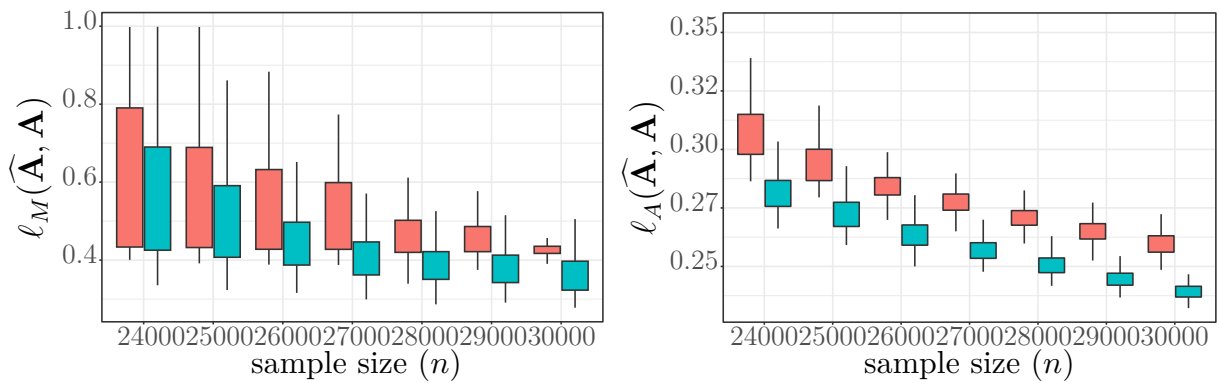

	\centering
	\begin{subfigure}{.49\textwidth}
		\input{ica-samps-max.tex}
	\end{subfigure}\hfil
	\begin{subfigure}{.49\textwidth}
		\vspace{-7mm}
		\input{ica-trunc.tex}
	\end{subfigure}
	\vspace{-5mm}
	\caption{Comparison of random slicing (red) and projection-based  (green) initializations at $d=150$.}
	\label{fig:trunc}
\end{figure}

\paragraph{Asymptotic Normality.} The next set of simulations aims to show the asymptotic normality of our estimator. To this end, we fixed $d=50$ and $n=6000$. We used random mixing matrices $\bA\in\calO(d)$, and report the histogram of different linear combinations based on 500 simulation runs in Figure \ref{fig:ica-clt}. Note that the true values are at $\langle\ba_1,\ba_2\rangle=\langle\ba_2,\ba_3\rangle=\langle\ba_1,\ba_4\rangle=0$ for the top row, and $\langle\ba_1,(\ba_1+\ba_2)/\sqrt{2}\rangle=1/\sqrt{2}$ and $\langle\ba_1,(\ba_1+\ba_2+\ba_3)/\sqrt{3}\rangle=1/\sqrt{3}$ for the bottom row. We overlaid the histogram with the normal density curve fitted with the mean and standard deviation of the data. The numerical results support our theoretical findings from Section~\ref{sec:asy-dist}.

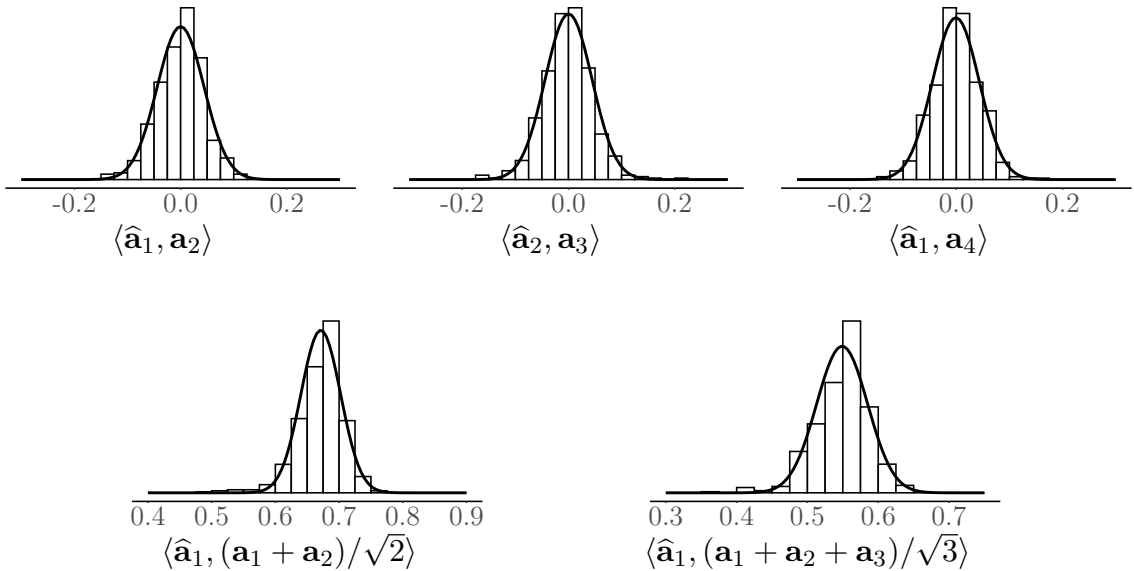
\begin{figure}[htb]
	\centering
	\begin{subfigure}{0.25\textwidth}
		\begin{tikzpicture}[x=.22pt,y=.22pt]
\definecolor{fillColor}{RGB}{255,255,255}
\begin{scope}
\definecolor{drawColor}{RGB}{255,255,255}
\definecolor{fillColor}{RGB}{255,255,255}

\end{scope}
\begin{scope}
\definecolor{fillColor}{RGB}{255,255,255}

\path[fill=fillColor] ( 20.71, 30.69) rectangle (623.25,355.85);
\definecolor{drawColor}{RGB}{0,0,0}

\path[draw=drawColor,line width= 0.6pt,fill=fillColor] ( 48.10, 45.47) rectangle ( 70.93, 45.47);

\path[draw=drawColor,line width= 0.6pt,fill=fillColor] ( 70.93, 45.47) rectangle ( 93.75, 45.47);

\path[draw=drawColor,line width= 0.6pt,fill=fillColor] ( 93.75, 45.47) rectangle (116.57, 45.47);

\path[draw=drawColor,line width= 0.6pt,fill=fillColor] (116.57, 45.47) rectangle (139.40, 45.47);

\path[draw=drawColor,line width= 0.6pt,fill=fillColor] (139.40, 45.47) rectangle (162.22, 45.47);

\path[draw=drawColor,line width= 0.6pt,fill=fillColor] (162.22, 45.47) rectangle (185.04, 45.47);

\path[draw=drawColor,line width= 0.6pt,fill=fillColor] (185.04, 45.47) rectangle (207.87, 54.78);

\path[draw=drawColor,line width= 0.6pt,fill=fillColor] (207.87, 45.47) rectangle (230.69, 57.10);

\path[draw=drawColor,line width= 0.6pt,fill=fillColor] (230.69, 45.47) rectangle (253.51, 78.05);

\path[draw=drawColor,line width= 0.6pt,fill=fillColor] (253.51, 45.47) rectangle (276.34,140.90);

\path[draw=drawColor,line width= 0.6pt,fill=fillColor] (276.34, 45.47) rectangle (299.16,213.05);

\path[draw=drawColor,line width= 0.6pt,fill=fillColor] (299.16, 45.47) rectangle (321.98,273.57);

\path[draw=drawColor,line width= 0.6pt,fill=fillColor] (321.98, 45.47) rectangle (344.80,341.07);

\path[draw=drawColor,line width= 0.6pt,fill=fillColor] (344.80, 45.47) rectangle (367.63,254.95);

\path[draw=drawColor,line width= 0.6pt,fill=fillColor] (367.63, 45.47) rectangle (390.45,112.97);

\path[draw=drawColor,line width= 0.6pt,fill=fillColor] (390.45, 45.47) rectangle (413.27, 82.71);

\path[draw=drawColor,line width= 0.6pt,fill=fillColor] (413.27, 45.47) rectangle (436.10, 54.78);

\path[draw=drawColor,line width= 0.6pt,fill=fillColor] (436.10, 45.47) rectangle (458.92, 45.47);

\path[draw=drawColor,line width= 0.6pt,fill=fillColor] (458.92, 45.47) rectangle (481.74, 45.47);

\path[draw=drawColor,line width= 0.6pt,fill=fillColor] (481.74, 45.47) rectangle (504.57, 45.47);

\path[draw=drawColor,line width= 0.6pt,fill=fillColor] (504.57, 45.47) rectangle (527.39, 45.47);

\path[draw=drawColor,line width= 0.6pt,fill=fillColor] (527.39, 45.47) rectangle (550.21, 45.47);

\path[draw=drawColor,line width= 0.6pt,fill=fillColor] (550.21, 45.47) rectangle (573.04, 45.47);

\path[draw=drawColor,line width= 0.6pt,fill=fillColor] (573.04, 45.47) rectangle (595.86, 45.47);

\path[draw=drawColor,line width= 1.1pt,line join=round] ( 48.10, 45.47) --
	( 53.58, 45.47) --
	( 59.06, 45.47) --
	( 64.54, 45.47) --
	( 70.01, 45.47) --
	( 75.49, 45.47) --
	( 80.97, 45.47) --
	( 86.45, 45.47) --
	( 91.92, 45.47) --
	( 97.40, 45.47) --
	(102.88, 45.47) --
	(108.36, 45.47) --
	(113.83, 45.47) --
	(119.31, 45.47) --
	(124.79, 45.47) --
	(130.27, 45.47) --
	(135.74, 45.47) --
	(141.22, 45.48) --
	(146.70, 45.49) --
	(152.18, 45.50) --
	(157.65, 45.53) --
	(163.13, 45.58) --
	(168.61, 45.65) --
	(174.09, 45.78) --
	(179.56, 45.97) --
	(185.04, 46.28) --
	(190.52, 46.74) --
	(196.00, 47.44) --
	(201.47, 48.46) --
	(206.95, 49.92) --
	(212.43, 51.97) --
	(217.91, 54.80) --
	(223.39, 58.62) --
	(228.86, 63.64) --
	(234.34, 70.14) --
	(239.82, 78.33) --
	(245.30, 88.45) --
	(250.77,100.64) --
	(256.25,115.00) --
	(261.73,131.48) --
	(267.21,149.93) --
	(272.68,170.00) --
	(278.16,191.20) --
	(283.64,212.88) --
	(289.12,234.26) --
	(294.59,254.47) --
	(300.07,272.59) --
	(305.55,287.76) --
	(311.03,299.20) --
	(316.50,306.31) --
	(321.98,308.69) --
	(327.46,306.23) --
	(332.94,299.05) --
	(338.41,287.54) --
	(343.89,272.32) --
	(349.37,254.15) --
	(354.85,233.92) --
	(360.32,212.53) --
	(365.80,190.85) --
	(371.28,169.66) --
	(376.76,149.61) --
	(382.24,131.20) --
	(387.71,114.75) --
	(393.19,100.43) --
	(398.67, 88.26) --
	(404.15, 78.18) --
	(409.62, 70.02) --
	(415.10, 63.55) --
	(420.58, 58.54) --
	(426.06, 54.75) --
	(431.53, 51.94) --
	(437.01, 49.89) --
	(442.49, 48.44) --
	(447.97, 47.43) --
	(453.44, 46.73) --
	(458.92, 46.27) --
	(464.40, 45.97) --
	(469.88, 45.77) --
	(475.35, 45.65) --
	(480.83, 45.57) --
	(486.31, 45.53) --
	(491.79, 45.50) --
	(497.26, 45.49) --
	(502.74, 45.48) --
	(508.22, 45.47) --
	(513.70, 45.47) --
	(519.17, 45.47) --
	(524.65, 45.47) --
	(530.13, 45.47) --
	(535.61, 45.47) --
	(541.09, 45.47) --
	(546.56, 45.47) --
	(552.04, 45.47) --
	(557.52, 45.47) --
	(563.00, 45.47) --
	(568.47, 45.47) --
	(573.95, 45.47) --
	(579.43, 45.47) --
	(584.91, 45.47) --
	(590.38, 45.47) --
	(595.86, 45.47);
\end{scope}
\begin{scope}
\definecolor{drawColor}{RGB}{0,0,0}

\end{scope}
\begin{scope}
\definecolor{drawColor}{RGB}{0,0,0}

\path[draw=drawColor,line width= 0.6pt,line join=round] ( 20.71, 30.69) --
	(623.25, 30.69);
\end{scope}
\begin{scope}
\definecolor{drawColor}{gray}{0.20}

\path[draw=drawColor,line width= 0.6pt,line join=round] (139.40, 27.94) --
	(139.40, 30.69);

\path[draw=drawColor,line width= 0.6pt,line join=round] (321.98, 27.94) --
	(321.98, 30.69);

\path[draw=drawColor,line width= 0.6pt,line join=round] (504.57, 27.94) --
	(504.57, 30.69);
\end{scope}
\begin{scope}
\definecolor{drawColor}{gray}{0.30}

\node[text=drawColor,anchor=base,inner sep=0pt, outer sep=0pt, scale=  0.88] at (139.40, -9.68) {-0.2};

\node[text=drawColor,anchor=base,inner sep=0pt, outer sep=0pt, scale=  0.88] at (321.98, -9.68) {0.0};

\node[text=drawColor,anchor=base,inner sep=0pt, outer sep=0pt, scale=  0.88] at (504.57, -9.68) {0.2};
\end{scope}
\begin{scope}
	\definecolor{drawColor}{RGB}{0,0,0}
	
	\node[text=drawColor,anchor=base,inner sep=0pt, outer sep=0pt, scale=  1.10] at (291.36,  -72.64) {\small $\langle\hat{\ba}_1,\ba_2\rangle$};
\end{scope}
\end{tikzpicture}
	\end{subfigure}\hfil
	\begin{subfigure}{0.25\textwidth}
		\begin{tikzpicture}[x=.22pt,y=.22pt]
\definecolor{fillColor}{RGB}{255,255,255}
\begin{scope}
\definecolor{drawColor}{RGB}{255,255,255}
\definecolor{fillColor}{RGB}{255,255,255}

\end{scope}
\begin{scope}
\definecolor{fillColor}{RGB}{255,255,255}

\path[fill=fillColor] ( 20.71, 30.69) rectangle (623.25,355.85);
\definecolor{drawColor}{RGB}{0,0,0}

\path[draw=drawColor,line width= 0.6pt,fill=fillColor] ( 48.10, 45.47) rectangle ( 70.93, 45.47);

\path[draw=drawColor,line width= 0.6pt,fill=fillColor] ( 70.93, 45.47) rectangle ( 93.75, 45.47);

\path[draw=drawColor,line width= 0.6pt,fill=fillColor] ( 93.75, 45.47) rectangle (116.57, 45.47);

\path[draw=drawColor,line width= 0.6pt,fill=fillColor] (116.57, 45.47) rectangle (139.40, 45.47);

\path[draw=drawColor,line width= 0.6pt,fill=fillColor] (139.40, 45.47) rectangle (162.22, 45.47);

\path[draw=drawColor,line width= 0.6pt,fill=fillColor] (162.22, 45.47) rectangle (185.04, 53.05);

\path[draw=drawColor,line width= 0.6pt,fill=fillColor] (185.04, 45.47) rectangle (207.87, 47.99);

\path[draw=drawColor,line width= 0.6pt,fill=fillColor] (207.87, 45.47) rectangle (230.69, 60.63);

\path[draw=drawColor,line width= 0.6pt,fill=fillColor] (230.69, 45.47) rectangle (253.51, 78.31);

\path[draw=drawColor,line width= 0.6pt,fill=fillColor] (253.51, 45.47) rectangle (276.34,151.58);

\path[draw=drawColor,line width= 0.6pt,fill=fillColor] (276.34, 45.47) rectangle (299.16,232.43);

\path[draw=drawColor,line width= 0.6pt,fill=fillColor] (299.16, 45.47) rectangle (321.98,330.96);

\path[draw=drawColor,line width= 0.6pt,fill=fillColor] (321.98, 45.47) rectangle (344.80,341.07);

\path[draw=drawColor,line width= 0.6pt,fill=fillColor] (344.80, 45.47) rectangle (367.63,237.48);

\path[draw=drawColor,line width= 0.6pt,fill=fillColor] (367.63, 45.47) rectangle (390.45,123.79);

\path[draw=drawColor,line width= 0.6pt,fill=fillColor] (390.45, 45.47) rectangle (413.27, 85.89);

\path[draw=drawColor,line width= 0.6pt,fill=fillColor] (413.27, 45.47) rectangle (436.10, 53.05);

\path[draw=drawColor,line width= 0.6pt,fill=fillColor] (436.10, 45.47) rectangle (458.92, 50.52);

\path[draw=drawColor,line width= 0.6pt,fill=fillColor] (458.92, 45.47) rectangle (481.74, 47.99);

\path[draw=drawColor,line width= 0.6pt,fill=fillColor] (481.74, 45.47) rectangle (504.57, 45.47);

\path[draw=drawColor,line width= 0.6pt,fill=fillColor] (504.57, 45.47) rectangle (527.39, 47.99);

\path[draw=drawColor,line width= 0.6pt,fill=fillColor] (527.39, 45.47) rectangle (550.21, 45.47);

\path[draw=drawColor,line width= 0.6pt,fill=fillColor] (550.21, 45.47) rectangle (573.04, 45.47);

\path[draw=drawColor,line width= 0.6pt,fill=fillColor] (573.04, 45.47) rectangle (595.86, 45.47);

\path[draw=drawColor,line width= 1.1pt,line join=round] ( 48.10, 45.47) --
	( 53.58, 45.47) --
	( 59.06, 45.47) --
	( 64.54, 45.47) --
	( 70.01, 45.47) --
	( 75.49, 45.47) --
	( 80.97, 45.47) --
	( 86.45, 45.47) --
	( 91.92, 45.47) --
	( 97.40, 45.47) --
	(102.88, 45.47) --
	(108.36, 45.47) --
	(113.83, 45.47) --
	(119.31, 45.47) --
	(124.79, 45.47) --
	(130.27, 45.47) --
	(135.74, 45.47) --
	(141.22, 45.48) --
	(146.70, 45.49) --
	(152.18, 45.50) --
	(157.65, 45.53) --
	(163.13, 45.58) --
	(168.61, 45.66) --
	(174.09, 45.80) --
	(179.56, 46.00) --
	(185.04, 46.33) --
	(190.52, 46.83) --
	(196.00, 47.57) --
	(201.47, 48.66) --
	(206.95, 50.22) --
	(212.43, 52.42) --
	(217.91, 55.44) --
	(223.39, 59.53) --
	(228.86, 64.92) --
	(234.34, 71.88) --
	(239.82, 80.69) --
	(245.30, 91.56) --
	(250.77,104.68) --
	(256.25,120.14) --
	(261.73,137.91) --
	(267.21,157.81) --
	(272.68,179.49) --
	(278.16,202.43) --
	(283.64,225.91) --
	(289.12,249.10) --
	(294.59,271.06) --
	(300.07,290.80) --
	(305.55,307.37) --
	(311.03,319.94) --
	(316.50,327.83) --
	(321.98,330.62) --
	(327.46,328.15) --
	(332.94,320.56) --
	(338.41,308.27) --
	(343.89,291.92) --
	(349.37,272.35) --
	(354.85,250.50) --
	(360.32,227.36) --
	(365.80,203.87) --
	(371.28,180.88) --
	(376.76,159.10) --
	(382.24,139.08) --
	(387.71,121.17) --
	(393.19,105.56) --
	(398.67, 92.30) --
	(404.15, 81.29) --
	(409.62, 72.37) --
	(415.10, 65.30) --
	(420.58, 59.82) --
	(426.06, 55.66) --
	(431.53, 52.58) --
	(437.01, 50.33) --
	(442.49, 48.74) --
	(447.97, 47.62) --
	(453.44, 46.86) --
	(458.92, 46.35) --
	(464.40, 46.02) --
	(469.88, 45.81) --
	(475.35, 45.67) --
	(480.83, 45.59) --
	(486.31, 45.54) --
	(491.79, 45.51) --
	(497.26, 45.49) --
	(502.74, 45.48) --
	(508.22, 45.47) --
	(513.70, 45.47) --
	(519.17, 45.47) --
	(524.65, 45.47) --
	(530.13, 45.47) --
	(535.61, 45.47) --
	(541.09, 45.47) --
	(546.56, 45.47) --
	(552.04, 45.47) --
	(557.52, 45.47) --
	(563.00, 45.47) --
	(568.47, 45.47) --
	(573.95, 45.47) --
	(579.43, 45.47) --
	(584.91, 45.47) --
	(590.38, 45.47) --
	(595.86, 45.47);
\end{scope}
\begin{scope}
\definecolor{drawColor}{RGB}{0,0,0}

\end{scope}
\begin{scope}
\definecolor{drawColor}{RGB}{0,0,0}

\path[draw=drawColor,line width= 0.6pt,line join=round] ( 20.71, 30.69) --
	(623.25, 30.69);
\end{scope}
\begin{scope}
\definecolor{drawColor}{gray}{0.20}

\path[draw=drawColor,line width= 0.6pt,line join=round] (139.40, 27.94) --
	(139.40, 30.69);

\path[draw=drawColor,line width= 0.6pt,line join=round] (321.98, 27.94) --
	(321.98, 30.69);

\path[draw=drawColor,line width= 0.6pt,line join=round] (504.57, 27.94) --
	(504.57, 30.69);
\end{scope}
\begin{scope}
\definecolor{drawColor}{gray}{0.30}

\node[text=drawColor,anchor=base,inner sep=0pt, outer sep=0pt, scale=  0.88] at (139.40, -9.68) {-0.2};

\node[text=drawColor,anchor=base,inner sep=0pt, outer sep=0pt, scale=  0.88] at (321.98, -9.68) {0.0};

\node[text=drawColor,anchor=base,inner sep=0pt, outer sep=0pt, scale=  0.88] at (504.57, -9.68) {0.2};
\end{scope}
\begin{scope}
	\definecolor{drawColor}{RGB}{0,0,0}
	
	\node[text=drawColor,anchor=base,inner sep=0pt, outer sep=0pt, scale=  1.10] at (291.36,  -72.64) {\small $\langle \hat{\ba}_{2},\ba_3\rangle$};
\end{scope}
\end{tikzpicture}
	\end{subfigure}\hfil 
	\begin{subfigure}{0.25\textwidth}
		\begin{tikzpicture}[x=.22pt,y=.22pt]
\definecolor{fillColor}{RGB}{255,255,255}
\begin{scope}
\definecolor{drawColor}{RGB}{255,255,255}
\definecolor{fillColor}{RGB}{255,255,255}

\end{scope}
\begin{scope}
\definecolor{fillColor}{RGB}{255,255,255}

\path[fill=fillColor] ( 20.71, 30.69) rectangle (623.25,355.85);
\definecolor{drawColor}{RGB}{0,0,0}

\path[draw=drawColor,line width= 0.6pt,fill=fillColor] ( 48.10, 45.47) rectangle ( 70.93, 45.47);

\path[draw=drawColor,line width= 0.6pt,fill=fillColor] ( 70.93, 45.47) rectangle ( 93.75, 45.47);

\path[draw=drawColor,line width= 0.6pt,fill=fillColor] ( 93.75, 45.47) rectangle (116.57, 45.47);

\path[draw=drawColor,line width= 0.6pt,fill=fillColor] (116.57, 45.47) rectangle (139.40, 45.47);

\path[draw=drawColor,line width= 0.6pt,fill=fillColor] (139.40, 45.47) rectangle (162.22, 45.47);

\path[draw=drawColor,line width= 0.6pt,fill=fillColor] (162.22, 45.47) rectangle (185.04, 45.47);

\path[draw=drawColor,line width= 0.6pt,fill=fillColor] (185.04, 45.47) rectangle (207.87, 50.39);

\path[draw=drawColor,line width= 0.6pt,fill=fillColor] (207.87, 45.47) rectangle (230.69, 60.25);

\path[draw=drawColor,line width= 0.6pt,fill=fillColor] (230.69, 45.47) rectangle (253.51, 77.49);

\path[draw=drawColor,line width= 0.6pt,fill=fillColor] (253.51, 45.47) rectangle (276.34,156.32);

\path[draw=drawColor,line width= 0.6pt,fill=fillColor] (276.34, 45.47) rectangle (299.16,208.05);

\path[draw=drawColor,line width= 0.6pt,fill=fillColor] (299.16, 45.47) rectangle (321.98,341.07);

\path[draw=drawColor,line width= 0.6pt,fill=fillColor] (321.98, 45.47) rectangle (344.80,331.22);

\path[draw=drawColor,line width= 0.6pt,fill=fillColor] (344.80, 45.47) rectangle (367.63,212.97);

\path[draw=drawColor,line width= 0.6pt,fill=fillColor] (367.63, 45.47) rectangle (390.45,163.71);

\path[draw=drawColor,line width= 0.6pt,fill=fillColor] (390.45, 45.47) rectangle (413.27, 75.03);

\path[draw=drawColor,line width= 0.6pt,fill=fillColor] (413.27, 45.47) rectangle (436.10, 50.39);

\path[draw=drawColor,line width= 0.6pt,fill=fillColor] (436.10, 45.47) rectangle (458.92, 47.93);

\path[draw=drawColor,line width= 0.6pt,fill=fillColor] (458.92, 45.47) rectangle (481.74, 47.93);

\path[draw=drawColor,line width= 0.6pt,fill=fillColor] (481.74, 45.47) rectangle (504.57, 45.47);

\path[draw=drawColor,line width= 0.6pt,fill=fillColor] (504.57, 45.47) rectangle (527.39, 45.47);

\path[draw=drawColor,line width= 0.6pt,fill=fillColor] (527.39, 45.47) rectangle (550.21, 45.47);

\path[draw=drawColor,line width= 0.6pt,fill=fillColor] (550.21, 45.47) rectangle (573.04, 45.47);

\path[draw=drawColor,line width= 0.6pt,fill=fillColor] (573.04, 45.47) rectangle (595.86, 45.47);

\path[draw=drawColor,line width= 1.1pt,line join=round] ( 48.10, 45.47) --
	( 53.58, 45.47) --
	( 59.06, 45.47) --
	( 64.54, 45.47) --
	( 70.01, 45.47) --
	( 75.49, 45.47) --
	( 80.97, 45.47) --
	( 86.45, 45.47) --
	( 91.92, 45.47) --
	( 97.40, 45.47) --
	(102.88, 45.47) --
	(108.36, 45.47) --
	(113.83, 45.47) --
	(119.31, 45.47) --
	(124.79, 45.47) --
	(130.27, 45.47) --
	(135.74, 45.47) --
	(141.22, 45.48) --
	(146.70, 45.49) --
	(152.18, 45.51) --
	(157.65, 45.54) --
	(163.13, 45.59) --
	(168.61, 45.68) --
	(174.09, 45.82) --
	(179.56, 46.04) --
	(185.04, 46.38) --
	(190.52, 46.90) --
	(196.00, 47.68) --
	(201.47, 48.81) --
	(206.95, 50.43) --
	(212.43, 52.70) --
	(217.91, 55.82) --
	(223.39, 60.01) --
	(228.86, 65.52) --
	(234.34, 72.62) --
	(239.82, 81.55) --
	(245.30, 92.54) --
	(250.77,105.75) --
	(256.25,121.25) --
	(261.73,138.99) --
	(267.21,158.76) --
	(272.68,180.19) --
	(278.16,202.74) --
	(283.64,225.70) --
	(289.12,248.21) --
	(294.59,269.36) --
	(300.07,288.18) --
	(305.55,303.75) --
	(311.03,315.28) --
	(316.50,322.15) --
	(321.98,324.00) --
	(327.46,320.71) --
	(332.94,312.47) --
	(338.41,299.73) --
	(343.89,283.15) --
	(349.37,263.58) --
	(354.85,241.95) --
	(360.32,219.22) --
	(365.80,196.29) --
	(371.28,174.00) --
	(376.76,152.98) --
	(382.24,133.76) --
	(387.71,116.64) --
	(393.19,101.79) --
	(398.67, 89.22) --
	(404.15, 78.83) --
	(409.62, 70.44) --
	(415.10, 63.82) --
	(420.58, 58.70) --
	(426.06, 54.84) --
	(431.53, 51.98) --
	(437.01, 49.91) --
	(442.49, 48.44) --
	(447.97, 47.42) --
	(453.44, 46.73) --
	(458.92, 46.27) --
	(464.40, 45.96) --
	(469.88, 45.77) --
	(475.35, 45.65) --
	(480.83, 45.57) --
	(486.31, 45.53) --
	(491.79, 45.50) --
	(497.26, 45.49) --
	(502.74, 45.48) --
	(508.22, 45.47) --
	(513.70, 45.47) --
	(519.17, 45.47) --
	(524.65, 45.47) --
	(530.13, 45.47) --
	(535.61, 45.47) --
	(541.09, 45.47) --
	(546.56, 45.47) --
	(552.04, 45.47) --
	(557.52, 45.47) --
	(563.00, 45.47) --
	(568.47, 45.47) --
	(573.95, 45.47) --
	(579.43, 45.47) --
	(584.91, 45.47) --
	(590.38, 45.47) --
	(595.86, 45.47);
\end{scope}
\begin{scope}
\definecolor{drawColor}{RGB}{0,0,0}

\end{scope}
\begin{scope}
\definecolor{drawColor}{RGB}{0,0,0}

\path[draw=drawColor,line width= 0.6pt,line join=round] ( 20.71, 30.69) --
	(623.25, 30.69);
\end{scope}
\begin{scope}
\definecolor{drawColor}{gray}{0.20}

\path[draw=drawColor,line width= 0.6pt,line join=round] (139.40, 27.94) --
	(139.40, 30.69);

\path[draw=drawColor,line width= 0.6pt,line join=round] (321.98, 27.94) --
	(321.98, 30.69);

\path[draw=drawColor,line width= 0.6pt,line join=round] (504.57, 27.94) --
	(504.57, 30.69);
\end{scope}
\begin{scope}
\definecolor{drawColor}{gray}{0.30}

\node[text=drawColor,anchor=base,inner sep=0pt, outer sep=0pt, scale=  0.88] at (139.40, -9.68) {-0.2};

\node[text=drawColor,anchor=base,inner sep=0pt, outer sep=0pt, scale=  0.88] at (321.98, -9.68) {0.0};

\node[text=drawColor,anchor=base,inner sep=0pt, outer sep=0pt, scale=  0.88] at (504.57, -9.68) {0.2};
\end{scope}
\begin{scope}
	\definecolor{drawColor}{RGB}{0,0,0}
	
	\node[text=drawColor,anchor=base,inner sep=0pt, outer sep=0pt, scale=  1.10] at (291.36,  -72.64) {\small $\langle\hat{\ba}_1,\ba_4\rangle$};
\end{scope}
\end{tikzpicture}
	\end{subfigure}\\
	\begin{subfigure}{0.25\textwidth}
		\begin{tikzpicture}[x=.22pt,y=.22pt]
\definecolor{fillColor}{RGB}{255,255,255}
\begin{scope}
\definecolor{drawColor}{RGB}{255,255,255}
\definecolor{fillColor}{RGB}{255,255,255}

\end{scope}
\begin{scope}
\definecolor{fillColor}{RGB}{255,255,255}

\path[fill=fillColor] ( 20.71, 30.69) rectangle (623.25,355.85);
\definecolor{drawColor}{RGB}{0,0,0}

\path[draw=drawColor,line width= 0.6pt,fill=fillColor] ( 48.10, 45.47) rectangle ( 75.49, 45.47);

\path[draw=drawColor,line width= 0.6pt,fill=fillColor] ( 75.49, 45.47) rectangle (102.88, 45.47);

\path[draw=drawColor,line width= 0.6pt,fill=fillColor] (102.88, 45.47) rectangle (130.27, 45.47);

\path[draw=drawColor,line width= 0.6pt,fill=fillColor] (130.27, 45.47) rectangle (157.65, 47.22);

\path[draw=drawColor,line width= 0.6pt,fill=fillColor] (157.65, 45.47) rectangle (185.04, 48.96);

\path[draw=drawColor,line width= 0.6pt,fill=fillColor] (185.04, 45.47) rectangle (212.43, 50.71);

\path[draw=drawColor,line width= 0.6pt,fill=fillColor] (212.43, 45.47) rectangle (239.82, 50.71);

\path[draw=drawColor,line width= 0.6pt,fill=fillColor] (239.82, 45.47) rectangle (267.21, 59.46);

\path[draw=drawColor,line width= 0.6pt,fill=fillColor] (267.21, 45.47) rectangle (294.59, 94.44);

\path[draw=drawColor,line width= 0.6pt,fill=fillColor] (294.59, 45.47) rectangle (321.98,173.15);

\path[draw=drawColor,line width= 0.6pt,fill=fillColor] (321.98, 45.47) rectangle (349.37,262.36);

\path[draw=drawColor,line width= 0.6pt,fill=fillColor] (349.37, 45.47) rectangle (376.76,341.07);

\path[draw=drawColor,line width= 0.6pt,fill=fillColor] (376.76, 45.47) rectangle (404.15,169.65);

\path[draw=drawColor,line width= 0.6pt,fill=fillColor] (404.15, 45.47) rectangle (431.53, 73.45);

\path[draw=drawColor,line width= 0.6pt,fill=fillColor] (431.53, 45.47) rectangle (458.92, 48.96);

\path[draw=drawColor,line width= 0.6pt,fill=fillColor] (458.92, 45.47) rectangle (486.31, 45.47);

\path[draw=drawColor,line width= 0.6pt,fill=fillColor] (486.31, 45.47) rectangle (513.70, 45.47);

\path[draw=drawColor,line width= 0.6pt,fill=fillColor] (513.70, 45.47) rectangle (541.09, 45.47);

\path[draw=drawColor,line width= 0.6pt,fill=fillColor] (541.09, 45.47) rectangle (568.47, 45.47);

\path[draw=drawColor,line width= 0.6pt,fill=fillColor] (568.47, 45.47) rectangle (595.86, 45.47);

\path[draw=drawColor,line width= 1.1pt,line join=round] ( 48.10, 45.47) --
	( 53.58, 45.47) --
	( 59.06, 45.47) --
	( 64.54, 45.47) --
	( 70.01, 45.47) --
	( 75.49, 45.47) --
	( 80.97, 45.47) --
	( 86.45, 45.47) --
	( 91.92, 45.47) --
	( 97.40, 45.47) --
	(102.88, 45.47) --
	(108.36, 45.47) --
	(113.83, 45.47) --
	(119.31, 45.47) --
	(124.79, 45.47) --
	(130.27, 45.47) --
	(135.74, 45.47) --
	(141.22, 45.47) --
	(146.70, 45.47) --
	(152.18, 45.47) --
	(157.65, 45.47) --
	(163.13, 45.47) --
	(168.61, 45.47) --
	(174.09, 45.47) --
	(179.56, 45.47) --
	(185.04, 45.47) --
	(190.52, 45.48) --
	(196.00, 45.49) --
	(201.47, 45.51) --
	(206.95, 45.54) --
	(212.43, 45.61) --
	(217.91, 45.73) --
	(223.39, 45.95) --
	(228.86, 46.31) --
	(234.34, 46.90) --
	(239.82, 47.86) --
	(245.30, 49.34) --
	(250.77, 51.57) --
	(256.25, 54.86) --
	(261.73, 59.55) --
	(267.21, 66.04) --
	(272.68, 74.77) --
	(278.16, 86.13) --
	(283.64,100.48) --
	(289.12,117.99) --
	(294.59,138.64) --
	(300.07,162.14) --
	(305.55,187.86) --
	(311.03,214.84) --
	(316.50,241.82) --
	(321.98,267.31) --
	(327.46,289.75) --
	(332.94,307.63) --
	(338.41,319.67) --
	(343.89,324.99) --
	(349.37,323.18) --
	(354.85,314.37) --
	(360.32,299.24) --
	(365.80,278.87) --
	(371.28,254.69) --
	(376.76,228.26) --
	(382.24,201.11) --
	(387.71,174.63) --
	(393.19,149.93) --
	(398.67,127.81) --
	(404.15,108.73) --
	(409.62, 92.83) --
	(415.10, 80.03) --
	(420.58, 70.05) --
	(426.06, 62.51) --
	(431.53, 56.98) --
	(437.01, 53.04) --
	(442.49, 50.33) --
	(447.97, 48.51) --
	(453.44, 47.32) --
	(458.92, 46.57) --
	(464.40, 46.10) --
	(469.88, 45.83) --
	(475.35, 45.66) --
	(480.83, 45.57) --
	(486.31, 45.52) --
	(491.79, 45.49) --
	(497.26, 45.48) --
	(502.74, 45.47) --
	(508.22, 45.47) --
	(513.70, 45.47) --
	(519.17, 45.47) --
	(524.65, 45.47) --
	(530.13, 45.47) --
	(535.61, 45.47) --
	(541.09, 45.47) --
	(546.56, 45.47) --
	(552.04, 45.47) --
	(557.52, 45.47) --
	(563.00, 45.47) --
	(568.47, 45.47) --
	(573.95, 45.47) --
	(579.43, 45.47) --
	(584.91, 45.47) --
	(590.38, 45.47) --
	(595.86, 45.47);
\end{scope}
\begin{scope}
\definecolor{drawColor}{RGB}{0,0,0}

\end{scope}
\begin{scope}
\definecolor{drawColor}{RGB}{0,0,0}

\path[draw=drawColor,line width= 0.6pt,line join=round] ( 20.71, 30.69) --
	(623.25, 30.69);
\end{scope}
\begin{scope}
\definecolor{drawColor}{gray}{0.20}

\path[draw=drawColor,line width= 0.6pt,line join=round] ( 48.10, 27.94) --
	( 48.10, 30.69);

\path[draw=drawColor,line width= 0.6pt,line join=round] (157.65, 27.94) --
	(157.65, 30.69);

\path[draw=drawColor,line width= 0.6pt,line join=round] (267.21, 27.94) --
	(267.21, 30.69);

\path[draw=drawColor,line width= 0.6pt,line join=round] (376.76, 27.94) --
	(376.76, 30.69);

\path[draw=drawColor,line width= 0.6pt,line join=round] (486.31, 27.94) --
	(486.31, 30.69);

\path[draw=drawColor,line width= 0.6pt,line join=round] (595.86, 27.94) --
	(595.86, 30.69);
\end{scope}
\begin{scope}
\definecolor{drawColor}{gray}{0.30}

\node[text=drawColor,anchor=base,inner sep=0pt, outer sep=0pt, scale=  0.88] at ( 48.10, -9.68) {0.4};

\node[text=drawColor,anchor=base,inner sep=0pt, outer sep=0pt, scale=  0.88] at (157.65, -9.68) {0.5};

\node[text=drawColor,anchor=base,inner sep=0pt, outer sep=0pt, scale=  0.88] at (267.21, -9.68) {0.6};

\node[text=drawColor,anchor=base,inner sep=0pt, outer sep=0pt, scale=  0.88] at (376.76, -9.68) {0.7};

\node[text=drawColor,anchor=base,inner sep=0pt, outer sep=0pt, scale=  0.88] at (486.31, -9.68) {0.8};

\node[text=drawColor,anchor=base,inner sep=0pt, outer sep=0pt, scale=  0.88] at (595.86, -9.68) {0.9};
\end{scope}
\begin{scope}
	\definecolor{drawColor}{RGB}{0,0,0}
	
	\node[text=drawColor,anchor=base,inner sep=0pt, outer sep=0pt, scale=  1.10] at (291.36,  -72.64) {\small $\langle \hat{\ba}_{1},(\ba_1+\ba_2)/\sqrt{2}\rangle$};
\end{scope}
\end{tikzpicture}
	\end{subfigure}\hfil
	\begin{subfigure}{0.25\textwidth}
		\begin{tikzpicture}[x=.22pt,y=.22pt]
\definecolor{fillColor}{RGB}{255,255,255}
\begin{scope}
\definecolor{drawColor}{RGB}{255,255,255}
\definecolor{fillColor}{RGB}{255,255,255}

\end{scope}
\begin{scope}
\definecolor{fillColor}{RGB}{255,255,255}

\path[fill=fillColor] ( 20.71, 30.69) rectangle (623.25,355.85);
\definecolor{drawColor}{RGB}{0,0,0}

\path[draw=drawColor,line width= 0.6pt,fill=fillColor] ( 48.10, 45.47) rectangle ( 78.53, 45.47);

\path[draw=drawColor,line width= 0.6pt,fill=fillColor] ( 78.53, 45.47) rectangle (108.96, 45.47);

\path[draw=drawColor,line width= 0.6pt,fill=fillColor] (108.96, 45.47) rectangle (139.40, 47.29);

\path[draw=drawColor,line width= 0.6pt,fill=fillColor] (139.40, 45.47) rectangle (169.83, 45.47);

\path[draw=drawColor,line width= 0.6pt,fill=fillColor] (169.83, 45.47) rectangle (200.26, 54.59);

\path[draw=drawColor,line width= 0.6pt,fill=fillColor] (200.26, 45.47) rectangle (230.69, 49.12);

\path[draw=drawColor,line width= 0.6pt,fill=fillColor] (230.69, 45.47) rectangle (261.12, 56.41);

\path[draw=drawColor,line width= 0.6pt,fill=fillColor] (261.12, 45.47) rectangle (291.55,116.63);

\path[draw=drawColor,line width= 0.6pt,fill=fillColor] (291.55, 45.47) rectangle (321.98,164.07);

\path[draw=drawColor,line width= 0.6pt,fill=fillColor] (321.98, 45.47) rectangle (352.41,235.24);

\path[draw=drawColor,line width= 0.6pt,fill=fillColor] (352.41, 45.47) rectangle (382.84,341.07);

\path[draw=drawColor,line width= 0.6pt,fill=fillColor] (382.84, 45.47) rectangle (413.27,193.27);

\path[draw=drawColor,line width= 0.6pt,fill=fillColor] (413.27, 45.47) rectangle (443.71, 94.73);

\path[draw=drawColor,line width= 0.6pt,fill=fillColor] (443.71, 45.47) rectangle (474.14, 58.24);

\path[draw=drawColor,line width= 0.6pt,fill=fillColor] (474.14, 45.47) rectangle (504.57, 45.47);

\path[draw=drawColor,line width= 0.6pt,fill=fillColor] (504.57, 45.47) rectangle (535.00, 45.47);

\path[draw=drawColor,line width= 0.6pt,fill=fillColor] (535.00, 45.47) rectangle (565.43, 45.47);

\path[draw=drawColor,line width= 0.6pt,fill=fillColor] (565.43, 45.47) rectangle (595.86, 45.47);

\path[draw=drawColor,line width= 1.1pt,line join=round] ( 48.10, 45.47) --
	( 53.58, 45.47) --
	( 59.06, 45.47) --
	( 64.54, 45.47) --
	( 70.01, 45.47) --
	( 75.49, 45.47) --
	( 80.97, 45.47) --
	( 86.45, 45.47) --
	( 91.92, 45.47) --
	( 97.40, 45.47) --
	(102.88, 45.47) --
	(108.36, 45.47) --
	(113.83, 45.47) --
	(119.31, 45.47) --
	(124.79, 45.47) --
	(130.27, 45.47) --
	(135.74, 45.47) --
	(141.22, 45.47) --
	(146.70, 45.47) --
	(152.18, 45.47) --
	(157.65, 45.48) --
	(163.13, 45.49) --
	(168.61, 45.51) --
	(174.09, 45.54) --
	(179.56, 45.58) --
	(185.04, 45.66) --
	(190.52, 45.77) --
	(196.00, 45.94) --
	(201.47, 46.20) --
	(206.95, 46.59) --
	(212.43, 47.14) --
	(217.91, 47.94) --
	(223.39, 49.05) --
	(228.86, 50.59) --
	(234.34, 52.67) --
	(239.82, 55.44) --
	(245.30, 59.07) --
	(250.77, 63.72) --
	(256.25, 69.59) --
	(261.73, 76.84) --
	(267.21, 85.65) --
	(272.68, 96.13) --
	(278.16,108.35) --
	(283.64,122.31) --
	(289.12,137.92) --
	(294.59,154.97) --
	(300.07,173.15) --
	(305.55,192.05) --
	(311.03,211.14) --
	(316.50,229.81) --
	(321.98,247.41) --
	(327.46,263.25) --
	(332.94,276.70) --
	(338.41,287.17) --
	(343.89,294.20) --
	(349.37,297.47) --
	(354.85,296.83) --
	(360.32,292.30) --
	(365.80,284.09) --
	(371.28,272.58) --
	(376.76,258.28) --
	(382.24,241.78) --
	(387.71,223.76) --
	(393.19,204.88) --
	(398.67,185.79) --
	(404.15,167.08) --
	(409.62,149.22) --
	(415.10,132.62) --
	(420.58,117.53) --
	(426.06,104.14) --
	(431.53, 92.49) --
	(437.01, 82.57) --
	(442.49, 74.29) --
	(447.97, 67.51) --
	(453.44, 62.07) --
	(458.92, 57.77) --
	(464.40, 54.45) --
	(469.88, 51.92) --
	(475.35, 50.03) --
	(480.83, 48.64) --
	(486.31, 47.64) --
	(491.79, 46.94) --
	(497.26, 46.44) --
	(502.74, 46.11) --
	(508.22, 45.88) --
	(513.70, 45.73) --
	(519.17, 45.63) --
	(524.65, 45.57) --
	(530.13, 45.53) --
	(535.61, 45.50) --
	(541.09, 45.49) --
	(546.56, 45.48) --
	(552.04, 45.47) --
	(557.52, 45.47) --
	(563.00, 45.47) --
	(568.47, 45.47) --
	(573.95, 45.47) --
	(579.43, 45.47) --
	(584.91, 45.47) --
	(590.38, 45.47) --
	(595.86, 45.47);
\end{scope}
\begin{scope}
\definecolor{drawColor}{RGB}{0,0,0}

\end{scope}
\begin{scope}
\definecolor{drawColor}{RGB}{0,0,0}

\path[draw=drawColor,line width= 0.6pt,line join=round] ( 20.71, 30.69) --
	(623.25, 30.69);
\end{scope}
\begin{scope}
\definecolor{drawColor}{gray}{0.20}

\path[draw=drawColor,line width= 0.6pt,line join=round] ( 48.10, 27.94) --
	( 48.10, 30.69);

\path[draw=drawColor,line width= 0.6pt,line join=round] (169.83, 27.94) --
	(169.83, 30.69);

\path[draw=drawColor,line width= 0.6pt,line join=round] (291.55, 27.94) --
	(291.55, 30.69);

\path[draw=drawColor,line width= 0.6pt,line join=round] (413.27, 27.94) --
	(413.27, 30.69);

\path[draw=drawColor,line width= 0.6pt,line join=round] (535.00, 27.94) --
	(535.00, 30.69);
\end{scope}
\begin{scope}
\definecolor{drawColor}{gray}{0.30}

\node[text=drawColor,anchor=base,inner sep=0pt, outer sep=0pt, scale=  0.88] at ( 48.10, -9.68) {0.3};

\node[text=drawColor,anchor=base,inner sep=0pt, outer sep=0pt, scale=  0.88] at (169.83, -9.68) {0.4};

\node[text=drawColor,anchor=base,inner sep=0pt, outer sep=0pt, scale=  0.88] at (291.55, -9.68) {0.5};

\node[text=drawColor,anchor=base,inner sep=0pt, outer sep=0pt, scale=  0.88] at (413.27, -9.68) {0.6};

\node[text=drawColor,anchor=base,inner sep=0pt, outer sep=0pt, scale=  0.88] at (535.00, -9.68) {0.7};
\end{scope}
\begin{scope}
	\definecolor{drawColor}{RGB}{0,0,0}
	
	\node[text=drawColor,anchor=base,inner sep=0pt, outer sep=0pt, scale=  1.10] at (291.36,  -72.64) {\small $\langle \hat{\ba}_{1},(\ba_1+\ba_2+\ba_3)/\sqrt{3}\rangle$};
\end{scope}
\end{tikzpicture}
	\end{subfigure}
	\vspace{-5mm}
	\caption{Asymptotic normality of linear forms of the mixing matrix $(d=50,\,n=6000)$.}
	\label{fig:ica-clt}
\end{figure}

\subsection{Real Data Example}

Finally, we consider a financial application of the proposed ICA method. Linear factor models are widely used to describe cross-sectional asset price movements and assess portfolio risks. See, e.g., \cite{connor2010portfolio}. PCA and asymptotic PCA are routinely used to extract latent factors. However, recent empirical evidence has suggested that they are inadequate in capturing higher-order risk and many have advocated the use of ICA. See, e.g., \cite{lassance2021portfolio, lassance2022optimal} and references therein. To illustrate the practical merits of our approach, we apply it to a benchmark dataset constructed by Ken French. The data contain monthly returns of 30 industrial portfolios from July 1972 to September 2022. Each portfolio takes equal weights for all companies from a certain industrial sector such as \texttt{Autos}, \texttt{Rtail}, and \texttt{Txtls}. The estimated factor loadings for the first four sources and their (asymptotic) 95\% confidence intervals are given in Figure \ref{fig:real-ci}.

\begin{figure}[htbp]
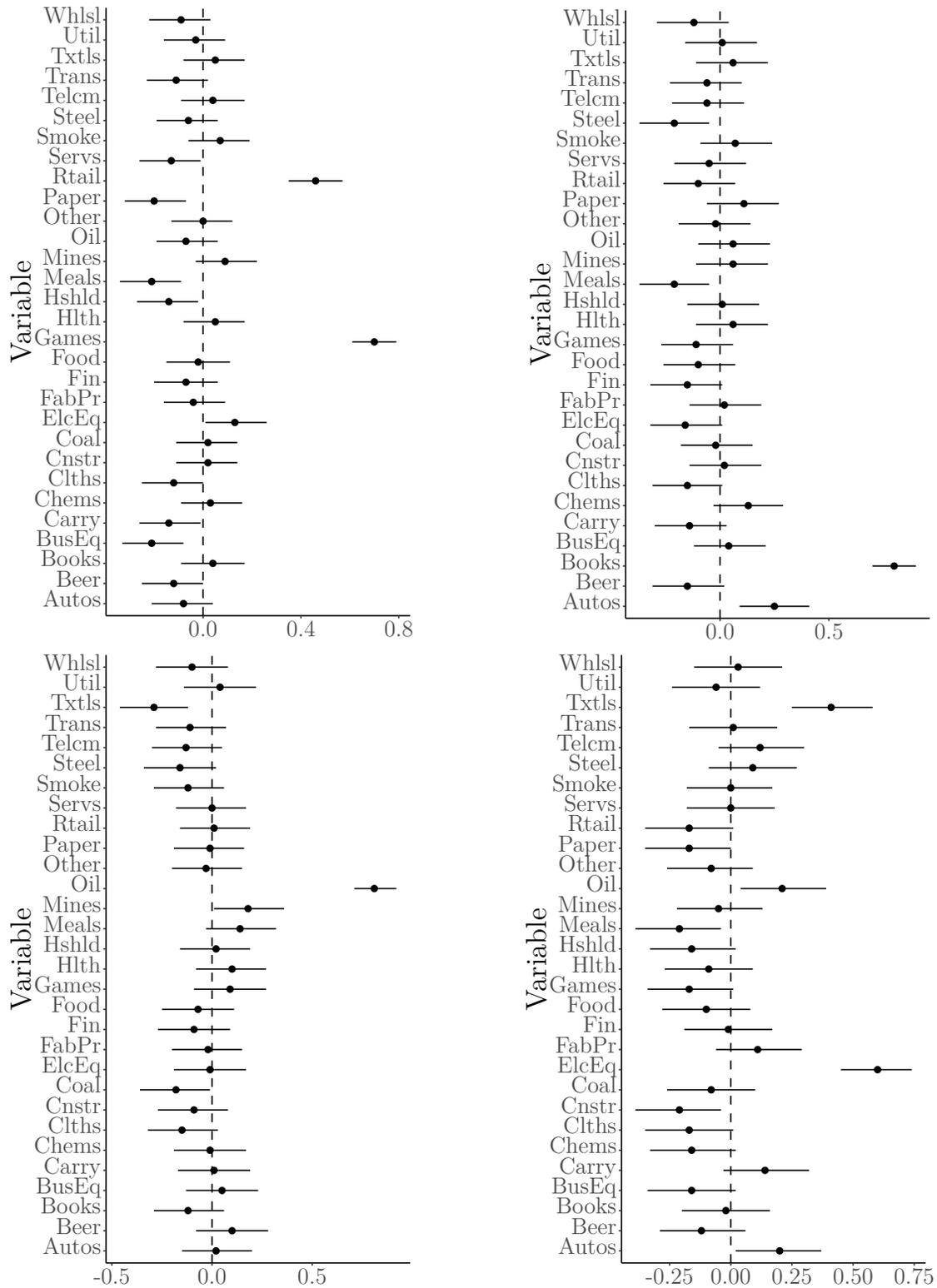

	\begin{subfigure}{.5\textwidth}
		\centering
		\input{real-conf-1.tex}
	\end{subfigure}\hfil
	\begin{subfigure}{.5\textwidth}
		\centering
		\input{real-conf-2.tex}
	\end{subfigure}
	\medskip
	\begin{subfigure}{.5\textwidth}
		\centering
		\input{real-conf-3.tex}
	\end{subfigure}\hfil
	\begin{subfigure}{.5\textwidth}
		\centering
		\input{real-conf-4.tex}
	\end{subfigure}
	\caption{Confidence intervals for the factor loadings of the first four sources.}
	\label{fig:real-ci}
\end{figure}



From Figure~\ref{fig:real-ci}, it is interesting to note that the independent component vectors (ICs) form a separation of the industry sectors. For example, the significant coefficients in the first IC correspond to the consumer sector: \texttt{Games}-which stands for sports and entertainment; \texttt{Rtail}- the retail sector, \texttt{Meals}- which includes restaurants, motels etc., \texttt{Paper}- for the paper and stationery products, \texttt{BusEq}- which stands for business equipment like computers, phones, radio, communication devices; and finally \texttt{Hshld} and \texttt{Beers}, which denote household items and alcoholic beverages respectively. The second IC has significant coefficients on \texttt{Books}-meaning the publishing industry of books, newspapers, periodicals etc; followed by \texttt{Autos}-standing for automobiles; \texttt{Steel}-referring to heavy industries involving steel; and finally \texttt{Meals}. We suspect that this IC is driven by the heavy industry sector but is somehow correlated with the publishing industry. The third IC has an energy component: as led by the significant coefficients for \texttt{Oil} (petroleum, natural gas etc.), \texttt{Coal}, \texttt{Mines} as well as \texttt{Txtls}-consisting of mills. Finally, the fourth IC once again has significant coefficients on \texttt{Txtls}, \texttt{Oil}, \texttt{Meals}, and \texttt{Autos} but is now accompanied by \texttt{ElcEq}-denoting electrical equipment in general, along with \texttt{Cnstr}-which stands for heavy construction materials.


\section{Concluding Remarks}\label{sec:conclude}

In this paper, we study the effect of dimensionality on statistical properties and computational complexity of ICA. Our work contributes in three different directions. First, we derive the information-theoretic lower bound on the sample complexity required to perform ICA in high dimensions. To accompany this lower bound we investigate why commonly used ICA techniques are necessarily suboptimal and propose a truncated tensor decomposition method that is indeed optimal. The proposed minimax optimal method is however computationally intractable which prompts us we study the optimal sample complexity under computational constraints. More specifically, under the low-degree polynomial framework of \cite{hopkins2018statistical,kunisky2022notes}, we show that the computational sample complexity threshold is at $n\gg d^2$, up to a logarithmic factor and thus demonstrate a gap in the information-theoretic and computational thresholds.

Secondly, in Section \ref{sec:tract}, we show how popular approaches such as FastICA may be suboptimal and identify initialization as the culprit in deriving computationally feasible and statistically efficient estimators in high dimensions. We demonstrate the shortcomings of initialization schemes prevalent in the literature and develop new initialization schemes that result in a minimax rate optimal estimator for ICA whenever $n\gg d^2$, matching the computational lower bound up to a logarithmic factor. Finally, in Section~\ref{sec:asy-dist}, we derive the asymptotic distributions of our estimators which can be readily used for statistical inferences.

We now point out a few related questions that should be investigated further. Firstly, note that for simplicity we have assumed that the observations are pre-whitened. In other words, we assume that the data are isotropic and the mixing matrices are orthonormal. In general, we need to estimate the covariance matrix first, before applying our proposed method. In Section 3.4, we also show that these two steps can be performed with sample splitting. It is possible that this data splitting is not essential, and we leave this question as a possible topic for future research.

More importantly, it is imperative to suggest methods of performing ICA for problems of even higher dimensions. We show that without further assumptions $n\asymp d^2$ is the correct computational threshold. However, it is possible that this sample complexity result can be improved greatly if one assumes more structure on the mixing matrices. For example, sparsity is a popular assumption of this nature. It would be interesting to devise an ICA algorithm that depends optimally on the sparsity level while having minimal dependence on the dimension $d$. Along this direction, it would be interesting to consider a noisy version of the ICA problem, where instead of $\bX=\bA\bS$ one observes $\bA\bS+\bE$ for some noise vector $\bE$ from some other distribution. This problem has been considered several times in the literature \citep[see, e.g.,][]{belkin2013blind} but an optimal dependence on the noise variance still remains elusive.

Finally, we point out the dependence of our methods on the moments of the independent sources. Following FastICA, we use an algorithm that measures departure from Gaussianity in terms of the fourth cumulants, i.e., the kurtosis. In principle, the same approach should provide a method of estimation whenever some higher-order cumulant is sufficiently different from zero (since all corresponding cumulants for the Gaussian distribution are zero). However, this is tied to the question of estimating an even higher-order moment tensor, which is presumed to be a computationally more difficult task. It would be interesting to determine this tradeoff between managing the computational burden vs making the moment assumptions less restrictive. For our analysis, we also assume finite $(8+\eps)$-th moments for rate optimal estimation and finite 12th moments for asymptotic distributions. It might be possible to weaken these assumptions further by using some suitably truncated estimators, for example. We leave this question for future research.

\bibliographystyle{plainnat}
\bibliography{references}


\appendix

\section{Proofs}

We now present the proofs of our main results. To facilitate the presentation, we defer the proofs of preliminary results and technical tools to Section \ref{sec:prooflemma}.

\subsection{Proofs of Results in Section~\ref{sec:inf-th}}
\begin{proof}[Proof of Theorem~\ref{th:ICAlower}]
	For $1\le k\le d,$ we take independent random variables 
	$$S_k=\alpha S_{k,1}+\sqrt{1-\alpha^2}S_{k,2}$$ 
	where $S_{k,1}\sim N(0,1)$ and $S_{k,2}$ is a Rademacher random variable. It is not hard to see that
	$$
	|\kappa_4(S_k)|=2(1-\alpha^2)^2.
	$$
	We consider the ICA model as $\bX=\bA\bS,$ where $\bA$ is a $d\times d$ orthogonal matrix, while $\bS$ is a $d$-dimensional vector consisting of i.i.d. random variables $S_k$ as defined above. 
	
	By standard results on packing numbers of orthonormal matrices, \citep[see e.g.,][]{szarek1997metric}, it is possible to construct a set $\calA$ of orthonormal matrices such that, for a sufficiently small $\delta>0$, we have
	\begin{equation}\label{eq:Aij-F-lbd}
		\sqrt{d\delta/2}\le \|\bA^{(i)}-\bA^{(j)}\|_{\rm F}\le
		\sqrt{d\delta} 
	\end{equation}
	for $i\neq j$ and $\bA^{(i)},\bA^{(j)}\in \calA$ and 
	\begin{equation}\label{eq:calA-card}
		|\calA|\ge 3^{d^2}
	\end{equation}		
	
	\noindent Note that the density function of each $S_k$ is
	\begin{align*}
		f(s)
		=&~
		\dfrac{1}{2\sqrt{2\pi}\alpha}
		\exp(-(s-\sqrt{1-\alpha^2})^2/2\alpha^2)
		+\dfrac{1}{2\sqrt{2\pi}\alpha}
		\exp(-(s+\sqrt{1-\alpha^2})^2/2\alpha^2)\\
		=&~
		\dfrac{1}{\sqrt{2\pi}\alpha}
		\exp\left(-\dfrac{s^2}{2\alpha^2}\right)
		\cdot
		\cosh\left(
		\dfrac{s\sqrt{1-\alpha^2}}{\alpha^2}\right)
		\cdot
		\exp\left(-\dfrac{(1-\alpha^2)}{2\alpha^2}\right).
	\end{align*}
	It follows that for two mixing matrices $\bP\neq\bQ\in \calO(d)$, we have the joint pdf's for $\bX$ are given by
	$$
	g_{\bP}(\bx)=\prod_{k=1}^df(\bp_k^{\top}\bx)
	\,\,
	\text{ and }
	g_{\bQ}(\bx)=\prod_{k=1}^df(\bq_k^{\top}\bx)
	$$
	when $\bA=\bP$ and $\bA=\bQ$ respectively.
	We then have
	\begin{align*}\label{eq:ratio-ubd}
		\dfrac{g_{\bP}(\bx)}{g_{\bQ}(\bx)}
		=&~\prod_{k=1}^d
		\left(
		\exp\left(
		-\dfrac{(\bp_k^{\top}\bx)^2
			-(\bq_k^{\top}\bx)^2}
		{2\alpha^2}
		\right)
		\cdot
		\dfrac{
			\cosh(\bx^{\top}\bp_k\sqrt{1-\alpha^2}/\alpha^2)		
		}{\cosh(\bx^{\top}\bq_k\sqrt{1-\alpha^2}/\alpha^2)}
		\right)\\
		=&~
		\prod_{k=1}^d
		\dfrac{
			\cosh(\bx^{\top}\bp_k\sqrt{1-\alpha^2}/\alpha^2)		
		}{\cosh(\bx^{\top}\bq_k\sqrt{1-\alpha^2}/\alpha^2)}.
		\numberthis
	\end{align*}
	The second equality uses the fact that $\bP,\bQ$ are orthonormal matrices, and thus
	$
	\|\bP\bx\|^2
	=\|\bQ\bx\|^2
	$. When $\bX=\bA\bS$ and we consider $\bA=\bA_i$ and $\bA=\bA_j$ for some $\bA^{(i)},\bA^{(j)}\in\calA$ and $i\neq j$ , it follows that  
	$$
	\ba^{(i)\top}_k\bX=\ba_k^{(i)\top}\bA^{(j)}\bS
	=\ba_k^{(i)\top}\ba_k^{(j)}S_k
	+\sum_{l\neq k}\ba_l^{(i)\top}\ba_l^{(j)}\bS_l.
	$$
	Consequently, the Kullback-Leibler divergence, when we observe $\bX_1,\dots,\bX_n$, can be bounded as
	\begin{align*}\label{eq:kldiv}
		&~{\rm KL}(g_{\bA^{(i)}};g_{\bA^{(j)}})\\
		=&~n\EE_{g_{\bA^{(i)}}}
		\left(
		\log\left(
		\dfrac{g_{\bA^{(i)}}(\bX)}{g_{\bA^{(j)}}(\bX)}
		\right)
		\right)\\
		=&~n
		\sum_{k=1}^d
		\EE
		\left(
		\log(\cosh(S_k\sqrt{1-\alpha^2}/\alpha^2))
		-\log
		\left(\cosh
		\left(
		\dfrac{\sqrt{1-\alpha^2}}{\alpha^2}
		\cdot
		\sum_{l=1}^d
		\langle
		\ba_k^{(i)},
		\ba_l^{(j)}
		\rangle
		S_l
		\right)
		\right)
		\right)\\
		=&~
		n\sum_{k=1}^d
		\sum_{m=0}^{\infty}
		\dfrac{{\rm cum}_{2m}(R)}{(2m)!}
		\left(
		\EE(S_k\sqrt{1-\alpha^2}/\alpha^2)^{2m}
		-\EE
		\left(
		\dfrac{\sqrt{1-\alpha^2}}{\alpha^2}
		\cdot
		\sum_{l=1}^d
		\langle
		\ba_k^{(i)},
		\ba_l^{(j)}
		\rangle
		S_l
		\right)^{2m}
		\right)\\
		=&~
		n\sum_{k=1}^d
		\sum_{m=2}^{\infty}
		\dfrac{{\rm cum}_{2m}(R)(1-\alpha^2)^m}{(2m)!\alpha^{4m}}
		\left(
		\EE\left(S_k^{2m}\right)
		-\EE
		\left(
		\sum_{l=1}^d
		\langle
		\ba_k^{(i)},
		\ba_l^{(j)}
		\rangle
		S_l
		\right)^{2m}
		\right)\\
		=&~
		n\sum_{k=1}^d
		\sum_{m=2}^{\infty}
		\dfrac{{\rm cum}_{2m}(R)(1-\alpha^2)^m}{(2m)!\alpha^{4m}}
		\left(
		\EE\left(\alpha Z+\sqrt{1-\alpha^2}R_k\right)^{2m}
		-\EE
		\left(
		\alpha Z
		+\sqrt{1-\alpha^2}
		\sum_{l=1}^d
		\langle
		\ba_k^{(i)},
		\ba_l^{(j)}
		\rangle
		R_l
		\right)^{2m}
		\right)\\
		=&~
		n\sum_{k=1}^d
		\sum_{m=2}^{\infty}
		\dfrac{{\rm cum}_{2m}(R)(1-\alpha^2)^m}{(2m)!\alpha^{4m}}
		\left(
		\EE X_k^{2m}
		-\EE
		Y_k^{2m}
		\right)\numberthis
	\end{align*}
	where we define the random variables within the summation as $X_k$ and $Y_k$ respectively. The third equality uses the fact that $\log(\cosh(x))$ is the cumulant generating function of Rademacher random variables. We also use the notation $Z$ and $R_k$ to mean independent  $N(0,1)$ and Rademacher random variables respectively. Note that
	$$
	X_k-Y_k
	=\sqrt{1-\alpha^2}
	\left(
	(\langle
	\ba_k^{(i)},
	\ba_k^{(j)}
	\rangle
	-1)R_k
	+
	\sum_{l\neq k}
	\langle
	\ba_k^{(i)},
	\ba_l^{(j)}
	\rangle
	R_l\right).
	$$
	Note that $\EE (X_k-Y_k)^{2m+1}=0$ for any $m\in \NN$ and any $k\in [d]$, and 
	$
	\EE(X_k-Y_k)^2
	={\rm Var}(X_k-Y_k)
	=(1-\alpha^2)
	\|\ba_k^{(i)}
	-\ba_k^{(j)}\|^2
	$. 
	
	Note that $Z,R_1,\dots, R_k$ are independent random variables which are symmetric around zero. Thus $\EE(\alpha Z+\sqrt{1-\alpha^2}R_k)^{2m-1}R_j=0$ for all $j\neq k$, and moreover
	\begin{align*}
		\EE(\alpha Z+\sqrt{1-\alpha^2}R_k)^{2m-1}R_k
		=&~\sum_{l=0}^{2m-1}
		{{2m-1}\choose{l}}
		(\alpha)^{2m-1-l}
		(1-\alpha^2)^{l/2}
		\EE Z^{2m-1-l}\EE R_k^{l+1}\\
		=&~\sum_{s=1}^{m}
		{{2m-1}\choose{2s-1}}
		(\alpha)^{2m-2s}
		(1-\alpha^2)^{(2s-1)/2}
		\EE Z^{2m-2s}\\
		\le&~\sqrt{1-\alpha^2}(\EE(1+\alpha Z)^{2m-1}-1).
	\end{align*}
	Then by binomial series expansion of $(Y_k-X_k+X_k)^{2m}$, it can be checked that for any $m\ge 3$, 
	\begin{align*}
		&~|\EE X_k^{2m}-\EE Y_k^{2m}|\\
		=&~
		\sqrt{1-\alpha^2}
		\cdot
		2m
		\cdot 
		(1-\langle\ba_k^{(i)},\ba_k^{(j)}\rangle)
		\EE 
		\left((\alpha Z+\sqrt{1-\alpha^2}R_k)^{2m-1}R_k\right)\\
		&~+
		\EE\left(
		\sum_{l=2}^{2m}
		{{2m}\choose{l}}
		(1-\alpha^2)^{l/2}
		\|\ba_k^{(i)}-\ba_k^{(j)}\|^{l}
		(\alpha Z+\sqrt{1-\alpha^2}R_k)^{2m-l}
		\left(R_k
		+
		\dfrac{1}{\|\ba^{(i)}_k-\ba^{(j)}_k\|}
		\sum_{l\neq k}
		\langle
		\ba_k^{(i)},
		\ba_l^{(j)}
		\rangle
		R_l
		\right)^l
		\right)
		\\
		\le&~ (1-\alpha^2)
		\|\ba_k^{(i)}
		-\ba_k^{(j)}\|^2
		\cdot
		\EE\left(
		1+
		\alpha Z
		+\sqrt{1-\alpha^2}
		\left(2R_k
		+
		\dfrac{1}{\|\ba^{(i)}_k-\ba^{(j)}_k\|}
		\sum_{l\neq k}
		\langle
		\ba_k^{(i)},
		\ba_l^{(j)}
		\rangle
		R_l
		\right)\right)^{2m}.
	\end{align*}
	Plugging this back into \eqref{eq:kldiv}, we have
	\begin{align*}
		&~{\rm KL}(g_{\bA^{(i)}};g_{\bA^{(j)}})\\
		=&~
		n\sum_{k=1}^d
		\sum_{m=2}^{\infty}
		\dfrac{{\rm cum}_{2m}(R)(1-\alpha^2)^m}{(2m)!\alpha^{4m}}
		\left(
		\EE X_k^{2m}
		-\EE
		Y_k^{2m}
		\right)\\
		=&~
		n\sum_{k=1}^d
		\dfrac{-2(1-\alpha^2)^2}{4!\alpha^8}
		\cdot
		(1-\alpha^2)^2
		\left(1-
		\sum_{l=1}^d
		\langle
		\ba_k^{(i)},\ba_l^{(j)}
		\rangle^4
		\right)\\
		&\hspace{1cm}~+
		n\sum_{k=1}^d
		\sum_{m=3}^{\infty}
		\dfrac{{\rm cum}_{2m}(R)(1-\alpha^2)^m}{(2m)!\alpha^{4m}}
		\left(
		\EE X_k^{2m}
		-\EE
		Y_k^{2m}
		\right)\\
		\le&~
		n\sum_{k=1}^d
		\dfrac{4(1-\alpha^2)^2}{4!\alpha^8}
		\cdot
		(1-\alpha^2)^2
		(1-\langle\ba^{(i)}_k,\ba^{(j)}_k\rangle^2)
		+
		n\sum_{k=1}^d
		\sum_{m=3}^{\infty}
		\dfrac{{\rm cum}_{2m}(R)(1-\alpha^2)^{m+1}}{(2m)!\alpha^{4m}}
		\|\ba_k^{(i)}
		-\ba_k^{(j)}\|^2\times\\
		&\hspace{1cm}\times
		\EE\left(
		1+
		\alpha Z
		+\sqrt{1-\alpha^2}
		\left(2R_k
		+
		\sum_{l\neq k}
		\dfrac{
			\langle
			\ba_k^{(i)},
			\ba_l^{(j)}
			\rangle}{\|\ba^{(i)}_k-\ba^{(j)}_k\|}
		R_l
		\right)\right)^{2m}
	\end{align*}
\begin{align*}
		\le &~ 
		4n(1-\alpha^2)^4\sum_{k=1}^d\|\ba^{(i)}_k-\ba^{(j)}_k\|^2
		\EE 
		\log\cosh
		\left(
		1/\alpha^2
		Z/\alpha
		+\dfrac{\sqrt{1-\alpha^2}}{\alpha^2}
		\left(2R_k
		+
		\sum_{l\neq k}
		\dfrac{
			\langle
			\ba_k^{(i)},
			\ba_l^{(j)}
			\rangle}{\|\ba^{(i)}_k-\ba^{(j)}_k\|}
		R_l
		\right)\right)\\
		\le&~
		4n(1-\alpha^2)^4\sum_{k=1}^d\|\ba^{(i)}_k-\ba^{(j)}_k\|^2
		\EE 
		\abs*{
			1/\alpha^2+
			Z/\alpha
			+\dfrac{\sqrt{1-\alpha^2}}{\alpha^2}
			\left(2R_k
			+
			\sum_{l\neq k}
			\dfrac{
				\langle
				\ba_k^{(i)},
				\ba_l^{(j)}
				\rangle}{\|\ba^{(i)}_k-\ba^{(j)}_k\|}
			R_l
			\right)}\\
		\le&~
		4n(1-\alpha^2)^4\sum_{k=1}^d\|\ba^{(i)}_k-\ba^{(j)}_k\|^2
		\cdot
		\left(
		\dfrac{1}{\alpha^2}
		+
		\dfrac{3(1-\alpha^2)+1}{\alpha^4}
		\right)^{1/2}\\
		\le&~
		8n(1-\alpha^2)^4
		\|\bA^{(i)}-\bA^{(j)}\|_{\rm F}^2/\alpha 
		\le 8n(1-\alpha^2)^4(d\delta)/\alpha.
	\end{align*}
	The first inequality follows by the upper bound we derived earlier. The next step uses the power series expansion of $\log(\cosh(x))$. The third inequality uses the fact that $\cosh(x)\le \exp(|x|)$. The fourth inequality is an application of Cauchy-Schwarz, while the last two steps involve the assumption $\alpha^2>2/3$.
	
	By the generalized Fano inequality \citep[see Lemma 3 of ][]{yu1997assouad}, we then have
	$$
	\underset{\hat{\bA}}{\inf}
	\underset{\bA\in \calA}{\sup}
	\EE\norm*{\hat{\bA}-\bA}_{\rm F}
	\ge 
	\sqrt{\dfrac{d\delta}{2}}
	\left(1-\dfrac{8n(1-\alpha^2)^4d\delta/\alpha+\log 2}{d^2\log 3}\right).
	$$
	We now take 
	$$
	\delta=\dfrac{d\alpha}{10(1-\alpha^2)^4n}
	$$
	to get that 
	\begin{equation}\label{eq:fro-inf-lb}
		\underset{\hat{\bA}}{\inf}\underset{\bA\in \calA}
		{\sup}
		\EE
		\norm*{\hat{\bA}-\bA}_{\rm F}
		\ge 
		c\sqrt{\dfrac{d^2\alpha^4}{(1-\alpha^2)^4n}}
		\ge 
		\dfrac{c}{\min_{1\le k\le d}|\kappa_4(S_k)|}\sqrt{\dfrac{d^2}{n}}.
	\end{equation}
	This finishes the proof for the $\ell_A$ error. The lower bound on the maximum columnwise error follows immediately by noting that
	$$
	\max_{1\le j\le d}
	\|\hat{\ba}_j-\ba_j\|
	\ge \left(
	\dfrac{1}{d}
	\sum_{j=1}^d
	\|\hat{\ba}_j-\ba_j\|^2
	\right)^{1/2}
	=\dfrac{1}{\sqrt{d}}
	\cdot \|\hat{\bA}-\bA\|_{\rm F}.
	$$
\end{proof}

\medskip

\begin{proof}[Proof of Theorem~\ref{th:samp-lbd}]
	From the independent samples $\bX_1,\dots,\bX_n$, let us define
	$$
	\bv=\bX_1/\|\bX_1\|.
	$$
	Note that $\bv\in \SS^{d-1}$. Hence
	\begin{align*}\label{eq:m4-low-bd}
		\|\hat{\scrM}_4^{\rm sample}-\scrM_4\|
		=&~\sup_{\bu\in \SS^{d-1}}
		\abs*{
			\dfrac{1}{n}
			\sum_{i=1}^n\langle\bX_i,\bu\rangle^4
			-\EE\langle\bX,\bu\rangle^4
		}
		\ge \abs*{
			\dfrac{1}{n}
			\sum_{i=1}^n\langle\bX_i,\bv\rangle^4
			-\EE\langle\bX,\bv\rangle^4
		}\\
		\ge&~
		\dfrac{1}{n}\cdot \langle\bX_1,\bv\rangle^4
		-\sup_{\bu\in \SS^{d-1}}
		\EE\langle\bX,\bu\rangle^4
		=
		\dfrac{\|\bX_1\|^4}{n}
		-\sup_{\bu\in \SS^{d-1}}
		\EE\langle\bX,\bu\rangle^4\\
		=&~\dfrac{\|\bS_1\|^4}{n}
		-\sup_{\bu\in \SS^{d-1}}
		\EE\langle\bS,\bu\rangle^4. \numberthis
	\end{align*}
	In the first inequality, we use a particular unit vector $\bv$ as defined earlier. The third inequality uses the fact that $\langle\bX_1,\bv\rangle^4\ge 0$. The last two equalities follow since $\bX_1=\bA\bS_1$ and $\bX=\bA\bS$, where $\bA\in \calO(d)$ is an orthonormal matrix. 
	
	First, by Lemma \ref{lem:Sk-tail-bds}, we have
	$$
	\PP(\|\bS_1\|^4\ge d^2/4)
	\ge \PP\left(\abs*{\|\bS_1\|^2-d}\le d/2\right)
	\ge 1-Cd^{-(4+\epsi/2)}.
	$$
	On the other hand, for any $\bu\in \SS^{d-1}$, we have from our assumptions that
	\begin{align*}
		\EE\langle\bS,\bu\rangle^4
		=&\sum_{j=1}^du_j^4\EE S_j^4+\sum_{j\neq k}u_j^2u_k^2\EE(S_j^2)\EE(S_k^2)
		\le \max\EE S_j^4+1.
	\end{align*}
	Plugging in these bounds into \eqref{eq:m4-low-bd} implies that
	\begin{equation}\label{eq:Enorm-lbd}
		\PP
		\left(
		\|\hat{\scrM}_4^{\rm sample}-\scrM_4\|
		\ge \dfrac{d^2}{4n}-(\max\EE S_j^4+1)
		\right)
		\ge 1-Cd^{-(4+\epsi/2)}.
	\end{equation}
	
	Let us define
	$$
	\scrX:=\hat{\scrM}_4^{\rm sample}(\bX)-\scrM_0;
	\,\,
	\scrT:=\sum_{k=1}^d\kappa_4(S_k)\ba_k^{\circ 4};
	\,\,
	\scrE:=\hat{\scrM}_4^{\rm sample}(\bX)-\scrM_4(\bX).
	$$
	Note that
	$
	\scrX=\scrT+\scrE
	$. By equation \eqref{eq:Enorm-lbd} we have that
	\begin{equation}\label{eq:Enorm-lbd2}
		\|\scrE\|
		=\|\hat{\scrM}_4^{\rm sample}(\bX)-\scrM_4(\bX)\|
		\ge \dfrac{d^2}{8n}\ge 4
	\end{equation}
	with probability at least $1-d^{-4-\epsi/2}$, whenever $n\le d^2/32$. By the definition of $\hat{\kappa}_4(\bX)$, this proves the first part of the theorem. Let
	$$
	\bu_*=\underset{\bv\in \SS^{d-1}}{\rm argmax}
	\,|\scrE\times_{1,2,3,4}\bv|.
	$$
	By definition of $\bu_*$, it then follows that
	\begin{align*}\label{eq:lamb-lbd}
		\lambda=|\scrX\times_{1,2,3,4}\bu_*|
		=&~\sup_{\bv\in \SS^{d-1}}|\scrX\times_{1,2,3,4}\bv|
		\ge|\scrX\times_{1,2,3,4}\bv_*|\\
		\ge&~|\scrE\times_{1,2,3,4}\bv_*|
		-|\scrT\times_{1,2,3,4}\bv_*|\\
		\ge&~
		\|\scrE\|-
		\abs*{
			\sum_{k=1}^d\kappa_4(S_k)\langle\ba_k,\bv_*\rangle^4
		}.
		\numberthis
	\end{align*}
	Let $j=\underset{1\le k\le d}{\rm argmax}|\langle\ba_k,\bu_*\rangle|$. We write
	$$
	\bu_*=\sqrt{1-\rho^2}\ba_j+\rho\bv 
	$$
	where $\bv\in\SS^{d-1}$ and $\bv\perp\ba_j$, while $0<\rho<1$. We consider only the case where $\langle\ba_j,\bu_*\rangle>0$, since the sign of $\bu_*$ is not determined here. Let us define the event 
	\begin{equation}\label{eq:tens-conc-bds-200}
		\calA\equiv
		\left(
		\Delta_1\le \sqrt{\dfrac{C_0d(\log d)}{n}},\,\,
		\Delta_2\le \sqrt{\dfrac{C_0d(\log d)}{n}},\,\,
		\Delta_3\le \sqrt{\dfrac{C_0d(\log d)}{n}}
		+\dfrac{Cd^{13/8}(\log d)}{n^{7/8}}
		\right)
	\end{equation}
	where $C_0$ is a numerical constant, while  $\scrE=\scrX-\displaystyle\sum_{k=1}^d\kappa_4(S_k)\ba_k^{\circ 4}$, and
	$$
	\Delta_{q}=\max_{1\le j\le d}\norm*{\scrE\times_{[4]/[q]}\ba_j}
	\text{ for }
	q=1,2,3.
	$$
	By Lemmas \ref{lem:eps3-bd}, \ref{lem:eps2-bd} and \ref{lem:eps1-bd}, we have that
	$$
	\PP(\calA)\ge 1-d^{-3}.
	$$
	
	Then under the event $\calA$, we have the upper bound as follows:
	\begin{align*}
		\lambda=&~
		~|\scrX\times_{1,2,3,4}\bu_*|\\
		\le&~
		~|\scrT\times_{1,2,3,4}\bu_*|
		+~|\scrE\times_{1,2,3,4}\bu_*|\\
		\le&~
		\abs*{
			\sum_{k=1}^d\kappa_4(S_k)\langle\ba_k,\bu_*\rangle^4
		}
		+(1-\rho^2)^2
		|\scrE\times_{1,2,3,4}\ba_j|
		+4(1-\rho^2)^{3/2}\rho\|\scrE\times_{1,2,3}\ba_j\|\\
		&+6(1-\rho^2)\rho^2\|\scrE\times_{1,2}\ba_j\|
		+4\rho^3\sqrt{1-\rho^2}\|\scrE\times_1\ba_j\|
		+\rho^4\|\scrE\|\\
		\le&~
		\kappa_4(S_j)(1-\rho^2)^2
		+\max_k{\kappa_4}(S_k)\rho^3
		+4\rho\Delta_1
		+6\rho^2\Delta_2
		+4\rho^3\Delta_3
		+\rho^4\|\scrE\|\\
		\le&~
		\kappa_4(S_j)(1-\rho^2)^2
		+\max_k{\kappa_4}(S_k)\rho^3
		+C\rho\sqrt{\dfrac{d\log d}{n}}
		+\dfrac{C\rho^3d^{13/8}(\log d)}{n^{7/8}}
		+\rho^4\|\scrE\|.
	\end{align*}
	Comparing this to the lower bound from \eqref{eq:lamb-lbd} implies that
	\begin{align*}
		(1-\rho^4)\|\scrE\|
		\le (1+(1-\rho^2)^2+\rho^3)\max_k\kappa_4(S_k)
		+C\sqrt{\dfrac{d(\log d)}{n}}
		+\dfrac{C\rho^3d^{13/8}(\log d)}{n^{7/8}}
	\end{align*}
	and hence
	\begin{align*}
		\sqrt{1-\rho^2}
		\le&~ 
		\left\{
		\dfrac{1}{\|\scrE\|}
		\cdot
		\left(
		\max_k\kappa_4(S_k)
		+C\sqrt{\dfrac{d(\log d)}{n}}
		+\dfrac{C\rho^3d^{13/8}(\log d)}{n^{7/8}}
		\right) 
		\right\}^{1/2}\\
		\le&~
		\left\{
		\dfrac{8n}{d^2}
		\cdot
		\left(
		\max_k\kappa_4(S_k)
		+C\sqrt{\dfrac{d(\log d)}{n}}
		+\dfrac{Cd^{13/8}(\log d)}{n^{7/8}}
		\right) 
		\right\}^{1/2}\\
		\le&\dfrac{C\sqrt{\log d}}{d^{1/8}},
	\end{align*}
	whenever $n\le d^2/C$ for a constant $C>0$. In the second inequality above, we use the lower bound on $\|\scrE\|$ from \eqref{eq:Enorm-lbd2}.
\end{proof}

\medskip

\begin{proof}[Proof of Theorem~\ref{th:ICA-infth}]
	We will first prove an upper bound on the error norm $\|\hat{\scrM}_4(\bX)-\scrM_4(\bX)\|$.	It follows from the moment assumption and Lemma \ref{lem:lin-comb-mmt} that 
	\begin{align*}
		\sup_{\bu\in \SS^{d-1}}\E \langle\bu,\bX\rangle^{8}
		=\sup_{\bu\in \SS^{d-1}}\EE\langle\bu,\bS\rangle^{8} \le L
	\end{align*}
	for some constant $L>0$. Thus
	$$
	{\rm Var}(\langle\bX_i,\bv\rangle^4)\le L'
	$$
	for some constant $L'>0$. Then we apply Proposition 2.4 of \cite{catoni2012challenging} to obtain that, for any $t\ge \exp(-n/2)$, we have
	\begin{equation}\label{eq:catoni-2}
		\left|\hat{\theta}_{\bu}-\E(\bu^\top \bX_i)^4\right|\le \sqrt{2L\log(t^{-1})\over n\left(1-2\log(t^{-1})/n\right)},
	\end{equation}
	with probability at least $1-2t$.
	
	It is clear that $\scrM_4(\bX)$ is in the feasible set of the optimization. We now apply \eqref{eq:catoni-2} with $t=\exp(-3d)$ to each $\bu\in \calN$. Taking the union bound over all $\bu\in \calN$, we have with probability at least 
	$$
	1-|\calN|\exp(-3d)=1-\exp(-(3-\log 9)d)\ge 1-2.2^{-d},
	$$
	that
	$$\underset{\bu\in \calN}{\max}|\hat{\theta}_\bu-\langle \calM_4(\bX),\,\bu\otimes \bu\otimes \bu\otimes \bu|\le\sqrt{\dfrac{6Ld}{n(1-6d/n)}}.$$
	
	Under this event, for any $\bu\in \calN,$
	\begin{eqnarray*}
		&&\langle \hat{\calM}_4(\bX)-\calM_4(\bX),\,\bu\otimes \bu\otimes \bu\rangle\\
		&\le&\left|\hat{\theta}_{\bu}-\langle\hat{\scrM}_4(\bX),\bu\otimes\bu\otimes\bu\otimes \bu\rangle\right|+\left|\hat{\theta}_{\bu}-\langle\scrM_4(\bX),\bu\otimes\bu\otimes\bu\otimes \bu\rangle\right|\\
		&\le&2\max_{\bu\in \calN}\left|\hat{\theta}_{\bu}-\langle\scrM_4(\bX),\bu\otimes\bu\otimes\bu\otimes \bu\rangle\right|\\
		&\le& 5\sqrt{Ld\over n\left(1-6d/n\right)}.
	\end{eqnarray*}
	from which we can conclude that
	$$\begin{aligned}
		&\|\hat{\scrM}_4(\bX)-\scrM_4(\bX)\|\\=&\underset{\bu\in\calS^{d-1}}{\sup}\left|\langle \hat{\scrM}_4(\bX)-\scrM_4(\bX),\,\bu\otimes \bu\otimes \bu\otimes \bu\rangle \right|
		\\\le &\underset{\bu\in\calN}{\sup}\left|\langle \hat{\scrM}_4(\bX)-\scrM_4(\bX),\,\bu\otimes \bu\otimes \bu\otimes \bu\rangle \right|+\dfrac{1}{4}\underset{\bu\in\calS^{d-1}}{\sup}\left|\langle \hat{\scrM}_4(\bX)-\scrM_4(\bX),\,\bu\otimes \bu\otimes \bu\otimes \bu\rangle \right|
	\end{aligned}
	$$
	which means that
	$$\|\hat{\scrM}_4(\bX)-\scrM_4(\bX)\|\le 7\sqrt{Ld\over n\left(1-6d/n\right)}$$
	with probability at least $1-1.4^{-d}.$	We now apply the tensor perturbation bounds from Theorems 2.3 and 2.4 of \cite{auddy2020perturbation} to obtain
	$$
	\ell_M(\hat{\bU},\bA)\le C_2\sqrt{\dfrac{d}{n}}.
	$$
	Note that $\ell_A(\hat{\bU},\bA)\le \ell_M(\hat{\bU},\bA)$ and this finishes the proof.
\end{proof}

\medskip

\begin{proof}[Proof of Theorem~\ref{th:comp-lower}]
	Let 
	$
	\bZ=(\bZ_1,\dots,\bZ_d)\in\RR^{d\times d},
	$ 
	be a matrix of i.i.d. standard Gaussian elements $Z_{ij}\stackrel{iid}{\sim}N(0,1)$. We consider the polynomial basis
	$$
	h_m(\bZ_j^{\top}\bX/\|\bX\|)
	\quad
	\text{ for }
	m\in[D],\,\,
	j\in [d],\,\,
	\text{and}\,\,
	i\in [n].
	$$
	Here $h_m$ is the $m^{\text{th}}$ probabilist's Hermite polynomial. We define
	$$
	\psi_{\alpha_1,\dots,\alpha_d}^{(j_1,\dots,j_d)}(\tilde{\bX})
	=
	\prod_{i=1}^n
	\prod_{j=1}^d
	\left({h_{\alpha_{i,j}}(\bZ_{j}^{\top}\bX_i/\|\bX_i\|)}
	-\EE_{\mu}
	{h_{\alpha_{i,j}}(\bZ_{j}^{\top}\bX_i/\|\bX_i\|)}
	\right)
	$$
	for $0\le \alpha_{i,j}\le D$. Note that this is a random polynomial basis. Since $\bZ_j$ are independent, it can be checked that this basis is incoherent. That is, there exists $0<\delta<1$ such that for any  polynomial $f$ in $\tilde{\bX}$, that has degree at most $C\log d$, we have
	$$
	(1-\delta)\langle f,f\rangle
	\le\sum_{t=1}^N\langle f,\psi_t\rangle^2
	\le 
	(1+\delta)\langle f,f\rangle.
	$$
	We will now show that Lemma~\ref{lem:comp-low} implies the non-existence of any consistent estimator of $\bA$ that is constructed from low-degree polynomials of $\bX_i$. To see this, suppose we have an estimator 
	$$
	\hat{\ba}
	={\rm poly}_{\le (D-1)/4}
	\left(\bX_i
	\text{ for }
	i=1,\dots,n-1
	\right).
	$$
	For the sake of contradiction, suppose $\hat{\ba}$ is a consistent estimator of $\ba_j$. Then, for $\bP\in\calO(d)$, consider the low-degree polynomial
	$$
	f_j(\tilde{\bX})=(\hat{\ba}^{\top}\bX_n)^4-(\|\hat{\ba}\|\bp_j^{\top}\bX_n)^4.
	$$
	
	\begin{lemma}\label{lem:comp-low2}
		For any $\bP,\bQ\in\calO(d)$, there exist constants $0<\eta_1,\eta_2<1$ depending only on $\EE S_1^6/\kappa_4(S_1)$, such that
		$$
		\dfrac{\EE_{\bA=\bQ}f_j(\tilde{\bX})}{\sqrt{\EE_{\bA=\bP}f_j^2(\tilde{\bX})}}
		\ge 4
		$$
		if $\ell(\bP,\bQ)=\dfrac{1}{\sqrt{2}}$ and $\PP_{\bA}\left(\|\langle\hat{\ba}-\ba_j\|\le\eta_1\right)\ge 1-\eta_2$ for some $1\le j\le d$ under $\bA=\bP$ and $\bA=\bQ$.
	\end{lemma}
	
	Since $f_j(\tilde{\bX})\in\RR_{\le D}(\bX_1,\dots,\bX_n)$, equation~\eqref{eq:lrcs-low-deg} and Lemma~\ref{lem:comp-low2} give a contradiction to Lemma~\ref{lem:comp-low} whenever $n\le d^2/(CD^2(\log d)^5)$. We thus arrive at the given computational minimax lower bound.
\end{proof}

\subsection{Proofs of Results in Section~\ref{sec:tract}}
\medskip

\begin{proof}[Proof of Theorem \ref{th:init}] 
	We define the target matrix
	$$
	\bM=\sum_{k=1}^d\kappa_4(S_k)(\ba_k\otimes\ba_k)(\ba_k\otimes\ba_k)^{\top}.
	$$
	Note that 
	$$
	\bM(\bG_s)={\rm unvec}(\bM{\rm vec}(\bG_s))
	=\sum_{k=1}^d
	\kappa_4(S_k)
	Z_k(\bG_s)
	\ba_k\ba_k^{\top}.
	$$
	where for each fixed $s$, we have $Z_k(\bG_s)=\ba_k^{\top}\bG_s\ba_k\stackrel{iid}{\sim}N(0,1)$, since $\ba_k\otimes\ba_k$ are orthogonal unit vectors, and ${\rm vec}(\bG_s)\sim N(\mathbf{0},\II_{d^2})$. Let us define
	$$
	\hat{\bP}={\rm Proj}_{d}\left(
	\hat{\bM}
	-\calM_{(12)(3,4)}\left(\scrM_0\right)\right)
	$$
	the rank-$d$ projection of the matrix $\hat{\bM}-\calM_{(12)(3,4)}\left(\scrM_0\right)$. Our estimator can be written as
	$$
	\hat{\scrM}
	\times_{3,4}\hat{\bP}G
	=
	\sum_{k=1}^d
	\kappa_4(S_k)(\ba_k^{\top}{\rm unvec}(\hat{\bP}\bG)\ba_k)
	\ba_k\ba_k^{\top}
	+\scrE\times_{3,4}\hat{\bP}\bG.
	$$
	We further expand the rightmost hand side to write
	\begin{align*}\label{eq:init-dev0}
		\hat{\scrM}
		\times_{3,4}\hat{\bP}G
		=&~\sum_{k=1}^d
		\kappa_4(S_k)(\ba_k^{\top}{\rm unvec}(\bG)\ba_k)
		\ba_k\ba_k^{\top}\\
		&+\sum_{k=1}^d
		\kappa_4(S_k)(\ba_k^{\top}{\rm unvec}((\hat{\bP}-\bP)\bG)\ba_k)
		\ba_k\ba_k^{\top}
		+\scrE\times_{3,4}\hat{\bP}\bG,\numberthis
	\end{align*}
	where $\bP=\displaystyle\sum_{k=1}^d(\ba_k\otimes\ba_k)(\ba_k\otimes\ba_k)^{\top}$. We will first bound the second term in the above expression. Note that
	\begin{align*}
		&~\norm*{\sum_{k=1}^d
			\kappa_4(S_k)(\ba_k^{\top}{\rm unvec}((\hat{\bP}-\bP)\bG)\ba_k)
			\ba_k\ba_k^{\top}}\\
		=&~
		\max_{1\le k\le d}\kappa_4(S_k)(\ba_k^{\top}{\rm unvec}((\hat{\bP}-\bP)\bG)\ba_k)\\
		=&~
		\max_{1\le k\le d}\kappa_4(S_k)(\ba_k\otimes\ba_k)^{\top}(\hat{\bP}-\bP)\bG\\
		\le &~
		C
		\max_{1\le k\le d}\kappa_4(S_k)
		\norm*{(\hat{\bP}-\bP)(\ba_k\otimes\ba_k)}\sqrt{\log d}\\
		\le&~
		C\sqrt{\dfrac{d^2(\log d)}{n}}
	\end{align*}
	with probability at least $1-d^{-3}$. Here we use Gaussian concentration inequalities and a union bound in the second last step, since $\bG\sim N(0,\II_{d^2})$. We also apply Lemma~\ref{lem:mat-init-conc} in the last step. 
	
	Moving on to the third term in \eqref{eq:init-dev0} we have
	\begin{align*}
		\norm*{
			\scrE\times_{3,4}\hat{\bP}\bG 	
		}=&~
		\norm*{
			\left(\scrE\times_{3,4}\hat{\bP}\right)\times_{3,4}\bG 	
		}\\
		=&~
		\norm*{
			\sum_{i,j}g_{ij}
			\left(\scrE\times_{3,4}\hat{\bP}\right)_{..ij}
		}\\
		\le&~
		C\max\left\{
		\norm*{\calM_{(1)(234)}(\scrE\times_{3,4}\hat{\bP})},
		\norm*{\calM_{(2)(134)}(\scrE\times_{3,4}\hat{\bP})}
		\right\}
		\times\sqrt{\log d}\\
		\le&~
		C(t_1+t_2)\sqrt{\dfrac{d^2\log d}{n}}.
	\end{align*}
	In the second step, we write $..ij$ to refer to the $(i,j)$th $(3,4)$ slice of $\scrE\times_{3,4}\hat{\bP}$. The next step uses concentration inequalities for matrix Gaussian sequences from \cite{tropp2015introduction}. In the last step we apply Lemma~\ref{lem:mat-init-conc2}. Plugging these inequalities back into \eqref{eq:init-dev0}, we then have
	\begin{equation}\label{eq:init-dev1}
		\norm*{
			\hat{\scrM}\times_{3,4}\hat{\bP}\bG 
			-\sum_{k=1}^d((\ba_k\otimes\ba_k)^{\top}\bG)\ba_k\ba_k^{\top}
		}
		\le C\left(\sqrt{\dfrac{d^2(\log d)}{n}}\right)
	\end{equation}
	with high probability. Next, since $\ba_k$ are orthogonal, it follows that $Z_k:=(\ba_k\otimes\ba_k)^{\top}\bG\stackrel{iid}{\sim}N(0,1)$. Suppose now that we repeatedly sample $\bG_s$ for $s=1,\dots,L$. By anti-concentration bounds for Gaussian random variables, \citep[see Lemma B.1 of ][]{anand2014sample}, (fixing $k=1$) without loss of generality, we have that
	\begin{equation}\label{eq:gauss-anti-conc-2}
		\PP\left(\max_{1\le s \le L}
		(\ba_1\otimes\ba_1)^{\top}\bG_s
		\le \left(\sqrt{2\log L}-\dfrac{\log\log L}{4\sqrt{\log L}}
		-\sqrt{\log 8}\right)
		\right) \le \dfrac{1}{4}.
	\end{equation}
	Note that by orthogonality of $\ba_k$, we have
	$$
	(\ba_k^{\top}\otimes\ba_k)^{\top}\bG_s
	\text{ and }
	(\II_{d^2}-(\ba_k\otimes\ba_k)(\ba_k\otimes\ba_k)^{\top})(\bG_s) 
	\text{ are independent}.
	$$
	Let 
	$$
	s_*=\argmax\limits_{1\le s\le L}
	\kappa_4(S_1)(\ba_1\otimes\ba_1)^{\top}\bG_s.
	$$
	Since the definition of $s_*$ depends only on $(\ba_1\otimes\ba_1)^{\top}\bG$, this implies that the distribution of $(\II_{d^2}-(\ba_1\otimes\ba_1)(\ba_1\otimes\ba_1)^{\top})(\bG)$ does not depend on $s_*$. Thus we obtain the eigengap
	\begin{align*}
		&~\kappa_4(S_1)(\ba_1\otimes\ba_1)^{\top}\bG_{s_*}
		-
		\max_{2\le k\le d}
		\kappa_4(S_k)(\ba_k\otimes\ba_k)^{\top}\bG_{s_*}\\
		\ge&~
		\max_{1\le k \le d}\kappa_4(S_k)
		\left(
		\sqrt{2\log L}-\dfrac{\log\log L}{4\sqrt{\log L}}
		-\sqrt{\log 8}
		-2\sqrt{\log(d/4)}
		\right)\\
		\ge&~
		\sqrt{\log d}
	\end{align*}
	with probability at least $\frac{1}{2}$, whenever $L\ge Cd^{2}$  and $n\ge Cd^2t^2$. Since $\bG_s$ are independent samples, we instead take $L=Cd^2(\log d)$. Then writing $L=L_1+\dots+L_t$, for $t=3(\log d)/(\log 2)$, it follows that $L_1,\dots,L_k\ge Cd$. We define $s_*={\rm argmax}~\ba_k^{\top}\bG\ba_k$, and $s_*^{(m)}={\rm argmax}_{1\le s\le L_m}~\langle \ba_k\otimes \ba_k,\btheta_s\rangle$ for $m=1,\dots,t$. We then have, by the independence of $\btheta_s$, that the above statement holds with probability $1-d^{-3}$. Referring back to \eqref{eq:init-dev1}, we then apply Davis-Kahan theorem to obtain that
	$$
	\sin\angle\left(
	\hat{\ba},\ba_1 
	\right)\le C\sqrt{\dfrac{d^2}{n}}
	\le \dfrac{1}{4}
	$$
	with probability at least $1-d^{-3}$, provided $n\ge Cd^2$. Here $\hat{\ba}$ is the top singular vector of $(\hat{\scrM}-\scrM_0)\times_{3,4}\hat{\bP}\bG_{s_*}$. This finishes the proof of the initialization. 
\end{proof}	

When accounting for the deflation step, we have an additional perturbation term in the gaussian matrix above. This can be bounded as follows:
\begin{align*}
	&~\norm*{
		(\sum_{k=1}^{j}
		\kappa_4(S_k)\ba_k^{\circ 4}
		-
		\sum_{k=1}^{j}
		\hat{\kappa}_4(S_k)\hat{\ba}_k^{\circ 4})\times_{3,4}\bG
	}\\
	\le&~
	\norm*{\sum_{k=1}^{j}
		\abs*{\kappa_4(S_k)(\ba_k\otimes\ba_k)^{\top}\bG
			-\hat{\kappa}_4(S_k)(\hat{\ba}_k\otimes\hat{\ba}_k)^{\top}\bG}
		\ba_k\ba_k^{\top}}
	+
	\norm*{
		\sum_{k=1}^{j}
		\hat{\kappa}_4(S_k)(\hat{\ba}_k\otimes\hat{\ba}_k)^{\top}\bG
		\left\{
		\hat{\ba}_k\hat{\ba}_k^{\top}
		-\ba_k\ba_k^{\top}
		\right\}
	}\\
	\le&~
	\max_{1\le k\le j}\abs*{\kappa_4(S_k)(\ba_k\otimes\ba_k)^{\top}\bG
		-\hat{\kappa}_4(S_k)(\hat{\ba}_k\otimes\hat{\ba}_k)^{\top}\bG}
	+
	\max_{1\le k\le j}
	\hat{\kappa}_4(S_k)(\hat{\ba}_k\otimes\hat{\ba}_k)^{\top}\bG
	\norm*{
		\sum_{k=1}^j
		\left\{
		\hat{\ba}_k\hat{\ba}_k^{\top}
		-\ba_k\ba_k^{\top}
		\right\}
	}
		\end{align*}
\begin{align*}
	\le&~
	\max_{1\le k\le j}
	\abs*{\kappa_4(S_k)-\hat{\kappa}_4(S_k)}
	\sqrt{C\log d}
	+
	\max_{1\le k\le j}\kappa_4(S_k)
	\max_{1\le k\le j}
	\left(
	\ba_k\otimes\ba_k
	-\hat{\ba}_k\otimes\hat{\ba}_k
	\right)^{\top}\bG\\
	&~
	+\norm*{
		\sum_{k=1}^j
		\left\{
		\hat{\ba}_k\hat{\ba}_k^{\top}
		-\ba_k\ba_k^{\top}
		\right\}
	}
	\times\sqrt{C\log d}\\
	\le&~C\sqrt{\dfrac{(d+d^2\mathbbm{1}(\epsi<4))t^2}{n}}
\end{align*}
with probability at least $1-\exp(-t)-n^{-\epsi/8}$. The last line follows using Theorem~\ref{th:comp-rate} and Lemma~\ref{lem:ica-op-norm} on the previously recovered terms $\hat{\ba}_j$. Since $n\ge Cd^2$, this finishes the proof.

\medskip

\begin{proof}[Proof of Theorem~\ref{th:ica-piter}]
	Let us define the noise tensor 
	\begin{equation}\label{eq:def-ica-err}
		\scrE:=\hat{\scrM}_4^{\rm sample}(\bX)-{\scrM}_4(\bX).
	\end{equation}
	Define 
	\begin{equation}\label{eq:def-eps12}
		\Delta:=\|\scrE\|;
		\Delta_1:=\max\limits_{j}\|\scrE\times_{q\neq 1,2,3}\ba_j\|;
		\,\,
		\Delta_2:=\max\limits_{j}\|\scrE\times_{q\neq 1,2}\ba_j\|;
		\text{ and }
		\,
		\Delta_3:=\|\scrE\times_{2,3,4}\ba_j\|.
	\end{equation}	
	By Lemmas \ref{lem:tens-conc}, \ref{lem:eps2-bd} and \ref{lem:eps1-bd}, we have that for any fixed $j\in [d]$, and any $\delta\in (0,1)$,
	\begin{equation}\label{eq:tens-conc-bds}
		\PP\left(
		\Delta_3
		\le
		\sqrt{\dfrac{Cd(\log (1/\delta))}{n}},\,\,
		\Delta_2\le \sqrt{\dfrac{Cd(\log d)}{n}},\,\,
		\Delta\le \sqrt{\dfrac{C(d\log \delta)^2}{n}}\right)
		\ge 1-\delta-d^{-3}-n^{-\epsi/8}.
	\end{equation}
	
	By the initialization procedure, and also since we choose the largest singular value at the end of each iteration, we have nontrivial initialization for the component satisfying
	$\kappa_4(S_j)= \max_{i\ge j}\kappa_4(S_i)$. We will use the notation 
	$$
	\kappa_i:=|\kappa_4(S_i)|\,\,\text{ for }i=1,\dots,d.
	$$
	Let
	$
	\eta_t:=\sin\angle\left(\hat{\ba}_{[t]},\ba_j\right)
	$.
	By the calculation preceding the statement of Theorem~\ref{th:ica-piter}, we have
	\begin{align*}
		\eta_t
		\le&~ 2|\kappa_4(S_j)|^{-1}\Delta_3+6|\kappa_4(S_j)|^{-1}\Delta_2\eta_{t-1}+6|\kappa_4(S_j)|^{-1}\Delta_1\eta_{t-1}^2+2|\kappa_4(S_j)|^{-1}\Delta\eta_{t-1}^3\\
		\le&~
		\dfrac{2}{|\kappa_4(S_j)|}\Delta_3
		+
		\eta_{t-1}
		\left(
		\dfrac{6}{|\kappa_4(S_j)|}
		\sqrt{\dfrac{C(d\log d)}{n}}
		+
		\dfrac{6\eta_{t-1}}{|\kappa_4(S_j)|}
		\sqrt{\dfrac{C(d\log \delta)^2}{n}}
		\right)\\
		\le&~
		\dfrac{2\Delta_3}{|\kappa_4(S_j)|}
		+\dfrac{\eta_{t-1}}{2},
	\end{align*}
	by \eqref{eq:def-eps12} and \eqref{eq:tens-conc-bds} in the last two steps along with the induction hypothesis, since $n\ge Cd^2(\log \delta)^2$, along with our initialization assumption $L_0\le 1/4$. Thus repeating the argument above, it can be shown that $\eta_{t-1}\le 1/4$, and after $T\ge C(\log d)$ steps, we have, under the event in \eqref{eq:tens-conc-bds} that 
	\begin{equation}\label{eq:piter-del3}
		\sin\angle\left(\hat{\ba}_{[T]},\ba_j\right)
		=\eta_T\le \dfrac{3\Delta_3}{|\kappa_4(S_j)|}.
	\end{equation}
	The proof now follows by the upper bound on $\Delta_3$ from Lemma~\ref{lem:sample-conc}.	
\end{proof}

\medskip

\begin{proof}[Proof of Theorem \ref{th:comp-rate}]
	We prove the statements by induction over $j$, the iteration counter of 
	Algorithm \ref{alg:alg-ica}. By the induction hypothesis and \eqref{eq:piter-del3}, we already have estimators such that
	$$
	\max_{1\le k\le j-1}
	\sin\angle\left(
	\hat{\ba}_{\pi(k)},\ba_k 
	\right)
	\le \dfrac{3\Delta_3}{\kappa_4(S_{j-1})}
	$$
	with high probability. Applying Theorem \ref{th:init}, we obtain that at the $j^{\text{th}}$ step, we have
	$$
	\sin\angle\left(
	\hat{\ba}^{(s)},
	\ba_j 
	\right)\le \min\left\{\dfrac{1}{C},\dfrac{\min_k|\kappa_4(S_k)|}{2\max_k|\kappa_4(S_k)|}\right\}
	$$
	for at least one $s\in [L]$, and some $\ba_j$ which has not been estimated previously. Now we apply the power iteration steps with $\hat{\ba}_{[0]}:=\hat{\ba}^{(s)}$. Theorem \ref{th:ica-piter} then implies that $\hat{\ba}:=\hat{\ba}_{[T]}$ satisfies
	$$
	\PP\left(
	\sin
	\angle\left(
	\hat{\ba},\ba_j
	\right)
	\le C\sqrt{\dfrac{d\log(1/\delta)}{n}}\right)
	\ge 1-\delta.
	$$
	It follows that this estimator $\hat{\ba}$ satisfies
	\begin{align*}
		&~|(\hat{\scrM}_4^{\rm sample}(\bX)-\scrM_0)\times_{1,2,3,4}\hat{\ba}|\\
		\ge&~
		|({\scrM}_4^{\rm sample}(\bX)-\scrM_0)\times_{1,2,3,4}\hat{\ba}|
		-\|\hat{\scrM}_4^{\rm sample}(\bX)-{\scrM}_4(\bX)\|\\
		\ge&~
		|\sum_{k=1}^d\kappa_4(S_k)\langle \hat{\ba},\ba_k\rangle^4|
		-C\sqrt{\dfrac{d^2(\log 1/\delta)}{n}}\\
		\ge&~
		\kappa_4(S_j)(1-\sin^2\angle\left(\hat{\ba},\ba_j\right))^4
		-|\max\kappa_4(S_k)|\sin\angle\left(\hat{\ba},\ba_j\right)
		-C\sqrt{\dfrac{d^2(\log 1/\delta)}{n}}\\
		\ge&~
		|\kappa_4(S_j)|-C\sqrt{\dfrac{d^2(\log 1/\delta)}{n}}.
	\end{align*}
	We have used Lemma \ref{lem:tens-conc} in the second inequality. Note that we consider the best initialization $s\in [L]$ by using $s_*:={\rm argmax}|(\hat{\scrT}-\scrM_0)\times_{1,2,3,4}\hat{\ba}_{[T]}^{(s)}|$. By the calculation above, it is immediate that
	\begin{equation}\label{eq:max-low-bd}
		|(\hat{\scrT}-\scrM_0)\times_{1,2,3,4}\hat{\ba}_{[T]}^{(s_*)}|
		\ge 	|\kappa_4(S_j)|-C\sqrt{\dfrac{d^2(\log 1/\delta)}{n}}.
	\end{equation}
	If $|\langle \hat{\ba}_{[T]}^{s_*},\ba_k\rangle|\le \dfrac{\min|\kappa_4(S_k)|}{4\max|\kappa_4(S_k)|}$ for all $k\in [d]$, we have that
	\begin{align*}
		&|(\hat{\scrM}_4^{\rm sample}(\bX)
		-\scrM_0)\times_{1,2,3,4}\hat{\ba}_{[T]}^{(s_*)}|\\
		\le&~ \|\hat{\scrM}_4^{\rm sample}(\bX)-{\scrM}_4(\bX)\|+
		|({\scrM}_4(\bX)-\scrM_0)\times_{1,2,3,4}\hat{\ba}_{[T]}^{(s_*)}|\\
		\le&~
		C\sqrt{\dfrac{d^2(\log 1/\delta)}{n}}
		+\sum_k\kappa_4(S_k)(\langle\hat{\ba}_{[T]}^{(s)},\ba_k\rangle)^4\\
		\le&~		C\sqrt{\dfrac{d^2(\log 1/\delta)}{n}}
		+\max\kappa_4(S_k)\cdot \max_{k\in [d]}(\langle\hat{\ba}_{[T]}^{(s)},\ba_k\rangle)^4\\
		\le&~
		C\sqrt{\dfrac{d^2(\log 1/\delta)}{n}}+\dfrac{\min |\kappa_4(S_k)|}{4}
		\le \min|\kappa_4(S_k)|/2,
	\end{align*}
	which contradicts \eqref{eq:max-low-bd} above. This implies that for $s=s_*$, the initialization condition from Theorem \ref{th:init} must be satisfied. We next apply the power iteration argument from Theorem \ref{th:ica-piter} to claim that
	$$
	\sin\angle\left(\hat{\ba}_{\pi(j)},\ba_j\right)
	\le \dfrac{3\Delta_3}{\kappa_4(S_j)}.
	$$
	This finishes the proof for $\ell_M(\hat{\bA},\bA)$ by an induction argument along with the upper bound on $\Delta_3$ from Lemma~\ref{lem:sample-conc}. Finally noting that $\ell_A(\hat{\bA},\bA)\le \ell_M(\hat{\bA},\bA)$ the proof is completed.
\end{proof}

\medskip

\begin{proof}[Proof of Theorem~\ref{th:ica-gen-rate}]
	Let us write the SVD $\bA=\bU\bD\bV^{\top}$. Then $\bSigma=\texttt{cov}(\bX)=\bU\bD^2\bU^{\top}$, and $\tilde{\bA}=\bSigma^{-1/2}\bA=\bU\bV^{\top}$ is an orthogonal mixing matrix. Thus $\tilde{\bX}=\bSigma^{-1/2}\bX$ are observations from an ICA model with orthonormal mixing matrix $\tilde{\bA}$, on which our results from the previous section will apply. To simplify notation we write $n_1=|\calS_1|$ and $n_2=|\calS_2|$.
	
	We have defined
	$$
	\hat{\bSigma}=\dfrac{1}{n_1}\sum_{i\in S_1}\bX_i\bX_i^{\top}.
	$$
	It follows by a truncated matrix Bernstein inequality that
	\begin{equation}\label{eq:var-conc}
		\PP\left(\norm*{
			\hat{\bSigma}-\bSigma 
		}\le C\sqrt{\dfrac{dt}{n_1}}
		\right)\le\exp(-t)+n_1^{-\gamma/8}.
	\end{equation}
	We now show that it is possible to derive versions of the concentration inequalities in Lemmas \ref{lem:tens-conc}-\ref{lem:sample-conc} as follows. For brevity, we outline the proof of only Lemma~\ref{lem:tens-conc} as the rest follow similarly. We define
	$$
	\tilde{\bX}_i=\hat{\bSigma}^{-1/2}\bX_i,\text{ for }i\in S_2.
	$$
	Note that $\bX_i$ are independent of $\hat{\bSigma}$, and hence, conditional on $\hat{\bSigma}$,
	$$
	\tilde{\bX}_i
	=
	\hat{\bSigma}^{-1/2}\bA\bS_i
	=
	\hat{\bSigma}^{-1/2}\bSigma^{1/2}\tilde{\bA}\bS_i 
	$$
	are independent samples for $i\in S_2$. Thus,
	\begin{align*}
		&\norm*{
			\dfrac{1}{n_2}\sum_{i\in S_2}\tilde{\bX}_i^{\circ 4}
			-\left(
			\sum_{k=1}^d\kappa_4(S_k)\ba_k^{\circ 4}
			+\scrM_0 
			\right)	
		}\\
		\le&~
		\norm*{\hat{\bSigma}^{-1/2}\bSigma}
		\norm*{
			\dfrac{1}{n_2}
			\sum_{i=1}^{n_2}\bS_i^{\circ 4}-\EE \bS^{\circ 4}
		}
		+\norm*{
			\EE
			\left((\hat{\bSigma}^{-1/2}\bSigma-\II_d)\bA\bS^{\circ 4}|\hat{\bSigma}
			\right)
		}\\
		\le&~
		C(t_1+t_2)\sqrt{\dfrac{d^2}{n_2}}+C\sqrt{\dfrac{dt_3}{n_1}}
	\end{align*}
	with probability at least $1-\exp(-t_1-t_3)-\dfrac{1}{Cn_1^{\gamma/8}(t_2\wedge t_3)^{2+\gamma/4}}$	using Lemma \ref{lem:tens-conc} for the first bound and equation \eqref{eq:var-conc} for the second bound. In an identical fashion, we can also prove the upper bounds on the required matricizations of $\hat{\scrM}_4(\tilde{\bX}_i)$ conditional on $\hat{\bSigma}$. Proceeding this way the conclusions from Lemmas \ref{lem:eps2-bd} and \ref{lem:eps1-bd} follow. Therefore, if we define
	$$
	\scrE=\dfrac{1}{n_2}\sum_{i\in S_2}
	\tilde{\bX}_i^{\circ 4}-\scrM_0-\sum_{k=1}^d\kappa_4(S_k)\tilde{\ba}_k^{\circ 4},
	$$
	we have that for
	$$
	\Delta_q:=\max\limits_{1\le j\le d}\|\scrE\times_{[4]/[4-q]}\ba_j\|
	$$
	for $q=1,2,3$, the event
	\begin{equation}\label{eq:tens-conc-bds-3}
		\calA\equiv
		\left(
		\|\scrE\|\le \sqrt{\dfrac{Cd^2(\log \delta)^2}{n_2}},\,\,
		\Delta_1\le \sqrt{\dfrac{Cd^2(\log\delta)^2}{n_2}},\,\,
		\Delta_2\le \sqrt{\dfrac{Cd(\log d)}{n_2}},\,\,
		\Delta_3\le \sqrt{\dfrac{Cd(\log (1/\delta))}{n_2}}
		\right)
	\end{equation}
	holds with probability at least $1-\delta-d^{-3}-n_2^{-\gamma/8}$. One can then follow the proof of Theorem \ref{th:comp-rate} to get the conclusion:
	\begin{equation}\label{eq:alg1-conc}
		\max_{1\le j\le d}
		\sin\angle\left(
		\hat{\ba}_{\pi(j)},\tilde{\ba}_{\pi(j)}\right)
		\le C\sqrt{\dfrac{d(\log (1/\delta))}{n_2}}
	\end{equation}
	with probability at least $1-\delta-d^{-3}-n_2^{-\gamma/8}$, for some permutation $\pi:[d]\to [d]$, where $\hat{\ba}_j$ are the outputs of Algorithm \ref{alg:alg-ica} when used on $\tilde{\bX}_i$ defined as above. Finally note that we redefine
	$\hat{\ba}_j:=\hat{\bSigma}^{1/2}\hat{\ba}_j$. It follows that for the permutation $\pi$ defined above,
	\begin{align*}
		&\max_{1\le j\le d}
		\norm*{
			{\rm sign}(\langle\hat{\ba}_{\pi(j)},\ba_j\rangle)\hat{\ba}_{\pi(j)}
			-\ba_j 
		}\\
		\le &~\max_{1\le j\le d}
		\norm*{
			{\rm sign}(\langle\hat{\ba}_{\pi(j)},\ba_j\rangle)
			\hat{\ba}_{\pi(j)}
			-\hat{\bSigma}^{1/2}\tilde{\ba}_j 
		}
		+\max_{1\le j\le d}
		\norm*{
			\left(\hat{\bSigma}^{1/2}-\bSigma^{1/2}\right)\tilde{\ba}_j}\\
		\le&~
		\|\hat{\bSigma}^{1/2}\|
		\cdot 
		\max_{1\le j\le d}
		\norm*{
			{\rm sign}(\langle\hat{\ba}_{\pi(j)},\tilde{\ba}_j\rangle)
			\hat{\ba}_{\pi(j)}-\tilde{\ba}_{\pi(j)}
		}
		+C\|\hat{\bSigma}-\bSigma\|\\
		\le &~C\sqrt{d(\log (1/\delta))\cdot
			\left(
			\dfrac{1}{n_1}+\dfrac{1}{n_2}
			\right)}
	\end{align*}
	using the bounds from \eqref{eq:var-conc} and \eqref{eq:alg1-conc}  above. This finishes the proof since $n_1=n_2=n/2$.
\end{proof}

\subsection{Proofs of Results in Section~\ref{sec:asy-dist}}

\medskip

\begin{proof}[Proof of Theorem \ref{th:ica-perp}]
	We assume without loss of generality that $\pi=Id$ and that $\langle\hat{\ba}_j,\,\ba_j\rangle>0$ for $1\le j\le d$. We define
	$$
	\scrE=\hat{\scrM}_4^{\rm sample}(\bX)-\scrM_4(\bX)
	$$
	and then use the asymptotic expansion of $\hat{\ba}_j$ from Lemma \ref{lem:asy}.
	
	By \eqref{eq:def-eps12}, Lemmas \ref{lem:tens-conc}, \ref{lem:eps2-bd} and \ref{lem:eps1-bd}, we have that the event
	\begin{equation}\label{eq:tens-conc-bds-21}
		\calA\equiv
		\left(
		\Delta_0\le \sqrt{\dfrac{Cd}{n(\log d)}},\,\,
		\Delta_3\le \sqrt{\dfrac{Cd(\log d)}{n}},\,\,
		\Delta_2\le \sqrt{\dfrac{Cd(\log d)}{n}},\,\,
		\|\scrE\|\le \sqrt{\dfrac{Cd^2(\log d)^2}{n}}\right)
	\end{equation}
	holds with probability at least $1-d^{-3}-n^{-\epsi/8}$. Next note that we have
	$$
	\bE_j
	:=
	\bP_{j,\perp}^{\top}\scrE\times_{2,3,4}\ba_j
	=\dfrac{1}{n}
	\sum_{k=1}^nS_{jk}^3
	(S_{1k}^3\dots S_{(j-1)k}^3\,0\,S_{(j+1)k}^3\,
	\dots S_{dk}^3)^{\top}.
	$$
	Then by a truncated version of the vector Bernstein inequality \citep[see][]{minsker2017some}, we have
	$$
	\|\bE_j\|\le C\sqrt{\dfrac{d\log(d)}{n}}
	$$
	with probability at least $1-d^{-3}-n^{-\epsi/8}$.
%
%
%
	Given a unit vector $\bu$, we have
	\begin{equation}\label{eq:linform}
	\bu^{\top}(\hat{\ba}_j-\ba_j)
	=\bu^{\top}\bP_{\ba_j,\perp}\hat{\ba}_j
	+\langle\bu,\ba_j\rangle(\langle\hat{\ba}_j,\ba_j\rangle-1)
	\end{equation}
	First, we use part i) of Lemma~\ref{lem:asy} to write that
	$$
	|\langle\bu,\ba_j\rangle(\langle\hat{\ba}_j,\ba_j\rangle-1)|
	\le \dfrac{C\langle\bu,\ba_j\rangle d\log(d)}{n}.
	$$
	Under the event $\calA$ defined in \eqref{eq:tens-conc-bds-21}, it follows from part ii) of Lemma \ref{lem:asy}, that
	\begin{equation*}
		\sup_{\bu\perp\ba_j,\|\bu\|=1}
		\abs*{
			\bu^{\top}
			\left(
			\hat{\ba}_j
			-\dfrac{1}{\kappa_j}
			\cdot \scrE\times_{2,3,4}\ba_j 
			\right) 
		}
		\le \dfrac{Cd(\log d)^2}{n}
	\end{equation*}
	for all $\bu$. By the Cauchy-Schwarz inequality we thus have,  for any $\bu\in \RR^d$ that
	$$
		\abs*{
			\bu^{\top}\bP_{\ba_j,\perp}
			\left(
			\hat{\ba}_j
			-\dfrac{1}{\kappa_j}
			\cdot \scrE\times_{2,3,4}\ba_j 
			\right) 
		}
		\le \dfrac{Cd(\log d)^2\|\bP_{\ba_j,\perp}\bu\|}{n}.
	$$
	Combining this with the bound along $\langle\bu,\ba_j\rangle$, we have from~\eqref{eq:linform} that
	\begin{equation}\label{eq:clt-rem-0}
		\sup_{\bu\in \RR^d}
		\abs*{
			\bu^{\top}
			(\hat{\ba}_j-\ba_j)
			-\dfrac{1}{\kappa_j}
			\bu^{\top}\bP_{\ba_j,\perp}
			\scrE\times_{2,3,4}\ba_j 
		}
		\le \dfrac{Cd(\log d)^2\|\bP_{\ba_j,\perp}\bu\|}{n}
		+\dfrac{C\langle\bu,\ba_j\rangle d(\log d)}{n}.
	\end{equation}	
	It remains to use the randomness of $\scrE$ and derive the asymptotic distribution. Note that
	$$
	\scrE=\dfrac{1}{n}\sum_{k=1}^n\bX_k^{\circ 4}-\EE \bX_1^{\circ 4}.
	$$
	Moreover $\bX_k=\bA\bS_k$ for $k=1,\dots,n$. Thus
	\begin{equation}\label{eq:clt-exp-0}
		\dfrac{1}{\kappa_j}\cdot
		\scrE\times_1
		\left(\bP_{\ba_j,\perp}\bu\right)\times_{2,3,4}\ba_j
		=~	\dfrac{1}{\kappa_j}\cdot
		\left(\bP_{\ba_j,\perp}\bu\right)^{\top}\bA 
		\left(
		\dfrac{1}{n}\sum_{k=1}^n
		S_{kj}^3
		\bS_k 
		-\EE S_{1j}^3\bS_1 
		\right).
	\end{equation}
	Since $\bS_k$ have independent components $S_{kj}$ for $k=1,\dots,n$ and $j=1,\dots,d$, we obtain
	$$
	{\rm Var}(S_{1j}^3\bS_1)
	={\rm Var}(S_{1j}^3)\cdot\II_d
	+({\rm Var}(S_{1j}^4)-{\rm Var}(S_{1j}^3))\be_j\be_j^{\top},
	$$
	Therefore
	\begin{align*}
		&{\rm Var}
		\left(
		\dfrac{1}{\kappa_j}\cdot
		\left(\bP_{\ba_j,\perp}\bu\right)^{\top}\bA 
		S_{1j}^3\bS_1 
		\right)=
		\dfrac{{\rm Var}(S_{1j}^3)}{\kappa_j^2}
		\cdot \bu^{\top}\bP_{\ba_j,\perp}\bu.
	\end{align*}
 	By our assumptions, 
	$
	\EE\left(\abs*{\tfrac{1}{\kappa_j}S_{1j}^3}^4\right)
	\le C$.	Similarly it follows by Lemma \ref{lem:lin-comb-mmt} that\\
	$
	\sup_{\bu\in \SS^{d-1}}
	\EE\abs*{
		\langle\bS_1,\bu\rangle
	}^{12}
	\le C
	$. Then using Holder's inequality, we have that
	\begin{align*}
		\EE\left(\abs*{\dfrac{1}{\kappa_j}S_{1j}^3}^3
		|\langle \bS_1,\bP_{\ba_j,\perp}\bu\rangle|^3\right)
		\le&~ 
		\sup_{\bu\in \SS^{d-1}}
		\left\{
		\EE\left(\abs*{\sum_j\dfrac{1}{\kappa_j}S_{1j}^3}^4\right)
		\right\}^{3/4}
		\left(\EE|\langle \bS_1,\bP_{\ba_j,\perp}\bu\rangle|^{12}\right)^{1/4}\\
		\le&~ C\|\bP_{\ba_j,\perp}\bu\|^3
	\end{align*}
	for any $\bu$. Since $\bS_{k}$ are i.i.d. samples, using the Berry-Esseen theorem, we have that
	\begin{align*}\label{eq:ber-ess-0}
		\sup_{x\in \RR}
		\sup_{\bu\in \SS^{d-1}}
		&~
		\abs*{\PP\left(
			\dfrac{\sqrt{n}}{\sigma_{\bu}}	
			\dfrac{1}{\kappa_j}
			\scrE\times_1
			\left(\bP_{\ba_j,\perp}\bu\right)\times_{2,3,4}\ba_j
			\le x
			\right)
			-\Phi(x)
		}\\
		\le&~ \dfrac{6}{\sqrt{n}}
		\sup_{\bu\in \SS^{d-1}}
		\EE\left(\abs*{\dfrac{1}{\kappa_j}S_{1j}^3}^3
		\dfrac{
		\abs*{
			\langle\bS_1,\bP_{\ba_j,\perp}\bu\rangle
		}^3}{\sigma_{\bu}^3}\right)\\
		\le&~ \dfrac{C}{\sqrt{n}}
		\sup_{\bu\in \SS^{d-1}}
		\dfrac{
			\|\bP_{\ba_j,\perp}\bu\|^3}{\sigma_{\bu}^3}
		\le \dfrac{C}{\sqrt{n}}.\numberthis
	\end{align*}
	Note that under the event $\calA$ defined in \eqref{eq:tens-conc-bds-21}, we have using \eqref{eq:clt-rem-0} that
	\begin{align*}
		\PP\left(\dfrac{\sqrt{n}}{\sigma_{\bu}}
		\bu^{\top}(\hat{\ba}_j-\ba_j)
		\le x\right) 
		\in 
		\bigg(
		&~
		\PP\bigg(	
		\dfrac{\sqrt{n}}{\sigma_{\bu}}
		\dfrac{1}{\kappa_j}\cdot
		\scrE\times_1
		\left(\bP_{\ba_j,\perp}\bu\right)\times_{2,3,4}\ba_j 
		\le x
		-\dfrac{\sqrt{n}\xi}{\sigma_{\bu}}\bigg)
		,\\
		&~\PP\bigg(
		\dfrac{\sqrt{n}}{\sigma_{\bu}}
		\dfrac{1}{\kappa_j}\cdot
		\scrE\times_1
		\left(\bP_{\ba_j,\perp}\bu\right)\times_{2,3,4}\ba_j 
		\le x
		+\dfrac{\sqrt{n}\xi}{\sigma_{\bu}}\bigg)
		\bigg).
	\end{align*}
	for some $\xi\in \left(0,\dfrac{Cd(\log d)^2}{n}\right)$. Note that
	$$
	\sup_{x\in \RR}
	\abs*{\Phi\left(x+\dfrac{\sqrt{n}\xi}{\sigma_{\bu}}\right)-\Phi(x)}
	\le \dfrac{1}{\sqrt{2\pi}}\cdot \dfrac{\sqrt{n}\xi}{\sigma_{\bu}}
	\le \dfrac{Cd(\log d)^2}{\sqrt{n}\sigma_{\bu}},
	$$
	and therefore, using the bound above, when $\sigma_{\bu}\ge c$ for some constant $c>0$, we have 
	\begin{align*}
		&~\sup_{x\in \RR}
		\sup_{\bu\in \SS^{d-1}:\sigma_{\bu}>c}
		\abs*{
			\PP\left(\dfrac{\sqrt{n}}{\sigma_{\bu}}
			\bu^{\top}(\hat{\ba}_j-\ba_j)
			\le x\right) 
			-\Phi(x)
		}\\
		\le&~ \dfrac{Cd(\log d)^2}{\sqrt{n}}\cdot\PP(\calA)+\PP(\calA^c)
		\le  \dfrac{Cd(\log d)^2}{\sqrt{n}}
		+\dfrac{C(\log d)^{3/2}}{\sqrt{d}},
	\end{align*}
	by our assumption that $\EE S_k^{12}\le C$, along with $n\ge Cd^2(\log d)^3$ in the last step.	The proof is now complete since by assumption we consider $\bu$ such that there is a constant $c>0$ such that  $\sigma_{\bu}>c$ for sufficiently large $d$.
\end{proof}

\medskip

\begin{proof}[Proof of Corollary \ref{cor:aij-dist}]
	The claim follows by considering $\bu=\bP_{\ba_i,\perp}\be_j$ in Theorem \ref{th:ica-perp} and then applying part i) of Corollary \ref{cor:ica-par}. We omit the details for brevity.
\end{proof}

\medskip

\begin{proof}[Proof of Corollary \ref{cor:ica-par}]
	We assume without loss of generality that $\pi=Id$ and that $\langle\hat{\ba}_j,\,\ba_j\rangle >0$ for $1\le j\le d$. We will use the asymptotic expansion of $\hat{\ba}_j$ from Lemma \ref{lem:asy}.
	
	We first derive a simpler expression for $\bA_{j,\perp}^{\top}\scrE\times_{2,3,4}\ba_j$. Let $\bS'\in \RR^{d-1}$ are the vectors obtained by removing the $j^{\text{th}}$ element of $\bS$. The vectors $\bS_k'$ are derived from $\bS_k$ similarly, for $k=1,\dots,n$. By the independence of $S_l$ for $1\le l\le d$, we have 
	$$
	\EE S_{1j}^3\bS_{1}'=\EE S_{1j}^3\EE \bS_1'=\mathbf{0}.
	$$
	Then by the definition of $\scrE$,
	\begin{align*}
		\bY_j:=
		\bA_{j,\perp}^{\top}\scrE\times_{2,3,4}\ba_j
		=&
		\bA_{j,\perp}^{\top}
		\bA 
		\left( 
		\dfrac{1}{n}
		\sum_{k=1}^n
		S_{kj}^3\bS_k 
		-\EE S_{1j}^3\bS_1 
		\right) 
		=
		\dfrac{1}{n}
		\sum_{k=1}^n
		S_{kj}^3\bS_k' 
		-\EE S_{1j}^3\bS_1'
		=\dfrac{1}{n}
		\sum_{k=1}^n
		S_{kj}^3\bS_k'.
	\end{align*}
	It can be checked that
	$$
	\bSigma:={\rm Var}(S_{1j}^3\bS'_1)
	=\EE S_{1j}^6
	\II_{d-1}.
	$$
	\paragraph{Case i: Concentration bound}
	In general, we can follow the proof of Lemma \ref{lem:eps1-bd}, using the vector Bernstein inequality from \cite{minsker2017some}, to get that
	\begin{align*}
		\|\bY_j\|\le C\sqrt{\dfrac{d(\log (1/\delta))}{n}}
	\end{align*}
	with probability at least $1-\delta -d^{-3}-n^{-\epsi/8}$. 
	
	\paragraph{Case ii: $n\ge d^3(\log d)^2$:}
	
	We now derive the asymptotic distribution of $1-\langle\ba_j,\hat{\ba}_j\rangle^2$, in the case where $n\ge Cd^3(\log d)^2$. Using the Berry-Esseen theorem for hyper-rectangles, for a set of independent random variables $Z_k\sim N(0,1)$, for $k=1,\dots,d-1$, we apply part 2 of Proposition 2 from \cite{chernozhukov2017central} to obtain
	$$
	\sup_{x_1,\dots,x_{d-1}>0}
	\abs*{
		\PP\left(
		\dfrac{ 
			\sqrt{n}|Y_{jk}|}{
			\EE S_{1j}^6
		}
		\le x_k 
		\text{ for all }
		k\neq j
		\right)
		-\PP(|Z_k|\le x_k
		\text{ for all }
		k\neq j
		)
	}
	\le
	\max_{k\neq j}
	\dfrac{Cd^{\frac{1}{4}}\EE S_{1j}^{12}\EE S_{1k}^4}{\sqrt{n}}.
	$$ 
	It follows that
	\begin{align*}
		&	\sup_{y\ge 0}
		\abs*{
			\PP\left(
			\dfrac{n\bY_j^{\top}\bY_j}{(\EE S_{1j}^6)^2}
			\le y\right)
			-\PP(\chi^2_d\le y)
		}\\
		=&~
		\sup_{y\ge 0}
		\abs*{
			\int_0^y
			\int_0^{y-x_1}
			\dots 
			\int_0^{y-x_1-\dots-x_{d-2}}
			\left[
			d\mu_n(x_1,\dots,x_n)
			-
			f_{Z^2}(x_1)
			\dots 
			f_{Z^2}(x_{d-1})
			dx_1\dots d_{x_{d-1}}
			\right]
		}\\
		\le&~\max_{k\neq j}
		\dfrac{Cd^{\frac{1}{4}}\EE S_{1j}^{12}\EE S_{1k}^4}{\sqrt{n}}.
	\end{align*}
	In the above we write $d\mu_n$ to denote the law of $(nY_{j1}^2,\dots,nY_{j(d-1)}^2)$. The last line uses the difference in hyper-rectangle probabilities derived earlier. Plugging in the expression for $\bY_j$ into part i) of Lemma \ref{lem:asy} finishes the proof in the case $n\ge Cd^3(\log d)^2$.	
\end{proof}

\medskip

\begin{proof}[Proof of Theorem \ref{th:joint}]
	The conclusion follows by retracing the steps of Theorem~\ref{th:ica-perp}, by using the expression from Lemma~\ref{lem:asy} for each linear form, and then using multivariate Berry-Esseen theorem for hyperrectangles, i.e., from part 2 of Proposition 2 from \cite{chernozhukov2017central}. We derive the covariance matrix since it is the only difference.
	
	Using the expression from the second part of Lemma~\ref{lem:asy}, it is enough to compute the asymptotic covariance 
	$$
	\Cov\left(
	\sqrt{n}\bu_i^{\top}\scrE\times_{2,3,4}\ba_i/\kappa_4(S_i),
	\sqrt{n}\bu_i^{\top}\scrE\times_{2,3,4}\ba_j/\kappa_4(S_j)
	\right).
	$$
	Since 
	$$
	\scrE\times_{2,3,4}\ba_i
	=\dfrac{1}{n}\sum_{l=1}^n
	S_{li}^3\bX_l
	-\EE(S_{l}^3\bX)
	=\bA 
	\left(
	\dfrac{1}{n}\sum_{l=1}^n
	S_{li}^3\bS_l
	-\EE(S_{l}^3\bS)
	\right).
	$$
	Since $\bS_l$ are independent samples for $l=1,\dots,n$, it is enough to compute the asymptotic covariance (when $i\neq j$) which are given by
	\begin{align*}
		&~\dfrac{1}{\kappa_4(S_i)\kappa_4(S_j)}
		\Cov(S_i^3\bu_i^{\top}\bA\bS,S_j^3\bu_j^{\top}\bA\bS)\\
		=&~
		\dfrac{1}{\kappa_4(S_i)\kappa_4(S_j)}
		\bu_i^{\top}\bA
		\Cov(S_i^3\bS,S_j^3\bS)
		\bA\bu_j\\
		=&~
		\dfrac{1}{\kappa_4(S_i)\kappa_4(S_j)}
		\bu_i^{\top}\bA
		(\EE S_i^4\EE S_j^4\be_j\be_i^{\top})
		\bA\bu_j\\
		=&~
		\dfrac{1}{\kappa_4(S_i)\kappa_4(S_j)}
		\bu_i^{\top}\ba_j\bu_j^{\top}\ba_i\EE S_i^4\EE S_j^4.
	\end{align*}
	The variances remain exactly same as Theorem~\ref{th:ica-perp}, and this finishes the proof.
\end{proof}

\medskip

\begin{proof}[Proof of Theorem \ref{th:ica-clt}]
	We define the matrix $\hat{\bA}=\left[\hat{\ba}_1\,\hat{\ba}_2\,\dots\,\hat{\ba}_d\right]$. Then for any $\bu\in \SS^{d-1}$,
	\begin{align*}
		\bu^{\top}\left(\hat{\bA}-\bA\right)\bv 
		=&~\sum_{j}v_j\bu^{\top}(\hat{\ba}_j-\ba_j)
		=\sum_{j}v_j\langle\bu,\ba_j\rangle
		\left(\langle\hat{\ba}_j,\ba_j\rangle-1\right)
		+\sum_jv_j\bu^{\top}\bP_{\ba_j,\perp}\hat{\ba}_j.
	\end{align*}
	By part i) of Lemma \ref{lem:asy}, we have under the event $\calA$ that
	\begin{align*}
		\sup_{\bu,\bv\in \SS^{d-1}}\abs*{\sum_{j}v_j\langle\bu,\ba_j\rangle \left(\langle\hat{\ba}_j,\ba_j\rangle-1\right)}
		\le&~ \max_j\left(1-\langle\hat{\ba}_{j},\ba_j\rangle\right)
		\le \max_j\left(1-\langle\hat{\ba}_{\pi(j)},\ba_j\rangle\right)\\
		\le&~ \dfrac{\eps_1^2}{\kappa_j^2}+
		\left(\dfrac{Cd(\log d)}{n}\right)^{3/2}
		\le \dfrac{Cd(\log d)}{n}.
	\end{align*}
	
	Similarly, by part ii) of Lemma \ref{lem:asy}, under the event $\calA$ defined in \eqref{eq:tens-conc-bds-21} we have that
	\begin{align*}
		\sup_{\bu,\bv\in \SS^{d-1}}
		\abs*{
			\sum_{j}v_j	
			\bu^{\top}\bP_{\ba_j,\perp}
			\left(
			\hat{\ba}_j-\dfrac{1}{\kappa_j}\cdot\scrE\times_{2,3,4}\ba_j 
			\right)
		}\le \dfrac{Cd^{3/2}(\log d)}{n}.
	\end{align*}
	This step uses the Cauchy-Schwarz inequality on the remainder and may be sub-optimal. By the last two steps, we therefore have
	\begin{equation}\label{eq:clt-rem}
		\sup_{\bu,\bv\in \SS^{d-1}}
		\abs*{
			\bu^{\top}(\hat{\bA}-\bA)\bv 
			-\sum_{j}	
			\dfrac{v_j}{\kappa_j}\cdot
			\scrE\times_1
			\left(\bP_{\ba_j,\perp}\bu\right)\times_{2,3,4}\ba_j 
		}
		\le \dfrac{Cd^{3/2}(\log d)}{n}.	
	\end{equation}
	It remains to use the randomness of $\scrE$ and derive the asymptotic distribution. We remind the reader that
	$$
	\scrE=\dfrac{1}{n}\sum_{k=1}^n\bX_k^{\circ 4}-\EE \bX_1^{\circ 4}.
	$$
	Moreover $\bX_k=\bA\bS_k$ for $k=1,\dots,n$. Thus
	\begin{equation}\label{eq:clt-exp}
		\sum_{j}	
		\dfrac{v_j}{\kappa_j}\cdot
		\scrE\times_1
		\left(\bP_{\ba_j,\perp}\bu\right)\times_{2,3,4}\ba_j
		=~	\sum_{j}	
		\dfrac{v_j}{\kappa_j}\cdot
		\left(\bP_{\ba_j,\perp}\bu\right)^{\top}\bA 
		\left(
		\dfrac{1}{n}\sum_{k=1}^n
		S_{kj}^3
		\bS_k 
		-\EE S_{1j}^3\bS_1 
		\right).
	\end{equation}
	Since $\bS_k$ have independent components $S_{kj}$ for $k=1,\dots,n$ and $j=1,\dots,d$, we obtain
	$$
	{\rm Var}(S_{1j}^3\bS_1)
	={\rm Var}(S_{1j}^3)\cdot\II_d
	+({\rm Var}(S_{1j}^4)-{\rm Var}(S_{1j}^3))\be_j\be_j^{\top},
	$$
	and for $j_1\neq j_2$,
	$$
	{\rm Cov}(S_{1j_1}^3\bS_1,S_{1j_2}^3\bS_1)
	=\EE(S_{1j_1}^4)\EE(S_{1j_2}^4)\be_{j_2}\be_{j_1}^{\top}.
	$$
	Therefore
	\begin{align*}
		&{\rm Var}
		\left(
		\sum_{j}	
		\dfrac{v_j}{\kappa_j}\cdot
		\left(\bP_{\ba_j,\perp}\bu\right)^{\top}\bA 
		S_{1j}^3\bS_1 
		\right)\\
		=&~\sum_j \dfrac{v_j^2{\rm Var}(S_{1j}^3)}{\kappa_j^2}
		\cdot \bu^{\top}\bP_{\ba_j,\perp}\bu
		+\sum_{j_1\neq j_2}\dfrac{v_{j_1}v_{j_2}
			\EE(S_{1j_1}^4)\EE(S_{1j_2}^4)}{\kappa_{j_1}\kappa_{j_2}}
		\cdot \bu^{\top}\bP_{\ba_{j_1},\perp}\be_{j_2}\be_{j_1}^{\top}
		\bP_{\ba_{j_2},\perp}\bu\\
		=&~\sum_j \dfrac{v_j^2{\rm Var}(S_{1j}^3)}{\kappa_j^2}
		\cdot \bu^{\top}\bP_{\ba_j,\perp}\bu\\
		=&~\bu^{\top}\bA
		\bD_{\bv}
		\bA^{\top}\bu
	\end{align*}
	where $\bD_{\bv}:={\rm diag}
	\left(
	\displaystyle\sum_{j\neq 1}\dfrac{v_j^2{\rm Var}(S_{1j}^3)}{\kappa_j^2}\,\,
	\sum_{j\neq 2}\dfrac{v_j^2{\rm Var}(S_{1j}^3)}{\kappa_j^2}\,\,
	\dots\,\,
	\sum_{j\neq d}\dfrac{v_j^2{\rm Var}(S_{1j}^3)}{\kappa_j^2}
	\right)$ is a diagonal matrix.
	
	\medskip
	
	Since $\bS_{k}$ are i.i.d. samples, we have from \eqref{eq:clt-exp} that for fixed $\bu,\bv$ s.t. 
	\begin{equation}\label{eq:def-sigmauv}
		\sigma_{\bu,\bv}^2:=\bu^{\top}\bA\bD_{\bv}\bA^{\top}\bu\ge c
	\end{equation}
	we have, using the univariate Berry-Esseen theorem, that
	\begin{equation}\label{eq:ber-ess}
		\sup_{x\in \RR}
		\abs*{\PP\left(
			\dfrac{\sqrt{n}}{\sigma_{\bu,\bv}}
			\sum_{j}	
			\dfrac{v_j}{\kappa_j}
			\scrE\times_1
			\left(\bP_{\ba_j,\perp}\bu\right)\times_{2,3,4}\ba_j
			\le x
			\right)
			-\Phi(x)
		}
		\le \dfrac{6}{\sqrt{n}}
		\EE\left(\abs*{\sum_j\dfrac{v_j}{\kappa_j}S_{1j}^3}^3
		|\langle\bS_1,\bu\rangle|^3\right).
	\end{equation}
	Since we assume that $\EE |S_{1j}|^{12}\le C$, it follows by the independence of $S_{1j}$ that
	$$
	\sup_{\bv}\EE\left(\abs*{\sum_j\dfrac{v_j}{\kappa_j}S_{1j}^3}^4\right)
	\le C.
	$$
	Similarly it follows that for any $\bu\in \SS^{d-1}$, that
	$
	\EE|\langle\bS_1,\bu\rangle|^{12}\le C
	$
	using Lemma \ref{lem:lin-comb-mmt}.	Then by Holder's inequality, we have that
	\begin{align*}
		\sup_{\bv}
		\EE\left(\abs*{\sum_j\dfrac{v_j}{\kappa_j}S_{1j}^3}^3\|\bS_1\|^3\right)
		\le&~ 
		\sup_{\bv}
		\left\{
		\EE\left(\abs*{\sum_j\dfrac{v_j}{\kappa_j}S_{1j}^3}^4\right)
		\right\}^{3/4}
		\left(\EE|\langle\bS_1,\bu\rangle|^{12}\right)^{1/4}
		\le C.
	\end{align*}
	It therefore follows from \eqref{eq:ber-ess} that 
	\begin{equation}\label{eq:ber-ess-2}
		\sup_{x\in \RR}
		\sup_{\bu,\bv\in \SS^{d-1}}
		\abs*{\PP\left(
			\dfrac{\sqrt{n}}{\sigma_{\bu,\bv}}
			\sum_{j}	
			\dfrac{v_j}{\kappa_j}
			\scrE\times_1
			\left(\bP_{\ba_j,\perp}\bu\right)\times_{2,3,4}\ba_j
			\le x
			\right)
			-\Phi(x)
		}
		\le \dfrac{C}{\sqrt{n}}.
	\end{equation}
	Note that under the event $\calA$ defined in \eqref{eq:tens-conc-bds-21}, we have using \eqref{eq:clt-rem} that
	\begin{align*}
		\PP\left(\dfrac{\sqrt{n}}{\sigma_{\bu,\bv}}
		\bu^{\top}(\hat{\bA}-\bA)\bv
		\le x\right) 
		\in 
		\bigg(
		&~
		\PP\bigg(\sum_{j}	
		\dfrac{\sqrt{n}}{\sigma_{\bu,\bv}}
		\dfrac{v_j}{\kappa_j}\cdot
		\scrE\times_1
		\left(\bP_{\ba_j,\perp}\bu\right)\times_{2,3,4}\ba_j 
		\le x
		-\dfrac{\sqrt{n}\xi}{\sigma_{\bu,\bv}}\bigg)
		,\\
		&~\PP\bigg(\sum_{j}	
		\dfrac{\sqrt{n}}{\sigma_{\bu,\bv}}
		\dfrac{v_j}{\kappa_j}\cdot
		\scrE\times_1
		\left(\bP_{\ba_j,\perp}\bu\right)\times_{2,3,4}\ba_j 
		\le x
		+\dfrac{\sqrt{n}\xi}{\sigma_{\bu,\bv}}\bigg)
		\bigg).
	\end{align*}
	for some $\xi\in \left(0,\dfrac{Cd^{3/2}(\log d)^2}{n}\right)$. Note that
	$$
	\sup_{x\in \RR}
	\abs*{\Phi\left(x+\dfrac{\sqrt{n}\xi}{\sigma_{\bu,\bv}}\right)-\Phi(x)}
	\le \dfrac{1}{\sqrt{2\pi}}\cdot \dfrac{\sqrt{n}\xi}{\sigma_{\bu,\bv}}
	\le \dfrac{Cd^{3/2}(\log d)^2}{\sqrt{n}\sigma_{\bu,\bv}},
	$$
	and therefore, using the bound from \eqref{eq:ber-ess-2}, we get that, when $\sigma_{\bu,\bv}\ge c$ for some constant $c>0$, we have 
	\begin{align*}
		\sup_{x\in \RR}
		\sup_{\bu,\bv\in \SS^{d-1}:\sigma_{\bu,\bv}>c}
		\abs*{
			\PP\left(\dfrac{\sqrt{n}}{\sigma_{\bu,\bv}}
			\bu^{\top}(\hat{\bA}-\bA)\bv
			\le x\right) 
			-\Phi(x)
		}
		\le&~ \dfrac{Cd^{3/2}(\log d)}{\sqrt{n}}\cdot\PP(\calA)+\PP(\calA^c)
		\\
		\le&~  \dfrac{Cd^{3/2}(\log d)}{\sqrt{n}}
		+C\dfrac{(\log d)^{3/2}}{\sqrt{d}}
	\end{align*}
	by our assumption that $\EE |S_{kj}|^{12}\le C$. The proof is now complete since by assumption we consider $\bu,\bv$ such that there is a constant $c>0$ such that  $\sigma_{\bu,\bv}>c$ for sufficiently large $d$.
\end{proof}

\subsection{Proofs of Lemmas}
\label{sec:prooflemma}
\begin{proof}[Proof of Lemma~\ref{lem:comp-low}]
	We will use the low-degree polynomial method with a basis chosen as follows. $\bX_1,\dots,\bX_n$ are $n$ independent copies of $\bX=\bA\bS$, an observation from the ICA model. For $1\le k\le d$ we take the source random variables $S_k$ to be independent copies of
	$$
	S=RE/\sqrt{2}
	$$
	where $R$ and $E$ are Rademacher and $\chi^2_1$ random variables respectively. Fixing two matrices $\bP,\bQ\in\calO(d)$, we compute the low-degree projection of the likelihood ratio as follows. Writing $\tilde{\bX}=(\bX_1,\dots,\bX_n)$, we have
	\begin{equation}\label{eq:def-low-Lamb}
		\Lambda_{\le D}^2= \sum_{t=1}^N\langle L,\psi_t\rangle^2
		=\sum_{t=1}^N
		\left(\EE_{\mu^{\otimes n}}\left(L(\tilde{\bX})\psi_t(\tilde{\bX})\right)
		\right)^2
		=\sum_{t=1}^N
		\left(\EE_{\nu^{\otimes n}}\left(\psi_t(\tilde{\bX})\right)
		\right)^2.
	\end{equation}
	Let 
	$
	\bZ=(\bZ_1,\dots,\bZ_d)\in\RR^{d\times d},
	$ 
	be a matrix of i.i.d. standard Gaussian elements $Z_{ij}\stackrel{iid}{\sim}N(0,1)$. We consider the polynomial basis
	$$
	h_m(\bZ_j^{\top}\bX/\sqrt{d})
	\quad
	\text{ for }
	m\in[D],\,\,
	j\in [d],\,\,
	\text{and}\,\,
	i\in [n].
	$$
	Here $h_m$ is the $m^{\text{th}}$ probabilist's Hermite polynomial. We define
	$$
	\psi_{\alpha_1,\dots,\alpha_d}^{(j_1,\dots,j_d)}(\tilde{\bX})
	=
	\prod_{i=1}^n
	\prod_{j=1}^d
	\left({h_{\alpha_{i,j}}(\bZ_{j}^{\top}\bX_i/\sqrt{d})}
	-\EE_{\mu}
	{h_{\alpha_{i,j}}(\bZ_{j}^{\top}\bX_i/\sqrt{d})}
	\right)
	$$
	for $1\le \alpha_{i,j}\le D$
	and 
	$1\le j_1,\dots,j_k\le d$. 
	Note that we have orthonormal matrices $\bP$ and $\bQ$. It is immediate that $\|\bP\bz_j\|_2=\|\bQ\bz_j\|_2$ for all $1\le j\le d$. Moreover since the odd moments of $S$ are zero, it can be checked that
	under both the measures $\mu$ and $\nu$ induced by $\bP$ and $\bQ$ respectively, we have
	$$
	\EE_{\eta}h_{2}(\bZ^{\top}\bX/\sqrt{d})=\EE_{\eta}\psi_{\alpha_1,\dots,\alpha_{n,d}}^{(j_1,\dots,j_d)}(\tilde{\bX})=0\,\,\text{ for }
	\sum_{l=1}^k\alpha_l=2r+1;\,\,
	\eta\in \{\mu,\nu\}
	$$
	and 
	$j,j_1,\dots,j_k\in [d],k\in \NN$.
	Thus all vectors $\utilde{\alpha}=(\alpha_1,\dots,\alpha_{n,d})$ such that  $\alpha_{i,j}=2$ or $\alpha_{i,j}=2k+1$, contributes zero to the value $\Lambda$. We now bound the even powers contributing to $\Lambda_{\le D}$.  Let $\bS'_i$ be independent copies of $\bS_i$ for $i=1,\dots,n$. We will compute
	
	\begin{align*}\label{eq:mean-lambda}
		&~\EE_{\bZ}\Lambda_{\le D} \\
		=&~
		\EE_{\bZ}
		\sum_{\tilde{\alpha}\in [D]^{n\times d},\tilde{j}\in [d]^{d}}
		\prod_{i=1}^n
		\left(
		\EE_{\bS_i}
		\prod_{j=1}^d
		\left(
		h_{\alpha_{i,j}}
		\left(\dfrac{\bZ_{j}^{\top}\bQ\bS_i}{\sqrt{d}}\right)
		-
		h_{\alpha_{i,j}}\left(\dfrac{\bZ_{j}^{\top}\bP\bS_i}{\sqrt{d}}\right)
		\right)
		\right)^2\\
		=&~
		\sum_{\tilde{\alpha},\tilde{j}}
		\EE_{\bZ}
		\prod_{i=1}^n
		\EE_{\bS_i,\bS_i'}
		\prod_{j=1}^d
		\left(
		h_{\alpha_{i,j}}
		\left(\dfrac{\bZ_{j}^{\top}\bP\bS_i}{\sqrt{d}}\right)
		-
		h_{\alpha_{i,j}}\left(\dfrac{\bZ_{j}^{\top}\bQ\bS_i}{\sqrt{d}}\right)
		\right)
		\left(
		h_{\alpha_{i,j}}
		\left(\dfrac{\bZ_{j}^{\top}\bP\bS'_i}{\sqrt{d}}\right)
		-
		h_{\alpha_{i,j}}\left(\dfrac{\bZ_{j}^{\top}\bQ\bS'_i}{\sqrt{d}}\right)
		\right)
		\\
		=&~
		\sum_{\tilde{\alpha},\tilde{j}}
		\EE_{\bS}
		\prod_{j=1}^d
		\EE_{\bZ_j}
		\prod_{i:\alpha_{i,j}\ge 4}
		\left(
		h_{\alpha_{i,j}}
		\left(\dfrac{\bZ_{j}^{\top}\bP\bS_{i}}{\sqrt{d}}\right)
		-
		h_{\alpha_{i,j}}\left(\dfrac{\bZ_{j}^{\top}\bQ\bS_{i}}{\sqrt{d}}\right)
		\right)
		\left(
		h_{\alpha_{i,j}}
		\left(\dfrac{\bZ_{j}^{\top}\bP\bS'_{i}}{\sqrt{d}}\right)
		-
		h_{\alpha_{i,j}}\left(\dfrac{\bZ_{j}^{\top}\bQ\bS'_{i}}{\sqrt{d}}\right)
		\right) \numberthis
	\end{align*}
	
	We next use the following lemma.
	\begin{lemma}\label{lem:mean-diff}
		\begin{align*}
			&	
			\EE_{\bS}
			\prod_{j=1}^d
			\EE_{\bZ_j}
			\prod_{i:\alpha_{i,j}\ge 4}
			\left(
			h_{\alpha_{i,j}}
			\left(\dfrac{\bZ_{j}^{\top}\bP\bS_{i}}{\sqrt{d}}\right)
			-
			h_{\alpha_{i,j}}\left(\dfrac{\bZ_{j}^{\top}\bQ\bS_{i}}{\sqrt{d}}\right)
			\right)
			\left(
			h_{\alpha_{i,j}}
			\left(\dfrac{\bZ_{j}^{\top}\bP\bS'_{i}}{\sqrt{d}}\right)
			-
			h_{\alpha_{i,j}}\left(\dfrac{\bZ_{j}^{\top}\bQ\bS'_{i}}{\sqrt{d}}\right)
			\right)\\
			\le &~
			\left(
			\dfrac{C\|\bP-\bQ\|_{\rm F}^2(\log d)^4}{d}
			\right)^{\|\alpha\|_0}
			\left(
			\dfrac{(\log(d))^3}{d}
			\right)^{\|\alpha\|_1/2}
			.
		\end{align*}
		where $\|\balpha_0\|=\sum_{i,j}\mathbbm{1}(\alpha_{i,j}\ge 4)$ and $\|\balpha\|_1=\sum_{i,j}\alpha_{i,j}\mathbbm{1}(\alpha_{i,j}\ge 4)$.
	\end{lemma}


	Using~\eqref{eq:mean-lambda} and the Lemma above, we then have,
	\begin{align*}
		&~\EE_{\bZ}\Lambda_{\le D}\\
		=&~
		\sum_{s=4}^D
		\sum_{m=1}^{s/4}
		\sum_{
			\substack{
				\balpha\in\NN^n\\
				\|\balpha\|_0=m\\
				\|\balpha\|_1=s}}
		\prod_{j=1}^d\\
		&~\EE
		\prod_{i:\alpha_{i,j}\ge 4}
		\left(
		h_{\alpha_{i,j}}
		\left(\dfrac{\bZ_{j}^{\top}\bP\bS_{i}}{\sqrt{d}}\right)
		-
		h_{\alpha_{i,j}}\left(\dfrac{\bZ_{j}^{\top}\bQ\bS_{i}}{\sqrt{d}}\right)
		\right)
		\left(
		h_{\alpha_{i,j}}
		\left(\dfrac{\bZ_{j}^{\top}\bP\bS'_{i}}{\sqrt{d}}\right)
		-
		h_{\alpha_{i,j}}\left(\dfrac{\bZ_{j}^{\top}\bQ\bS'_{i}}{\sqrt{d}}\right)
		\right)\\
		\le &~
		\sum_{s=4}^D
		\sum_{m=1}^{s/4}
		\sum_{
			\substack{
				\balpha\in\NN^n\\
				\|\balpha\|_0=m
				\|\balpha\|_1=s}}
		\left(
		\dfrac{C\|\bP-\bQ\|_{\rm F}^2(\log d)^4}{d}
		\right)^{m}
		\left(
		\dfrac{(\log(d)^3)^2}{d}
		\right)^{s/2}	\\
		\le&~
		\sum_{s=4}^D
		\sum_{m=1}^{s/4}
		{{nd}\choose m}m^{s/2}
		\left(
		\dfrac{C\|\bP-\bQ\|_{\rm F}^2(\log d)^4}{d}
		\right)^{m}
		\left(
		\dfrac{(\log(d))^3}{d}
		\right)^{s/2}	\\
		\le&~
		\sum_{s=4}^D
		\left(
		\dfrac{s(\log(d))^3}{4d}
		\right)^{s/2}
		\sum_{m=1}^{s/4}
		\left(
		\dfrac{Cnd\|\bP-\bQ\|_{\rm F}^2(\log d)^4}{d}
		\right)^{m}\\
		=&~
		\sum_{s=4}^D
		\left(
		\dfrac{s(\log(d))^3}{4d}
		\right)^{s/2}
		\left(
		Cn\|\bP-\bQ\|_{\rm F}^2(\log d)^4
		\right)^{s/4}\\
		=&~
		\sum_{s=4}^D
		\left(
		\dfrac{Cs^2(\log d)^{10}n\|\bP-\bQ\|_{\rm F}^2}{16d^2}
		\right)^{s/4}\\
		=&~
		\sum_{s=4}^D
		\left(
		\dfrac{CD^2(\log d)^{10}n\|\bP-\bQ\|_{\rm F}^2}{d^2}
		\right)^{s/4}
		\le~
		1+
		\dfrac{CD^2(\log d)^{10}n\|\bP-\bQ\|_{\rm F}^2}{d^2}
	\end{align*}
	if $n\le d^{2}/\left(CD^2\|\bP-\bQ\|_{\rm F}^2(\log d)^{10}\right)$ for a sufficiently large constant $C>0$. To finish the proof we show how to construct $\bQ$ from $\bP$. For any $\bP\in\calO(d)$, we take
	$$
	\bq_1=\dfrac{1}{\sqrt{2}}\bp_1+\dfrac{1}{\sqrt{2}}\bp_2;
	\,\,
	\bq_2=\dfrac{1}{\sqrt{2}}\bp_1-\dfrac{1}{\sqrt{2}}\bp_2;
	\,\,
	\bq_k=\bp_k 
	\,\,
	\text{for }
	3\le k\le d.
	$$
	This implies $\|\bP-\bQ\|_{\rm F}^2=2(\sqrt{2}-1)$ but $\max_{1\le k\le d}\sin\angle\left(\bp_k,\bq_k\right)=\dfrac{1}{\sqrt{2}}$. Thus
	$$
	\EE_{\bZ}\Lambda_{\le D}
	\le 1+\eps
	\quad
	\text{when }
	n\le \dfrac{d^{2}\eps}{CD^2(\log d)^{10}}
	$$
	for a sufficiently large constant $C>0$. 
\end{proof}

\medskip

\begin{proof}[Proof of Lemma \ref{lem:tens-conc}]
	We will use Lemma \ref{lem:varY-conc} as an intermediate step. Let us define $\bY_i=\bS_i\otimes\bS_i$ and $\bY=\bS\otimes\bS$. We also define $\boldsymbol{\alpha}:=\EE(\bY)={\rm vec}(\II_d)$. Then, since $\bA$ is orthonormal, we have
	\begin{align*}\label{eq:tens-conc}
		&~\norm*{
			\left(\dfrac{1}{n}\displaystyle
			\sum_{l=1}^n\bX_l^{\circ 4}-\EE\bX_1^{\circ 4}\right)}\\
		=&~
		\sup_{\bv\in \SS^{d-1}}
		\abs*{
			\dfrac{1}{n}\displaystyle
			\sum_{l=1}^n\langle \bS_l,\bv\rangle^{4}-
			\EE\langle\bS,\bv\rangle^{4}
		}\\
		=&~
		\sup_{\bv\in \SS^{d-1}}
		\abs*{
			(\bv\otimes \bv)^{\top}
			\left(\dfrac{1}{n}\displaystyle
			\sum_{l=1}^n \bY_l\bY_l^{\top}-
			\EE\bY\bY^{\top}\right)
			(\bv\otimes\bv)
		}\\
		\le &~
		\sup_{\bv\in \SS^{d-1}}
		\abs*{
			(\bv\otimes \bv)^{\top}
			\left(\dfrac{1}{n}\displaystyle
			\sum_{l=1}^n (\bY_l-\boldsymbol{\alpha})(\bY_l-\boldsymbol{\alpha})^{\top}-
			\EE(\bY\bY^{\top}-\boldsymbol{\alpha}\boldsymbol{\alpha}^{\top})\right)
			(\bv\otimes\bv)
		}\\
		&~+
		2\sup_{\bv\in \SS^{d-1}}
		\abs*{
			(\bv\otimes\bv)^{\top}\alpha 
		}
		\abs*{
			\left	
			\langle
			\dfrac{1}{n}\sum_{l=1}^n \bY_l-\boldsymbol{\alpha},\bv\otimes\bv\right\rangle
		}	\\
		\le&~C(t_1+t_2)\sqrt{\dfrac{d^2}{n}}+
		2\sup_{\bv\in \SS^{d-1}}
		\abs*{
			\bv^{\top}
			\left(\dfrac{1}{n}\sum_{l=1}^n \bS_l\bS_l^{\top}-\EE(\bS\bS^{\top})\right)\bv}\\
		\le&~C(t_1+t_2)\sqrt{\dfrac{d^2}{n}}+
		2\norm*{
			\dfrac{1}{n}\sum_{l=1}^n \bS_l\bS_l^{\top}-\EE(\bS\bS^{\top})}.\numberthis
	\end{align*}
	with probability at least $1-\exp(-Cn^{\epsi/8})-\exp(-Ct_1)-\dfrac{1}{Cn^{\epsi/8}t_2^{2+\epsi/4}}$. Here we use Lemma \ref{lem:varY-conc} to bound the first term, and the fact that $\boldsymbol{\alpha}={\rm vec}(\II_d)$ to get that for all $\bv\in \SS^{d-1}$,
	$$
	\abs*{(\bv\otimes\bv)^{\top}\boldsymbol{\alpha}}
	=\sum_{i=1}^dv_i^2
	=\bv^{\top}\EE(\bS\bS^{\top})\bv
	=1.
	$$
	Finally note that $\EE\bS=0$. Then applying a truncation technique similar to the proofs of Lemmas \ref{lem:varY-conc}, \ref{lem:eps2-bd} and \ref{lem:eps1-bd} we obtain
	$$
	\norm*{
		\dfrac{1}{n}\sum_{l=1}^n \bS_l\bS_l^{\top}-\EE(\bS\bS^{\top})}
	\le C(t_1+t_2)\sqrt{\dfrac{d}{n}}
	$$
	with probability at least $1-\exp(-Cn^{\epsi/8})-\exp(-Ct_1)-\dfrac{1}{Cn^{\epsi/8}t_2^{2+\epsi/4}}$. Plugging in this bound into \eqref{eq:tens-conc} finishes the proof.
\end{proof}

\medskip



\begin{proof}[Proof of Lemma \ref{lem:mat-init-conc}]
	We define $\tilde{\bY}_i=\bS_i\otimes\bS_i$ for $i=1,\dots,n$ and $\tilde{\bY}=\bS\otimes\bS$. Moreover, let $\boldsymbol{\alpha}:={\rm vec}(\II_d)$. Note first that since $\bA$ is orthonormal,
	$$
	(\bA\otimes\bA)\boldsymbol{\alpha}
	={\rm vec}(\bA{\rm unvec}(\boldsymbol{\alpha})\bA^{\top})
	={\rm vec}(\bA(\II_d)(\bA)^{\top})
	={\rm vec}(\II_d)=\boldsymbol{\alpha}.
	$$
	As before, we define $\hat{\bGamma}:=\dfrac{1}{n}\displaystyle\sum_{i=1}^n\bX_i\bX_i^{\top}-\II_d$. Hence
	$$
	{\rm vec}(\hat{\bGamma})
	={\rm vec}\left(
	\dfrac{1}{n}\displaystyle\sum_{i=1}^n\bX_i\bX_i^{\top}-\II_d\right)
	=(\bA\otimes\bA)({\rm vec}(\bar{\bY})-\boldsymbol{\alpha}).
	$$
	Since $\bX=\bA\bS$ for an orthonormal matrix $\bA$, it follows by the above equations that
	\begin{align*}
		\hat{\bM}
		=&~
		(\bA\otimes\bA)
		\left(
		\dfrac{1}{n}
		\sum_{i=1}^n
		\tilde{\bY}_i\tilde{\bY}_i^{\top}
		-\tilde{\bY}\tilde{\bY}^{\top}
		\right)(\bA\otimes\bA)^{\top}+\balpha\balpha^{\top}\\
		=&~
		(\bA\otimes\bA)
		\left(
		\dfrac{1}{n}
		\sum_{i=1}^n
		(\tilde{\bY}_i-\balpha)(\tilde{\bY}_i-\balpha)^{\top}
		+\balpha(\tilde{\bY}-\balpha)^{\top}
		+(\tilde{\bY}-\balpha)\balpha^{\top}
		-\tilde{\bY}\tilde{\bY}^{\top}
		\right)(\bA\otimes\bA)^{\top}\\
		=&~
		(\bA\otimes\bA)
		\left(
		\dfrac{1}{n}
		\sum_{i=1}^n
		(\tilde{\bY}_i-\balpha)(\tilde{\bY}_i-\balpha)^{\top}
		-\balpha\balpha^{\top}
		-(\bar{\tilde{\bY}}-\balpha)
		(\bar{\tilde{\bY}}-\balpha)^{\top}
		\right)(\bA\otimes\bA)^{\top}.
	\end{align*}
	On the other hand,
	\begin{align*}
		\calM_{(12)(34)}(\scrM_4(\bX))
		=&~
		(\bA\otimes\bA)
		\EE(\tilde{\bY}\tilde{\bY}^{\top})
		(\bA\otimes\bA)^{\top}\\
		=&~
		(\bA\otimes\bA)
		\left(
		\EE
		(\tilde{\bY}-\balpha)(\tilde{\bY}-\balpha)^{\top}
		-\balpha\balpha^{\top}
		\right)
		(\bA\otimes\bA)^{\top}.
	\end{align*}
	Thus,
	\begin{align*}
		&~\norm*{\hat{\bM}-\calM_{(12)(34)}(\scrM_4(\bX))}\\
		=&~
		\norm*{
			\dfrac{1}{n}
			\sum_{i=1}^n
			(\tilde{\bY}_i-\balpha)(\tilde{\bY}_i-\balpha)^{\top}
			-\EE(\tilde{\bY}_1-\balpha)(\tilde{\bY}_1-\balpha)^{\top}
			-(\bar{\tilde{\bY}}-\balpha)
			(\bar{\tilde{\bY}}-\balpha)^{\top}
		}\\
		\le&~
		\norm*{
			\dfrac{1}{n}
			\sum_{i=1}^n
			(\tilde{\bY}_i-\balpha)(\tilde{\bY}_i-\balpha)^{\top}
			-\EE(\tilde{\bY}_1-\balpha)(\tilde{\bY}_1-\balpha)^{\top}}
		+\norm*{\bar{\tilde{\bY}}-\balpha}^2\\
		\le&~
		C(t_1+t_2)\sqrt{\dfrac{d^2}{n}}
	\end{align*}
	with the required probability. The last step of the proof follows from Lemma \ref{lem:varY-conc}, along with the fact that
	$$
	\norm*{\bar{\tilde{\bY}}-\balpha}
	=
	\norm*{\dfrac{1}{n}\sum_{i=1}^n\bS_i\bS_i^{\top}-\II_d}_{\rm F}
	\le C(t_1+t_2)\sqrt{\dfrac{d^2}{n}}
	$$
	with the same argument as in Lemma~\ref{lem:varY-conc} but now used for  concentration around the $d\times d$ matrix given by ${\rm Var}(\bS)$.
	.
\end{proof}

\medskip

\begin{proof}[Proof of Lemma~\ref{lem:mat-init-conc2}]
	Writing $\bY_k=\bX_k\otimes\bX_k$ for $k=1,\dots,n$, and $\scrE:=
	\mathcal{M}_{(12)(3,4)}^{-1}(\hat{\bM}_1)-\scrM_4(\bX)$, we have
	\begin{align*}
		&~{\calM}_{(12)(34)}(\scrE)\\
		=&~
		\dfrac{1}{n}\sum_{k=1}^n
		(\bX_i\otimes\bX_i)
		(\bX_i\otimes\bX_i)^{\top}
		-{\rm vec}(\II_d){\rm vec}(\hat{\bGamma})^{\top}
		-{\rm vec}(\hat{\bGamma}){\rm vec}(\II_d)^{\top}
		-{\calM}_{(12)(34)}(\scrM_4(\bX))\\
		=&~
		\dfrac{1}{n}\sum_{k=1}^n
		(\bY_k-{\rm vec}(\II_d))
		(\bY_k-{\rm vec}(\II_d))^{\top}
		+
		{\rm vec}(\II_d){\rm vec}(\II_d)^{\top}
		-\EE\bY_1\bY_1^{\top}\\
		=&~
		\dfrac{1}{n}\sum_{k=1}^n
		(\bY_k-{\rm vec}(\II_d))
		(\bY_k-{\rm vec}(\II_d))^{\top}
		+
		{\rm vec}(\II_d){\rm vec}(\II_d)^{\top}
		-\EE\bY_1\bY_1^{\top}\\
		=&~
		\dfrac{1}{n}\sum_{k=1}^n
		(\bY_k-{\rm vec}(\II_d))
		(\bY_k-{\rm vec}(\II_d))^{\top}
		-
		\EE\left(
		(\bY_1-{\rm vec}(\II_d))
		(\bY_1-{\rm vec}(\II_d))^{\top}
		\right).
	\end{align*}
	In the last two steps we use the fact that $\EE \bY_k={\rm vec}(\II_d)$. We define
	$$
	\hat{\bP}
	:={\rm Proj}_d
	\left(\hat{\bM}_2-\calM_{(12)(3,4)}(\scrM_0)\right).
	$$
	which is independent of $\scrE$. Let us also define the true projecion
	$$
	{\bP}
	:={\rm Proj}_d
	\left(\calM_{(12)(3,4)}(\scrM_4(\bX)-\scrM_0)\right)
	=\sum_{k=1}^d
	(\ba_k\otimes\ba_k)(\ba_k\otimes\ba_k)^{\top}.
	$$
	By Lemma~\ref{lem:mat-init-conc} and the Davis-Kahan theorem it follows that
	$$
	\norm*{\hat{\bP}-\bP}\le C\sqrt{\dfrac{d^2}{n}}
	$$
	with high probability. Note that our estimator can be written as
	\begin{align*}
		\hat{\scrM}
		=&~
		(\scrM_4(\bX)-\scrM_0+\scrE)
		\times_{(3,4)}
		\hat{\bP}\\
		=&~
		\scrM_4(\bX)-\scrM_0
		+(\scrM_4(\bX)-\scrM_0)
		\times_{(3,4)}
		(\hat{\bP}-\bP)
		+\scrE
		\times_{(3,4)}
		\hat{\bP}
	\end{align*}
	which implies
	\begin{align*}
		&~\calM_{(1)(234)}(\hat{\scrM}-
		(\scrM_4(\bX)-\scrM_0))\\
		=&~
		\sum_{k=1}^d\kappa_4(S_k)
		\ba_k(\ba_k\otimes ((\hat{\bP}-\bP)(\ba_k\otimes\ba_k)))^{\top}
		+\calM_{(1)(234)}
		(\scrE
		\times_{(3,4)}
		\hat{\bP}).
	\end{align*}
	Note that the first matrix on the right hand side above can be bounded as
	$$
	\norm*{\sum_{k=1}^d\kappa_4(S_k)
		\ba_k(\ba_k\otimes ((\hat{\bP}-\bP)(\ba_k\otimes\ba_k)))^{\top}}
	\le \max_{1\le k\le d}\kappa_4(S_k)\cdot \|\hat{\bP}-\bP\|
	\le C\sqrt{\dfrac{d^2}{n}}
	$$
	with high probability. We thus bound the second term
	\begin{align*}
		&~\norm*{
			\calM_{(1)(234)}(\scrE
			\times_{(3,4)}\hat{\bP}) 
		}\\
		=&~
		\sup_{\bu\in\SS^{d-1},
			\bv\in\SS^{d^3-1}}
		\sum_{i,j,k,l}
		u_iv_{jkl}
		(
		((\EE_n-\EE)
		(\bY-{\rm vec}(\II_d))
		(\bY-{\rm vec}(\II_d))^{\top}
		)
		\hat{\bP})_{(i,j)(k,l)}
		.
	\end{align*}
	where $\EE_n$ is the sample averaging operator over $n/2$ i.id. samples. Let $\tilde{\bX}_m$ be independent copies of $\bX$ for $m=1,\dots,n$ that are also independent of $\bX_m$. Note first that $\hat{\bP}$ is independent of $\bY_i$s. Then by a decoupling argument, we have for a constant $C>0$ that
	\begin{align*}
		&~ \norm*{
			{\calM}_{(1)(234)}
			(\scrE
			\times_{(3,4)}\hat{\bP}
			)
		}\\
		\le&~
		C
		\sup_{\bu\in\SS^{d-1},
			\bv\in\SS^{d^3-1}}
		\sum_{i,j,k,l}
		u_iv_{jkl}
		(
		(\EE_n
		(\bY-{\rm vec}(\II_d))
		(\tilde{\bY}-{\rm vec}(\II_d))^{\top}
		)
		\hat{\bP})_{(i,j)(k,l)}
		.
	\end{align*}
	with high probability. We now compute the norm on the right. Note that we have
	\begin{align*}\label{eq:m123-bd0}
		&~
		\sup_{\bu,\bv}
		\sum_{i,j,k,l}
		u_iv_{jkl}
		(
		((\EE_n-\EE)
		(\bY-{\rm vec}(\II_d))
		(\bY-{\rm vec}(\II_d))^{\top}
		)
		\hat{\bP})_{(i,j)(k,l)}\\
		=&~
		\sup_{\bu,\bv}
		\sum_{j=1}^d
		\bigg\langle
		\bv_j,
		\bigg\{
		\dfrac{2}{n}
		\sum_{m=1}^{n/2}
		(\bY_m-{\rm vec}(\II_d))
		\hat{\bP}
		(\bY_m-{\rm vec}(\II_d))^{\top}\\
		&\hspace{4cm}~-\EE((\bY_m-{\rm vec}(\II_d))
		\hat{\bP}
		(\bY_m-{\rm vec}(\II_d))^{\top})
		\bigg\}(\II_d\otimes\ba_j)\bu 
		\bigg\rangle\\
		\le&~
		\sup_{\bu,\bv}
		\left(\sum_{j=1}^d\|\bv_j\|^2\right)^{1/2}
		\times\\
		&~\quad\times 
		\bigg(\sum_{j=1}^d
		\bigg\Vert
		\bigg\{
		\dfrac{2}{n}
		\sum_{m=1}^{n/2}
		(\bY_m-{\rm vec}(\II_d))
		\hat{\bP}
		(\bY_m-{\rm vec}(\II_d))^{\top}\\
		&\hspace{4cm}~
		-\EE((\bY_m-{\rm vec}(\II_d))
		\hat{\bP}
		(\bY_m-{\rm vec}(\II_d))^{\top})
		\bigg\}(\II_d\otimes\ba_j)\bu 
		\bigg\Vert^2\bigg)^{1/2}\\
		\le&~
		\bigg(\sum_{j=1}^d
		\bigg\Vert
		\bigg\{
		\dfrac{2}{n}
		\sum_{m=1}^{n/2}
		(\bY_m-{\rm vec}(\II_d))
		\hat{\bP}
		(S_{jm}\bS_m-\be_j)^{\top}
		-\EE((\bY_m-{\rm vec}(\II_d))
		\hat{\bP}
		(S_{jm}\bS_m-\be_j)^{\top})
		\bigg\}
		\bigg\Vert^2\bigg)^{1/2}.\numberthis
	\end{align*}
	Here we define $\bv_j:=(\II_d\otimes\ba_j\otimes\II_d)\bv$ for $j=1,\dots,d$. Conditional on $\hat{\bP}$ we now apply truncated matrix Bernstein inequality (following the arguments of Lemmas~\ref{lem:eps2-bd}-\ref{lem:eps1-bd}) to obtain
	\begin{align*}
		\label{eq:m123-bd1}
		&\bigg\Vert
		\bigg\{
		\dfrac{2}{n}
		\sum_{m=1}^{n/2}
		(\bY_m-{\rm vec}(\II_d))
		\hat{\bP}
		(S_{jm}\bS_m-\be_j)^{\top}
		-\EE((\bY_m-{\rm vec}(\II_d))
		\hat{\bP}
		(S_{jm}\bS_m-\be_j)^{\top})
		\bigg\}
		\bigg\Vert\\
		\le&~
		C(t_1+t_2)
		\bigg\{\sqrt{
			\dfrac{{\rm trace}(\hat{\bP}{\rm Var}(\bY))
				\norm*{{\rm Var}(S_{j}\bS)}}{n}}	
		+\sqrt{
			\dfrac{{\rm trace}({\rm Var}(S_j\bS))\norm*{{\rm Var}(\bY)}}{n}}
		\bigg\}
		\numberthis
	\end{align*}
	with probability at least $1-n^{-C\eps/8}-\exp(-Ct_1)-\dfrac{1}{Cn^{\eps/8}t_2^{2+\eps/4}}$. To see this, we first truncate the observations according to $\calS(t):=\{m\in [n/2]:\|\bS_m\|\le t\}$, and then choose $t$ such that $|[n/2]/\calS(t)|\le C$ with high probability. On $m\in\calS(t)$, we apply matrix Bernstein inequality. See the proofs of Lemmas~\ref{lem:eps2-bd}-\ref{lem:eps1-bd} for the details and more of the similar arguments. Note that
	\begin{align*}
		\hat{\bP}
		=&~{\rm Proj}_d
		\left(
		\dfrac{1}{n}
		\sum_{m=1}^n
		\bY_m\bY_m^{\top}
		-{\rm vec}(\II_d)(\bar{\bY}-{\rm vec}(\II_d))^{\top}
		-(\bar{\bY}-{\rm vec}(\II_d)){\rm vec}(\II_d)^{\top}
		\right)\\
		=&~{\rm Proj}_d
		\left(
		(\bA\otimes\bA)
		\left(
		\dfrac{1}{n}
		\sum_{m=1}^n
		(\bS_m\otimes\bS_m)
		(\bS_m\otimes\bS_m)^{\top}
		-{\rm vec}(\II_d)\tilde{\bGamma}^{\top}
		-\tilde{\bGamma}{\rm vec}(\II_d)^{\top}
		\right)
		(\bA\otimes\bA)^{\top}
		\right)\\
		=&~
		(\bA\otimes\bA)
		{\rm Proj}_d
		\left(
		\dfrac{1}{n}
		\sum_{m=1}^n
		(\bS_m\otimes\bS_m)
		(\bS_m\otimes\bS_m)^{\top}
		-{\rm vec}(\II_d)(\tilde{\bGamma})^{\top}
		-(\tilde{\bGamma}){\rm vec}(\II_d)^{\top}
		\right)
		(\bA\otimes\bA)^{\top}\\
		=:&~(\bA\otimes\bA)
		\tilde{\bP}
		(\bA\otimes\bA)^{\top}.
	\end{align*}
	where ${\rm Proj}_d$ refers to the projection to the top $d$ singular vectors, and $\tilde{\Gamma}=\dfrac{1}{n}
	\displaystyle\sum_{m=1}^n(\bS_m\otimes\bS_m)
	-{\rm vec}(\II_d)$. This implies that 
	$
	{\rm trace}(\hat{\bP}{\rm Var}({\bY}))
	={\rm trace}(\tilde{\bP}{\rm Var}(\bS\otimes\bS))
	$. We compute
	\begin{align*}
		({\rm Var}(\bS\otimes\bS))_{ijkl}
		={\rm Cov}(S_iS_j,S_kS_l)
		=\begin{cases}
			{\rm Var}(S_i^2)\quad&\text{if }i=j=k=l\\
			1\quad&\text{if }{\rm card}(\{i,j,k,l\})=2\\
			0\quad\text{otherwise.}
		\end{cases}
	\end{align*}
	It can be checked by a direct calculation that this implies 
	\begin{align*}
		&~{\rm trace}(\tilde{\bP}{\rm Var}(\bS\otimes\bS))\\
		\le &~
		\max{\rm Var}(S_i^2){\rm trace}(\tilde{\bP})
		+\sum_{i,j}(\be_i\otimes\be_j)^{\top}\tilde{\bP}(\be_j\otimes\be_i)^{\top}
		+\sum_i 
		(\be_i\otimes\be_i)^{\top}\tilde{\bP}
		\sum_j(\be_j\otimes\be_j)^{\top}\\
		\le&~
		d\max{\rm Var}(S_i^2)
		+\sum_{k=1}^d\sum_{i,j}u_{k,ij}u_{k,ji}
		+\norm*{\tilde{\bP}}
		\times\norm*{\sum_{i}\be_i\otimes\be_i}^2\\
		\le&~
		d\max{\rm Var}(S_i^2)
		+\sum_{k=1}^d\sum_{i,j}u_{k,ij}^2
		+d\\
		\le&~ Cd.
	\end{align*}
	In the second inequality we use an SVD of $\tilde{\bP}=\bU\bU^{\top}$ where $\bU\in \RR^{d^2\times d}$ is an orthonormal matrix. We also use the fact that
	${\rm trace}(\tilde{\bP})={\rm rank}(\tilde{\bP})=d$ and that $\|\tilde{\bP}\|=1$. Plugging this bound into \eqref{eq:m123-bd1} we obtain
	\begin{align*}
		&~\norm*{
			\dfrac{1}{n}
			\sum_{m=1}^n
			\hat{\bP}
			(\bY_m-{\rm vec}(\II_d))
			(S_{jm}\bS_m-\be_j)^{\top}}\\
		\le&~
		C(t_1+t_2)\sqrt{
			\dfrac{d}{n}	
			\left(
			\norm*{{\rm Cov}(\bY)}
			+C\sqrt{\dfrac{d^2}{n}}
			\right)
		}
		\le C(t_1+t_2)\sqrt{
			\dfrac{d}{n}}.
	\end{align*}
	In the second inequality we use Lemma~\ref{lem:varY-conc}. The third inequality uses the fact that $\|{\rm Var}(\bY)\|\le C$ and $n\ge Cd^2$. Finally plugging this bound into \eqref{eq:m123-bd0} and \eqref{eq:m123-bd1} we have
	\begin{align*}
		\sup_{\bu,\bv}
		\sup_{\bu,\bv}
		\sum_{i,j,k,l}
		u_iv_{jkl}
		(
		((\EE_n-\EE)
		(\bY-{\rm vec}(\II_d))
		(\bY-{\rm vec}(\II_d))^{\top}
		)
		\hat{\bP})_{(i,j)(k,l)}
		\le&~C(t_1+t_2)\sqrt{\dfrac{d^2}{n}}.
	\end{align*}
	with probability at least $1-n^{-C\eps/8}-\exp(-Ct_1)-\dfrac{1}{Cn^{\eps/8}t_2^{2+\eps/4}}$. This finishes the proof for $\calM_{(1)(234)}(\scrE)$. The proof for $\calM_{(2)(134)}$ follows by an analogous argument by swapping the indices $i$ and $j$.
\end{proof}

\medskip

\begin{proof}[Proof of Lemma~\ref{lem:sample-conc}]
	The proof follows from Lemmas~\ref{lem:tens-conc}, \ref{lem:eps2-bd} and \ref{lem:eps1-bd} using the definitions of $\Delta_k$. 
\end{proof}

\medskip

\begin{proof}[Proof of Lemma~\ref{lem:comp-low2}]
	Without loss of generality, we assume that
	$$
	1-(\bp_1^{\top}\bq_1)^2
	=\max_{1\le k\le d}
	1-(\bp_k^{\top}\bq_k)^2
	=\dfrac{1}{2}.
	$$
	Then we prove the statement for $j=1$. We define
	$$
	\bv=\hat{\ba}/\|\hat{\ba}\|.
	$$	
	By assumption on $\hat{\ba}$, and since $\ba_1\in\SS^{d-1}$, we can show that
	$$
	\abs*{1-\|\hat{\ba}\|}\le\eta_1,
	\quad
	\text{and}
	\quad
	|\langle\ba_1,\bv\rangle|
	\ge
	\dfrac{1-\eta_1}{\sqrt{1+\eta_1}}.
	$$
	Note that $\bX_n$ is independent of $\hat{\ba}$. Thus we first take the expectation over $\bX_n$ to get
	\begin{align*}
		\EE_{\bA=\bQ}f_1(\tilde{\bX}|\bX_1,\dots,\bX_{n-1})
		=&~
		\EE(\hat{\ba}^{\top}\bQ\bS_n)^4
		-\EE(\|\hat{\ba}\|\bp_1^{\top}\bQ\bS_n)^4\\
		=&~
		\kappa_4(S_1)\|\hat{\ba}\|^4
		\sum_{k=1}^d
		((\bq_k^{\top}\bv)^4
		-(\bq_k^{\top}\bp_1)^4)\\
		\ge&~
		\kappa_4(S_1)\|\hat{\ba}\|^4
		\left(
		(\bq_1^{\top}\bv)^4
		-(\bq_1^{\top}\bp_1)^4
		\right)
		-
		\kappa_4(S_1)\|\hat{\ba}\|^4
		(1-(\bq_1^{\top}\bp_1)^2)^2\\
		=&~
		\kappa_4(S_1)\|\hat{\ba}\|^4\left((\bq_1^{\top}\bv)^4-1\right)
		+2\kappa_4(S_1)\|\hat{\ba}\|^4
		(1-(\bq_1^{\top}\bp_1)^2)
		(\bq_1^{\top}\bp_1)^2\\
		=&~
		\kappa_4(S_1)\|\hat{\ba}\|^4
		\left(
		(\bq_1^{\top}\bv)^4-
		\dfrac{1}{2}
		\right)
		.
	\end{align*}
	Since $\PP\left(\|\hat{\ba}-\bq_1\|\le \eta_1\right)\ge 1-\eta_2$, we obtain
	\begin{equation}\label{eq:pr-comp-num}
		\EE_{\bA=\bQ}f(\tilde{\bX})
		\ge\kappa_4(S_1)(1-\eta_1)^4
		\left(
		\dfrac{(1-\eta_1)^4}{(1+\eta_1)^2}-\eta_2/2\right).
	\end{equation}
	To write the denominator, note that for $\bu\in\SS^{d-1}$, we have
	\begin{align*}
		&~\EE\left(
		(\bu^{\top}\bS)^4-S_1^4)
		\right)^2\\
		=&~
		\EE(\bu^{\top}\bS)^8+\EE S_1^8-2\EE(\bu^{\top}\bS)^4S_1^4\\
		=&~
		\kappa_8(S_1)\sum u_i^8
		+\sum_{k=0}^7
		{7\choose k}\EE(\bu^{\top}\bS)^{7-k}
		\kappa_{k+1}(S_1)
		\sum_{i=1}^du_i^{k+1}
		+\EE S_1^8
		-2
		\left(
		\EE S_1^8u_1^4
		+(\EE S_1^4)^2
		\sum_{j\neq 1}u_j^4
		\right)\\
		&~-6
		\left(\EE S_1^6u_1^2(1-u_1^2)
		+\EE S_1^4\sum_{i\ge 2}u_i^2(1-u_1^2)\right)
		\\
		=&~
		7\EE(\bu^{\top}\bS)^{6}
		+
		35\kappa_4(S_1)\EE(\bu^{\top}\bS)^{4}
		\sum u_i^4
		+
		21\kappa_6(S_1)\sum u_i^6
		+\kappa_8(S_1)\sum u_i^8\\
		&~
		+\EE S_1^8
		-2
		\left(
		\EE S_1^8u_1^4
		+(\EE S_1^4)^2
		\sum_{j\neq 1}u_j^4
		\right)
		-6
		\left(\EE S_1^6u_1^2(1-u_1^2)
		+\EE S_1^4\sum_{i\ge 2}u_i^2(1-u_1^2)\right)\\
		\le&~ C\EE S_1^6(1-u_1)
	\end{align*}
	for a numerical constant $C>0$. Then, conditional on $\bX_1,\dots,\bX_{n-1}$ we can write that
	\begin{align*}
		\EE\left(
		(\bv^{\top}\bP\bS)^4-S_1^4\right)
		\le ~(C\EE S_1^6)^2\EE(1-\bp_1^{\top}\hat{\bv})
		=(C\EE S_1^6)^2\left(1-\dfrac{1-\eta_1}{\sqrt{1+\eta_1}}+\eta_2\right)
	\end{align*}
	since $\PP(\|\hat{\bv}-\bp_1\|\le \eta_1)\ge 1-\eta_2$ when $\bA=\bP$.	Dividing equation~\eqref{eq:pr-comp-num} by the expression above, we have 
	$$
	\dfrac{\EE_{\bQ}f(\tilde{\bX})}{\sqrt{\EE_{\bP}f^2(\tilde{\bX})}}
	\ge 
	\dfrac{\kappa_4(S_1)(1-\eta_1)^4
		\left(
		\dfrac{(1-\eta_1)^4}{(1+\eta_1)^2}-\eta_2/2\right)}{
		C\EE S_1^6\left(1-\dfrac{1-\eta_1}{\sqrt{1+\eta_1}}+\eta_2\right)
	}
	\ge 2
	$$
	provided $\eta_1$ and $\eta_2$ are sufficiently small.
\end{proof}

\medskip

\begin{proof}[Proof of Lemma~\ref{lem:mean-diff}]
	Writing $\bu_i=(\bP-\bQ)\bS_i/\sqrt{d}$ and $\bv_i=(\bP-\bQ)\bS'_i/\sqrt{d}$, we have using Hanson-Wright inequality conditional on $\bS,\bS'$ that
	\begin{align*}
		\PP(\bZ^{\top}\bu_i\bv_i^{\top}\bZ\ge C\bu^{\top}\bv(\log d))\le d^{-C}
	\end{align*}
	for a sufficiently large constant $C>0$. Note next that
	\begin{align*}
		\bu_i^{\top}\bv_i
		=
		\dfrac{1}{d}
		\bS^{\top}(\bP-\bQ)^2\bS'
		\le&~
		\dfrac{1}{d}
		(\bS^{\top}(\bP-\bQ)^2\bS)^{1/2}
		(\bS'^{\top}(\bP-\bQ)^2\bS')^{1/2}\\
		\le&~
		\dfrac{1}{d}
		({\rm tr}(\bP-\bQ)^2+C\|(\bP-\bQ)^2\|_{\rm F}(\log d))\\
		\le&~
		\dfrac{C\|\bP-\bQ\|_{\rm F}^2(\log d)}{d}
	\end{align*}
	with probability at least $1-d^{-C}$, once again using Hanson-Wright inequality and the sub-Gaussianity of $S_i$. Combining this bound with the Gaussian quadratic form above we get
	\begin{equation}\label{eq:mean-diff1}
		\PP\left(
		\max_{1\le i\le n}
		|\bS_i^{\top}\bS'_i|\ge C\sqrt{d};
		\max_{1\le i\le n}
		\max_{1\le j\le d}
		\bZ_j^{\top}\bu_i\bv_i^{\top}\bZ_j\ge 
		\dfrac{C\|\bP-\bQ\|_{\rm F}^2(\log d)^2}{d}
		\right)\le d^{-C},
	\end{equation}
	provided $n={\rm poly}(d)$.	Now by mean value theorem, and properties of derivatives of Hermite polynomials, we obtain
	\begin{align*}
		h_{\alpha_{i,j}}\left(\dfrac{\bZ_j^{\top}\bP\bS_i}{\sqrt{d}}\right)-
		h_{\alpha_{i,j}}\left(\dfrac{\bZ_j^{\top}\bQ\bS_i}{\sqrt{d}}\right)
		=~\alpha_{i,j}
		\bZ_j^{\top}\bu 
		h_{\alpha_{i,j}-1}(\bZ_j^{\top}\bw)
	\end{align*}
	where
	$$
	\bw=(\lambda\bP\bS_i+(1-\lambda)\bQ\bS_i)/\sqrt{d}
	$$
	for some $\lambda\in [0,1]$. Similarly writing the term involving $\bS_i'$ we have
	\begin{align*}\label{eq:herm-taylr}
		&~\left(
		h_{\alpha_{i,j}}\left(\dfrac{\bZ_j^{\top}\bP\bS_i}{\sqrt{d}}\right)-
		h_{\alpha_{i,j}}\left(\dfrac{\bZ_j^{\top}\bQ\bS_i}{\sqrt{d}}\right)
		\right)
		\left(
		h_{\alpha_{i,j}}\left(\dfrac{\bZ_j^{\top}\bP\bS'_i}{\sqrt{d}}\right)-
		h_{\alpha_{i,j}}\left(\dfrac{\bZ_j^{\top}\bQ\bS'_i}{\sqrt{d}}\right)
		\right)\\
		=&~
		\alpha_{i,j}^2
		\bZ_j^{\top}\bu\bv^{\top}\bZ_j
		h_{\alpha_{i,j}-1}(\bZ_j^{\top}\bw_i)
		h_{\alpha_{i,j}-1}(\bZ_j^{\top}\bw'_i)\\
		\le&~
		\dfrac{C\alpha_{i,j}^2\|\bP-\bQ\|_{\rm F}^2(\log d)^2}{d}
		\cdot 
		h_{\alpha_{i,j}-1}(\bZ_j^{\top}\bw_i)
		h_{\alpha_{i,j}-1}(\bZ_j^{\top}\bw'_i)
		\numberthis
	\end{align*}
	where $\bw_i=(\lambda\bP+(1-\lambda)\bQ)\bS_i/\sqrt{d}$ and $\bw'_i=(\lambda'\bP+(1-\lambda')\bQ)\bS'_i/\sqrt{d}$ for some $\lambda,\lambda'\in [0,1]$. We now derive an upper bound on
	\begin{align*}
		\prod_{j=1}^d
		\EE_{\bZ_j}
		\prod_{i:\alpha_{i,j}\ge 4}
		h_{\alpha_{i,j}-1}(\bZ_j^{\top}\bw_i)
		h_{\alpha_{i,j}-1}(\bZ_j^{\top}\bw'_i).
	\end{align*}
	that holds with high probability over $\bS_i,\bS'_i$. For a fixed $j\in[d]$, let the set $\calS_j:=\{i:\alpha_{i,j}\ge 4\}$ be $\calS_j=\{i_1,\dots,i_m\}$. Note that $\bS_i$ and $\bS'_i$ are independent samples for $i\in\calS_j$. On the other hand by properties of Gaussian random variables we can write, for example, that conditional on the $\bS_i$'s
	$$
	\bZ_j^{\top}\bS_m 
	=\bZ_j^{\top}\calP_{-m}\bS_m
	+\|\calP_{-m}^{\perp}\bS_m\|Z
	$$
	where $\calP_{-m}$ is the projection onto the span of $\{\bS_i,\,\bS'_i:i\in \calS_j/i_m\}$, and $Z\sim N(0,1)$ is independent of $\bZ_j^{\top}\bS_i$ and $\bZ_j^{\top}\bS'_i$ for $i\in\calS_j$, $i\neq i_m$. We compute the above expectation by successively conditioning on  $\{\bZ_j^{\top}\bS_i,\bZ_j^{\top}\bS'_i:i\in\{i_1,\dots,i_k\}\}$ for $k=1,\dots,m-1$. By Lemma~\ref{lem:Herm-exp}, we have
	\begin{align*}\label{eq:herm-bd1}
		&~\abs*{\EE\left( h_{{\alpha}_{i_m,j}-1}
			\left(
			\bZ_j^{\top}\bw_{i_m}
			\right)
			h_{{\alpha}_{i_m,j}-1}
			\left(
			\bZ_j^{\top}\bw'_{i_m}
			\right)
			\big| \bS_i,\bS_i':i\in\{i_1,\dots,i_{m-1}\}
			\right)}\\
		\le&~
		C\EE\left(
		\dfrac{\|\calP_{-m}\bS_m\|}{\sqrt{d}}
		\max\left\{
		\dfrac{\|\calP_{-m}\bS_m\|}{\sqrt{d}},
		\dfrac{|\bS_m^{\top}\calP^{\perp}_{-m}\bS_m|}{
			\|\calP^{\perp}_{-m}\bS_m\|\times \|\calP^{\perp}_{-m}\bS'_m\|}
		\right\}
		\right)^{\alpha_{i_m,j}-1}\\
		&\times
		h_{\alpha_{i_m,j}-1}
		\left(\dfrac{\bZ_j^{\top}\calP_{-m}\bS_m}{\|\calP_{-m}\bS_m\|}\right)
		h_{\alpha_{i_m,j}-1}
		\left(\dfrac{\bZ_j^{\top}\calP_{-m}\bS'_m}{\|\calP_{-m}\bS'_m\|}\right)\\
		\le&~
		C\left(
		\dfrac{(\log d)^{5/2}}{d}
		\right)^{\alpha_{i_m,j}-1}.\numberthis
	\end{align*}
	In the last inequality we use the following high probability bounds on $\bS_i,\bS'_i$. Note that
	$$
	\PP\left(
	\bS_m^{\top}\calP_{-m}\bS_m\ge {tr}(\calP_{-m})+C_1\|\calP_{-m}\|_{\rm F}(\log d) 
	\right)\le d^{-C_2}.
	$$
	for a sufficiently large constants $C_1,C_2>0$. Since $\calP_{-m}$ has rank at most $D=C(\log d)$, the bound on the norms follow. Similarly note that
	$$
	\PP\left(
	\bS_m^{\top}\bS'_m\ge C\sqrt{d} 
	\right)\le d^{-C}
	$$
	by the independence of $\bS$ and $\bS'$. Thus $\bS_m^{\top}\calP_{-m}^{\perp}\bS'_m\le C\sqrt{d}$ with probability at least $1-d^{-C}$. Finally note that 
	$$
	\PP\left(\max_{1\le j\le d}\max_{1\le k\le nd}\bZ_j^{\top}\bw_k\ge C_1\sqrt{\log d}\right)\le d^{-C_2}
	$$
	provided $n={\rm poly}(d)$. Here we write $\bw_k$ to denote unit vectors along $\calP_{-{k}}\bS_k$ and $\calP_{-{k}}\bS'_k$ for $k=2,\dots,m$. Then using upper bounds on Hermite polynomials it follows that, for example,
	$$
	|h_{\alpha_{i_m,j}-1}
	\left(\bZ_j^{\top}\bw_k\right)|
	\le |h_{\alpha_{i_m,j}-1}(C\sqrt{\log d})|
	\le C(\log d)^{(\alpha_{i_m,j}-1)/2}.
	$$
	The other terms can be bounded similarly. Repeating the argument in \eqref{eq:herm-bd1} for all $i,j$, along with \eqref{eq:herm-taylr}, we have
	\begin{align*}\label{eq:hprod-bd1}
		&~
		\EE_{\bS}
		\prod_{j=1}^d
		\EE_{\bZ_j}
		\prod_{i:\alpha_{i,j}\ge 4}
		\left(
		h_{\alpha_{i,j}}
		\left(\dfrac{\bZ_{j}^{\top}\bP\bS_{i}}{\sqrt{d}}\right)
		-
		h_{\alpha_{i,j}}\left(\dfrac{\bZ_{j}^{\top}\bQ\bS_{i}}{\sqrt{d}}\right)
		\right)
		\left(
		h_{\alpha_{i,j}}
		\left(\dfrac{\bZ_{j}^{\top}\bP\bS'_{i}}{\sqrt{d}}\right)
		-
		h_{\alpha_{i,j}}\left(\dfrac{\bZ_{j}^{\top}\bQ\bS'_{i}}{\sqrt{d}}\right)
		\right)\\
		\le&~
		\EE 
		\left(
		h_{\alpha_{i_1,1}}
		\left(\dfrac{\bZ_{1}^{\top}\bP\bS_{i_1}}{\sqrt{d}}\right)
		-
		h_{\alpha_{i_1,1}}\left(\dfrac{\bZ_{j}^{\top}\bQ\bS_{i_1}}{\sqrt{d}}\right)
		\right)
		\left(
		h_{\alpha_{i_1,1}}
		\left(\dfrac{\bZ_{1}^{\top}\bP\bS'_{i_1}}{\sqrt{d}}\right)
		-
		h_{\alpha_{i_1,1}}\left(\dfrac{\bZ_{1}^{\top}\bQ\bS'_{i_1}}{\sqrt{d}}\right)
		\right)\\
		&\times\prod_{(i,j)\neq (i_1,1):\alpha_{ij}\ge 4}
		\dfrac{C\alpha_{i,j}^2\|\bP-\bQ\|_{\rm F}^2(\log d)^2}{d}
		\left(
		\dfrac{(\log d)^{5/2}}{d}
		\right)^{\alpha_{i,j}-1}\numberthis 
	\end{align*}
	Finally we bound the first term in the product as:
	\begin{align*}
		&\EE_{\bS_{i_1},\bS'_{i_1}}
		\EE_{\bZ_1}
		\left(
		h_{\alpha_{i_1,1}}
		\left(\dfrac{\bZ_{1}^{\top}\bP\bS_{i_1}}{\sqrt{d}}\right)
		-
		h_{\alpha_{i_1,1}}\left(\dfrac{\bZ_{j}^{\top}\bQ\bS_{i_1}}{\sqrt{d}}\right)
		\right)
		\left(
		h_{\alpha_{i_1,1}}
		\left(\dfrac{\bZ_{1}^{\top}\bP\bS'_{i_1}}{\sqrt{d}}\right)
		-
		h_{\alpha_{i_1,1}}\left(\dfrac{\bZ_{1}^{\top}\bQ\bS'_{i_1}}{\sqrt{d}}\right)
		\right)\\
		=&~
		\EE_{\bS}
		\bigg(
		\EE_{\bZ_1}
		h_{\alpha_{i_1,1}}
		\left(\dfrac{\bZ_1^{\top}\bP\bS_{i_1}}{\sqrt{d}}\right)
		h_{\alpha_{i_1,1}}
		\left(\dfrac{\bZ_1^{\top}\bP\bS'_{i_1}}{\sqrt{d}}\right)
		+
		\EE_{\bZ_1}
		h_{\alpha_{i_1,1}}
		\left(\dfrac{\bZ_1^{\top}\bQ\bS_{i_1}}{\sqrt{d}}\right)
		h_{\alpha_{i_1,1}}
		\left(\dfrac{\bZ_1^{\top}\bQ\bS'_{i_1}}{\sqrt{d}}\right)\\
		&
		-\EE_{\bZ_1}
		h_{\alpha_{i_1,1}}
		\left(\dfrac{\bZ_1^{\top}\bQ\bS_{i_1}}{\sqrt{d}}\right)
		h_{\alpha_{i_1,1}}
		\left(\dfrac{\bZ_1^{\top}\bP\bS'_{i_1}}{\sqrt{d}}\right)
		-\EE_{\bZ_1}
		h_{\alpha_{i_1,1}}
		\left(\dfrac{\bZ_1^{\top}\bP\bS_{i_1}}{\sqrt{d}}\right)
		h_{\alpha_{i_1,1}}
		\left(\dfrac{\bZ_1^{\top}\bQ\bS'_{i_1}}{\sqrt{d}}\right)
		\bigg)\\
		=&~
		2\EE\left(\dfrac{\bS_i^{\top}\bS_i'}{d}\right)^{\alpha_{i_1}}
		-2\EE\left(\dfrac{\bS_i^{\top}\bP^{\top}\bQ\bS_i'}{d}\right)^{\alpha_{i_1}}.
	\end{align*}
	The equalities follow using part 1 of Lemma~\ref{lem:Herm-exp} and the property that $\EE h_{\alpha}^2(X)=1$ when $X\sim N(0,1)$. We use the lemmas:
	\begin{lemma}\label{lem:herm-mdiff2}
		$\EE\left(\dfrac{\bS_i^{\top}\bS_i'}{d}\right)^{\alpha}
		-\EE\left(\dfrac{\bS_i^{\top}\bP^{\top}\bQ\bS_i'}{d}\right)^{\alpha}
		\le ~
		\dfrac{C\alpha\|\bP-\bQ\|_{\rm F}^2}{d^2}
		\left(\dfrac{2}{d}\right)^{\alpha/2}		
		$
		.
	\end{lemma}

	\begin{lemma}\label{lem:Herm-exp} For $\rho\in (-1,1)$, any random variable $X$, and $Z\sim N(0,1)$, we have 
		$$\EE h_{\alpha}(\rho X+\sqrt{1-\rho^2}Z|X)=\rho^{\alpha}h_{\alpha}(X)$$
		Moreover, 
		$$
		\abs*{\EE\left(
			h_{\alpha}(\rho_1 X_1+\sqrt{1-\rho_2^2}Z_1)
			h_{\alpha}(\rho_2 X_2+\sqrt{1-\rho_2^2}Z_2)\big| X_1,X_2\right)}
		\le C\rho_1^{\alpha}\max\{\rho_2,\rho_{12}\}^{\alpha}h_{\alpha}(X_1)h_{\alpha}(X_2)
		$$
		for a constant $C>0$ and $Z_1,Z_2\sim N(0,1)$ with ${\rm corr}(Z_1,Z_2)=\rho_{12}$.
	\end{lemma}
	Plugging in the above bound into \eqref{eq:hprod-bd1} we have
	\begin{align*}
		&~\EE_{\bS}
		\prod_{j=1}^d
		\EE_{\bZ_j}
		\prod_{i:\alpha_{i,j}\ge 4}
		\left(
		h_{\alpha_{i,j}}
		\left(\dfrac{\bZ_{j}^{\top}\bP\bS_{i}}{\sqrt{d}}\right)
		-
		h_{\alpha_{i,j}}\left(\dfrac{\bZ_{j}^{\top}\bQ\bS_{i}}{\sqrt{d}}\right)
		\right)
		\left(
		h_{\alpha_{i,j}}
		\left(\dfrac{\bZ_{j}^{\top}\bP\bS'_{i}}{\sqrt{d}}\right)
		-
		h_{\alpha_{i,j}}\left(\dfrac{\bZ_{j}^{\top}\bQ\bS'_{i}}{\sqrt{d}}\right)
		\right)\\
		\le&~
		\dfrac{C\alpha_{i_1,1}\|\bP-\bQ\|_{\rm F}^2}{d^2}
		\left(\dfrac{2}{d}\right)^{\alpha_{i_1,1}/2}
		\times\prod_{(i,j)\neq (i_1,1):\alpha_{ij}\ge 4}
		\dfrac{C\alpha_{i,j}^2\|\bP-\bQ\|_{\rm F}^2(\log d)^2}{d}
		\left(
		\dfrac{(\log d)^{5/2}}{d}
		\right)^{\alpha_{i,j}-1}\\
		\le&~
		\left(
		\dfrac{C\alpha_{i,j}^2\|\bP-\bQ\|_{\rm F}^2(\log d)^2}{d}
		\right)^{\|\balpha\|_0}
		\prod_{i,j:\alpha_{i,j}\ge 4}
		\left(
		\dfrac{(\log d)^{5/2}}{d}
		\right)^{\alpha_{i,j}/2}\\
		=&~
		\left(
		\dfrac{C\|\bP-\bQ\|_{\rm F}^2(\log d)^4}{d}
		\right)^{\|\balpha\|_0}
		\left(
		\dfrac{(\log d)^{5/2}}{d}
		\right)^{\|\balpha\|/2}.
	\end{align*}
	This finishes the proof.
\end{proof}

\medskip

\begin{proof}[Proof of Lemma~\ref{lem:herm-mdiff2}]
	Since $S_j$ are independent random variables with a distribution symmetric around zero, we have
	$$
	\EE (S_j/\|\bS\|)^{2k-1}=0
	\quad
	\text{for all }k\in\NN.
	$$
	By assumption $S_k$ are independent sub-exponential random variables with $\EE S_k=0$, ${\rm Var}(S_k)=1$. Thus,
	\begin{align*}
		\PP(\|S\|^2\le d/2)
		\le&~
		\PP\left(\sum_{k=1}^d(S_k^2-1)\ge d\right)
		\le \dfrac{C}{(d/2)^{\alpha}}
	\end{align*}
	by Lemma~\ref{lem:Sk-tail-bds} for $4\alpha$, where $\alpha\le D$. Thus
	\begin{align*}\label{eq:mean-diff}
		&~
		\EE_{\bS,\bS'}
		\left\{
		\left(
		\dfrac{\langle\bS,\bS'\rangle}{d}
		\right)^{\alpha_{i}}
		-
		\left(
		\dfrac{\langle\bS,\bP^{\top}\bQ\bS'\rangle}{d}
		\right)^{\alpha_{i}}
		\right\}\\
		=&~
		\EE_{\bS,\bS'}
		\left(
		\dfrac{1}{d^{\alpha_i}}
		\big[
		\langle\bS,\bS'\rangle^{\alpha_i}
		-\langle\bS,\bP^{\top}\bQ\bS'\rangle^{\alpha_i}
		\big]
		\right)\\
		\le&~
		\left(\dfrac{2}{d}\right)^{\alpha_i}
		\abs*{
			\EE\langle\bS,\bS'\rangle^{\alpha_i}
			-
			\EE\langle\bS,\bP^{\top}\bQ\bS'\rangle^{\alpha_i}
		}
		+\PP(\min\{\|\bS\|\}\le d/2)\\
		\le&~
		\left(\dfrac{C}{d}\right)^{\alpha_i/2}
		\abs*{
			\EE\langle\bS,\bS'_u\rangle^{\alpha_i}
			-
			\EE\langle\bS,\bP^{\top}\bQ\bS'_u\rangle^{\alpha_i}
		}
		+C\left(
		\dfrac{1}{d}
		\right)^{\alpha_i}
		,\numberthis
	\end{align*}
	where we write $\bS'_u:=\bS/\|\bS\|$, and use the high probability lower bound on $\|\bS\|$. We next expand the moments in terms of cumulants as follows. Writing $\bO:=\bP^{\top}\bQ\in\calO(d)$, we have:
	\begin{align*}
		&~\EE_{\bS}\langle\bS,\bS'_u\rangle^{\alpha}
		-
		\EE_{\bS}\langle\bS,\bO\bS'_u\rangle^{\alpha}\\
		=&~
		\sum_{k=1}^{\alpha}
		\left(B_{\alpha,k}(\kappa_1(\langle\bS,\bS'_u\rangle),\dots,\kappa_{\alpha-k+1}(\langle\bS,\bS'_u\rangle))
		-B_{\alpha,k}(\kappa_1(\langle\bS,\bO\bS'_u\rangle),\dots,\kappa_{\alpha-k+1}(\langle\bS,\bO\bS'_u\rangle))
		\right)
		\\
		=&~
		\sum_{k=1}^{\alpha}
		\bigg(
		B_{\alpha,k}(0,1,0,\kappa_4(\langle\bS,\bS'_u\rangle),\dots,\kappa_{\alpha-k+1}(\langle\bS,\bS'_u\rangle))\\
		&~\hspace{2cm}-B_{\alpha,k}(0,1,0,\kappa_4(\langle\bS,\bO\bS'_u\rangle),\dots,\kappa_{\alpha-k+1}(\langle\bS,\bO\bS'_u\rangle))
		\bigg)\\
		=:&~
		\sum_{k=1}^{\alpha}
		\left(
		g(\alpha,k,\langle \bS,\bS'_u\rangle)
		-g(\alpha,k,\langle \bS,\bO\bS'_u\rangle)
		\right)
	\end{align*}
	where $B_{\alpha,k}$ are incomplete Bell polynomials. Note that the odd cumulants $\kappa_{2k+1}(\langle\bS,\bS'_u\rangle)=0$ and $\kappa_{2k+1}(\langle\bS,\bO\bS'_u\rangle)=0$ since $S_k$ are symmetric around zero. Moreover $\kappa_{2}(\langle\bS,\bS'_u\rangle)=\kappa_{2}(\langle\bS,\bO\bS'_u\rangle)=\sum (S_i')^2/\|\bS'\|^2=1$. Finally we have by the independence of $S_k$ that
	$$
	\kappa_{2k}(\langle\bS,\bS'_u\rangle)=\sum_{i=1}^d\kappa_{2k}(S_i(S_u)_i')
	\quad 
	\text{and}
	\quad 
	\kappa_{2k}(\langle\bS,\bO\bS'_u\rangle)=\sum_{i=1}^d\kappa_{2k}
	(S_i\bo_i^{\top}\bS'_u).
	$$
	The incomplete Bell polynomials can be written as:
	\begin{align*}
		B_{\alpha,k}(x_1,\dots,x_{\alpha-k+1})=
		\sum\dfrac{\alpha!}{j_1!j_2!\dots j_{\alpha-k+1}!}\prod_{l=1}^{\alpha-k+1}
		\left(
		\dfrac{x_l}{l!}
		\right)^{j_l}.
	\end{align*}
	Here the sum is over all $j_1,\dots,j_{\alpha-k+1}\in \{0\}\cup\NN$ such that
	$$
	\sum_{l=1}^{\alpha-k+1}j_l=k\quad
	\text{and}
	\quad
	\sum_{l=1}^{\alpha-k+1}lj_l=\alpha.
	$$
	Since odd cumulants are zero in our case, the sum only runs over $(0,j_2,0,j_4,\dots,j_{2\floor{(n-k+1)/2}})$. By the above equations it can then be checked that
	$$
	g(\alpha,k,\langle \bS,\bS'_u\rangle)
	=g(\alpha,k,\langle \bS,\bO\bS'_u\rangle)=0
	\quad
	\text{for }k>\alpha/2.
	$$
	When $k=\alpha/2$, the only possible configuration of $j_l$'s for which $g(\cdot)$ above takes nonzero values is given by $j_2=\alpha/2$ and $j_l=0$ for $l\neq 2$. Thus,
	\begin{align*}
		&g(\alpha,\alpha/2,\langle \bS,\bS'_u\rangle)
		-g(\alpha,\alpha/2,\langle \bS,\bO\bS'_u\rangle)\\
		=&~
		\dfrac{\alpha!}{(\alpha/2)!}
		\left[
		\left( \dfrac{\kappa_2(\langle\bS,\bS'_u\rangle)}{2!}\right)^{\alpha/2}
		-\left( \dfrac{\kappa_2(\langle\bS,\bO\bS'_u\rangle)}{2!}\right)^{\alpha/2}\right]
		=~
		\dfrac{\alpha!}{(\alpha/2)!}
		\left[
		\left( \dfrac{1}{2!}\right)^{\alpha/2}
		-\left( \dfrac{1}{2!}\right)^{\alpha/2}\right]
		=0.
	\end{align*}
	Similarly for $k=\alpha/2-1$, the only possible configuration of $j_l$'s for which $g(\cdot)$ above takes nonzero values is given by $j_2=(\alpha-4)/2$, $j_4=1$ and $j_l=0$ for $l\neq 2,4$. Thus,
	\begin{align*}
		&g(\alpha,\alpha/2-1,\langle \bS,\bS'_u\rangle)
		-g(\alpha,\alpha/2-1,\langle \bS,\bO\bS'_u\rangle)\\
		=&~
		\dfrac{\alpha!}{(\alpha/2-1)!}
		\left[
		\left( \dfrac{\kappa_2(\langle\bS,\bS'_u\rangle)}{2!}\right)^{\alpha/2-2}
		\dfrac{\kappa_4(\langle\bS,\bS'_u\rangle)}{4!}
		-\left( \dfrac{\kappa_2(\langle\bS,\bO\bS'_u\rangle)}{2!}\right)^{\alpha/2-2}
		\dfrac{\kappa_4(\langle\bS,\bO\bS'_u\rangle)}{4!}
		\right]\\
		=&~
		\dfrac{\alpha!}{(\alpha/2-1)!}
		\left( \dfrac{1}{2!}\right)^{\alpha/2-2}
		\times
		\dfrac{1}{4!}
		\left[
		\sum_{i=1}^d
		(\kappa_4(S_iS'_i/\|\bS'\|)-\kappa_4(S_i\bo_i^{\top}\bS'/\|\bS'\|))
		\right]\\
		=&~
		\dfrac{\alpha!}{(\alpha/2-1)!}
		\left( \dfrac{1}{2!}\right)^{\alpha/2-2}
		\times
		\dfrac{1}{4!}
		\left[
		\kappa_4(S_1)
		\sum_{i=1}^d
		\dfrac{(S'_i)^4-(\bo_i^{\top}\bS')^4}{\|\bS'\|^4}
		\right]
	\end{align*}
	Thus taking a second expectation over $\bS'$ we have
	\begin{align*}
		&~\EE_{\bS'}\left(
		g(\alpha,\alpha/2-1,\langle \bS,\bS'_u\rangle)
		-g(\alpha,\alpha/2-1,\langle \bS,\bO\bS'_u\rangle)
		\right)
		\\
		\le&~
		\dfrac{\kappa_4(S_1)\alpha!}{(d/2)^24!(\alpha/2-1)!}
		\left( \dfrac{1}{2!}\right)^{\alpha/2-2}
		\left[
		\sum_{i=1}^d
		\{\EE(S'_i)^4-\EE(\bo_i^{\top}\bS')^4\}
		\right]+C\left(
		\dfrac{1}{d}
		\right)^{\alpha}\\
		\le&~
		\dfrac{C\kappa_4(S_1)^2\alpha!}{(d/2)^24!(\alpha/2-1)!}
		\left( \dfrac{1}{2!}\right)^{\alpha/2-2}
		\left[
		\sum_{i=1}^d
		\{
		1-\sum_{j=1}^do_{ij}^4\}
		\right]\\
		=&~
		\dfrac{C\kappa_4(S_1)^2\alpha!}{(d/2)^24!(\alpha/2-1)!}
		\left( \dfrac{1}{2!}\right)^{\alpha/2-2}
		\left[
		\sum_{i=1}^d
		\sum_{j=1}^d
		\{
		o_{ij}^2(1-o_{ij}^2)\}
		\right]
	\end{align*}
	where we use the high probability event $\PP(\|\bS'\|\ge \sqrt{d/2})\ge 1-d^{-\alpha}$ in the first inequality. The higher order terms, for $k\le \alpha/2-2$ can be bounded similarly using successive differences in the products of cumulants of orders $(j_1,\dots,j_{\alpha-k+1})$.

	We then have
	\begin{align*}
		&~\EE_{\bS,\bS'}\langle\bS,\bS'_u\rangle^{\alpha}
		-
		\EE_{\bS,\bS'}\langle\bS,\bO\bS'_u\rangle^{\alpha}\\
		\le&~
		C\alpha 
		\left(\dfrac{\kappa_4(S_1)^2}{d^2}
		\left[
		\sum_{i=1}^d
		\sum_{j=1}^d
		\{
		o_{ij}^2(1-o_{ij}^2)\}
		\right]
		\right)
		\max\{
		\EE_{\bS,\bS'}
		\langle\bS,\bS'_u
		\rangle^{\alpha},
		\EE_{\bS,\bS'}
		\langle\bS,\bO\bS'_u
		\rangle^{\alpha}
		\}
		+C\left(
		\dfrac{1}{d}
		\right)^{\alpha}\\
		\le&~
		C\alpha 
		\left(\dfrac{\kappa_4(S_1)^2}{d^2}
		\left[
		\sum_{i=1}^d
		\sum_{j=1}^d
		\{
		o_{ij}^2(1-o_{ij}^2)\}
		\right]
		\right)
		\left\{
		1+
		\EE_{\bS'}
		\sum_{k=1}^d
		(\max\{S'_k,\bo_k^{\top}\bS'\}/\|\bS'\|)^{\alpha}
		\right\}
		\\
		\le&~
		C\alpha 
		\left(\dfrac{\kappa_4(S_1)^2}{d^2}
		\left[
		\sum_{i=1}^d
		\sum_{j=1}^d
		\{
		o_{ij}^2(1-o_{ij}^2)\}
		\right]
		\right)
		.
	\end{align*}
	The first inequality follows by comparing the expression for the difference in expectations to the case where $\bO=\II_d$. To finish the proof, we use the fact that $\bO=\bP^{\top}\bQ$ and simplify the expression. In particular,
	\begin{align*}
		~\sum_{i=1}^d
		\sum_{j=1}^d
		\{
		o_{ij}^2(1-o_{ij}^2)\}
		\le &~
		\sum_{i=1}^d (1-o_{ii}^2)+\sum_{i\neq j}o_{ij}^2
		\le ~
		2\sum_{i=1}^d (1-o_{ii})+\sum_{i\neq j}o_{ij}^2\\
		=&~
		\sum_{i=1}^d (2-2\bp_i^{\top}\bq_i)+\sum_{i\neq j}(\bP^{\top}\bQ)_{ij}^2
		=~
		\sum_{i=1}^d \|\bp_i-\bq_i\|^2+\sum_{i\neq j}(\bP^{\top}\bQ)_{ij}^2\\
		\le&~
		\|\bP-\bQ\|_{\rm F}^2+\|\II_d-\bP^{\top}\bQ\|_{\rm F}^2\\
		=&~
		2\|\bP-\bQ\|_{\rm F}^2.
	\end{align*}	
	Plugging this bound into \eqref{eq:mean-diff} we have the lemma.
\end{proof}

\medskip

\begin{proof}[Proof of Lemma~\ref{lem:Herm-exp}]
	The proof follows equations 82-83 of \cite{mao2021optimal}. 
	It is a standard fact that
	$$
	h_{\alpha}(x+y)=\sum_{l=0}^{\alpha}
	\sqrt{\dfrac{l!}{\alpha!}}
	{\alpha\choose l}x^{\alpha-l}h_l(y).
	$$
	Thus
	\begin{align*}
		&\EE
		h_{\alpha}(\rho X+\sqrt{1-\rho^2}Z)\\
		=&~
		\sum_{l=0}^{\alpha}
		\sqrt{\dfrac{l!}{\alpha!}}
		{\alpha\choose l}
		\left(\rho X\right)^{\alpha-l}
		\EE h_l\left(
		\sqrt{1-\rho^2}Z
		\right)\\
		=&~
		\sum_{l=0}^{\alpha}
		\sqrt{\dfrac{l!}{\alpha!}}
		{\alpha\choose l}
		\left(\rho X\right)
		^{\alpha-l}
		(l-1)!!(l!)^{-1/2}
		\left(
		(1-\rho^2)-1
		\right)^{l/2}\mathbbm{1}(l\text{ is even})\\
		=&~
		\sum_{l=0}^{\alpha}
		\sqrt{\dfrac{l!}{\alpha!}}
		{\alpha\choose l}
		\left(\rho X\right)
		^{\alpha-l}
		(l-1)!!(l!)^{-1/2}
		\left(
		(1-\rho^2)-1
		\right)^{l/2}\mathbbm{1}(l\text{ is even})\\
		=&~
		\rho^{\alpha}
		\sum_{l=0}^{\alpha}
		\sqrt{\dfrac{l!}{\alpha!}}
		{\alpha\choose l}
		X^{\alpha-l}
		(l-1)!!(l!)^{-1/2}
		\mathbbm{1}(l\text{ is even})\\
		=&~
		\rho^{\alpha}h_{\alpha}(X).
	\end{align*}
	In the third equality, we use the standard fact that
	$$
	\EE h_{l}(\sqrt{1-\rho^2}Z)=(l-1)!!(l!)^{-1/2}(-\rho)^l\mathbbm{1}(l\text{ is even}).
	$$	
	The second last equality follows by using the third equality for $\rho=1$. Similarly,
	\begin{align*}
		&\EE
		h_{\alpha}(\rho_1 X_1+\sqrt{1-\rho_2^2}Z_1)
		h_{\alpha}(\rho_2 X_2+\sqrt{1-\rho_2^2}Z_2)\\
		=&~
		\sum_{l_1=0}^{\alpha}
		\sum_{l_2=0}^{\alpha}
		\sqrt{\dfrac{l_1!l_2!}{(\alpha!)^2}}
		{\alpha\choose l_1}
		{\alpha\choose l_2}
		\left(\rho_1 X_1\right)^{\alpha-l_1}
		\left(\rho_2 X_2\right)^{\alpha-l_2}
		\EE h_{l_1}\left(
		\sqrt{1-\rho_1^2}Z_1
		\right)
		h_{l_2}\left(
		\sqrt{1-\rho_2^2}Z_2
		\right)
	\end{align*}
	We can then write $Z_2=\rho_{12}Z_1+\sqrt{1-\rho_{12}^2}Z_3$ where $Z_3$ is independent of $Z_1$. Then following the previous calculation, we get
	\begin{align*}
		&~\EE\left( h_{l_1}\left(
		\sqrt{1-\rho_1^2}Z_1
		\right)
		h_{l_2}\left(
		\sqrt{1-\rho_2^2}Z_2
		\right)\right)\\
		=&~
		\EE
		\EE\left( h_{l_1}\left(
		\sqrt{1-\rho_1^2}Z_1
		\right)
		h_{l_2}\left(
		\sqrt{1-\rho_2^2}Z_2
		\right)	|Z_1\right)\\
		=&~
		\EE
		h_{l_1}\left(
		\sqrt{1-\rho_1^2}Z_1
		\right)
		\EE 
		h_{l_2}
		\left(
		\sqrt{1-\rho_2^2}
		(\rho_{12}Z_1+\sqrt{1-\rho_{12}^2}Z_3)
		\right)		\\
		=&~
		\EE
		h_{l_1}\left(
		\sqrt{1-\rho_1^2}Z_1
		\right)
		\EE 
		h_{l_2}
		\left
		(\rho_{12}\sqrt{1-\rho_2^2}Z_1+
		\sqrt{1-\rho_{12}^2-\rho_2^2+\rho_2^2\rho_{12}^2}Z_3
		\right)
		\\
		=&~\rho_{12}^{l_2}
		\EE
		h_{l_1}\left(
		\sqrt{1-\rho_1^2}Z_1
		\right)
		h_{l_2}\left(Z_1\right)
		\le C\rho_{12}^{l_2}\rho_1^{l_1}
	\end{align*}
	using the expansion of Hermite polynomials in the last step. This implies that
	$$
	\abs*{\EE
		h_{\alpha}(\rho_1 X_1+\sqrt{1-\rho_2^2}Z_1)
		h_{\alpha}(\rho_2 X_2+\sqrt{1-\rho_2^2}Z_2)}
	\le C\rho_1^{\alpha}\max\{\rho_2,\rho_{12}\}^{\alpha}
	|\EE h_{\alpha}(X_1)h_{\alpha}(X_2)|.
	$$ 
	Here $\rho_{12}$ is the correlation between $Z_1$ and $Z_2$.
\end{proof}

\medskip

\begin{lemma}\label{lem:ica-op-norm}
	Let $\bX_1,\ldots,\bX_n$ be $n$ independent copies of $\bX$ such that $\calL(\bX)\in\calP_{\rm ICA}(\bA;\epsi, M_1, M_2)$ for some $\bA\in \calO(d)$ and $M_1,M_2>0$. Let $\hat{\bA}_j$ and $\bA_j$ denote the matrix consisting of the first $j$ columns of $\hat{\bA}$ and $\bA$ respectively. We then have
	$$
	\max_{1\le j\le d}
	\norm*{\hat{\bA}_j-\bA_j}
	\le C\sqrt{\dfrac{d(\log d)}{n}}+\dfrac{Cd^{3/2}(\log d)}{n}
	+C\sqrt{\dfrac{d^2\log(\delta)^2}{n}}\mathbbm{1}(\epsi<4)
	$$
	with probability at least $1-d^{-3}-\delta-n^{-\epsi/8}-n^{-(\epsi-4)/12}\mathbbm{1}(\epsi\ge 4)$, provided $n\ge (d\log\delta)^2$.
\end{lemma}

\medskip

\begin{proof}[Proof of Lemma \ref{lem:ica-op-norm}]
	We assume that $n\ge C(d\log\delta)^2$ so that we have
	$$
	\scrE:=\hat{\scrM}_4^{\rm sample}(\bX)-{\scrM}_4(\bX);\,\,	
	\Delta_1:=\max_{1\le j\le d}\|\scrE\times_{2,3,4}\ba_j\|,\,\,
	\Delta_2:=\max_{1\le j\le d}\|\scrE\times_{3,4}\ba_j\|
	.	
	$$ 
	satisfies
	\begin{equation}\label{eq:tens-conc-bds-1}
		\PP\left(
		\Delta_1
		\le
		\sqrt{\dfrac{Cd(\log (d))}{n}},\,\,
		\Delta_2\le \sqrt{\dfrac{Cd(\log d)}{n}},\,\,
		\Delta\le \sqrt{\dfrac{C(d\log \delta)^2}{n}}\right)
		\ge 1-\delta-d^{-3}-n^{-\epsi/8}.
	\end{equation}
	by the concentration bounds from Lemmas \ref{lem:tens-conc}-\ref{lem:eps1-bd}. For any $1\le j\le d$, let us define the matrix $\bE_j\in\RR^{d\times j}$
	$$
	\bE:=[\scrE\times_{2,3,4}\ba_1\,
	\scrE\times_{2,3,4}\ba_2\,
	\dots\,
	\scrE\times_{2,3,4}\ba_j].
	$$
	By the expression from Lemma \ref{lem:asy}, we have that
	$$
	\norm*{
		\hat{\bA}-\bA-\bE
	}
	\le
	\sqrt{j}\times (\eps_1\eps_2+\eps_1^2\|\scrE\|)
	\le \dfrac{Cd^{3/2}(\log d)}{n}.
	$$
	The proof follows by the bound on $\|\bE_j\|\le\|\bE_d\|$ from Lemma~\ref{lem:Ejnorm}.
\end{proof}

\medskip


%

\begin{lemma}\label{lem:piter} Let $\gamma_j:={\rm sign}(\langle\hat{\ba}_{j,[t]},\ba_j\rangle)$. For $\scrE, \eps_1,\eps_2$ defined in \eqref{eq:def-eps12}, we have that
	$$	\sin\angle\left(\ba_j,
	\hat{\ba}_{j,[t+1]}
	-\dfrac{\gamma_j}{\kappa_j}\scrE\times_{2,3,4}\ba_j
	\right)
	\le \dfrac{CL_t\eps_2+L_t^2(\kappa_j+C\|\scrE\|)}
	{\kappa_j}.$$
\end{lemma}

\medskip

\begin{proof}[Proof of Lemma \ref{lem:piter}]
	
	By the induction hypothesis, we have estimates $\hat{\ba}_{[t]}$ such that 
	$$
	\norm*{\ba_j-{\rm sign}((\ba_j)^\top\hat{\ba}_{[t]})
		\hat{\ba}_{[t]}}
	\le \sqrt{2}L_t.
	$$
	To simplify the proof, we assume that $\gamma_j:={\rm sign}(\langle\hat{\ba}_{j,[t]},\ba_j\rangle)=1$. The other cases can be handled by reversing the sign of $\hat{\ba}_{j,[t]}$ accordingly. We have, for any $\bv\perp \ba_j$ that
	\begin{equation}\label{eq:piter-num}
		\langle\bv,(\hat{\scrT}-\scrM_0)\times_{2,3,4}\hat{\ba}_{j,[t]}\rangle
		=\sum_{i\neq j}
		\kappa_4(S_i)
		\langle\bv,\ba_i\rangle
		(\langle\hat{\ba}_{[t]},\ba_i\rangle)^3
		+\scrE\times_1\bv\times_{q>1}\hat{\ba}_{[t]}
	\end{equation}
	By Cauchy-Schwarz inequality, the first term is
	\begin{equation}\label{eq:piter-sigbd}
		\abs*{\sum_{i\neq j}\kappa_i
			\langle\bv,\ba_i\rangle
			(\langle\hat{\ba}_{[t]},\ba_i\rangle)^3}
		\le \max_{i\neq j}
		\{\kappa_i |\langle\bv,\ba_i\rangle|\}
		(\sin\angle\left(\hat{\ba}_{[t]},\ba_j\right))^3
		\le \kappa_{\max}L_t^3.
	\end{equation}
	For the second term, note that
	\begin{align*}\label{eq:piter-errbd}
		\scrE\times_1\bv\times_{q>1}\hat{\ba}_{[t]}
		=&~\scrE\times_1\bv\times_{q>1}(
		\ba_j+
		\hat{\ba}_{j,[t]}-\ba_j)\\
		\le& ~\scrE\times_1\bv\times_{q>1}\ba_j
		+3\sqrt{2}L_t
		\norm*{\scrE\times_{2,3,4}\ba_j}\\
		&~+\sum_{A\subset [p]/\{1\},\, |A|\ge 2}
		\scrE\times_1\bv\times_{q\in A}(\hat{\ba}_{j,[t]}-\ba_j)
		\times_{q'\notin A}\ba_j\\
		\le&
		~\scrE\times_1\bv\times_{2,3,4}\ba_j
		+CL_t\eps_2
		+CL_t^2\|\scrE\|.\numberthis
	\end{align*}	
	Plugging in the upper bounds from \eqref{eq:piter-sigbd}, \eqref{eq:piter-errbd} into \eqref{eq:piter-num}, yields
	\begin{equation}\label{eq:sinth-num}
		\sup_{\bv\perp\bu_j^{(1)}}
		\langle\bv,
		(\hat{\scrT}-\scrM_0)\times_{2,3,4}\hat{\ba}_{j,[t]}
		-\scrE\times_{2,3,4}\ba_j
		\rangle
		\le 
		CL_t\eps_2
		+L_t^2(L_t\kappa_{\max}+C\|\scrE\|).
	\end{equation}
	It remains to bound the norm. We have
	\begin{align*}\label{eq:sinth-den}
		\|(\hat{\scrT}-\scrM_0)\times_{2,3,4}\hat{\ba}_{j,[t]}\|
		\ge &~(\hat{\scrT}-\scrM_0)\times_1\ba_{j}
		\times_{2,3,4}\hat{\ba}_{j,[t]}\\
		\ge &~\kappa_j
		(\langle \hat{\ba}_{j,[t]},\ba_j\rangle^3
		-\scrE\times_1\ba_j\times_{2,3,4}\hat{\ba}_{j,[t]}
		\\
		\ge&~\kappa_j(1-L_t^2)^3
		-\scrE\times_{1,2,3,4}\ba_j-CL_t\eps_2-CL_t^2\|\scrE\|
		\numberthis
	\end{align*}
	following the steps of \eqref{eq:piter-sigbd}, \eqref{eq:piter-errbd} with $\bu_j^{(1)}$ instead of $\bv\perp\bu_j^{(1)}$.
	Dividing \eqref{eq:sinth-num} by \eqref{eq:sinth-den}, one has
	\begin{align*}
		\sup_{\bv\perp\bu_j^{(1)}}
		\dfrac{
			\langle\bv,(\hat{\scrT}-\scrM_0)\times_{2,3,4}\hat{\ba}_{j,[t]}
			-\scrE\times_{2,3,4}\ba_j
			\rangle}{\|(\hat{\scrT}-\scrM_0)
			\times_{2,3,4}\hat{\ba}_{j,[t]}\|}
		\le &
		\dfrac{CL_t\eps_2+L_t^2(L_t\kappa_{\max}+C\|\scrE\|)}
		{\kappa_j(1-L_t^2)^3
			-\scrE\times_{1,2,3,4}\ba_j
			-CL_t\eps_2-CL_t^2\|\scrE\|}
		\\
		\le & \dfrac{CL_t\eps_2+L_t^2(\kappa_j+C\|\scrE\|)}
		{\kappa_j}
	\end{align*}
	since by assumption, $L_0\le  \min\{1/C,\frac{\kappa_{\min}}{2\kappa_{\max}}\}$, and by induction hypothesis it can be verified that, $L_t\le \min\{1/C,\frac{\kappa_{\min}}{2\kappa_{\max}}\}$. Then  
	by Lemma \ref{lem:tens-conc},
	$$
	3\scrE\times_{1,2,3,4}\ba_j+CL_t\eps_2+CL_t^2\|\scrE\|
	\le C\sqrt{\dfrac{d^2(\log d)^2}{n}}
	\le \kappa_j.
	$$
	Moreover, by similar calculations, we also have that 
	\begin{align*}
		\sup_{\bv\perp\ba_j}
		\scrE\times_1\bv\times_{2,3,4}\ba_j
		\left(\dfrac{1}{\kappa_j}-
		\dfrac{1}
		{\|(\hat{\scrT}-\scrM_0)\times_{2,3,4}\hat{\ba}_{j,[t]}\|}
		\right)
		\le &
		~\dfrac{C\eps_1\eps_2L_t+\eps_1L_t^2(\kappa_j+C\|\scrE\|)
		}{\kappa_j^2}.
	\end{align*}
	Combining all the bounds we have
	\begin{equation}\label{eq:sinth-fo-bd}
		\sin\angle\left(\ba_j,
		\hat{\ba}_{j,[t+1]}
		-\dfrac{1}{\kappa_j}\scrE\times_{2,3,4}\ba_j
		\right)
		\le \dfrac{CL_t\eps_2+L_t^2(\kappa_j+C\|\scrE\|)}
		{\kappa_j}.
	\end{equation}
	Thus,
	$$
	L_{t+1}\le \dfrac{\eps_1}{\kappa_j}
	+\dfrac{CL_t\eps_2+L_t^2(\kappa_j+C\|\scrE\|)}
	{\kappa_j}.
	$$
\end{proof}

\medskip

\begin{lemma}\label{lem:varY-conc}
	Let $\bX_1,\dots,\bX_n$ be $n$ independent copies of $\bX$ s.t. $\calL(\bX)\in\calP_{\rm ICA}(\bA;\epsi,M_1,M_2)$ for some $\bA\in\calO(d)$ and $\epsi,M_1,M_2>0$. Let $\boldsymbol{\alpha}:={\rm vec}(\II_d)$. Then defining $\bY_i:=\bS_i\otimes\bS_i\in \RR^{d^2}$, one has for any $t_1,t_2>0$ that
	$$
	\norm*{
		\dfrac{1}{n}\sum_{i=1}^n(\bY_i-\boldsymbol{\alpha})(\bY_i-\boldsymbol{\alpha})^{\top}
		-\EE (\bY-\boldsymbol{\alpha})(\bY-\boldsymbol{\alpha})^{\top}
	}
	\ge C(t_1+t_2)\sqrt{\dfrac{d^2}{n}}$$
	with probability at most 
	$
	e^{-Cn^{\epsi/8}}+\exp(-Ct_1)+\dfrac{1}{Cn^{\epsi/8}t_2^{2+\epsi/4}}
	$.
\end{lemma}
\medskip

\begin{proof}[Proof of Lemma \ref{lem:varY-conc}]	
	We use the truncation strategy which is standard in proving concentration bounds of this nature, previously used in \cite{adamczak2010} and \cite{vershynin2011approximating}, for example. We use the fact that random vectors $\bS_l\in \RR^d$ with independent components. Note that $\EE\bY_l={\rm vec}(\II_d)=\boldsymbol{\alpha}$.	Then we divide the empirical deviation as follows:
	

	\begin{align*}
		&~\norm*{
			\dfrac{1}{n}\sum_{i=1}^n(\bY_i-\boldsymbol{\alpha})(\bY_i-\boldsymbol{\alpha})^{\top}
			-\EE (\bY-\boldsymbol{\alpha})(\bY-\boldsymbol{\alpha})^{\top}}\\
		=&~
		\norm*{
			\dfrac{1}{n}
			\sum_{l=1}^n
			(\bS_l\otimes \bS_l-\boldsymbol{\alpha})(\bS_l\otimes\bS_l-\boldsymbol{\alpha})^{\top}
			-\EE (\bS_1\otimes \bS_1-\boldsymbol{\alpha})(\bS_1\otimes\bS_1-\boldsymbol{\alpha})^{\top}}\\
		\le &~
		\norm*{
			\dfrac{1}{n}
			\sum_{l=1}^n
			(\bY_l-\boldsymbol{\alpha})(\bY_l-\boldsymbol{\alpha})^{\top}\mathbbm{1}(\|\bY_l\|\le C(nd^2)^{1/4})
			-\EE (\bY_1-\boldsymbol{\alpha})(\bY_1-\boldsymbol{\alpha})^{\top}\mathbbm{1}(\|\bY_1\|\le C(nd^2)^{1/4})}	\\
		&+\norm*{
			\dfrac{1}{n}
			\sum_{l=1}^n(\bY_l-\boldsymbol{\alpha})(\bY_l-\boldsymbol{\alpha})^{\top}\mathbbm{1}(\|\bY_l\|\ge C(nd^2)^{1/4})}\\
		&+\norm*{
			\EE(\bY_1-\boldsymbol{\alpha})(\bY_1-\boldsymbol{\alpha})^{\top}\mathbbm{1}(\|\bY_1\|\ge C(nd^2)^{1/4})
		}.
	\end{align*}
	We will bound each of these terms separately. To bound the first term, we use matrix Bernstein inequality. We define
	$$
	\bM_l=(\bY_l-\boldsymbol{\alpha})(\bY_l-\boldsymbol{\alpha})^{\top}\mathbbm{1}(\|\bY_l\|\le C(nd^2)^{1/4}).
	$$
	Then since $\|\boldsymbol{\alpha}\|=\sqrt{d}$, it is immediate that
	$$
	\|\bM_l\|=\|\bY_l-\boldsymbol{\alpha}\|^2\mathbbm{1}(\|\bY_l\|\le C(nd^2)^{1/4})\le C\sqrt{nd^2}
	$$
	almost surely. On the other hand, note that
	\begin{align*}
		(\EE((\bY-\boldsymbol{\alpha})(\bY-\boldsymbol{\alpha})^{\top})^2)_{(k,k)(k,k)}	
		=&~\EE(\sum_{i=1}^d(S_i^2-1)^2+\sum_{i\neq j}S_i^2S_j^2)(S_k^2-1)^2\le Cd^2\\
		(\EE((\bY-\boldsymbol{\alpha})(\bY-\boldsymbol{\alpha})^{\top})^2)_{(k,k)(l,l)}
		=&~2\EE(S_k^2(S_k^2-1)S_l^2(S_l^2-1))\\
		=&~2(\kappa_4(S_k)+2)(\kappa_4(S_l)+2)\hspace{2cm}\text{if }k\neq l\\
		(\EE((\bY-\boldsymbol{\alpha})(\bY-\boldsymbol{\alpha})^{\top})^2)_{(k,k)(m,n)}=&~0
		\hspace{6.2cm}
		\text{if }m\neq n,\\
		(\EE((\bY-\boldsymbol{\alpha})(\bY-\boldsymbol{\alpha})^{\top})^2)_{(k,l)(m,n)}
		=&~Cd^2\mathbbm{1}(\{k,l\}=\{m,n\})\,
		\hspace{2.5cm}
		\text{if }k\neq l.
	\end{align*}
	Note that the $d$ rows, with indices $(k,k)$, contribute to a singular value of at most $Cd^2+C\sqrt{d}\le Cd^2$. The other rows each contribute to at most $Cd^2$.
	
	By the orthogonality of the rows $\be_i\otimes\be_j$, it then follows by a direct calculation that
	$$
	\norm*{
		\EE((\bY-\boldsymbol{\alpha})(\bY-\boldsymbol{\alpha})^{\top})^2
	}
	\le Cd^2.
	$$
	Then one has
	$$
	\|\bN_l\|=\|\EE(\bM_l-\EE\bM_l)^2\|
	\le Cd^2.
	$$
	By matrix Bernstein inequality \citep[see, e.g., Theorem 5.4.1 of ][]{vershynin2018high} we then have
	\begin{align}\label{eq:norm-trunc}
		&\norm*{
			\dfrac{1}{n}
			\sum_{l=1}^n
			(\bY_l-\boldsymbol{\alpha})(\bY_l-\boldsymbol{\alpha})^{\top}\mathbbm{1}(\|\bY_l\|\le C(nd^2)^{1/4})
			-\EE (\bY-\boldsymbol{\alpha})(\bY-\boldsymbol{\alpha})^{\top}\mathbbm{1}(\|\bY\|\le C(nd^2)^{1/4})}\nonumber\\
		\le&~ Ct_1\sqrt{\dfrac{d^2}{n}}
	\end{align}
	with probability at least $1-\exp(-Ct_1)$. 
	
	It remains to bound the operator norm of the untruncated average. In the remainder of the proof we will assume that $n\ge Cd^2$ for a large enough $C>0$. This means $nd^2\ge Cd^4$. By Lemma \ref{lem:Sk-tail-bds}, we have
	$$
	p:=
	\PP(\|\bY_l\|\ge C(nd^2)^{1/4})
	=\PP(\|\bS\|\ge C(nd^2)^{1/8})
	=\PP(\|\bS\|^2-d\ge C(nd^2)^{1/4})
	\le \dfrac{d^{2+\epsi/4}}{(nd^2)^{1+\epsi/8}}.
	$$
	Let us define $M=\{1\le l\le n: \|\bS_l\|\ge C(nd^2)^{1/8}\}$. Since $\mathbbm{1}(\|\bS_l\|\ge C(nd^2)^{1/8})$ are independent ${\rm Bernoulli}(p)$ random variables, we have 
	$$
	\EE|M|=np\le 
	\dfrac{nd^{2+\epsi/4}}{(nd^2)^{1+\epsi/8}}\le n^{-\epsi/8}.
	$$
	Moreover, 	
	$$
	\PP\left(|M|\ge C\right)
	\le \exp(-Cn^{\epsi/8})
	$$
	by the Chernoff bound for Bernoulli random variables. We have using Lemma \ref{lem:Sk-tail-bds} that for any $t>0$,
	\begin{equation}\label{eq:trunc-tail-bd}
		\PP(\max_{1\le l\le n} \|\bS_l\|^2\ge d+t)
		\le n\PP(\|\bS\|^2-d\ge t)
		\le \dfrac{nd^{2+\epsi/4}}{Ct^{4+\epsi/2}}
	\end{equation}
	Note that $\|\bY_l\|=\|\bS_l\|^2$. Thus we can bound the ``untruncated average" as
	\begin{align*}\label{eq:norm-untrunc-11}
		&~\PP\left(
		\norm*{\dfrac{1}{n}\sum_{l=1}^n
			(\bY_k-\boldsymbol{\alpha})(\bY_k-\boldsymbol{\alpha})^{\top}
			\mathbbm{1}(\|\bY_l\|\ge C(nd^2)^{1/4})
		}\ge (d/\sqrt{n}+t_2)^2
		\right)	\\
		\le &
		~\PP\left(
		\dfrac{1}{n}\cdot |M|\max\limits_{1\le l\le n}\|\bS_l\|^4
		\ge (d/\sqrt{n}+t_2)^2
		\right)\\
		\le & 	
		~\PP(|M|\ge C)+\PP(|M|\le C,\, \max\limits_{1\le k\le n}\|\bS_k\|^4
		\ge (d+C\sqrt{n}t_2)^2)
		\\
		\le &
		~\exp(-Cn^{\epsi/8})+
		n\PP(\|\bS\|^2-d\ge C\sqrt{n}t_2)
		\\
		\le&
		~ \exp(-Cn^{\epsi/8})+
		\dfrac{nd^{2+{\epsi/4}}}
		{Cn^{2+\epsi/4}t_2^{4+\epsi/2}}	\\
		\le& \exp(-Cn^{\epsi/8})
		+\dfrac{d^{2+\epsi/4}}{Cn^{1+\epsi/4}t_2^{4+{\epsi/2}}}.\numberthis
	\end{align*}
	Finally, once again using the fact that $\bY_1:=\bS_1\otimes \bS_1$, we have
	\begin{align*}\label{eq:norm-untrunc-2}
		\norm*{\EE (\bY-\boldsymbol{\alpha})(\bY-\boldsymbol{\alpha})^{\top}
			\mathbbm{1}(\|\bY\|\ge C(nd^2)^{1/4})}
		\le&
		~(\sup_{\bv:\|\bv\|=1}\EE\langle\bY,\bv\rangle^4)^{1/4}
		(\PP(\|\bS\|^2-d\ge C(nd^2)^{1/4}))^{1/2}\\
		\le&~ C
		\cdot \dfrac{1}{n^{1+\epsi/8}}.\numberthis
	\end{align*}
	Here we use Cauchy-Schwarz inequality in the first step, and the fact that for a unit vector $\bv$ we have
	\begin{align*}
		\EE \langle\bY,\bv\rangle^4
		=&~\sum_{(i,j)=1}^{d^2}\EE(S_i^4S_j^4)v_{(i,j)}^4
		+\sum_{(i,j)\neq (k,l)}\EE(S_i^2S_j^2S_k^2S_l^2)v_{(i,j)}^2v_{(k,l)}^2\\
		\le&~ C\sum_{(i,j)=1}^{d^2}v_{(i,j)}^4+C\sum_{(i,j)\neq (k,l)}v_{(i,j)}^2v_{(k,l)}^2\le C.
	\end{align*}
	Adding equations \eqref{eq:norm-trunc}, \eqref{eq:norm-untrunc-11} and \eqref{eq:norm-untrunc-2} finishes the proof.
\end{proof}


\medskip

\begin{lemma}\label{lem:eps3-bd} Let $\bX_1,\dots,\bX_n$ be $n$ independent copies of a random vector $\bX$ such that $\calL(\bX)\in\calP_{\rm ICA}(\bA;\epsi,M_1,M_2)$ for $\epsi,M_1,M_2>0$. Then for any $t>C\log d$, we have
	$$
	\max\limits_{1\le j\le d}
	~\norm*{\dfrac{1}{n}
		\displaystyle
		\sum_{l=1}^n
		\langle\bX_l,\ba_j\rangle\bX_l\circ\bX_l\circ\bX_l
		-\EE\langle\bX,\ba_j\rangle\bX\circ\bX\circ\bX}
	\le 
	C\sqrt{\dfrac{dt}{n}}
	+\dfrac{Cd^{13/8}t}{n^{7/8}}
	$$
	with probability at least $1-\exp(-t)-
	(nd)^{-\tfrac{\epsi}{8}}$.
\end{lemma}

\medskip

\begin{proof}[Proof of Lemma \ref{lem:eps3-bd}]
	Note that for any $j\in [d]$, 
	$$
	\langle\bX_l,\ba_j\rangle=\be_j^{\top}\bS_l:=S_{lj}.
	$$
	By our model assumptions, $\{S_{lj}\}$ are independent random  variables with mean zero, variance one and $\EE |S_{lj}|^{8+\eps}\le C$. 
	\begin{align*}
		&\norm*{
			\dfrac{1}{n}
			\displaystyle
			\sum_{l=1}^n
			\langle\bX_l,\ba_j\rangle\bX_l\circ\bX_l\circ\bX_l
			-
			\EE\langle\bX,\ba_j\rangle\bX\circ\bX\circ\bX
		}\\
		=&\norm*{
			\dfrac{1}{n}
			\displaystyle
			\sum_{l=1}^n
			S_{lj}
			\bS_l\circ\bS_l\circ\bS_l
			-
			\EE S_j\bS\circ\bS\circ\bS
		}\\
		\le&\norm*{
			\dfrac{1}{n}
			\displaystyle
			\sum_{l=1}^n
			S_{lj}
			(\bS_l\circ\bS_l-(\II_d))\circ\bS_l
			-
			\EE\S_j(\bS\circ\bS-\II_d)\circ\bS
		}
		+\norm*{\II_d\circ\left(\dfrac{1}{n}\sum_{l=1}^nS_{lj}\bS-\be_j\right)}.
	\end{align*}
	We bound these two terms separately. Note that the second term can be bounded using truncated matrix Bernstein inequality (see the proof of Lemma \ref{lem:eps1-bd}) to obtain
	$$
	\norm*{\II_d\circ\left(\dfrac{1}{n}\sum_{l=1}^nS_{lj}\bS-\be_j\right)}
	\le \dfrac{Cdt}{n}
	$$
	with probability at least $1-\exp(-t)-d^{-1-\epsi/8}(n)^{-\epsi/8}$.
	
	For the first term, we use a similar truncation as done in Lemma \ref{lem:varY-conc}. We define
	$$
	\bM:=S_j(\bS\otimes\bS-{\rm vec}(\II_d))\circ\bS^{\top}
	\in \RR^{d^2\times d}
	$$
	and let $\bM_l$ be independent copies of $\bM$ for $l=1,2,\dots,n$.	We write
	\begin{align*}
		&\norm*{
			\dfrac{1}{n}
			\displaystyle
			\sum_{l=1}^n
			S_{lj}
			(\bS_l\circ\bS_l-(\II_d))\circ\bS_l
			-
			\EE\S_j(\bS\circ\bS-\II_d)\circ\bS
		}\\
		\le&~
		\norm*{
			\dfrac{1}{n}
			\displaystyle
			\sum_{l=1}^n
			S_{lj}
			(\bS_l\otimes\bS_l-{\rm vec}(\II_d))\bS_l^{\top}
			-
			\EE\S_j(\bS\otimes\bS-{\rm vec}(\II_d))\circ\bS^{\top}
		}\\
		=&~
		\norm*{
			\dfrac{1}{n}\sum_{l=1}^n\bM_{l}\mathbbm{1}(\|\bM_l\|\le Cd^{13/8}n^{1/8})
			-\EE \bM\mathbbm{1}(\|\bM\|\le Cd^{13/8}n^{1/8})
		}\\
		&+
		\norm*{
			\dfrac{1}{n}
			\sum_{l=1}^n
			\bM_{l}\mathbbm{1}(\|\bM_l\|\ge Cd^{13/8}n^{1/8})
		}\\
		&+
		\norm*{
			\EE \bM\mathbbm{1}(\|\bM\|\le Cd^{13/8}n^{1/8})
		}.
	\end{align*}
	One can check by a direct calculation that
	$$
	\max
	\left\{\norm*{\EE\bM\bM^{\top}},
	\norm*{\EE\bM^{\top}\bM}
	\right\}
	\le Cd.
	$$
	Then the first term can be bounded by matrix Bernstein inequality as follows. We have
	\begin{align*}
		\norm*{
			\dfrac{1}{n}\sum_{l=1}^n\bM_{l}\mathbbm{1}(\|\bM_l\|\le Cd^{13/8}n^{1/8})
			-\EE \bM\mathbbm{1}(\|\bM\|\le Cd^{13/8}n^{1/8})
		}
		\le C\sqrt{\dfrac{dt}{n}}+\dfrac{Cd^{13/8}n^{1/8}t}{n}
	\end{align*}
	with probability at least $1-\exp(-t)$.
	
	For the second term, as before we compute $p=\PP(\|\bM\|\ge Cd^{3/4}n^{1/2})$. Note that
	\begin{align*}
		\abs*{\norm*{\bM}^2-d^3S_j^2}
		=&
		\abs*{|S_j|^2\|\bS\|^2\cdot\|(\bS\otimes \bS)-{\rm vec}(\II_d)\|^2
			-d^3S_j^2}\\
		=&
		\abs*{S_j^2\left(\sum S_k^2\right)
			\left(
			(\sum S_k^2)^2
			+d
			-2\sum S_k^2 
			\right)
			-d^3S_j^2}\\
		\le& 
		S_j^2\abs*{
			\sum_k(S_k^2-1)
		}^3.
	\end{align*}
	It follows by the moment inequalities in Lemma \ref{lem:Sk-tail-bds}, see also Remark 2 of \cite{latala1997estimation}, that
	$$
	\EE\abs*{\norm*{\bM}^2-d^3S_j^2}^{(4+\epsi/2)}
	\le C\EE\abs*{
		\sum_k(S_k^2-1)
	}^{(8+\epsi)/2}
	\le Cd^{4+\epsi/2}.
	$$
	Then by Markov's inequality, we obtain
	\begin{align*}
		\PP(\|\bM\|\ge Cd^{3/2}\cdot (nd)^{1/8} +Cn^{1/8}d^{5/8})\le \dfrac{1}{(dn)^{1+\epsi/8}}.
	\end{align*}
	Therefore if we define ${\calQ}=\{1\le l\le n: \|\bM_l\|
	\ge Cd^{3/2}\cdot (nd)^{1/8} +n^{1/8}d^{5/8}\}$, we have
	$$
	|\calQ|\le C
	$$
	with probability at least $1-\exp(-n^{\epsi/8})$, using Chernoff bounds on Binomial random variables. Moreover, it also follows by a union bound over all possible $l$ and $j$ that
	$$
	\max_{1\le j\le d}
	\max_{1\le l\le n}
	\|\bM_l\|
	\le Cd^{3/2}\cdot (nd)^{1/8} +Cn^{1/8}d^{5/8}
	$$
	with probability at least $1-(nd)^{-\epsi/8}$. Thus the second term in the deviation bound  satisfies
	\begin{align*}\label{eq:eps32-untrunc-1}
		\norm*{
			\dfrac{1}{n}\sum_{l=1}^n
			\bM_l
			\mathbbm{1}
			\left(
			\|\bM_{l}\|\ge Cn^{1/8}d^{13/8}+Cn^{1/8}d^{5/8}
			\right)}
		\le &~\dfrac{1}{n}\cdot |\calQ|\cdot 	
		\max_{1\le l\le n}\max_{1\le j\le d}
		|\bM_{l}|
		\le \dfrac{Cn^{1/8}d^{13/8}}{n}\numberthis
	\end{align*}
	with probability at least $1-(nd)^{-\epsi/8}$. The third term can be bounded by Cauchy-Schwarz inequality through an argument identical to \eqref{eq:norm-untrunc-2}, and we omit it for brevity. Taking a union bound over $j\in [d]$ in the first term (note that the other terms have been bounded uniformly in $j$), we finish the proof.
\end{proof}

\medskip

\begin{lemma}\label{lem:eps2-bd}
	Let $\bX_1,\ldots,\bX_n$ be $n$ independent copies of $\bX$ s.t.  $\calL(\bX)\in\calP_{\rm ICA}(\bA;\epsilon, M_1, M_2)$ for some $\bA\in \calO(d)$ and $\epsilon, M_1,M_2>0$. We have, for any $t>C\log d$ that
	$$
	\Delta_2:=\max_{1\le j\le d}\|(\hat{\scrM}_4^{\rm sample}(\bX)-\scrM_4(\bX))\times_3\ba_j\times_4\ba_j\|\ge Ct\sqrt{\dfrac{d}{n}}
	$$
	with probability at most $\exp(-t)+(nd)^{-\epsi/8}+n^{-\epsi/2}$.
\end{lemma}

\medskip

\begin{proof}[Proof of Lemma \ref{lem:eps2-bd}]
	Note that for any $j\in [d]$, 
	$
	\langle\bX_l,\ba_j\rangle=\be_j^{\top}\bS_l:=S_{lj}
	$.
	By our model assumptions, $\{S_{lj}\}$ are independent random  variables with mean zero, variance one and $\EE |S_{lj}|^{8+\epsi}\le C$. Let us define the vectors $\bY'_{l(j)}=S_{lj}\bS_l$.
	
	By Markov's inequality, for any $t>d^{-1/2}$,
	$$
	\PP(|S_{lj}|\ge (ndt)^{1/8})\le \left(
	ndt\right)^{-(1+\epsi/8)}.
	$$	
	Since $\{S_{l(j)}:1\le j\le d\}$ are independent,
	\begin{align*}
		\EE |\sum_{i\neq j}S_{lj}^2(S_{li}^2-1)|^{\tfrac{8+\epsi}{2}}
		=&~\EE|S_{lj}|^{8+\epsi}\EE\left(
		\abs*{\sum_{i\neq j}(S_{li}^2-1)}^{\tfrac{8+\epsi}{2}}\right)
		\le  ~ Cd^{\tfrac{8+\epsi}{4}}.
	\end{align*}	
	using \eqref{eq:mmnt-bd} for  $\alpha=(8+\epsi)$. Thus by Markov inequality, we also have, for any $t>0$, that
	$$
	\PP\left(
	\abs*{\sum_{i\neq j}S_{lj}^2(S_{li}^2-1)}\ge t
	\right)
	\le \dfrac{Cd^{(8+\epsi)/4}}{t^{(8+\epsi)/2}}.
	$$
	Thus, for any $t>d^{-1/2}$,
	\begin{align*}\label{eq:Ylj-1-nrm}
		\PP\left(\max_{1\le j\le d}\|\bY'_{l(j)}\|\ge C(ndt)^{1/4}\right)
		\le&~ 
		d^{-\tfrac{\epsi}{8}}(nt)^{-(1+\epsi/8)}
		+
		\dfrac{d}{(nt)^{(8+\epsi)/4}}.
		\numberthis
	\end{align*}
	\begin{align*}
		&\norm*{\dfrac{1}{n}\sum_{l=1}^n
			\langle\bX_l,\ba_j\rangle^2\bX_l\bX_l^{\top}
			-\EE\langle\bX_1,\ba_j\rangle^2\bX_1\bX_1^{\top}}\\	
		=&~
		\norm*{
			\dfrac{1}{n}
			\sum_{l=1}^n
			S_{lj}^2\bS_l\bS_l^{\top}
			-\EE S_{1j}^2\bS_1\bS_1^{\top}}\\
		\le&~
		\norm*{
			\dfrac{1}{n}
			\sum_{l=1}^n
			\bY'_{l(j)}(\bY'_{l(j)})^{\top}
			\mathbbm{1}\left(\|\bY'_{l(j)}\|\le \dfrac{C(nd)^{1/4}}{(\log d)}\right)
			-\EE \bY'_{1(j)}(\bY'_{1(j)})^{\top}
			\mathbbm{1}\left(\|\bY'_{1(j)}\|\le \dfrac{C(nd)^{1/4}}{(\log d)}\right)}\\
		&~+\norm*{
			\dfrac{1}{n}\sum_{l=1}^n
			\bY'_{l(j)}(\bY'_{l(j)})^{\top}
			\mathbbm{1}\left(\|\bY'_{l(j)}\|\ge \dfrac{C(nd)^{1/4}}{(\log d)}\right)
		}
		+\norm*{
			\EE \bY'_{1(j)}(\bY'_{1(j)})^{\top}
			\mathbbm{1}\left(\|\bY'_{1(j)}\|\ge \dfrac{C(nd)^{1/4}}{(\log d)}\right)
		}.
	\end{align*}
	The rest of the proof is identical to that of Lemma \ref{lem:varY-conc}. By matrix Bernstein inequality, we can bound
	\begin{align*}\label{eq:eps12-trunc}
		&~\norm*{
			\dfrac{1}{n}
			\sum_{l=1}^n
			\bY'_{l(j)}(\bY'_{l(j)})^{\top}
			\mathbbm{1}\left(\|\bY'_{l(j)}\|\le \dfrac{C(nd)^{1/4}}{(\log d)}\right)
			-\EE \bY'_{1(j)}(\bY'_{1(j)})^{\top}
			\mathbbm{1}\left(\|\bY'_{1(j)}\|\le \dfrac{C(nd)^{1/4}}{(\log d)}\right)}\\
		\le&~ \sqrt{\dfrac{Cdt}{n}}	\numberthis
	\end{align*}
	with probability at least $1-\exp(-t)$. 
	
	For the second term, as before, let us define 
	$$
	M'':=\left\{1\le l\le n:\max_{1\le j\le d}\|\bY'_{l(j)}\|\ge \dfrac{C(nd)^{1/4}}{(\log d)}\right\}.$$ 
	It can be checked by \eqref{eq:Ylj-1-nrm} that
	$$
	\EE(M'')\le (nd)^{-\epsi/8}+n^{-2/3-\epsi/4},
	\,\,
	\text{ and }
	\,\,
	\PP(M''\ge C)\le \exp(-n^{2/3})
	$$
	by Chernoff bound for binomial random variables. Following the steps of \eqref{eq:norm-untrunc-11} , we then have
	\begin{align*}\label{eq:eps12-untrunc-1}
		&~\norm*{
			\dfrac{1}{n}\sum_{l=1}^n
			\bY'_{l(j)}(\bY'_{l(j)})^{\top}
			\mathbbm{1}\left(\|\bY'_{l(j)}\|\ge \dfrac{C(nd)^{1/4}}{(\log d)}\right)}\\
		\le &~\dfrac{1}{n}\cdot |M''|\cdot 	
		\max_{1\le l\le n}\max_{1\le j\le d}
		|\bY_{l(j)}|^2
		\cdot \max_{1\le l\le n}
		\|\bS_l\|^2\\
		\le&~ \dfrac{C(nd)^{1/2}}{n}
		=C\sqrt{\dfrac{d}{n}}\numberthis
	\end{align*}
	with probability at least $1-(nd)^{-\epsi/8}-n^{-\epsi/2}$, where we use the bounds from \eqref{eq:trunc-tail-bd} and \eqref{eq:norm-untrunc-11}. The third term can be bounded by Cauchy-Schwarz inequality through an argument identical to \eqref{eq:norm-untrunc-2}. In particular,
	\begin{align*}
		&~\norm*{
			\EE \bY'_{1(j)}(\bY'_{1(j)})^{\top}
			\mathbbm{1}\left(\|\bY'_{1(j)}\|\ge \dfrac{C(nd)^{1/4}}{(\log d)}\right)
		}	\\
		\le&~
		\sup_{\bv\in \SS^{d-1}}
		\left(\EE\langle\bY'_{1(j)},\bv\rangle^4\right)^{1/2}
		\cdot \left(\PP\left(\|\bY'_{1(j)}\|\ge \dfrac{C(nd)^{1/4}}{(\log d)}\right)\right)^{1/2}\\
		\le&~
		\sup_{\bv\in \SS^{d-1}}
		\left(\EE S_{1(j)}^8\langle\bS_{1(j)},\bv\rangle^8\right)^{1/4}
		\cdot \dfrac{1}{\sqrt{n}}
		\le \dfrac{C}{\sqrt{n}}
	\end{align*} 
	using \eqref{eq:Ylj-1-nrm} with $t=\log d$ in the last step. Taking a union bound over $j\in [d]$ in the first term (note that the other terms have been bounded uniformly in $j$), we finish the proof.
\end{proof}

\medskip

\begin{lemma}\label{lem:eps1-bd} Let $\bX_1,\dots,\bX_n$ be $n$ independent copies of a random vector $\bX$ such that $\calL(\bX)\in\calP_{\rm ICA}(\bA;\epsi,M_1,M_2)$ for $\epsi,M_1,M_2>0$. Given any $\delta\in(0,1)$, we have for any $j\in \{1,\dots,d\}$, that
	$$\PP\left(
	\max\limits_{1\le j\le d}
	~\norm*{\dfrac{1}{n}
		\displaystyle
		\sum_{l=1}^n
		\langle\bX_l,\ba_j\rangle^3\bX_l
		-\EE\langle\bX_1,\ba_j\rangle^3\bX_1}
	\ge 
	C\sqrt{\dfrac{d(\log (1/\delta))}{n}}\right)~\le 
	\dfrac{\delta}{2}
	+(nd)^{-\epsi/8}+n^{-\epsi/2}.
	$$
\end{lemma}

\medskip

\begin{proof} As before, $\langle\bX_l,\ba_j\rangle=S_{lj}$. We define the random vectors 
	$
	\bY_{l(j)}:=\langle\bX_l,\ba_j\rangle^3\bS_l=S_{lj}^3\bS_l
	$. 
	Then,
	$$
	\|\bY_{l(j)}\|^2=S_{lj}^8+(d-1)S_{lj}^6+\sum_{i\neq j}S_{lj}^6(S_{li}^2-1).
	$$	
	By Markov's inequality, for any $t>d^{-1/2}$,
	$$
	\PP(|S_{lj}|\ge (ndt)^{1/8})\le \left(
	ndt\right)^{-(1+\epsi/8)}.
	$$	
	Moreover, since $\{S_{l(j)}:1\le j\le d\}$ are independent,
	\begin{align*}
		\EE |\sum_{i\neq j}S_{lj}^6(S_{li}^2-1)|^{\tfrac{8+\epsi}{6}}
		=&~\EE|S_{lj}|^{8+\epsi}\EE\left(
		\abs*{\sum_{i\neq j}(S_{li}^2-1)}^{\tfrac{8+\epsi}{6}}\right)
		\le C\abs*{\EE\left(
			\abs*{\sum_{i\neq j}(S_{li}^2-1)}^{\tfrac{8+\epsi}{3}}\right)}^{1/2}
		\le ~ Cd^{\tfrac{8+\epsi}{12}}.
	\end{align*}	
	using Cauchy-Schwarz inequality, along with \eqref{eq:mmnt-bd} for  $\alpha=2(8+\epsi)/3$. Thus by Markov inequality, we also have, for any $t>d^{-1/2}$, that
	$$
	\PP\left(
	\abs*{\sum_{i\neq j}S_{lj}^6(S_{li}^2-1)}\ge Cndt
	\right)
	\le \dfrac{Cd^{(8+\epsi)/12}}{(ndt)^{(8+\epsi)/6}}
	=\left(
	nt\right)^{-\tfrac{8+\epsi}{6}}d^{-\tfrac{8+\epsi}{12}}.
	$$
	Thus, for any $t>d^{-1/2}$,
	\begin{align*}\label{eq:Ylj-nrm}
		\PP\left(\max_{1\le j\le d}\|\bY_{l(j)}\|^2\ge Cndt\right)
		\le&~ \left(
		nt\right)^{-\tfrac{8+\epsi}{6}}
		d^{1-\tfrac{8+\epsi}{12}}+
		\left(nt\right)^{-\tfrac{8+\epsi}{8}}d^{-\tfrac{\epsi}{8}}\\
		\le&~ n^{-\tfrac{8+\epsi}{6}}d^{-\tfrac{8+7\epsi}{12}}t^{-\tfrac{8+\epsi}{6}}
		+n^{-\tfrac{8+\epsi}{8}}d^{-\tfrac{\epsi}{8}}
		t^{-\tfrac{(8+\epsi)}{8}}\numberthis
	\end{align*}
	We then use the truncation strategy to bound
	\begin{align*}
		&\max_{1\le j\le d}\norm*{\dfrac{1}{n}\sum_{l=1}^n
			\langle\bX_l,\ba_j\rangle^3\bX_l
			-\EE\langle\bX_1,\ba_j\rangle^3\bX_1}\\	
		=&~
		\max_{1\le j\le d}
		\norm*{
			\dfrac{1}{n}
			\sum_{l=1}^n
			\bY_{l(j)}
			-\EE \bY_{l(j)}}\\
		\le&~
		\max_{1\le j\le d}
		\norm*{
			\dfrac{1}{n}
			\sum_{l=1}^n
			\bY_{l(j)}
			\mathbbm{1}
			\left(\max_{1\le j\le d}\|\bY_{l(j)}\|\le \dfrac{C\sqrt{nd}}{(\log d)}\right)
			-\EE \bY_{1(j)}\mathbbm{1}
			\left(\max_{1\le j\le d}\|\bY_{1(j)}\|\le \dfrac{C\sqrt{nd}}{(\log d)}\right)}\\
		&~+\max_{1\le j\le d}\norm*{
			\dfrac{1}{n}\sum_{l=1}^n
			\bY_{lj}
			\mathbbm{1}\left(\max_{1\le j\le d}\|\bY_{l(j)}\|\ge \dfrac{C\sqrt{nd}}{(\log d)}\right)
		}
		+\max_{1\le j\le d}\norm*{
			\EE \bY_{1(j)}
			\mathbbm{1}
			\left(\|\bY_{1(j)}\|\ge \dfrac{C\sqrt{nd}}{(\log d)}\right)
		}.
	\end{align*}
	Note that ${\rm trace}({\rm Var}(\bY_{l(j)}))\le Cd$. Then, we bound the first term using a vector Bernstein inequality to get
	\begin{align*}\label{eq:eps1-trunc}
		\max_{1\le j\le d}&~
		\norm*{
		\dfrac{1}{n}
		\sum_{l=1}^n
		\bY_{l(j)}
		\mathbbm{1}
		\left(\max_{1\le j\le d}\|\bY_{l(j)}\|\le \dfrac{C\sqrt{nd}}{(\log d)}\right)
		-\EE \bY_{1(j)}\mathbbm{1}
		\left(\max_{1\le j\le d}\|\bY_{1(j)}\|\le \dfrac{C\sqrt{nd}}{(\log d)}\right)}\\
		\le&~ \sqrt{\dfrac{Cd(t+t^2/(\log d))}{n}}\numberthis
	\end{align*}
	with probability at least $1-\exp(-t)$. Following the strategy of \eqref{eq:norm-untrunc-11}, we define $M':=\{1\le l\le n:\max_{1\le j\le d}\|\bY_{l(j)}\|^2\ge Cnd/(\log d)^2\}$. Then we have, using \eqref{eq:Ylj-nrm} for $t=\dfrac{1}{(\log d)^2}$ that
	$$
	\EE|M'|\le d^{-2/3}+n^{-\epsi/8}\,\,
	\text{ and }\,\,
	\PP(|M'|\ge C)\le \exp(-Cd^{2/3}\wedge n^{\epsi/8}),
	$$
	by Chernoff bounds for Binomial random variables. As previously done in Lemmas \ref{lem:varY-conc} and \ref{lem:eps2-bd}, we bound the second term as 
	\begin{align*}
		\norm*{
			\dfrac{1}{n}\sum_{l=1}^n
			\bY_{lj}
			\mathbbm{1}(\|\bY_{l(j)}\|\ge C\sqrt{nd}/(\log d))
		}
		\le&~ \dfrac{1}{n}\cdot |M'|\cdot \max_{1\le l\le n}\max_{1\le j\le d}\|\bY_{l(j)}\|
		\le~ \dfrac{C\sqrt{nd}}{n}
	\end{align*}
	with probability at least $1-\exp(-d^{2/3}\wedge n^{\epsi/8})-n^{-\tfrac{1}{3}}d^{-\tfrac{8+7\epsi}{12}}
	+n^{-\tfrac{\epsi}{8}}d^{-\tfrac{\epsi}{8}}$, taking a union bound over $1\le l\le n$ in \eqref{eq:Ylj-nrm} with $t=1$ in the last step. Finally by Cauchy-Schwarz inequality, we have
	\begin{align*}
		&~\norm*{
			\EE \bY_{1(j)}
			\mathbbm{1}(\|\bY_{1(j)}\|\ge C\sqrt{nd}/(\log d))
		}\\
		\le&~
		\sup_{\bv\in \SS^{d-1}}\left(\EE S_{1(j)}^6\langle \bS_1,\bv\rangle^2\right)^{1/2}
		\cdot
		\left(\PP(\|\bY_{1(j)}\|\ge \dfrac{C\sqrt{nd}}{\log d}\right)^{1/2}
		\le  \dfrac{C}{\sqrt{n}}
	\end{align*}
	using \eqref{eq:Ylj-nrm} with $t=\log d$ in the last step. Taking $t=4(\log (1/\delta))$ in \eqref{eq:eps1-trunc}  finishes the proof. 
\end{proof}

\medskip

\begin{lemma}\label{lem:Ejnorm}
	Let $\bX_1,\dots,\bX_n$ be $n$ independent copies of a random vector $\bX$ such that $\calL(\bX)\in\calP_{\rm ICA}(\bA;\epsi,M_1,M_2)$ for $M_1,M_2>0$. Let us define
	$$
	\scrE=\dfrac{1}{n}\sum_{l=1}^n\bX_l^{\circ 4}-\EE(\bX^{\circ 4})
	$$
	and the matrix $\bE=[\scrE\times_{1,2,3}\ba_1\,\scrE\times_{1,2,3}\ba_2\,
	\dots \scrE\times_{1,2,3}\ba_d]$.
	Then
	$$
	\PP\left(\|\bE\|\le C(\sqrt{t}+1)\sqrt{\dfrac{d+d^2\mathbbm{1}(\epsi<4)}{n}}\right)\le \exp(-t)+n^{-\epsi/8}+n^{-(\epsi-4)/12}\mathbbm{1}(\epsi>4).
	$$
\end{lemma}

\begin{proof}[Proof of Lemma \ref{lem:Ejnorm}]
	The proof follows the same truncation strategy from Lemmas \ref{lem:tens-conc}, \ref{lem:eps2-bd} and \ref{lem:eps1-bd}. Note that
	$$
	\scrE\times_{1,2,3}\ba_j
	=
	\bA
	\left(\dfrac{1}{n}\sum_{l=1}^nS_{lj}^3\bS_l-\EE S_j^3\bS\right).
	$$
	Writing
	$$
	\bY=(S_1^3\,S_2^3\,\dots\,S_d^3)
	$$
	and letting $\bY_1,\dots,\bY_n$ denote independent copies of $\bY$, we have 
	\begin{align*}
		\|\bE\|
		=&~\norm*{
			\dfrac{1}{n}\sum_{l=1}^n\bS_l\bY_l^{\top}
			-\EE\bS\bY^{\top}	
		}.
	\end{align*}
	Note that
	$
	\|\bS\bY^{\top}\|^2=\displaystyle\sum_{i=1}^dS_i^2\sum_{j=1}^dS_j^6
	$,
	so that
	\begin{align*}
		\PP\left(
		\|\bS\bY^{\top}\|
		\ge t
		\right)
		\le \dfrac{\EE\left(\|\bS\bY^{\top}\|^2\right)^{(8+\epsi)/8}}{t^{(8+\epsi)/4}}
		\le 
		\dfrac{d^{(8+\epsi)/4}}{t^{(8+\epsi)/4}}
	\end{align*}
	by Theorem 1 of \cite{latala1997estimation}. Thus,
	$$
	\PP(\max\{\|\bS\bY^{\top}\|\}\le C(nd^2t)^{1/2})
	\le \dfrac{1}{n^{\eps/8}t^{(8+\epsi)/8}}.
	$$
	On the other hand if $\epsi\ge 4$, we have
	$$
	\PP\left(\sum (S_i^6-\EE S_i^6)\ge t\right)
	\le \dfrac{d^{(8+\eps)/12}}{t^{(8+\eps)/6}}
	$$
	by Corollary 2 of \cite{latala1997estimation}, which implies that
	$$
	\PP(\max\{\|\bY\|\}\le C(ndt)^{1/4})
	\le \dfrac{1}{n^{(\epsi-4)/12}t^{(8+\epsi)/12}}.
	$$
	A similar bound over $\|\bS\|$ holds and this implies,
	$$
	\PP(\max\{\|\bS\bY^{\top}\|\}\le C(ndt)^{1/2})
	\le \dfrac{1}{n^{(\epsi-4)/12}t^{(8+\epsi)/12}}.
	$$
	
	Next, as done in the previous lemmas, we write
	\begin{align*}
		&\norm*{
			\dfrac{1}{n}\sum_{l=1}^n\bS_l\bY_l^{\top}
			-\EE\bS\bY^{\top}	
		}\\
		\le&~
		\norm*{\dfrac{1}{n}\sum_{l=1}^n\bS_l\bY_l^{\top}
			\mathbbm{1}(\max\{\|\bS_l\bY_l^{\top}\|\}\le \eta)
			-\EE\bS\bY^{\top}\mathbbm{1}(\max\{\|\bS_l\bY_l^{\top}\|\}\le \eta)
		}\\
		&+\norm*{
			\dfrac{1}{n}
			\sum_{l=1}^n\bS_l\bY_l^{\top}
			\mathbbm{1}(\max\{\|\bS_l\bY_l^{\top}\|\}\ge \eta)
		}\\
		&+\norm*{
			\EE 
			\bS_l\bY_l^{\top}
			\mathbbm{1}(\max\{\|\bS_l\bY_l^{\top}\|\}\ge \eta)
		}.
	\end{align*}
	Note that
	\begin{align*}
		\EE(\bS\bY^{\top})(\bS\bY^{\top})^{\top}
		=&\EE\left(
		(\sum S_j^6)\bS\bS^{\top}
		\right)
		={\rm diag}
		\left(
		\EE S_1^8+\sum_{j\neq 1} S_j^6
		\,\dots\,
		\EE S_d^8+\sum_{j\neq d}S_j^6
		\right).
	\end{align*}
	By a similar calculation on $\EE (\bS\bY^{\top})^{\top}(\bS\bY^{\top})$, it follows that
	\begin{align*}
		\max\left\{
		\norm*{\EE(\bS\bY^{\top})(\bS\bY^{\top})^{\top}},
		\norm*{\EE(\bS\bY^{\top})^{\top}(\bS\bY^{\top})}
		\right\}
		\le Cd.
	\end{align*}
	Thus by matrix Bernstein inequality,
	\begin{align*}
		\norm*{\dfrac{1}{n}\sum_{l=1}^n\bS_l\bY_l^{\top}
			\mathbbm{1}(\max\{\|\bS_l\bY_l^{\top}\|\}\le \eta)
			-\EE\bS\bY^{\top}\mathbbm{1}(\max\{\|\bS_l\bY_l^{\top}\|\}\le \eta)
		}
		\le \sqrt{\dfrac{Cdt}{n}}+\dfrac{Ct\eta}{n}
	\end{align*}
	with probability at least $1-\exp(-t)$. For the second term, as before we define
	$$
	\tilde{M}:=\{1\le l\le n:\max\{\|\bS_l\bY_l^{\top}\|\}\ge \eta\}.
	$$
	It follows that
	$$
	\EE|\tilde{M}|\le n^{-\epsi/8}
	\text{ and }
	\PP(|\tilde{M}|\ge C)\le \exp(-n^{\epsi/8})
	\text{ if }
	\begin{cases}
		\eta=C\sqrt{nd^2}\text{ and }\epsi>0\\
		\eta=C\sqrt{nd}\text{ and }\epsi\ge 4.
	\end{cases}
	$$
	by Chernoff bounds. Hence the second term can be bounded by
	\begin{align*}
		\norm*{\dfrac{1}{n}\sum_{l=1}^n\bS_l\bY_l^{\top}
			\mathbbm{1}(\max\{\|\bS_l\bY_l^{\top}\|\}\ge C\eta}
		\le\dfrac{1}{n}\cdot |\tilde{M}|\cdot
		\max_{1\le l\le n}\max\{\|\bS_l\bY_l^{\top}\|\}
		\le \dfrac{C\eta}{n}
	\end{align*}
	with probability at least $1-n^{-\epsi/8}-n^{-(\epsi-4)/12}\mathbbm{1}(\epsi\ge 4)$. The third term can be bounded by Cauchy-Schwarz inequality as done in Lemmas \ref{lem:eps2-bd} and \ref{lem:eps1-bd}.
\end{proof}

\medskip 

\begin{lemma}\label{lem:asy} Let $\pi$ be the permutation from Theorem \ref{th:comp-rate}. We define the $d\times (d-1)$ matrix $\bA_{j,\perp}=[\ba_1\,\ba_2\,\dots\,\ba_{j-1}\,\ba_{j+1}\,\dots\,\ba_d]$, i.e.,  with all the columns of $\bA$ except $\ba_j$. We then have
	$$
	\langle\hat{\ba}_{\pi(j)},\ba_j\rangle^2-1+
	\dfrac{1}{\kappa_j^2}\cdot(\bA_{j,\perp}^{\top}\scrE\times_{2,3,4}\ba_j )^{\top}(\bA_{j,\perp}^{\top}\scrE\times_{2,3,4}\ba_j)
	\le \left(\dfrac{Cd(\log d)}{n}\right)^{3/2}
	$$
	and 
	$$
	\sup_{\bu\in\SS^{d-1},\bu\perp\ba_j}
	\abs*{
		\langle\bu,\hat{\ba}_{\pi(j)}\rangle
		-\dfrac{\kappa_j+\scrE\times_{1,2,3,4}\ba_j}{\kappa_j^2}
		\left(
		\bu^{\top}\scrE\times_{2,3,4}\ba_j
		+3\bu^{\top}(\scrE\times_{3,4}\ba_j)^2\ba_j 
		\right)	
	}	\le \dfrac{Cd^2(\log d)^{2}}{n^{3/2}}.
	$$
	for all $1\le j\le d$, with probability at least $1-d^{-3}-n^{-\epsi/8}$.	
\end{lemma}

\medskip

\begin{proof}[Proof of Lemma \ref{lem:asy}]
	We assume without loss of generality that $\pi=Id$. Note that we use the power iteration estimator $\hat{\ba}_j:=\hat{\ba}_{j,[T+1]}$ for some $T>C\log d$. Thus, $\check{\ba}_j:=\hat{\ba}_{j,[T]}$ satisfy
	$$
	\sin\angle\left(\check{\ba}_j,\ba_j\right)\le \sqrt{\dfrac{Cd(\log d)}{n}}
	$$ 
	by Theorem \ref{th:comp-rate}. In addition to the error quantities used earlier, we also define
	$$
	\eps_0=\max_j\scrE\times_{1,2,3,4}\ba_j=\max_j
	\abs*{\dfrac{1}{n}\sum_{l=1}^n(S_{lj}^4-\EE S_{1j}^4)}.
	$$
	By \eqref{eq:mmnt-bd}, we have that
	$$
	\PP\left(\eps_0\ge \sqrt{\dfrac{Cd}{n(\log d)}}\right)\le 
	\dfrac{(\log d)^{1+\epsi/8}}{d^{\epsi/8}}.
	$$	
	Let us define the event
	\begin{equation}\label{eq:tens-conc-bds-2}
		\calA\equiv
		\left(
		\eps_0\le \sqrt{\dfrac{Cd}{n(\log d)}},\,\,
		\eps_1\le \sqrt{\dfrac{Cd(\log d)}{n}},\,\,
		\eps_2\le \sqrt{\dfrac{Cd(\log d)}{n}},\,\,
		\|\scrE\|\le \sqrt{\dfrac{Cd^2(\log d)^2}{n}}\right)
	\end{equation}
	By Lemmas \ref{lem:tens-conc}, \ref{lem:eps2-bd} and \ref{lem:eps1-bd}, we have that the event $\calA$ holds with probability at least $1-d^{-3}-n^{-\epsi/8}$.	
	
	By the definition of the power iteration estimator, $\hat{\ba}_j$ is the top eigenvector of the symmetric matrix
	\begin{align*}\label{eq:defG}
		\bM:=\left((\scrT-\scrM_0)\times_{2,3,4}\check{\ba}_j\right)
		\left((\scrT-\scrM_0)\times_{2,3,4}\check{\ba}_j\right)^{\top}
		=&~\hat{\kappa}_j^2\ba_j\left(\ba_j\right)^\top
		+\bY\bY^{\top}
		+\hat{\kappa}_j\ba_j\bY^{\top}
		+\hat{\kappa}_j\bY\left(\ba_j\right)^{\top}\\
		=:&~\hat{\kappa}_j^2\ba_j\left(\ba_j\right)^\top+\bG
		\numberthis
	\end{align*}
	where $\hat{\kappa}_j:=\kappa_j(\langle\ba_j,\check{\ba}_j\rangle)^3$
	and 
	\begin{equation}\label{eq:defY}
		\bY=\sum_{i\neq j}\kappa_i
		\left(\langle\ba_i,\,\check{\ba}_j\rangle\right)^3\ba_i
		+\scrE\times_{2,3,4}\check{\ba}_j
		=:\bA_{j,\perp}\tilde{\boldsymbol{\kappa}}+\hat{\bE},
	\end{equation}
	where $\tilde{\boldsymbol{\kappa}}\in \RR^{d-1}$ has entries $\kappa_i\left(\langle\ba_i,\,\check{\ba}_j\rangle\right)^3$, and 
	$\hat{\bE}:=\scrE\times_{2,3,4}\check{\ba}_j$. 
	
	By Theorem \ref{th:comp-rate}, we obtain
	\begin{equation}\label{eq:kapp-hat}
		\kappa_j\left(1-\dfrac{Cd(\log d)}{n}\right)^3
		\le\kappa_j|\langle\check{\ba}_j,\ba_j\rangle|^3
		=|\hat{\kappa}_j|\le \kappa_{\max}.
	\end{equation}
	On the other hand, under the event $\calA$,
	\begin{align*}
		\|\tilde{\boldsymbol{\kappa}}\|\le\kappa_{\max}
		\max\limits_{i\neq j}
		\left(\langle\ba_i,\check{\ba}_j\rangle\right)^2
		\cdot \sin\angle
		\left(\check{\ba}_j,\ba_j\right)
		\le \left(\dfrac{Cd(\log d)}{n}\right)^{3/2}
	\end{align*}
	by Theorem \ref{th:comp-rate}. Following the steps of  \eqref{eq:piter-errbd}, we next have, under the event $\calA$,
	\begin{align*}\label{eq:asy-err-bd}
		\|\hat{\bE}\|
		=\|\scrE\times_{2,3,4}\check{\ba}_j\|
		\le&~\|\scrE\times_{2,3,4}\ba_j\|
		+C\sin\angle\left(\check{\ba}_j,\ba_j\right)\cdot \eps_2
		~+C\sin\angle\left(\check{\ba}_j,\ba_j\right)^2\cdot\|\scrE\|\\
		\le&~\sqrt{\dfrac{Cd(\log d)}{n}}\numberthis
	\end{align*}
	by \eqref{eq:def-eps12} and Lemmas \ref{lem:tens-conc}, \ref{lem:eps2-bd} and \ref{lem:eps1-bd}. We then have by equations ~\eqref{eq:defG} and \eqref{eq:defY} that, under $\calA$,
	\begin{equation}\label{eq:YGnorm-bd}
		\max\{\|\bG\|,\,\|\bY\|\}\le \sqrt{\dfrac{Cd(\log d)}{n}}.
	\end{equation}

	Using resolvent based series expansion of projection matrices, we have the following expression. See Theorem 1 from \cite{xia2021normal} and Lemma 1 of \cite{koltchinskii2016asymptotics}. We use the notation $\bP_{j,\perp}:=\bA_{j,\perp}\bA_{j,\perp}^{\top}$.
	\begin{align*}\label{eq:proj-expr}
		\hat{\ba}_j\hat{\ba}_j^{\top}-\ba_j\ba_j^{\top}
		=&~\dfrac{1}{\hat{\kappa}_j^2}
		\bP_{j,\perp}
		\bG\ba_j\ba_j^{\top}+
		\dfrac{1}{\hat{\kappa}_j^2}
		\ba_j\ba_j^{\top}
		\bG
		\bP_{j,\perp}\\
		&~
		+\dfrac{1}{\hat{\kappa}_j^4}
		\left(
		\ba_j\ba_j^{\top}\bG\bP_{j,\perp}\bG\bP_{j,\perp}
		+\bP_{j,\perp}\bG\ba_j\ba_j^{\top}\bG\bP_{j,\perp}
		+\bP_{j,\perp}\bG\bP_{j,\perp}\bG\ba_j\ba_j^{\top}
		\right)\\
		&~-
		\dfrac{1}{\hat{\kappa}_j^4}
		\left(
		\bP_{j,\perp}\bG\ba_j\ba_j^{\top}\bG\ba_j\ba_j^{\top}
		+\ba_j\ba_j^{\top}\bG\bP_{j,\perp}\bG\ba_j\ba_j^{\top}
		+\ba_j\ba_j^{\top}\bG\ba_j\ba_j^{\top}\bG\bP_{j,\perp}
		\right)\\
		&~+R_3(\bG).\numberthis
	\end{align*}
	Moreover, $\|R_3(\bG)\|\le \dfrac{C\|\bG\|^3}{\hat{\kappa}_j^6}\le \left(\dfrac{Cd(\log d)}{n}\right)^{3/2}$ under the event $\calA$.
	
	To derive an asymptotic expression for linear forms $\langle\bu,\hat{\ba}_j\rangle$ we make two cases.
	
	\paragraph{Case 1:~$\bu=\ba_j$.} We pre and post multiply the matrix valued expressions in \eqref{eq:proj-expr} by $\ba_j$. Then,
	\begin{align*}
		\langle\hat{\ba}_j,\ba_j\rangle^2-1
		=&~-\dfrac{1}{\hat{\kappa}_j^4}\cdot\ba_j^{\top}\bG\bP_{j,\perp}\bG\ba+\ba_j^{\top}R_3(\bG)\ba_j\\
		=&~-\dfrac{1}{\hat{\kappa}_j^4}\cdot(\bA_{j,\perp}^{\top}\bG\ba_j)^{\top}(\bA_{j,\perp}^{\top}\bG\ba_j)
		+\ba_j^{\top}R_3(\bG)\ba_j.
	\end{align*}
	By \eqref{eq:defG} and \eqref{eq:defY},
	\begin{align*}
		\bA_{j,\perp}^{\top}\bG\ba_j
		=&~(\hat{\kappa}_j+\ba_j^{\top}\hat{\bE})
		(\tilde{\boldsymbol{\kappa}}+\bA_{j,\perp}^{\top}\hat{\bE})\\
		=&~\hat{\kappa}_j\bA_{j,\perp}^{\top}\scrE\times_{2,3,4}\ba_j
		+\hat{\kappa}_j\cdot \tilde{\boldsymbol{\kappa}}
		+(\ba_j^{\top}\hat{\bE})(\tilde{\boldsymbol{\kappa}}+\bA_{j,\perp}^{\top}\hat{\bE})
		+\hat{\kappa}_j
		\bA_{j,\perp}^{\top}
		(\hat{\bE}-\scrE\times_{2,3,4}\ba_j)
	\end{align*}
	and hence, under the event $\calA$,
	\begin{align*}
		&~\norm*{
			\bA_{j,\perp}^{\top}\bG\ba_j-
			\hat{\kappa}_j\bA_{j,\perp}^{\top}\scrE\times_{2,3,4}\ba_j }\\
		\le&~
		\|\hat{\bE}\|^2
		+(\kappa_{\max}+\|\hat{\bE}\|)(\|\tilde{\boldsymbol{\kappa}}\|
		+C\sin\angle\left(\check{\ba}_j,\ba_j\right)\eps_2
		+C\sin\angle\left(\check{\ba}_j,\ba_j\right)^2\|\scrE\|)\\
		\le&~\dfrac{Cd(\log d)}{n}
	\end{align*}
	by \eqref{eq:asy-err-bd}, along with Lemmas \ref{lem:tens-conc}, \ref{lem:eps2-bd} and Theorem \ref{th:comp-rate}. This implies that under the event $\calA$,
	$$
	\|\bA_{j,\perp}^{\top}\bG\ba_j\|\le |\hat{\kappa}_j|\cdot\eps_1+
	\dfrac{Cd(\log d)}{n}
	\le \sqrt{\dfrac{Cd(\log d)}{n}}
	$$
	by \eqref{eq:def-eps12}, \eqref{eq:kapp-hat} and Theorem \ref{th:comp-rate}. We then have, using the bounds on $\hat{\kappa}_j$, $\|\bG\|$ and $\|R_3(\bG)\|$, that
	\begin{align*}\label{eq:asy-aj}
		&~
		\abs*{\langle\hat{\ba}_j,\ba_j\rangle^2-1+
			\dfrac{1}{\kappa_j^2}\cdot(\bA_{j,\perp}^{\top}\scrE\times_{2,3,4}\ba_j )^{\top}(\bA_{j,\perp}^{\top}\scrE\times_{2,3,4}\ba_j)}\\
		\le &~
		\dfrac{1}{\kappa_j^4}
		\norm*{
			\bA_{j,\perp}^{\top}\bG\ba_j-
			\hat{\kappa}_j\bA_{j,\perp}^{\top}\scrE\times_{2,3,4}\ba_j }
		\cdot \|\bA_{j,\perp}^{\top}\bG\ba_j\|
		+\ba_j^{\top}R_3(\bG)\ba_j
		\\
		\le&~
		\dfrac{Cd(\log d)}{n}\cdot\|\bA_{j,\perp}^{\top}\bG\ba_j\|
		+\|R_3(\bG)\|\\
		\le&~\dfrac{Cd^{3/2}(\log d)^{3/2}}{n^{3/2}}.\numberthis
	\end{align*}
	
	\paragraph{Case 2: $\bu\in \SS^{d-1},\bu\perp\ba_j$.}  We pre and post multiply the matrix valued expressions in \eqref{eq:proj-expr} by $\bu$ and $\ba_j$ respectively. Then,
	\begin{align*}
		\langle\bu,\hat{\ba}_j\rangle\langle\hat{\ba}_j,\ba_j\rangle
		=&~\dfrac{1}{\hat{\kappa}_j^2}\bu^{\top}\bG\ba_j 
		+\dfrac{1}{\hat{\kappa}_j^4}
		\left(\bu^{\top}\bG\bP_{j,\perp}\bG\ba_j
		-\bu^{\top}\bG\ba_j\ba_j^{\top}\bG\ba_j 
		\right)\\
		=&~\dfrac{1}{\hat{\kappa}_j^2}\bu^{\top}\bG\ba_j 
		+\dfrac{1}{\hat{\kappa}_j^4}
		\left(\bu^{\top}\bG\bA_{j,\perp}\bA_{j,\perp}^{\top}\bG\ba_j
		-\bu^{\top}\bG\ba_j\ba_j^{\top}\bG\ba_j 
		\right).
	\end{align*}
	As before we use \eqref{eq:defG} and \eqref{eq:defY} to write	
	$$
	\bu^{\top}\bG\ba_j=\bu^{\top}\bY\bY^{\top}\ba_j+\hat{\kappa}_j\bu^{\top}\bY
	=
	\bu^{\top}(\bA_{j,\perp}\tilde{\boldsymbol{\kappa}}+\hat{\bE})
	(\ba_j^{\top}\hat{\bE}+\hat{\kappa}_j).
	$$	
	We will write, following the steps of \eqref{eq:piter-errbd}, that
	\begin{align*}
		\hat{\bE}
		=&~\scrE\times_{2,3,4}\check{\ba}_j
		=~\scrE\times_{2,3,4}\ba_j+3(\scrE\times_{3,4}\ba_j)^2\ba_j 
		+O\left(\eps_1^2(\eps_2+\|\scrE\|)\right). 
	\end{align*}
	Under the event $\calA$, we obtain
	\begin{align*}
		\abs*{\bu^{\top}(\bA_{j,\perp}\tilde{\boldsymbol{\kappa}}+\hat{\bE})
			-\bu^{\top}\scrE\times_{2,3,4}\ba_j
			-3\bu^{\top}(\scrE\times_{3,4}\ba_j)^2\ba_j 
		}
		\le&~\|\tilde{\boldsymbol{\kappa}}\|+C\eps_1^2(\eps_2+\|\scrE\|)
		\le~\dfrac{Cd^2(\log d)^2}{n^{3/2}}.
	\end{align*}
	Thus, under $\calA$,
	\begin{align*}\label{eq:asy-uGa}
		&\abs*{
			\bu^{\top}\bG\ba_j 
			-(\kappa_j+\scrE\times_{1,2,3,4}\ba_j)
			\left(
			\bu^{\top}\scrE\times_{2,3,4}\ba_j
			+3\bu^{\top}(\scrE\times_{3,4}\ba_j)^2\ba_j 
			\right)
		}\\
		\le&~ 
		(\|\hat{\kappa}_j-\kappa_j\|
		+
		|\ba_j^{\top}\hat{\bE}-\scrE\times_{1,2,3,4}\ba_j|)
		\|\bu^{\top}(\bA_{j,\perp}\tilde{\boldsymbol{\kappa}}+\hat{\bE})\|
		+(\hat{\kappa}_j+\ba^{\top}\hat{\bE})
		\cdot \dfrac{Cd^2(\log d)^2}{n^{3/2}}\\
		\le&~
		\dfrac{Cd(\log d)\eps_2}{n}+\dfrac{Cd^2(\log d)^2}{n^{3/2}}
		\le~\dfrac{Cd^2(\log d)^2}{n^{3/2}}.\numberthis
	\end{align*}
	Similarly by \eqref{eq:YGnorm-bd}, under the event $\calA$,
	$$
	\|\bu^{\top}\bG\bA_{j,\perp}\|
	=\|\bu^{\top}\bY\bY^{\top}\bA_{j,\perp}\|
	\le \|\bY\|^2\le \dfrac{Cd(\log d)}{n}.
	$$
	Finally, under the event $\calA$,
	$$
	\ba_j^{\top}\bG\ba_j
	=(\ba_j^{\top}\hat{\bE})^2+2\hat{\kappa}_j\ba_j^{\top}\hat{\bE}
	\le \|\scrE\times_{1,2,3,4}\ba_j\|+\dfrac{Cd(\log d)}{n}
	\le \dfrac{Cd(\log d)}{n}.
	$$
	Plugging in the bounds into the expression for $\langle\bu,\ba_j\rangle$ we thus have
	\begin{align*}
		\abs*{
			\langle\bu,\hat{\ba}_j\rangle\langle\hat{\ba}_j,\ba_j\rangle
			-\dfrac{1}{\hat{\kappa}_j^2}\bu^{\top}\bG\ba_j}
		\le&~ \dfrac{1}{\hat{\kappa}_j^4}\cdot
		\left(\|\bu^{\top}\bG\bA_{j,\perp}\|\|\bA_{j,\perp}^{\top}\bG\ba_j\|
		+|\bu^{\top}\bG\ba_j|\cdot|\ba_j^{\top}\bG\ba_j|
		\right)\\
		\le&~ \dfrac{2\|\bG\|}{\kappa_j^4}
		\cdot \dfrac{Cd(\log d)}{n}
		=\left(\dfrac{Cd(\log d)}{n}\right)^{3/2}.
	\end{align*}
	Finally using \eqref{eq:asy-uGa} we have
	\begin{align*}
		\abs*{
			\langle\bu,\hat{\ba}_j\rangle\langle\hat{\ba}_j,\ba_j\rangle
			-\dfrac{\kappa_j+\scrE\times_{1,2,3,4}\ba_j}{\hat{\kappa}_j^2}
			\left(
			\bu^{\top}\scrE\times_{2,3,4}\ba_j
			+3\bu^{\top}(\scrE\times_{3,4}\ba_j)^2\ba_j 
			\right)	
		}	\le \dfrac{Cd^2(\log d)^{2}}{n^{3/2}}.
	\end{align*}
	Note that by Theorem \ref{th:comp-rate},
	$
	\langle\hat{\ba}_j,\ba_j\rangle\ge 1-\tfrac{Cd(\log d)}{n},
	$
	and thus, also using \eqref{eq:kapp-hat}, we have
	\begin{equation}\label{eq:asy-ajperp}
		\abs*{
			\langle\bu,\hat{\ba}_j\rangle
			-\dfrac{\kappa_j+\scrE\times_{1,2,3,4}\ba_j}{\kappa_j^2}
			\left(
			\bu^{\top}\scrE\times_{2,3,4}\ba_j
			+3\bu^{\top}(\scrE\times_{3,4}\ba_j)^2\ba_j 
			\right)	
		}	\le \dfrac{Cd^2(\log d)^{2}}{n^{3/2}}.
	\end{equation}
	This finishes the proof.
\end{proof}

\medskip

\begin{lemma}\label{lem:Sk-tail-bds}
	If $\bS\in \RR^d$ has independent random entries $S_j$ satisfying $\EE S_j=0$, $\EE S_j^2=1$ and $\EE |S_j|^{\alpha}\le C$ for some $\alpha>4$, then we have
	$$
	\mathbb{P}\left(\abs*{\|\bS\|^2-d}\ge t\right)\le 
	\dfrac{\max\{
		(16d)^{\alpha/4},
		d\max_j\EE|2S_j|^{\alpha}
		\}}{t^{\alpha/2}}.
	$$
\end{lemma}

\begin{proof}[Proof of Lemma \ref{lem:Sk-tail-bds}]
	Note that 
	$$
	\|\bS\|^2-d=\sum_{j=1}^dY_j
	$$
	where $Y_j=S_j^2-1$ are independent mean zero random variables. By symmetrization, followed by an application of Rosenthal inequalities (see Corollary 2 and Remark 2 of \cite{latala1997estimation}), we have
	\begin{equation}\label{eq:mmnt-bd}
		\EE\left(\abs*{\sum_{j=1}^dY_j}\right)^{\alpha/2}\le 
		2^{\alpha}\max\left\{
		\left(\sum_j\EE Y_j^2\right)^{\alpha/4},\,\sum_{j}\EE|Y_j|^{\alpha/2}
		\right\}
		\le 2^{\alpha}C\max\{
		d^{\alpha/4},
		d\max_j\EE|S_j|^{\alpha}
		\}.
	\end{equation}
	The conclusion then follows by Markov's inequality.
\end{proof}	

\medskip

\begin{lemma}\label{lem:lin-comb-mmt}
	For a random vector $\bS\in \RR^d$ whose elements are independent random variables $S_k$ satisfying $\EE S_k=0$, $\EE S_k^2=1$ and $\EE |S_k|^{\alpha}\le C$ for some $\alpha,C>0$, we have
	$$
	\sup_{\bu\in \SS^{d-1}}
	\EE|\langle\bS,\bu\rangle|
	\le 2^{\alpha}\max\EE|S_k|^{\alpha}
	\le C.
	$$
\end{lemma}

\begin{proof}For any $\bu\in \SS^{d-1}$, we have
	\begin{align*}
		\EE|\langle \bS,\bu\rangle|^{\alpha}
		=&~\EE\abs*{\sum_kS_ku_k}^{\alpha}
		\le C\EE\abs*{\sum_k \eps_kS_ku_k}^{\alpha}\\
		\le&~
		2^{\alpha}
		\max\left\{
		\left(\EE(\sum_k S_ku_k)^2\right)^{\alpha/2},\,
		\sum_{k=1}^d
		\EE|S_ku_k|^{\alpha}
		\right\}\\
		\le &~2^{\alpha}
		\max\left\{
		\left(\sum_k u_k^2\right)^{\alpha/2},\,
		\sum_{k=1}^d
		\EE|S_ku_k|^{\alpha}
		\right\}
		\le 2^{\alpha}\max\EE|S_k|^{\alpha}
	\end{align*}
	In the above we have used $\eps_k$ to denote independent Rademacher random variables, and Khintchine inequalities, conditional on $S_k$, for the first inequality. The second inequality uses Rosenthal inequalities on the symmetrized random variables $\eps_kS_k$. The last step uses moment assumptions on $S_k$. See Remark 2 of \cite{latala1997estimation} for more details.
\end{proof}

\end{document}